\documentclass[envcountsame,envcountchap]{svmono}

\usepackage{pgfplots}
\pgfplotsset{compat=1.17}
\usepackage{color}
\usepackage{amsmath}
\usepackage{amssymb}
\usepackage{amsthm}
\usepackage{bbm}
\usepackage{enumitem}
\usepackage{imakeidx}
\usepackage{mathrsfs}
\usepackage{mathtools}
\usepackage{microtype}
\usepackage{eucal}
\usepackage{hyperref}

\usepackage{makeidx}          
\usepackage{graphicx}         
\usepackage{multicol}         
\usepackage[bottom]{footmisc} 

\usepackage{tikz}

\hypersetup{citecolor=violet, colorlinks=true, linkcolor=blue}

\makeindex             
\smartqed

\newtheorem{assumption}[theorem]{Assumption}

\usepackage{fix-cm}
\makeatletter
\newcommand\semiHuge{\@setfontsize\semiHuge{22.72}{27.38}}

\def\curv{\mathop{\mathrm{curv}}}
\def\peri{\mathop{\mathrm{peri}}}
\def\cov{\mathop{\mathrm{cov}}}
\def\var{\mathop{\mathrm{var}}}
\definecolor{MIT}{cmyk}{.24, 1.00, .78, .17}

\newcommand{\unif}{\mathsf{Unif}}

\newcommand{\bs}{b^\star}	
\newcommand{\supp}{\operatorname{supp}}

\newcommand{\bcdot}{\,\begin{picture}(-1,1)(-1,-3)\circle*{3}\end{picture}\ \,}

\newcommand{\cA}{\mathcal{A}}

\newcommand{\cD}{\mathcal{D}}

\newcommand{\cF}{\mathcal{F}}
\newcommand{\cG}{\mathcal{G}}
\newcommand{\cH}{\mathcal{H}}

\newcommand{\cK}{\mathcal{K}}
\newcommand{\cL}{\mathcal{L}}
\newcommand{\cM}{\mathcal{M}}
\newcommand{\cN}{\mathcal{N}}

\newcommand{\cP}{\mathcal{P}}
\newcommand{\cQ}{\mathcal{Q}}

\newcommand{\cS}{\mathcal{S}}
\newcommand{\cT}{\mathcal{T}}

\newcommand{\cV}{\mathcal{V}}
\newcommand{\cW}{\mathcal{W}}
\newcommand{\cX}{\mathcal{X}}
\newcommand{\cY}{\mathcal{Y}}

\newcommand{\bH}{\mathbf{H}}

\newcommand{\deq}{\coloneqq}
\newcommand{\one}{\mathbbm{1}}

\newcommand{\1}{{\rm 1}\kern-0.24em{\rm I}}

\newcommand{\R}{\mathbb{R}}
\newcommand{\p}{\mathbb{P}}
\renewcommand{\E}{\mathbb{E}}
\newcommand{\Z}{\mathbb{Z}}
\newcommand{\N}{\mathbb{N}}

\newcommand{\field}[1]{\mathbb{#1}}

\newcommand{\Y}{\field{Y}}

\newcommand{\x}{\mathcal{X}}

\newcommand{\pn}{\p_{\kern-0.25em n}}
\newcommand{\pnm}{\p_{\kern-0.25em n,m}}
\newcommand{\psubm}{\p_{\kern-0.25em m}}
\newcommand{\psubp}{\p_{\kern-0.25em p}}
\newcommand{\cfi}{\cF_{\kern-0.25em \infty}}

\newcommand{\Lipone}{\mathrm{\overline{Lip}}_1}

\newcommand{\ima}{\textbf{i}}
\newcommand{\mmid}{\mathbin{\|}}
\newcommand{\simiid}{\overset{\text{i.i.d.}}{\sim}}
\newcommand{\Sob}{\mathcal H}

\newcommand{\torus}{\mathbb T^\dd}

\newcommand{\comp}{\mathsf{c}}
\newcommand{\argmin}{\mathop{\mathrm{arg\,min}}}

\newcommand{\ud}{\mathrm{d}}

\newcommand{\eps}{\varepsilon}

\newcommand{\T}{\mathsf{T}}

\DeclareMathOperator{\interior}{int}
\DeclareMathOperator{\law}{law}
\DeclareMathOperator{\tr}{tr}

\newlength{\minipagewidth}
\newcommand{\mb}[1]{\mathbf{#1}}
\newcommand{\msf}[1]{\mathsf{#1}}
\newcommand{\qtxt}[1]{\quad \text{#1} \quad}

\newcommand{\bpi}{\boldsymbol{\pi}}
\newcommand{\BW}{\mathsf{BW}}
\DeclareMathOperator{\cone}{cone}
\DeclareMathOperator{\diag}{diag}
\DeclareMathOperator{\divergence}{div}
\DeclareMathOperator{\dom}{dom}
\DeclareMathOperator{\prox}{prox}
\DeclareMathOperator{\KL}{\mathsf{KL}}
\newcommand{\dd}{d}
\newcommand{\dFR}{\msf d_{\msf{FR}}}
\newcommand{\gradFR}{\nabla_{\msf{FR}}}
\newcommand{\gradW}{\nabla\mkern-10mu\nabla}
\newcommand{\gradWFR}{\gradW_{\msf{FR}}}
\newcommand{\id}{\mathrm{id}}
\newcommand{\md}{\mathsf d}
\newcommand{\mg}{\mathfrak g}
\newcommand{\Cb}{C_{\mathsf{b}}}
\newcommand{\MMD}{\mathsf{MMD}}
\newcommand{\SW}{\mathsf{SW}}

\newcommand{\bbar}[1]{\overline{\overline{#1}}}

\renewcommand{\boxed}[1]{#1}
\renewcommand{\emptyset}{\varnothing}
\renewcommand{\sharp}{\#}
\renewcommand{\top}{\T}
\renewcommand{\Delta}{\varDelta}
\renewcommand{\Gamma}{\varGamma}
\renewcommand{\Omega}{\varOmega}
\renewcommand{\Phi}{\varPhi}
\renewcommand{\Sigma}{\varSigma}
\renewcommand{\Theta}{\varTheta}
\renewcommand{\Upsilon}{\varUpsilon}

\newcommand{\mbb}[1]{\mathbb #1}

\makeatother

\allowdisplaybreaks{}
\usetikzlibrary{arrows, cd, positioning, quotes}

\makeindex[intoc, title=Index]

\begin{document}

\frontmatter

\thispagestyle{empty}

\vskip3cm
\begin{center}
    \huge \bfseries\sffamily  Statistical Optimal Transport
\end{center}
\vskip2cm
\begin{center}
    \large \begin{tabular}{c}Sinho Chewi \\ Yale \end{tabular} \hspace{2em} \begin{tabular}{c} Jonathan Niles-Weed \\ NYU \end{tabular} \hspace{2em} \begin{tabular}{c} Philippe Rigollet \\ MIT \end{tabular}
\end{center}
\vskip3cm
\begin{center}
    \large\emph{{E}cole d'Et\'{e} de Probabilit\'{e}s de Saint-Flour XLIX}
\end{center}

\setcounter{tocdepth}{1}
\tableofcontents

\mainmatter

\Preface{}
\label{chap:preface}
The history of optimal transport begins in 1781 with a memoir by Gaspard Monge that he submitted to the Acad\'emie des Sciences~\cite{Mon81}. Since then, it has grown into a mature mathematical field with many important discoveries, such as  Kantorovich's duality theory, Brenier's theorem, Otto's calculus, the JKO scheme, and the Lott--Sturm--Villani definition of the Ricci curvature of geodesic spaces, to name a few. The first comprehensive treatment of optimal transport dates back to the seminal volumes of Rachev and R\"uschendorf~\cite{RacRus98a,RacRus98b}. We also refer the reader to the excellent texts of Villani~\cite{Vil03, Vil09}, Ambrosio, Gigli, and Savar\'e~\cite{AmbGigSav08}, and more recently, Santambrogio~\cite{San15} for a comprehensive treatment of this subject from the mathematical perspective, and to the notes of Ambrosio and Gigli~\cite{AmbGig13}, Ambrosio, Bru\'e, and Semola~\cite{AmbBruSem21}, and the short monograph of Figalli and Glaudo~\cite{FigGla23} for quicker introductions.
Even a quick inspection of their tables of contents reveals that Monge's question was the gateway to many more, and that the field of optimal transport has many unexpected connections, ranging from geometry to partial differential equations.

More recently, optimal transport has made a resounding entrance into the field of machine learning under the impetus of Marco Cuturi, who showed that Wasserstein distances could be computed efficiently using the Sinkhorn algorithm~\cite{Cut13}. This initial spark was followed by an extensive toolbox, built on optimal transport, that covered multiple tasks across various areas of machine learning and graphics.  The development of this toolbox, now called \emph{computational optimal transport}, was led by Marco Cuturi and Gabriel Peyr\'e and collected in their inspiring manuscript~\cite{PeyCut19}. The far-reaching scope of optimal transport across machine learning and data science rests on the fact that many objects such as point clouds, polygonal meshes, or even documents can be encoded as probability measures. In turn, the Wasserstein metric between these probability measures offers a semantically meaningful notion of distance.

So what is \emph{statistical} optimal transport? This modifier comes largely as an echo to Peyr\'e and Cuturi's \emph{computational} optimal transport. It is an umbrella term that captures the remarkably diverse points of contact between statistics and optimal transport. The aim of the present monograph is to provide an introduction to a selection of topics within this burgeoning field according to our tastes (see~\cite{PanZem20Invitation} for a complementary treatment).

Historically, Wasserstein distances have been employed in statistics as a tool to quantify the rate of convergence of empirical probability measures $\mu_n$ to their limit $\mu$.
This line of work was inaugurated in a celebrated work of Dudley~\cite{Dud69}, who provided bounds for $W_1(\mu_n, \mu)$.  Wasserstein distances are particularly well-suited for
quantifying this convergence for several reasons. First, unlike the total variation distance or Kullback--Leibler divergence, the Wasserstein distance between a discrete distribution $\mu_n$ and a potentially continuous one $\mu$ remains finite and informative. Second, by definition, bounding the Wasserstein distance amounts to exhibiting a coupling between the two measures. Third and finally, thanks to Kantorovich duality, a bound on the Wasserstein distance translates into a strong uniform bound on test functions. For example, when $p=1$, $W_1(\mu_n, \mu)\le \eps$ implies  that $|\int f\, \ud \mu_n - \int f\, \ud \mu|\le \eps$  for all functions $f$ that are $1$-Lipschitz; in fact the two statements are equivalent (see Section~\ref{sec:duality_w1} for details).

In the last decade, following the impetus of machine learning, optimal transport has percolated to many more aspects of statistics. One of the most exciting directions is a new avenue of research that had largely been out of reach of classical methods in the past. In this class of problems, the \emph{coupling} of data is the main obstacle to statistical analysis. More concretely, consider a classical statistical setup where one observes independent copies of a pair of random variables $(X,Y)$ where $X$ is thought of as input and $Y$ output. Regression falls in this framework, as does, more generally, all of supervised learning. In particular, the observed $X$ and $Y$ are coupled.
A more challenging model arises when $X$ and $Y$ are observed in an \emph{uncoupled} fashion: independent copies of $X$ and independent copies of $Y$. Such a setup arises naturally in single-cell genomics where the destructive nature of the prevailing sequencing  process does not allow for taking multiple measurements of the same cell. This conundrum is a key obstacle to cellular trajectory reconstruction where one aims at reconstructing the time evolution of a cell in a genetic landscape; see~\cite{SchShuTab19,bunne2022proximal} for more details.
However, under a natural physical hypothesis that the distribution of $Y$ is obtained from the distribution of $X$ by evolving it in a conservative vector field for a short time, it is natural to model the unobserved coupling between $X$ and $Y$ as arising from optimal transport.
These problems and more raise fundamental questions about the estimation of Wasserstein distances and the corresponding couplings, which are taken up in the first part of this monograph.

The aforementioned applications make use of the role of optimal transport in endowing the space of probability measures with an interpretable and useful notion of distance.
A deeper study of this space, however, uncovers a rich underlying geometrical structure admitting descriptions of curvature, geodesics, gradient flows, etc.
This geometric perspective, first advocated by Felix Otto in his seminal article~\cite{Ott01}, provides statisticians with powerful new tools for the design and analysis of algorithms for manipulating probability distributions.
A key development in this regard was the edifying interpretation, by Jordan, Kinderlehrer, and Otto~\cite{JorKinOtt98}, of the Langevin diffusion as a Wasserstein gradient flow of the KL divergence. Since the Langevin diffusion is popularly employed as a sampling algorithm in Bayesian statistics, this discovery has ushered in a decade of research linking sampling to optimization over the Wasserstein space.
More broadly, Wasserstein gradient flows and their variants yield new algorithmic paradigms and fresh perspectives for diverse problems including the nonparametric MLE, the dynamics and training of neural networks, and variational inference (see Chapter~\ref{chap:appWGF}).
Geometric considerations also lead to novel applications, such as the geometric averaging of data which can be cast as probability measures, e.g., images or speech.
The subject of Wasserstein geometry is studied in the second half of this monograph.

\paragraph*{How to read this book.}

This monograph aims to offer a concise introduction to optimal transport, quickly transitioning to its applications in statistics and machine learning. It is primarily tailored for students and researchers in these fields, yet it remains accessible to a broader audience of applied mathematicians and computer scientists.

Chapter 1 serves as the gateway to the subsequent chapters by presenting the foundational concepts of optimal transport that will be used throughout. The remaining chapters are largely independent, with the exceptions of chapters 5 and 6, and chapters 7 and 8, which should be studied together. Figure~\ref{fig:chapter_dependencies} illustrates the dependencies between the various chapters.

\begin{figure}[h!]
\centering
\begin{tikzpicture}[node distance=2cm, rounded corners, >=Stealth]

\coordinate (A) at (0,0);
\coordinate (B) at (-2,-1);
\coordinate (C) at (-1,-1);
\coordinate (D) at (0,-1);
\coordinate (E) at (1,-1);
\coordinate (F) at (1,-2);
\coordinate (G) at (2,-1);
\coordinate (H) at (2,-2);

\node[draw, rectangle, rounded corners] (1) at (A) {1};
\node[draw, rectangle, rounded corners] (2) at (B) {2};
\node[draw, rectangle, rounded corners] (3) at (C) {3};
\node[draw, rectangle, rounded corners] (4) at (D) {4};
\node[draw, rectangle, rounded corners] (5) at (E) {5};
\node[draw, rectangle, rounded corners] (6) at (F) {6};
\node[draw, rectangle, rounded corners] (7) at (G) {7};
\node[draw, rectangle, rounded corners] (8) at (H) {8};

\draw[->] (1) -- (2);
\draw[->] (1) -- (3);
\draw[->] (1) -- (4);
\draw[->] (1) -- (5);
\draw[->] (1) -- (7);
\draw[->, dotted] (2) -- (3);
\draw[->, dotted] (3) -- (4);
\draw[->] (5) -- (6);
\draw[->, dotted] (5) -- (7);

\draw[->] (7) -- (8);

\end{tikzpicture}
    \caption{Dependencies between chapters. Solid arrows show prerequisites; dotted arrows indicate references.
}\label{fig:chapter_dependencies}
\end{figure}
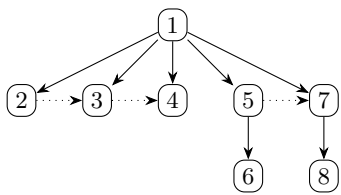

Each chapter concludes with a series of exercises, allowing readers to apply the concepts learned to questions not addressed in the main text.

\paragraph*{Acknowledgments.}

\noindent{\bf Sinho Chewi.}
Optimal transport was not what I originally planned to study as a graduate student, but my research career has been infinitely richer as a result.
For this, I am grateful to my advisor and co-author Philippe for fearlessly inducting me into this beautiful area, and my academic sibling and co-author Jon for paving the way and constantly acting as my source of inspiration.

I owe my understanding of optimal transport to many sources, beginning with Villani's expertly written monograph~\cite{Vil03} and including countless hours of discussion with collaborators and members of Philippe's group. I am happy to have the opportunity to contribute to the growth of the community through this pedagogical text.

I would like to acknowledge the Institute for Advanced Study for the hospitable working environment and support through the Eric and Wendy Schmidt Fund.

\smallskip{}
\begin{flushright}
    \emph{New Haven, CT, July 2024.}
\end{flushright}

\medskip

\noindent{\bf Jonathan Niles-Weed.}
I have been lucky to learn about optimal transport from many brilliant collaborators; among them, Jason Altschuler, Francis Bach, Sivaraman Balakrishnan, Quentin Berthet, Marco Cuturi, Vincent Divol, Alberto Gonz\'alez-Sanz, Tudor Manole, Aram-Alexandre Pooladian, Austin Stromme, and Larry Wasserman have had an especially large impact on me. I am deeply grateful to them for their knowledge and insight.
It is an honor to be able to share some of their work in this book.

I am indebted to students at NYU who participated in a graduate seminar on portions of this material in 2021 and whose feedback significantly improved the exposition.
I also thank Sivaraman Balakrishnan (again) for his perceptive comments on an early version of this manuscript.

Of course, I would be nowhere without Sinho Chewi and Philippe Rigollet, my colleagues and friends.
You both have inspired me beyond measure.
Collaborating with you has been one of the highlights of my career, and I hope to have the privilege of many more collaborations ahead.

My work on these notes was primarily supported by an National Science Foundation CAREER award DMS-2339829 and several other NSF grants (DMS-2015291, DMS-2210583), along with a Sloan Foundation fellowship and gifts from Apple and Google.

\smallskip{}
\begin{flushright}
    \emph{New York, NY, July 2024.}
\end{flushright}

\medskip

\noindent{\bf Philippe Rigollet.} These notes have grown significantly since I was honored to give the 2019 St.\ Flour lecture series, which initially covered roughly Chapters 1, 7, and 8. Much of the material presented here was unfamiliar to me when I first delivered the lectures.

First and foremost, I extend my heartfelt thanks to the organizers (espectially Christophe Bahadoran) and attendees of the 2019 St.\ Flour Summer School. The valuable feedback I've received from the audience has been instrumental in motivating the extensive expansion of topics covered in these notes. My only regret is the delay in producing these notes, which prevents them from being published concurrently with those of my wonderful co-lecturers, Nicolas Curien and Elchanan Mossel. Though we won't be sharing a volume, we will forever be bonded as Chevaliers du Taste-fourme. 

I have learned most from others and would like to thank my collaborators who have taught me so much on this topic: Jason Altschuler, Julien Clancy, Max Daniels, Aiden Forrow, Borjan Geshkovski, Florian Gunsilius, Jan-Christian H\"utter, Cyril Letrouit, Jean-Michel Loubes, Chen Lu, Tyler Maunu, Mor Nitzan, Quentin Paris, Geoff Schiebinger, Austin Stromme, George Stepaniants, Felipe Suarez, William Torous, Kaizheng Wang, Yuling Yan, Aleksandr Zimin.  In particular, for enlightening discussions on optimal transport and other topics, I would like to specifically acknowledge  S\'ebastien Bubeck, Ramon van Handel, Thibaut Le Gouic, Vianney Perchet,  Yury Polyanskiy, Maxim Raginsky, and Justin Solomon. Many of the ideas in Chapters 5 and 6 were core discussion topics during the program on \emph{Geometric Methods in Optimization and Sampling} at the Simons Institute in Berkeley during Fall 2021. I extend my gratitude to all the participants, my co-organizers, and Peter Bartlett, who made this program so wonderful and stimulating.

Since then, the community around statistical optimal transport has grown significantly, and these notes have benefited from interactions with many people. Starting with Marco Cuturi, who introduced me to optimal transport when we were in Princeton together, I also learned a lot of optimal transport from  Guillaume Carlier, Victor Chernozhukov, Lena\"ic Chizat, Simone Di Marino, Augusto Gerolin, Promit Ghosal, Marc Hallin, Zaid Harchaoui, Kengo Kato, Anna Korba, Alexei Kroshnin, Qin Li, Jan Maas, Axel Munk, Robert McCann, Dan Mikulincer, Youssef Marzouk, Soumik Pal, Victor Panaretos, Gabriel Peyr\'e, Filippo Santambrogio, Bodhi Sen, Vladimir Spokoiny, Yair Shenfeld, and Jia-Jie Zhu among others.

 Part of this material has been taught at MIT in 2023 and 2024, at Universit\'e Paris Sorbonne in 2022, and two other summer schools in 2023: The Princeton Machine Learning Theory Summer School and the CIME summer school in Cetraro, Italy. The audience has given me excellent feedback, contributing to the improvement of these notes. Special thanks to Th\'eo Dumont, Max Daniels, and Giulia Bertagnolli for typing up the respective material.
 
 Some typos and errors were fixed by Michael Diao, Haruki Kono, Hugo Lavenant, Aimee Maurais, Yaroslav Mukhin, Madhav Sankaranarayanan, Sabarish Sainathan, Yucheng Shang, Vishwak Srinivasan, Panos Tsimpos, Oliver Wang, Liane Xu, and Julie Zhu. We thank them for helping make this manuscript more readable.

Last but not least, I would like to thank my wonderful co-authors, Sinho Chewi and Jonathan Niles-Weed. Thank you both for joining me on this epic adventure. I've learned more from you than anyone else, and you have been as much students as you have been teachers to me.

My work on these notes was primarily supported by NSF grant CCF-1838071 and several other grants from the National Science Foundation during the period 2019-2024 (DMS-1712596, CCF-1740751, DMS-2022448, CCF-2106377). I am also thankful for a gift from Apple. A first draft was conceived in Spring 2019, when I was supported by the Eric and Wendy Schmidt Fund at the Institute for Advanced Study.

\smallskip{}
\begin{flushright}
    \emph{Cambridge, MA, July 2024.}
\end{flushright}

\chapter{Optimal transport}
\label{chap:OT}

\section{The optimal transport problem}\label{sec:ot_problem}

In his 1781 memoir~\cite{Mon81}, Monge formulated the following problem:  how can one transport a given pile of sand to fill a given ditch so as to minimize the cost of transporting the sand? This problem can be modeled using probability distributions. Indeed, note first that for this task to be solvable, the pile and the ditch must occupy the same volume. Without loss of generality, let us normalize this volume to be 1. We are therefore led to consider two probability measures, $\mu$ and $\nu$ over $\R^\dd$. It is often convenient to reason about two random variables, $X \sim \mu$, $Y\sim \nu$. This is our \textbf{input} to  a \textbf{constrained optimization problem}.

\subsection{The Monge and Kantorovich problems}

Back to our sand analogy, transporting the pile means finding a (measurable) function, called a \emph{transport map}  $T: \R^\dd \to \R^\dd$, which indicates that the sand located at $x \in \R^\dd$ should be moved to $T(x) \in \R^\dd$. For the transport map to actually complete the job (filling the ditch), one needs to ensure that $T(X)\sim \nu$ whenever $X \sim \mu$. We say that $T$ \emph{pushes} $\mu$ to $\nu$ or that $\nu$ is the \emph{pushforward measure} of $\mu$ (through $T$) and write $T_{\#}\mu=\nu$. This is our \textbf{constraint}.

Turning now to our \textbf{objective} function, recall that Monge's question involved minimizing the cost of transporting the sand. There are many ways to measure this cost (effort, fuel consumption, etc.) so to simplify our exposition, we simply measure it in terms of the Euclidean distance travelled by the sand. The sand at location $x$ travels a distance of $\|T(x)-x\|$. Therefore, the average distance travelled is
$$
\int \|T(x)-x\|\, \mu ( \ud x)\,.
$$

The Monge formulation\index{Monge problem} of the optimal transport problem is therefore to  minimize the above objective subject to the constraint that $T$ pushes $\mu$ to $\nu$:
$$
\inf_{T: T_\# \mu =\nu} \int \|T(x)-x\|\, \mu ( \ud x)\,.
$$
Note that many choices for the transport cost may be considered. In full generality, it is customary to consider a general cost $c(X, T(X))$, where $c(x, y)$ measures the cost of transporting $x \in \R^\dd$ to $y \in \R^\dd$. In this general framework, we may even allow $X$ and $Y$ to be defined on two different spaces, not necessarily $\R^\dd$. In these notes, we focus primarily on the cases where $c(x, y) = \|x - T(x)\|$ or $c(x, y)=\|x-T(x)\|^2$, which give rise to the \emph{Wasserstein distances}. The space $\R^\dd$ may also be replaced with more complex spaces such as Riemannian manifolds, but this is generally beyond the scope of these lectures (with the exception of Section~\ref{sec:gaussian_mixtures}).

While the Monge problem is easy to formulate, we need to ask several questions:
\begin{itemize}
\item Does there always exists such a valid transport map or, conversely, is the constraint set empty?
\item If there is a minimizer, is it unique? How to characterize it? Note that our constraint is not convex, which makes finding an answer to this question rather difficult.
\end{itemize}

A simple example gives an answer to the first question. Indeed, take $d=1$, assume that $\mu=\delta_0$ is a point mass at $0$, and that $\nu= \frac{1}{2}\, \delta_{-1} + \frac{1}{2} \,\delta_{1}$ is a mixture of two point masses. Whatever our choice of the transport map $T$, the pushforward $T_\# \mu$ is the point mass $\delta_{T(0)}$ at $T(0)$, so we cannot achieve the transport at all, at least with a deterministic map.

Intuitively, we would like:
$$
T(0)=
\begin{cases}
-1& \text{w.p.} \ \frac{1}{2}\\[0.25em]
1& \text{w.p.} \ \frac{1}{2}\\
\end{cases}
\qquad \text{and} \qquad T(x)=x\,,\ \forall\ x \neq 0\,.
$$
Such a $T$ is not a function but a Markov kernel: it assigns a probability distribution to each point $x \in \R$.

The second question remained without a satisfactory answer for almost two centuries until the Soviet mathematician Leonid Kantorovich~\cite{Kan42} introduced a relaxation of the problem that exactly allows for Markov kernels, as discussed in the example above, in a groundbreaking two-pager.\index{Kantorovich problem} Equivalently, this formulation involves \emph{couplings} as opposed to maps. 

Let $\mu, \nu$ be two probability measures over $\R^\dd$ and let $\gamma$ be a \emph{coupling} between these two distributions, that is, a joint distribution over $\R^\dd \times \R^\dd$ such that its first marginal is $\mu$ and its second marginal is $\nu$: for any Borel set $A \in \R^\dd$, we have
$$
\gamma(A \times \R^\dd)= \mu(A) \qquad \text{and} \qquad  \gamma(\R^\dd \times A)= \nu(A)\,.
$$
The terminology \emph{coupling} comes from the fact that while $X\sim \mu$ and $Y\sim \nu$ were random variables that had nothing to do with each other, the coupling forces them to live on the same probability space by describing their probabilistic dependence. 
Here and throughout these notes, we use the notation $\Gamma_{\mu,\nu}$ for the set of couplings of $\mu$ and $\nu$.

Let $c: \R^\dd\times \R^\dd \to [0, \infty)$ be a measurable cost function. The general Kantorovich formulation of the optimal transport problem consists of the following optimization problem:
\begin{equation}
\label{KOT}\tag{\textsf{KOT}}
\inf_{\gamma \in \Gamma_{\mu, \nu}} \int c(x, y)\, \gamma(\ud x,\ud y)\,.
\end{equation}

\subsection{Couplings}

To get a better understanding of the Kantorovich problem, it is informative to explore the set $\Gamma_{\mu, \nu}$.

Perhaps the simplest coupling is the independent coupling $\gamma=\mu \otimes \nu$ where $X\sim \mu$ and $Y \sim \nu$ are simply assumed to be independent: for any Borel sets $A, B \subseteq \R^\dd$, 
$$
\gamma(A \times B)=\mu(A) \cdot \nu(B)\,.
$$

The next proposition collects preliminary facts about $\Gamma_{\mu, \nu}$.

\begin{proposition}\label{prop:couplings}
    Let $\mu, \nu$ be two probability measures on $\R^d$. The set $\Gamma_{\mu, \nu}$ of couplings between $\mu$ and $\nu$ is \emph{non-empty}, \emph{convex}, and \emph{compact} with respect to the topology of weak convergence.
\end{proposition}
\begin{proof}
    Because the independent coupling always exists, we know that $\Gamma_{\mu, \nu} \neq \emptyset$. 

    To show that $\Gamma_{\mu, \nu}$ is convex, consider two couplings $\gamma_0, \gamma_1 \in \Gamma_{\mu, \nu}$ and for any $\lambda \in (0,1)$ define the mixture $\gamma_\lambda=(1-\lambda)\, \gamma_0+\lambda\, \gamma_1$. Observe that for any Borel set $A \in \R
^d$,
\begin{align*}
    \gamma_\lambda (A \times \R^d)&= (1-\lambda)\,\gamma_0 (A \times \R^d)+ \lambda\,\gamma_1 (A \times \R^d)\\
    &=(1-\lambda)\,\mu(A)+ \lambda\,\mu (A)=\mu(A)\,.
\end{align*}
Hence the first marginal of $\gamma_\lambda$ is given by $\mu$ and by the same argument its second marginal is given by $\nu$. Thus $\gamma_\lambda \in \Gamma_{\mu, \nu}$ for any $\lambda \in (0,1)$, whence $\Gamma_{\mu, \nu}$ is convex.

To complete the proof of our proposition, we show that 
$\Gamma_{\mu,\nu}$ is compact. By Prokhorov's theorem (Theorem~\ref{thm:prokhorov}), it is sufficient to show that $\Gamma_{\mu, \nu}$ is closed and (uniformly) tight. To that end, recall that from Prokhorov's theorem, the constant sequences $(\mu)_n, (\nu)_n$ are both tight, so that for any $\eps>0$, there exists a compact set $K \subset \R^\dd$  such that $\mu(K^\comp)+ \nu(K^\comp)<\eps$. Then the set $K\times K$ is also compact and for any $\gamma \in \Gamma_{\mu,\nu}$,
$$
\gamma((K \times K)^\comp)\le \gamma(\R^\dd \times K^\comp)+\gamma(K^\comp \times \R^\dd)=\mu(K^\comp)+ \nu(K^\comp)<\eps\,.
$$
Hence, $\Gamma_{\mu,\nu}$ is tight.
Moreover, since $\gamma \in \Gamma_{\mu,\nu}$ is equivalent to
\begin{align*}
    \int f(x) \, \gamma(\ud x,\ud y) = \int f\,\ud \mu\qquad\text{and}\qquad\int f(y)\,\gamma(\ud x,\ud y) = \int f \,\ud \nu
\end{align*}
for all bounded continuous $f : \R^\dd\to\R$, by the definition of weak convergence (Theorem~\ref{thm:portmanteau}) it follows that $\Gamma_{\mu,\nu}$ is closed.
Therefore, Prokhorov's theorem yields that $\Gamma_{\mu, \nu}$ is compact.
\end{proof}

\begin{figure}[h]
    \centering
    \includegraphics[width=0.45\textwidth]{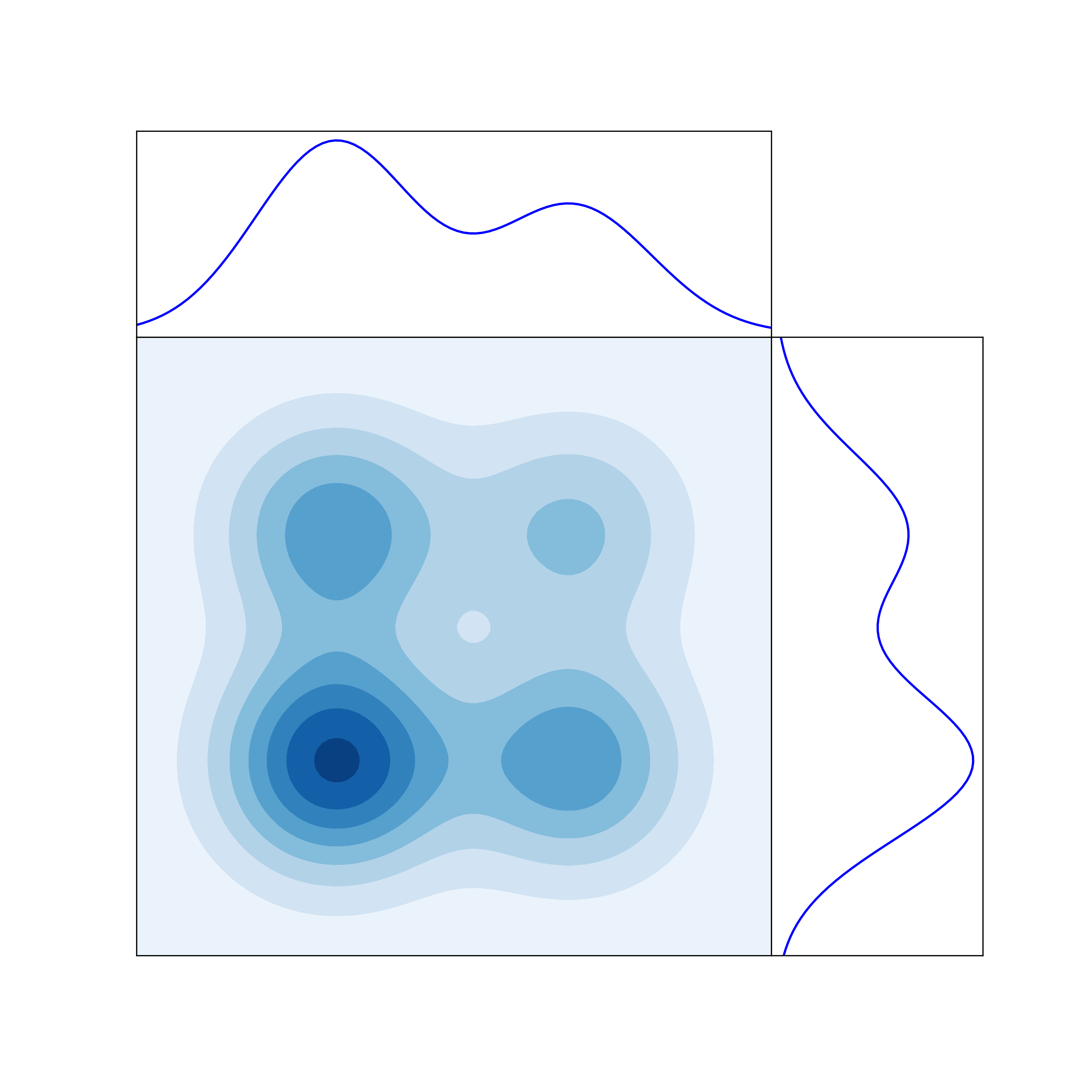}
    \includegraphics[width=0.45\textwidth]{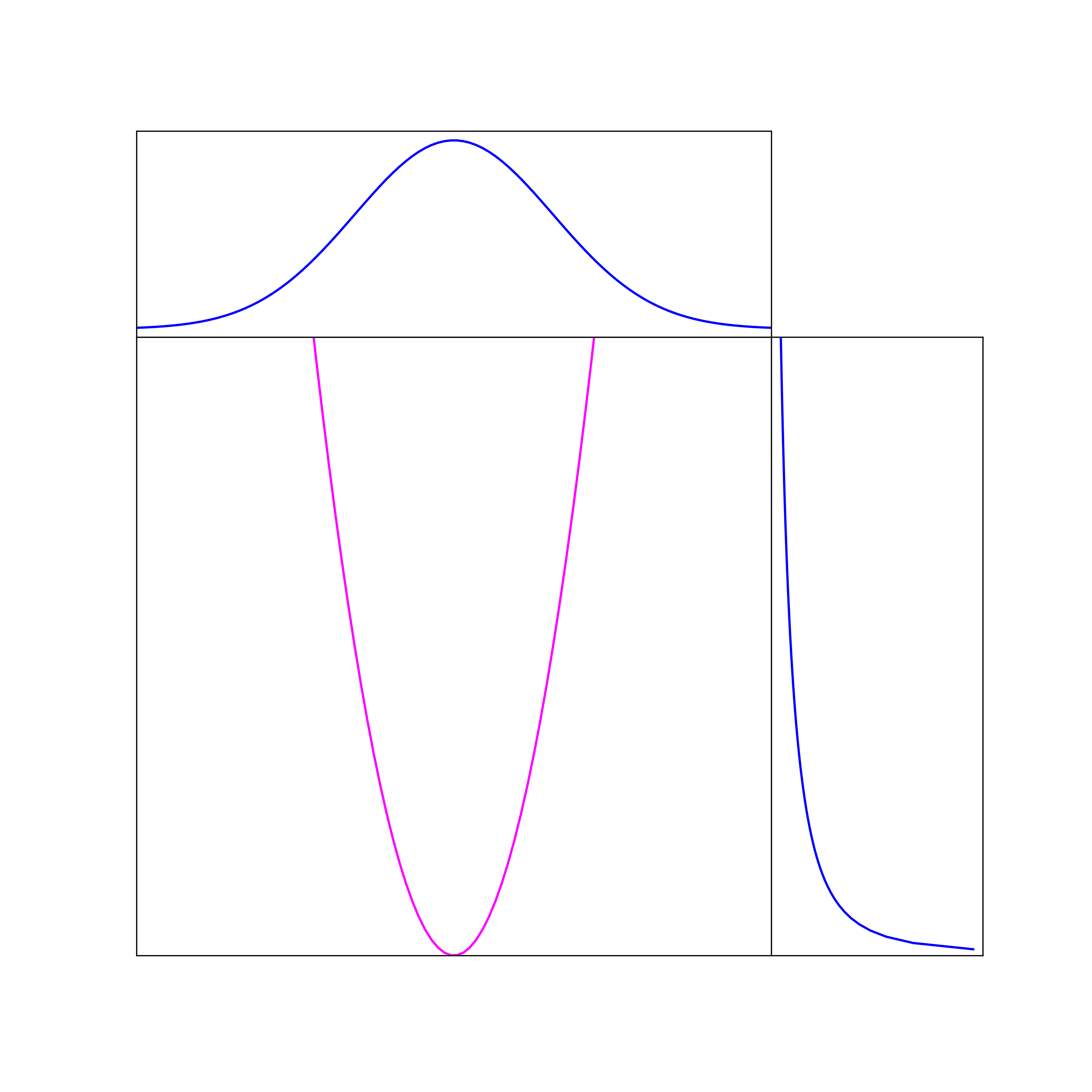}
    \caption{(Left) Independent coupling of a mixture of two Gaussians. (Right) Deterministic coupling of $X\sim \cN(0,1)$ with $Y \sim \chi_1^2$.}\label{fig:coupling}
\end{figure}

A coupling $\gamma \in \Gamma_{\mu, \nu}$ captures the dependence between two random variables  $X \sim \mu$ and $Y \sim \nu$. 
As mentioned above, one option is to assume that $X$ and $Y$ are independent, which gives rise to the independent coupling.
In Figure~\ref{fig:coupling} (Left), we plot the independent coupling between two mixtures of Gaussians.
At the opposite extreme of the independent coupling, suppose by way of example that $X \sim \cN(0,1)$ and $Y \sim \chi^2_1$ and observe that $Y$ has the same distribution as $X^2$. Then we can take the deterministic coupling such that $Y=X^2$:
    $$
    \gamma(\ud x, \ud y)=\mu(\ud x)\, \delta_{x^2}(\ud y)\,.
    $$
    We plot this coupling in Figure~\ref{fig:coupling} (Right); observe that it is degenerate.
    
To continue our exploration of couplings, assume that $X \sim \cN(0,1)$ and $Y \sim \cN(0,1)$, then we can take any coupling where
    \begin{align}\label{eq:bivariate_gaussian}
        \begin{pmatrix} X \\Y\end{pmatrix}\sim \cN\biggl(\begin{pmatrix}0 \\0\end{pmatrix}, \begin{pmatrix}1 & \rho \\\rho & 1\end{pmatrix}\biggr)
    \end{align}
    and $\rho \in [-1, 1]$ is the correlation between $X$ and $Y$. See Figure~\ref{fig:bivariate_gaussian}.

    \begin{figure}[h]
        \centering
        \includegraphics[width=0.95\textwidth]{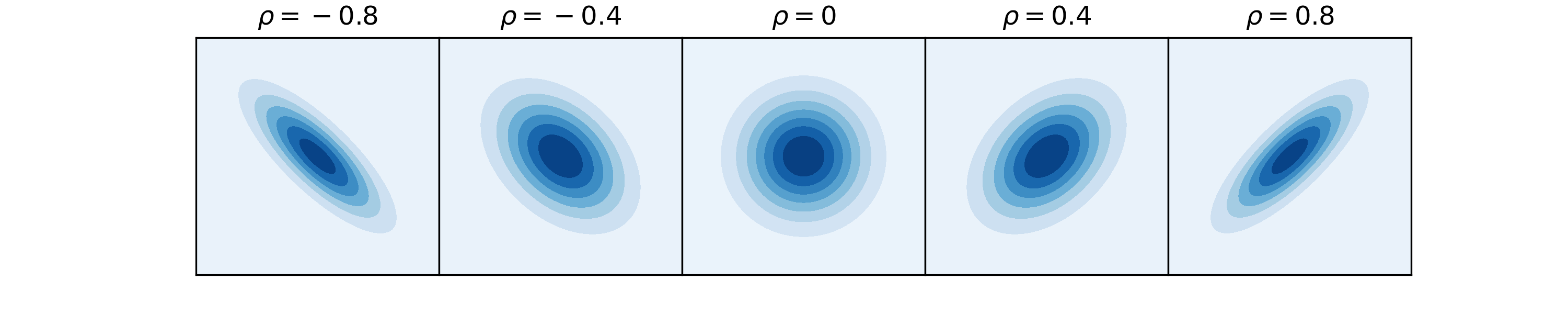}
        \caption{The bivariate Gaussian coupling~\eqref{eq:bivariate_gaussian} for five different values of $\rho$.}\label{fig:bivariate_gaussian}
    \end{figure}

The content of Brenier's theorem later in this chapter is that under mild regularity conditions, the solution of the Kantorovich problem with quadratic cost is achieved by a deterministic coupling. 
As we discuss next, when $\mu$ and $\nu$ are discrete measures, this fact can be understood by examining the geometry of the set $\Gamma_{\mu, \nu}$.

\subsection{Discrete optimal transport}\label{sec:discreteOT}

The case where $\mu$ and $\nu$ are two discrete distributions is of special practical relevance. For example, $\mu, \nu$ can be empirical measures on a point cloud. Consider the case where 
$$
\mu = \sum_{i=1}^m p_i \delta_{x_i}\,, \qquad \text{and} \qquad \nu =  \sum_{j=1}^n q_j \delta_{y_j}\,.
$$
In this case, a coupling $\gamma$ of $X \sim \mu$ and $Y \sim \nu$ is characterized by a non-negative matrix $P \in \R^{m\times n}$ where $P_{i,j}=\gamma(X=x_i, Y=y_j)$. The marginal constraints on $\gamma \in \Gamma_{\mu, \nu}$ readily translate into
    \begin{align*}
        \forall i\in [m]\,, \;\sum_{j\in [n]} P_{i,j} = p_i\,, \qquad\text{and}\qquad\forall j\in [n]\,, \;\sum_{i\in[m]} P_{i,j} = q_j\,.
    \end{align*}
Introducing $\mb 1_m$, $\mb 1_n$ for the all-ones vectors of sizes $m$ and $n$, respectively, these constraints can be represented concisely as $P \mb 1_n = p$, $P^\T \mb 1_m = q$, where $p=(p_1, \ldots, p_n)^\T$ and $q=(q_1, \ldots, q_m)^\T$.

Like the coupling, the cost $c$ can also be captured by an $m \times n$ matrix $C$ where $C_{i,j}=c(x_i, y_j)$ for $i\in [m]$, $j\in [n]$. The Kantorovich optimal transport problem~\eqref{KOT} is therefore equivalent to
$$
 \min_{P \in \R^{m\times n}_+} \sum_{i, j \in [n]} C_{i, j} P_{i, j}\qquad\text{s.t.}\qquad P\mb 1_n = p\,,\; P^\T \mb 1_m = q\,,
$$
which can also be written more concisely as
\begin{align*}
     \min_{P \in \R^{m\times n}_+}{\langle C,P\rangle}\qquad\text{s.t.}\qquad P\mb 1_n = p\,,\; P^\T \mb 1_m = q\,,
\end{align*}
where $\langle C,  P \rangle = \tr(C^\T P)$ is the Frobenius inner product on the set of $m \times n$ real matrices.

In particular, when $m = n$ and all of the weights $p_i$, $q_j$ are equal to $1/n$, the set of valid coupling matrices $P$ is (a multiple of) the set of doubly stochastic matrices, also known as the \emph{Birkhoff polytope}\index{Birkhoff polytope}:
\begin{align}\label{eq:birkhoff}
    \msf{Birk}
    &\deq \{\gamma \in \R_+^{n\times n} : \gamma\mb 1_n = \mb 1_n, \; \mb 1_n^\T \gamma =  \mb 1_n^\T\}\,.
\end{align}
Then,~\eqref{KOT} reduces to
\begin{equation}
    \label{eq:discreteOT}
     \min_{P \in n^{-1}\msf{Birk}} {\langle C,  P \rangle} \,.
\end{equation}
The extreme points of the Birkhoff polytope are permutation matrices: they are binary matrices  $\pi \in \{0, 1\}^{n \times n}$ with exactly one non-zero entry in each row and column.
In particular, general principles of convex geometry imply that the solution to any linear program of the form~\eqref{eq:discreteOT}
can be taken to be a matrix of the form $n^{-1} \pi$.
A transport plan of this form is induced by a deterministic map (the permutation), and hence in this case there is a solution to the Monge problem.
As we shall see in the subsequent sections, extreme points of $\Gamma_{\mu,\nu}$ play a special role more generally in the geometry of the optimal transport problem.

\section{Wasserstein distances}\label{sec:wp_distances}

The Kantorovich problem~\eqref{KOT} makes sense for a wide variety of cost functions, with different interpretations in each case.
For instance, one natural example comes from taking $c(x,y)=\one_{x \neq y}$ to be the trivial metric. In this case,~\eqref{KOT} gives:
$$
\inf_{\gamma \in \Gamma_{\mu,\nu}} \gamma(X \neq Y)
$$
which is a well-known formulation of the total variation distance. (See Exercise~\ref{qu:tv}.)
Note, however, that the trivial distance $\one_{x\neq y}$ is unrelated to the geometry of $\R^\dd$. In particular, it does not say whether $x$ and $y$ are far from each other but only if they are different. This limitation manifests itself in the total variation. Indeed, if $\mu=\delta_x$ and $\nu=\delta_y$, then the objective of~\eqref{KOT} is equal to 1 as soon as $x \neq y$.

To obtain a geometrically meaningful quantity from the Kantorovich problem, we need to choose a cost that reflects the actual distance between $x$ and $y$.
This idea gives rise to the \emph{Wasserstein distances}.\index{Wasserstein distance}
For any $p \ge 1$, let $\cP_p(\R^\dd)$ be the set of probability measures over $\R^\dd$ equipped with the Euclidean norm $\|\cdot\|$ that have finite $p$-th moment: 
$$
\mu \in \cP_p(\R^\dd) \quad \Leftrightarrow \quad \int \|x\|^p\, \mu(\D x)<\infty\,.
$$

The $p$-Wasserstein distance between two probability measures $\mu, \nu \in \cP_p(\R^\dd)$ is defined by
\begin{equation*}
W_p(\mu, \nu)=\inf_{\gamma \in \Gamma_{\mu, \nu}}\left( \int \|x-y\|^p\, \gamma(\D x,\D y)\right)^{1/p}\,,
\end{equation*}
where we recall that $\Gamma_{\mu, \nu}$ is the set of couplings between $\mu$ and $\nu$. 

We first show that in fact the above infimum is attained.  To that end, define 
$$
I(\gamma) \deq \int \|x-y\|^p\, \gamma(\D x,\D y)
$$
and observe that by definition, there exists a sequence $(\gamma_n)_n$ in $\Gamma_{\mu,\nu}$ such that $I(\gamma_n) \to W_p^p(\mu, \nu)$.
Since $\Gamma_{\mu,\nu}$ is compact (Proposition~\ref{prop:couplings}),
there is a subsequence of $(\gamma_n)_n$ which converges to some $\bar\gamma \in \Gamma_{\mu,\nu}$. By definition $W_p(\mu, \nu) \le I(\bar \gamma)$.
Since $(x,y)\mapsto \|x-y\|^p$ is unbounded, $I$ is not continuous, but it is lower semicontinuous, so $I(\bar\gamma) \le \liminf_{n\to\infty} I(\gamma_n) = W_p^p(\mu,\nu)$ (part three of the portmanteau theorem, Theorem~\ref{thm:portmanteau}). Hence $I(\bar \gamma)=W_p^p(\mu,\nu)$.
Note that the only property of the cost function we used in this proof is lower semicontinuity so this argument readily extends to more general costs.

We can therefore adopt the following definition of Wasserstein distances. Note that these distances should really be called Kantorovich--Rubinstein distances\index{Kantorovich--Rubinstein distance|see{Wasserstein distance}} but we stick to the modern trend of ``Wassersteinification''.
\begin{definition}\label{def:wass}
The $p$-Wasserstein\index{Wasserstein distance} distance between two probability measures $\mu, \nu \in \cP_p(\R^\dd)$ is defined by
\begin{equation*}
W_p(\mu, \nu)=\min_{\gamma \in \Gamma_{\mu, \nu}}\left( \int \|x-y\|^p\, \gamma(\D x,\D y)\right)^{1/p}\,.
\end{equation*}
\end{definition}

\begin{proposition}\label{prop:wp_is_metric}
The $p$-Wasserstein distance defines a metric over $\cP_p(\R^\dd)$, that is for every $\mu, \nu \in \cP_p(\R^\dd)$, it holds
\begin{enumerate}
\item $W_p(\mu, \nu)\ge 0$
\item $W_p(\mu, \nu)=W_p(\nu, \mu)$
\item $W_p(\mu, \nu)=0$ iff $\mu=\nu$
\item $W_p(\mu, \nu)\le W_p(\mu, \rho)+W_p(\rho, \nu)$ for any $\rho \in \cP_p(\R^\dd)$.
\end{enumerate}
\end{proposition}
\begin{proof}
Note first that 1. and 2. hold trivially.

We now turn to the proof of 3. If $\mu=\nu$, then the measure $\gamma(\ud x, \ud y)=\mu(\ud x)\,\delta_x(\ud y)$ is a valid coupling: $\gamma \in \Gamma_{\mu, \nu}$.
Concretely, $\gamma$ is the law of $(X, X)$ for $X\sim \mu$. Therefore
$$
0\le W_p^p(\mu, \mu)\le \int \|x-y\|^p\,\gamma(\ud x,\ud y)= \int \|x-x\|^p\,\mu(\ud x)=0\,.
$$  
To show the other direction of 3., observe that if $W_p(\mu,\nu)=0$, there exists $\bar \gamma \in \Gamma_{\mu,\nu}$ such that $(X,Y) \sim \bar \gamma$, and $X=Y$ almost surely; in particular, they must have the same distribution: $\mu=\nu$.

To complete the proof, we check the triangle inequality 4. To that end, we employ the \emph{gluing lemma} (Lemma~\ref{lem:gluing}) which ensures that there exists $X,Y,Z$ such that $X \sim \mu$, $Y\sim \nu$, $Z \sim \rho$ and such that $(X,Z)$ and $(Z,Y)$ are optimally coupled.

Then
\begin{align*}
W_p(\mu, \nu) &\le \big(\E\|X-Y\|^p\big)^{1/p}\\
&=\big(\E\|X-Z+Z-Y\|^p\big)^{1/p}\\
&\le \big(\E\|X-Z\|^p\big)^{1/p}+\big(\E\|Z-Y\|^p\big)^{1/p}\\
&=W_p(\mu, \rho)+ W_p(\rho, \nu)\,,
\end{align*}
where in the first line we used the suboptimality of the coupling  $(X,Y)$, in the third line we used the triangle inequality for $L^p$ norms, and in the last line, we used the optimality of the couplings $(X,Z)$ and $(Z,Y)$.
\end{proof}

\begin{example}[Wasserstein distances in simple cases]
\begin{enumerate}
\item Fix $x, y \in \R^\dd$. Then 
$$
W_p(\delta_x, \delta_y)=\|x-y\|\,.
$$ 
Therefore, $(\R^\dd, \|\cdot\|)$ is isometrically embedded in $(\cP_p(\R^\dd), W_p)$ via $x \mapsto \delta_x$.
\item Fix $x, y \in \R^\dd$ and $0\le \lambda,\tau\le 1$. Then 
$$
W_p\big(\lambda\, \delta_x+ (1-\lambda)\, \delta_y,\, \tau\, \delta_x+ (1-\tau)\, \delta_y\big)=|\lambda-\tau|^{1/p}\,\|x-y\|\,.
$$
\end{enumerate}
\end{example}

Note that it follows from ordering of the $L^p$ norms that $W_p(\mu, \nu) \le W_q(\mu, \nu)$ whenever $p \le q$. In particular, the smallest of the Wasserstein distances is $W_1$.

Wasserstein distances induce a useful topology on random variables: they \emph{metrize weak convergence} on compact spaces; see Appendix~\ref{app:proba} for background. More specifically, a sequence $(\mu_n)_n$ satisfies $W_p(\mu_n,\mu) \to 0$ if and only if it converges weakly to $\mu$, denoted $\mu_n \hookrightarrow \mu$, and the $p$-th moment converges: $\int \|\cdot\|^p \, \ud \mu_n \to \int \|\cdot\|^p \, \ud \mu$.  This ``metrization'' property can be found in all of the main texts on optimal transport and has often been employed as a justification for the use of Wasserstein distance as opposed to other distances. This is hardly a discriminating feature, however, and many other distances (L\'evy--Prokhorov, Fortet--Mourier, etc.) also have this property; see \cite[Chapter~6]{Vil09}. In fact, this folklore result does not do justice to the quantitative meaning of $W_p(\mu, \nu)\le \eps$ for some $\eps$. 

For example, the following statement implies that if two random variables with sufficiently light tails are close in $p$-Wasserstein distance for any $p > 1$, then all of their moments must also be close.
We formalize the assumption that the tails are light by considering \emph{sub-exponential} random variables, that is, random variables $Z$ satisfying 
\begin{equation}
	\E e^{|Z|} \leq 2\,.
\end{equation}
For such random variables, we have the following bound.

\begin{proposition}\label{prop:wass_controls_moments}
Let $X\sim \mu$ and $Y\sim \nu$ be two sub-exponential random variables. Then, for any $p > 1$, there exists a constant $C_p>0$ such that for any integer $\ell \ge 1$, it holds
$$
\big| \E|X|^\ell-\E|Y|^\ell\big| \le (C_p \ell)^{\ell}\, W_p(\mu, \nu)\,.
$$
\end{proposition}
\begin{proof}
By convexity of the function $x \mapsto |x|^\ell$, $\ell \ge 1$, it holds
$$
|X|^\ell-|Y|^\ell \le \ell\, |X-Y|\,(|X|^{\ell-1}\vee |Y|^{\ell-1})\,.
$$
Taking expectation on both sides and applying H\"older's inequality yields for any coupling $(X,Y)$, 
$$
\big| \E|X|^\ell-\E|Y|^\ell\big| \le \ell\, \big(\E|X-Y|^p\big)^{1/p}\,\big(\E(|X|^{\ell-1}\vee |Y|^{\ell-1})^q\big)^{1/q}\,,\;\; \frac1p+\frac1q=1\,.
$$
Taking the optimal coupling between $X$ and $Y$ yields
$$
\big| \E|X|^\ell-\E|Y|^\ell\big| \le\ell\, W_p(\mu, \nu)\,\big(\E(|X|^{\ell-1}\vee |Y|^{\ell-1})^q\big)^{1/q}\,.
$$
To conclude, recall that that it is a standard property of sub-exponential random variables~\cite{Ver18} that if $Z$ is sub-exponential, then $(\E[|Z|^k])^{1/k}\le k$ for all $k \ge1$. Thus
\begin{align*}
\big(\E(|X|^{\ell-1}\vee |Y|^{\ell-1})^q\big)^{\frac{1}{(\ell-1)\,q}}
&\le 2(\ell-1)\,q\le 2{\frac{\ell p}{p-1}}\,.
\end{align*}
\end{proof}

The above result is encouraging: obtaining bounds on the Wasserstein distance between two measures implies quantitative bounds on the distance between their moments. It could be the case, though, that the Wasserstein distance tends to be quite large compared to other commonly used distances or divergences such as total variation or the Kullback--Leibler divergence; see \cite[Chapter~2]{Tsy09} for a list of such distances, their comparison, and relevance to statistical problems.

It turns out, however, that the Wasserstein distance can often be controlled by other commonly used distances.
For example,  the next result shows that on a bounded domain, the Wasserstein distance is dominated by the total variation distance (see Exercise~\ref{qu:tv} for background). Moreover, its proof is our first illustration of how to bound Wasserstein distances---it suffices to exhibit a (suboptimal) coupling $\gamma$ such that $\E_\gamma\|X-Y\|^p$ is controlled appropriately.

\begin{theorem}\label{thm:wpTV}
    Let $\mu, \nu \in \cP_p(\R^\dd)$ be two distributions with densities $f$ and $g$ respectively. Then, for any $p \ge 1$, it holds
    $$
W_p^p(\mu, \nu) \le 2^{p-1} \inf_{x_0 \in \R^\dd}\int \|x-x_0\|^p\, |f(x)-g(x)| \,\ud x \,.
    $$
    In particular, if  the supports of both $\mu$ and $\nu$ are included in the same ball of diameter $D$, then
    $$
W_p^p(\mu, \nu) \le D^p\,d_{\sf TV}(\mu, \nu)\,,
    $$
    where $d_{\sf TV}(\mu, \nu)$ is the total variation distance between $\mu$ and $\nu$ and is defined by
    $$
d_{\sf TV}(\mu, \nu)=\frac12\int |f(x)-g(x)|\, \ud x\,.
    $$
\end{theorem}
\begin{proof}
  Assume that $\mu \neq \nu$ as otherwise the statement is trivial.
  As mentioned before the statement of the theorem, we construct an explicit coupling between $\mu$ and $\nu$. To that end, consider the three positive functions $(f-g)_+$, $(f-g)_-$, and $f\wedge g$ (see Figure~\ref{fig:weighted_tv} for reference) and observe that:

  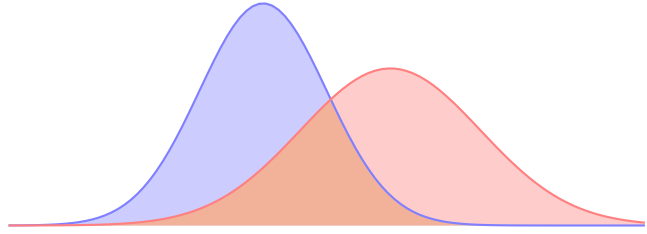
\begin{figure}[ht]
  \centering
\begin{tikzpicture}
\begin{axis}[
  no markers,
  domain=0:10,
  samples=100,
  ymin=0,
  ymax=0.5, 
  height=6cm, 
  width=10cm,
  xtick=\empty,
  ytick=\empty,
  enlargelimits=false,
  clip=false,
  axis on top,
  grid = major,
  axis lines=none, 
  enlarge y limits={upper,value=0.2} 
  ]

  \pgfmathdeclarefunction{gauss}{2}{%
    \pgfmathparse{1/(#2*sqrt(2*pi))*exp(-((x-#1)^2)/(2*#2^2))}
  }
  
  \addplot [fill=blue!40, draw=none, domain=0:10, opacity=0.5] {gauss(4,1)} \closedcycle;

  \addplot [fill=red!40, draw=none, domain=0:10, opacity=0.5] {gauss(6,sqrt(2))} \closedcycle;

  \addplot [fill=orange!60, draw=none, domain=0:10, opacity=0.5] {min(gauss(4,1),gauss(6,sqrt(2)))} \closedcycle;
  
  \addplot [blue!50, thick] {gauss(4,1)};
  \addplot [red!50, thick] {gauss(6,sqrt(2))};
\end{axis}
\end{tikzpicture}
      \caption{The integrals $\int(f-g)_+$, $\int (f-g)_-$, and $\int f\wedge g$ correspond to the blue, red, and orange regions, respectively.}\label{fig:weighted_tv}
  \end{figure}
    $$
\int \bigl((f-g)_+ - (f-g)_-\bigr) = \int f -\int g =1-1=0
    $$
    so that
    $$
\int (f-g)_+= \int(f-g)_- \eqqcolon t>0
    $$
and
$$
\int f\wedge g =\frac12 \left( \int f  + \int g - \int (f-g)_+-\int(f-g)_-\right)=1-t\,.
$$
Next, we normalize these functions to obtain three densities
    $$
     h_+= \frac1t\,(f-g)_+\,, \quad
     h_-= \frac1t\,(f-g)_-\,, \quad
     h_\wedge=\frac1{1-t}\,f\wedge g\,.
    $$
We can rewrite $f$ and $g$ as mixtures of the above densities:
$$
f=th_++ (1-t) h_\wedge\,, \qquad g= th_-+ (1-t) h_\wedge\,.
$$
Next, let $Z_+, Z_-$, and $Z_\wedge$ be three independent random variables with densities $h_+, h_-$, and $h_\wedge$ respectively and let $B$ be a Bernoulli random variable with parameter $t \in (0,1]$, independent of $Z_+, Z_-$, and $Z_\wedge$.

We are now in a position to define our coupling between $\mu$ and $\nu$. To that end, let $(X,Y)$ be a random pair such that
\begin{align*}
    X&=BZ_++(1-B)Z_\wedge\,,\\
    Y&=BZ_-+(1-B)Z_\wedge\,,
\end{align*}
and observe that the distribution $\gamma$ of $(X,Y)$ is indeed a valid coupling between $\mu$ and $\nu$. Using this fact together with the inequality $\|x-y\|^p \le 2^{p-1}\,(\|x-x_0\|^p + \|y -x_0\|^p)$, we get
\begin{align*}
    W_p^p(\mu, \nu)&\le \E_\gamma\|X-Y\|^p\\
    &= \p(B=0)\cdot 0 + \p(B=1)\int \|x-y\|^p\, h_+(x)\,h_-(y)\,\ud x\, \ud y\\
    &\le t2^{p-1} \,\left(\int  \|x-x_0\|^p\, h_+(x)\,\ud x+ \int\|y-x_0\|^p \, h_-(y)\,\ud y \right)\\
    &=t2^{p-1} \,\left(\int  \|x-x_0\|^p\, (h_+(x)+ h_-(x))\,\ud x \right)\\
    &=2^{p-1}\, \left(\int  \|x-x_0\|^p\, \bigl((f-g)_+(x)+ (f-g)_-(x)\bigr)\,\ud x \right)\\    
    &=2^{p-1}\, \left(\int  \|x-x_0\|^p\, |f(x)-g(x)|\,\ud x \right)
    \end{align*}
and the result follows by minimizing the right-hand side with respect to $x_0$.  The second statement follows easily by taking $x_0$ to be the center of said ball.
\end{proof}

The assumption that $\mu$ and $\nu$ are absolutely continuous is superfluous and the exact same proof follows by manipulating measures rather than densities, albeit with slightly more opaque notation; see \cite[Theorem~6.15]{Vil09}.

\section{Optimal transport in one dimension}\label{sec:ot_1d}

To gain a bit of insight into optimal transport, we look at the simpler case where $\mu$ and $\nu$ are probability measures on the real line. In this case, we may define their associated cumulative distribution functions.

Recall that the \emph{cumulative distribution function} (CDF) of a random variables $Z$ is the function $F:\R \to [0,1]$ defined by
$$
F(t) \deq \p(Z \le t)\,,\qquad t \in \R\,.
$$ 
Since $F$ is monotonically non-decreasing, we may define its \emph{pseudo-inverse}  $F^\dagger$ by
$$
F^\dagger(u)=\inf\{t \in \R\,:\, F(t) \ge u\}\,, \qquad u \in [0,1]\,,
$$
with the convention that $\inf \emptyset=\infty$. While $F^\dagger$ is not an inverse per se, it does satisfy the following property:
\begin{equation}\label{EQ:pseudoinverse}
F^\dagger(u) \le t \ \Leftrightarrow \ u \le F(t)
\end{equation}

This property is often used to simulate random variables. Let $U\sim\unif([0,1])$ be a uniform random variable, then $Z\sim F^{\dagger}(U)$ has CDF $F$. Indeed, for any $t \in \R$, 
\begin{equation}\label{EQ:unifCDF}
\p(Z \le t)=\p(F^{\dagger}(U) \le t)=\p(U \le F(t))=F(t)\,.
\end{equation}

The following theorem characterizes optimal transport in one dimension in terms of CDFs.

\begin{theorem}\label{thm:W1_1D}
Let $\mu, \nu \in \cP_1(\R)$ be two probability distributions with  CDFs $F_\mu$ and $F_\nu$ respectively. Let $U\sim\unif([0,1])$ be a uniform random variable and denote by $\bar \gamma$ the distribution of $(F_\mu^{\dagger}(U), F_{\nu}^{\dagger}(U))$. Then $\bar \gamma \in \Gamma_{\mu, \nu}$ is a valid coupling between $\mu$ and $\nu$ and it is optimal:
$$
W_{1}(\mu, \nu)=\int |x-y|\, \bar \gamma(\ud x,\ud y)=\min_{\gamma \in \Gamma_{\mu, \nu}}\int |x-y|\, \gamma(\ud x,\ud y)\,.
$$
Moreover, 
$$
W_{1}(\mu, \nu)=\int_{-\infty}^\infty \big|F_{\mu}(t) -F_{\nu}(t)\big|\,\ud t\,.
$$
\end{theorem}
\begin{proof}
It follows from~\eqref{EQ:unifCDF} that $\bar \gamma \in \Gamma_{\mu, \nu}$ and it remains to check that it is optimal. To that end, observe that for any $\gamma \in \Gamma_{\mu, \nu}$, it follows from Fubini's theorem that for $(X,Y) \sim \gamma$,
\begin{align*}
&\int |x-y|\, \gamma(\ud x,\ud y)
=\iint_{-\infty}^\infty \big(\1_{x \le t < y} + \1_{y \le t < x}\big)\,\ud t\,  \gamma(\ud x,\ud y)\\
&\qquad =\int_{-\infty}^\infty \big(\gamma(X \le t < Y) + \gamma(Y \le t < X)\big)\,\ud t \\
&\qquad =\int_{-\infty}^\infty \big(\gamma(X \le t) + \gamma(Y \le t) -2\gamma(X\le t, Y \le t)\big)\,\ud t \\
&\qquad \ge \int_{-\infty}^\infty \big(F_{\mu}(t) +F_{\nu}(t) -2\,(F_\mu(t)\wedge F_{\nu}(t))\big)\,\ud t\\
&\qquad =\int_{-\infty}^\infty \big|F_{\mu}(t) -F_{\nu}(t)\big|\,\ud t\,.
\end{align*}
To show that the above inequality becomes an equality when $\gamma=\bar \gamma$, observe that
\begin{align*}
\bar \gamma(X\le t, Y \le t)&=\p(F_\mu^{\dagger}(U)\le t, F_{\nu}^{\dagger}(U)\le t)\\
&=\p(U \le F_\mu(t),U \le F_\nu(t))\\
&=\p(U \le F_\mu(t)\wedge  F_\nu(t))\\
&=F_\mu(t)\wedge F_{\nu}(t)\,.
\end{align*}
We have proved
$$
\int |x-y|\, \gamma(\ud x,\ud y) \ge \int_{-\infty}^\infty \big|F_{\mu}(t) -F_{\nu}(t)\big|\,\ud t=\int |x-y|\, \bar\gamma(\ud x,\ud y)
$$
so that $\bar \gamma$ is an optimal coupling.
\end{proof}

If $Z$ admits a density, then its CDF $F$ is actually a left inverse of $F^\dagger$, i.e., $F \circ F^\dagger = \mathrm{Id}$.
If $\mu$ has a density, this fact implies that the optimal coupling $\bar \gamma$ takes the following special form. If $X \sim \mu$, then
\begin{align}\label{eq:1d_opt_coupling}
    (X, F^{\dagger}_\nu\circ F_\mu(X)) \sim \bar \gamma\,.
\end{align}
In other words, the solution to the Monge problem and the Kantorovich problem coincide since we have found a transport \emph{map} $\bar T=F^{-1}_\nu\circ F_\mu$ such that $\bar T_{\#}\mu=\nu$ and
\begin{align*}
    \int |x-\bar T(x)|\,\mu(\ud x)
    &= \min_{\gamma \in \Gamma_{\mu, \nu}}\int |x-y|\, \gamma(\ud x,\ud y) \\
    &= \min_{T: T_{\#}\mu=\nu}\int |x-T(x)|\,\mu(\ud x)\,.
\end{align*}

Although we have focused on the $W_1$ distance in this section, the coupling $\bar\gamma$ given in~\eqref{eq:1d_opt_coupling} turns out to be universally optimal, in the sense that it is optimal for the Kantorovich problem for any strictly convex cost (a cost of the form $c(x,y) = h(x-y)$ where $h : \R\to\R$ is strictly convex); this includes all $W_p$ distances for $p > 1$.
See Exercise~\ref{ex:1d_ot}.

Note that $\bar T$ is a \emph{monotone} increasing function as the composition of two increasing functions. Continuous monotone increasing functions in one dimension are derivatives of convex functions, suggesting that this property may be generalized to higher dimension by considering gradients of convex functions. Existence of such monotone transport maps in higher dimensions is the content of the influential result of Brenier~\cite{Bre87}, which we explore next.

\section{Brenier's theorem}\label{sub:brenier}

The $p$-Wasserstein distance is a natural object for any $p \geq 1$.
However, the cases $p = 1, 2$ possess remarkable special structure, and we focus on them in much of what follows.
We first explore the case $p = 2$, which is notable for its close connection to convex analysis.

Recall that
\begin{equation}\label{W2}\tag{$\mathsf{W_2^2}$}
W_2^2(\mu, \nu)=\min_{\gamma \in \Gamma_{\mu, \nu}}\int \|x-y\|^2\, \gamma(\ud x,\ud y)\,.
\end{equation}

\begin{theorem}[Brenier]\label{thm:brenier}\index{Brenier's theorem}
Let $\mu, \nu \in \cP_2(\R^\dd)$ be two probability measures such that $\mu$ has a density and let $X \sim \mu$. If $\bar \gamma$ is an optimal coupling for~\eqref{W2},
$$
\int \|x-y\|^2\, \bar \gamma(\ud x,\ud y)=\min_{\gamma \in \Gamma_{\mu, \nu}}\int \|x-y\|^2 \,\gamma(\ud x,\ud y)=W_2^2(\mu, \nu)\,,
$$
then there exists a convex function $\varphi:\R^\dd \to \R$ such that $(X, \nabla \varphi(X))\sim \bar \gamma \in \Gamma_{\mu, \nu}$.
\end{theorem}

The crucial assumption of Brenier's theorem is that $\mu$ has a density, that is, that $\mu$ is absolutely continuous with respect to the Lebesgue measure.
In particular, since convex functions are differentiable Lebesgue-almost everywhere, this assumption guarantees that $\nabla \varphi(X)$ makes sense.

Before turning to the proof, we first consider a first statistical implication of Brenier's theorem.
Brenier's theorem asserts that, as long as $\mu$ has a density, for any $\nu \in \cP_2(\R^\dd)$, there exists a convex function $\phi$ so that $\nabla \phi_\#\mu = \nu$.
Since gradients of convex functions are natural analogues of monotone functions in higher dimensions, this theorem is therefore a significant generalization of the classical univariate fact mentioned in Section~\ref{sec:ot_1d}, that if $U \sim \mathsf{Unif}([0,1])$, then $F^\dagger_\nu(U) \sim \nu$.

In one dimension, the function $F^\dagger_\nu$ is known as the quantile function of $\nu$, and is of fundamental statistical significance.
Brenier's theorem therefore can be used to define a multivariate notion of quantiles~\cite{CheGalHalHen17,HalBarCueMat21}.
This point of view has proven to be extremely fruitful and has led to a wide range of statistical applications~\cite{Hal22}.
(For more details, see the discussion section.)

Returning to the content of Brenier's theorem, at this point it is not obvious what optimal transport has to do with gradients of convex functions.
We therefore begin by studying such gradients to gain intuition.

\subsection{Gradients of convex functions}\label{sec:grad_cvx_fn}

Note first that a continuous function $f:\R\to \R$ is such that $f=\varphi'$ for some differentiable convex function $\varphi$ if and only if $f$ is non-decreasing. Indeed, convexity of $\varphi$ implies that for any $x,y \in \R$:
\begin{align}
\varphi(x)-\varphi(y)&\le (x-y)\,\varphi'(x)\,, \label{EQ:2point1}\\
\varphi(y)-\varphi(x)& \le (y-x)\,\varphi'(y)\,.\label{EQ:2point2}
\end{align}
Summing the above two inequalities yields
$$
(x-y)\,(\varphi'(x)-\varphi'(y))\ge 0\,,
$$
so that $\varphi'$ is non-decreasing. 

Is there an analogue of this statement for functions on $\R^\dd$? Of course we immediately get that for any $x,y \in \R^\dd$
$$
\langle x-y, \tilde \nabla \varphi(x)-\tilde \nabla \varphi(y)\rangle\ge 0\,,
$$
where $\tilde \nabla \varphi(x) \in \partial \varphi(x)$ denotes a subgradient of $\varphi$ at $x$ (see Appendix~\ref{app:convex} for preliminaries on convex analysis). Unfortunately, while in dimension 1, the two-point inequalities~\eqref{EQ:2point1}--\eqref{EQ:2point2} imply inequalities for any arrangement of points, in higher dimension this is no longer the case and we need to capture additional information.
 
In fact, convexity implies many such inequalities: for any integer $k\ge 2$, and any collection of points $x_1, \ldots, x_k \in \R^\dd$, we have
\begin{align*}
\varphi(x_i)-\varphi(x_{i+1}) &\le \langle x_i-x_{i+1},\tilde \nabla \varphi(x_{i})\rangle\,, \qquad i=1, \ldots, k-1\,,\\
\varphi(x_{k})-\varphi(x_{1}) &\le \langle x_{k}-x_1,\tilde \nabla \varphi(x_{k})\rangle\,.
\end{align*}
Summing these inequalities yields:
\begin{equation}
\label{EQ:gradcycm}
\sum_{i=1}^{k}\langle x_i-x_{i+1},\tilde \nabla \varphi(x_{i})\rangle\ge 0\,,
\end{equation}
with the convention that $x_{k+1}=x_1$.\footnote{With this convention, the points $x_1 \to x_2 \to \dots \to x_k \to x_1$ form a ``cycle.''}

\subsection{Cyclical monotonicity}

Since there may exist several points in the subdifferential of $\varphi$ at $x$, we first  describe the graph $\{(x,\tilde \nabla \varphi(x))\,:\, x\in \R^d\}$ before thinking about $\tilde \nabla \varphi (\cdot)$ as a map from $\R^d$ to $\R^d$. We first define an important property of such graphs.

\begin{definition}
A set $A \subset \R^\dd\times \R^\dd$ is said to be \emph{cyclically monotone}\index{cyclical monotonicity} if for any integer $k \ge 2$, and points $(x_i, y_i) \in A$, $i=1, \ldots, k$, it holds
\begin{equation}
\label{EQ:CM}
\sum_{i=1}^{k}\langle x_i-x_{i+1}, y_{i}\rangle\ge 0\,,
\end{equation}
with the convention that $x_{k+1}=x_1$.
\end{definition}

In light of~\eqref{EQ:gradcycm}, the set $\partial \varphi \subset \R^\dd\times \R^\dd$ is cyclically monotone whenever $\varphi$ is convex. It turns out that all cyclically monotone subsets of $\R^\dd\times \R^\dd$ are of this form.

\begin{theorem}[Rockafellar]\label{thm:rock}\index{Rockafellar's theorem}
A set $A \subset \R^d\times \R^d$ is cyclically monotone if and only if there exists a closed convex function $\varphi:\R^d \to \R\cup\{\infty\}$ such that
$$
A\subseteq \partial\varphi\,.
$$
\end{theorem}

The proof of this classical theorem of convex analysis can be found in Appendix~\ref{app:convex}.

Note that  condition~\eqref{EQ:CM} is equivalent to the requirement that
\begin{align}\label{eq:cyclically_monotone}
    \sum_{i=1}^k\|x_{i}-y_i\|^2 \le \sum_{i=1}^k\|x_{i+1}-y_i\|^2
\end{align}
for any points $(x_i, y_i) \in A$, $i = 1, \dots, k$.
This formulation enables us to see the connection with optimal transport. 
Indeed, consider the following example for illustration purposes. Let 
$$
\mu=\frac{1}{n}\sum_{j=1}^n\delta_{a_j}\,,\qquad \nu=\frac{1}{n}\sum_{j=1}^n\delta_{b_j}\,.
$$

In this discrete case, the set of all couplings between $\mu$ and $\nu$ can be identified with the Birkhoff polytope (see Section \ref{sec:discreteOT}), whose extreme points are rescaled permutation matrices.
Since the discrete optimal transport problem is a linear program, solutions can be taken to be extreme points.
We may therefore restrict our attention to couplings given by permutations.
Concretely, a permutation $\sigma$ of $\{1, \dots, n\}$ corresponds to the coupling
\begin{equation*}
    \gamma = \frac 1n \sum_{j=1}^n \delta_{(a_{j}, b_{\sigma(j)})}\,.
\end{equation*}
Such a coupling is optimal if its cost is minimal among all permutations, that is, if
\begin{equation}
  \label{eq:discrete_opt}
\sum_{j=1}^n\|a_{j}-b_{\sigma(j)}\|^2 \le \sum_{j=1}^n\|a_{j}-b_{\tau(j)}\|^2\,, \qquad \forall \tau\,.  
\end{equation}

This condition is precisely equivalent to the support of $\gamma$ being cyclically monotone.
Indeed, by relabeling the atoms of $\nu$, we may assume without loss of generality that $\sigma$ is the identity permutation.
Then the support of $\gamma$ consists of the pairs $(a_{j}, b_j)$, $j \in \{1, \dots, n\}$.
Given any subset of $k$ distinct points $(x_i, y_i) = (a_{j_i}, b_{j_i}) \in \supp(\gamma)$, $i = 1, \dots, k$, let $\tau$  be the cyclic permutation of $\{j_1, \ldots, j_k\}$ that leaves other indices unchanged
\begin{equation*}
\tau(j) = 
    \begin{cases}
        j_{i-1} & \text{if $j = j_i$, $i \in \{2, \dotsc, k\}$} \\
        j_k & \text{if $j = j_1$} \\
        j & \text{otherwise.}
    \end{cases}
\end{equation*}
Then~\eqref{eq:discrete_opt} implies~\eqref{eq:cyclically_monotone}.
In fact, since~\emph{any} permutation $\tau$ can be decomposed into cycles, similar reasoning then shows that~\eqref{eq:cyclically_monotone} is also a sufficient condition for~\eqref{eq:discrete_opt} to hold.

The preceding discussion indicates that, in the discrete case, the support of an optimal coupling is cyclically monotone.
A similar phenomenon holds in the general case; however, the argument given above is not valid when $\gamma$ does not assign positive mass to points in its support.
Nevertheless, the following result shows that a similar strategy can be made to work by reasoning about small neighborhoods (e.g., balls) rather than individual points.
Some care is required to ensure that it is possible to modify $\gamma$ on such neighborhoods while maintaining the constraint $\gamma \in \Gamma_{\mu, \nu}$.

\begin{proposition}\label{prop:suppCM}
Let  $\bar \gamma \in \Gamma_{\mu, \nu}$ be an optimal coupling between $\mu$ and $\nu$ in the sense that
$$
\int \|x-y\|^2\, \bar \gamma(\ud x,\ud y)=\min_{\gamma \in \Gamma_{\mu, \nu}}\int \|x-y\|^2\, \gamma(\ud x,\ud y)=W_2^2(\mu, \nu)\,.
$$
Then $\supp(\bar \gamma)$ is cyclically monotone.
\end{proposition}
\begin{proof}
Suppose that $S \deq \supp(\bar \gamma)$ is \emph{not} cyclically monotone. Then there exists $k \ge 2$ and $(x_i,y_i) \in S$, $i=1, \ldots, k$ such that
$$
\sum_{i=1}^k\|x_{i}-y_i\|^2 > \sum_{i=1}^k\|x_{i+1}-y_i\|^2\,,
$$
and by continuity of the Euclidean norm, there exist neighborhoods $U_i, V_i$ of $x_i, y_i$ respectively for $i=1, \ldots, k$ such that $\bar \gamma (U_i \times V_i)>0$ and
\begin{equation}\label{EQ:noCM}
\sum_{i=1}^k\|\tilde x_{i}-\tilde y_i\|^2 > \sum_{i=1}^k\|\tilde x_{i+1}'-\tilde y_i'\|^2 \,,
\end{equation}
for all $\tilde x_i, \tilde x_i' \in U_i, \tilde y_i, \tilde y_i' \in V_i$, $i=1, \ldots, k$. 

Now, let $\gamma_i, i=1 \ldots, k$ be a family of (conditional) probability distributions on $\R^\dd\times \R^\dd$ defined such that $\gamma_i(A)=\bar \gamma(A \mid U_i \times V_i)$ for any Borel set $A \subset \R^\dd\times \R^\dd$. Next, let $\gamma_i^{(1)}$ and $\gamma_i^{(2)}$ denote the first and second marginal of $\gamma_i$ respectively and define the mixture:
$$
\gamma  = \bar \gamma + \frac{c}{k}\sum_{i=1}^k(\gamma_{i+1}^{(1)} \otimes \gamma_i^{(2)} - \gamma_i)\,,
$$
where $c>0$ is to be chosen later and with the convention that $\gamma_{k+1}=\gamma_1$.

Note that for any Borel set $A \subset \R^\dd\times \R^\dd$, it holds
\begin{align*}
\gamma(A)&\ge \bar \gamma(A) - \frac{c}{k}\sum_{i=1}^k \gamma_i(A)\\
&=\bar \gamma(A) - \frac{c}{k}\sum_{i=1}^k \bar \gamma(A\mid U_i \times V_i)\\
&=\bar \gamma(A) - \frac{c}{k}\sum_{i=1}^k \frac{\bar \gamma(A\cap (U_i \times V_i))}{\bar \gamma( U_i \times V_i)}\\
&\ge \bar \gamma(A) - \frac{c\bar \gamma(A)}{k}\sum_{i=1}^k \frac{1}{\bar \gamma( U_i \times V_i)}\,.
\end{align*}
Thus $\gamma(A)\ge 0$ if $c < \min_{i\in[k]} \bar \gamma( U_i \times V_i)$.  Moreover, $\gamma(\R^\dd \times \R^\dd)=1$ so that $\gamma$ is indeed a probability distribution over $\R^\dd \times \R^\dd$. 

To check that $\gamma \in \Gamma_{\mu,\nu}$ observe that for any Borel set $B \subset \R^\dd$,
\begin{align*}
\gamma(B\times \R^\dd)&=\mu(B)+\frac{c}{k}\sum_{i=1}^k(\gamma_{i+1}^{(1)}(B) - \gamma_i(B \times \R^\dd))\\
&=\mu(B)+\frac{c}{k}\sum_{i=1}^k(\gamma_{i+1}^{(1)}(B) - \gamma_i^{(1)}(B))\\
&=\mu(B)+\frac{c}{k}\,(\gamma_{k+1}^{(1)}(B)-\gamma_1^{(1)}(B))=\mu(B)\,.
\end{align*}
Similarly
\begin{align*}
\gamma(\R^\dd\times B)&=\nu(B)+\frac{c}{k}\sum_{i=1}^k(\gamma_{i}^{(2)}(B) - \gamma_i(\R^\dd\times B))\\
&=\nu(B)+\frac{c}{k}\sum_{i=1}^k(\gamma_{i}^{(2)}(B) - \gamma_i^{(2)}(B) )=\nu(B)\,.
\end{align*}
Next observe that
\begin{align*}
\int \|x-y\|^2\, & \gamma(\ud x,\ud y)-\int \|x-y\|^2\, \bar \gamma(\ud x,\ud y)\\
&=\frac{c}{k}\sum_{i=1}^k\biggl(\int_{U_{i+1}\times V_i}\|x-y\|^2\,\gamma_{i+1}^{(1)}(\ud x)\, \gamma_{i}^{(2)}(\ud y)\\
&\qquad\qquad\qquad\qquad{} -\int_{U_i\times V_i}\|x-y\|^2\,\gamma_i(\ud x,\ud y)\biggr)\\
&<0\,,
\end{align*}
by~\eqref{EQ:noCM}. This contradicts optimality of $\bar \gamma$. 
\end{proof}

\subsection{Proof of Brenier's theorem}

We are now in a position to prove Brenier's theorem. 

Let $\bar \gamma$ be an optimal coupling. In light of Proposition~\ref{prop:suppCM}, $\supp(\bar \gamma)$ is cyclically monotone. By Rockafeller's Theorem~\ref{thm:rock}, this implies that there exists a convex function $\varphi :\R^\dd \to \R \cup \{\infty\}$ such that $\bar \gamma(Y \in \partial \varphi(X))=1$. But since $\varphi$ is convex, it is locally Lipschitz on the interior of its domain, and the boundary of the domain has measure zero by convexity. Hence, $\varphi$ is almost everywhere differentiable with respect to the Lebesgue measure over its domain by Rademacher's theorem. Since $\mu$ has a density, this implies that $\varphi$ is differentiable $\mu$ almost everywhere. Therefore, $\bar \gamma(Y = \nabla \varphi(X))=1$ or in other words, if $X \sim \mu$, then $(X, \nabla\varphi(X)) \sim \bar \gamma$.

\section{Kantorovich duality}\label{sec:kantorovich_duality}

Brenier's theorem shows that an optimal coupling for \eqref{W2} if $\mu$ has a density is a deterministic coupling given by the gradient of a convex function $\varphi$.
This result raises the question of whether it is possible to solve an optimization problem to find $\varphi$ directly, or whether it is possible to certify that a convex function $\varphi$ corresponds to an optimal coupling.
These questions can be answered by employing tools from convex duality.

In the fully discrete setting (see Section~\ref{sec:discreteOT}),~\eqref{W2} is a linear program or LP (linear objective \& linear constraints), which admits a useful theory of duality.
This intuition carries over to the general setting (and the link can be made precise through approximation arguments, see~\cite[Chapter 11]{Dud02}).
In fact, it was through optimal transport that Kantorovich actually introduced LP duality, which has furnished algorithmic advances continuously since its inception.

\subsection{The dual Kantorovich problem}\label{subsec:dualK}

The dual problem to~\eqref{W2} is a maximization problem. To find its expression, encode the constraint $\gamma \in \Gamma_{\mu, \nu}$ as
\begin{align*}
    &\sup_{f,g \in \Cb}\left\{ \int f(x)\, \mu(\ud x) + \int g(y) \,\nu(\ud y) - \int \big(f(x)+g(y)\big)\, \gamma(\ud x,\ud y)\right\} \\
    &\qquad =\begin{cases}
    0\,, & \text{if } \gamma \in \Gamma_{\mu, \nu}\,,\\
    \infty\,,  & \text{otherwise}\,,
    \end{cases}
\end{align*}
where the supremum is taken over the set $\Cb$ of continuous and bounded functions over $\R^\dd$. Thus,~\eqref{W2} is equivalent to
\begin{align*}
    &\inf_{\gamma \in \cM_+} \biggl\{\int \|x-y\|^2\, \gamma(\ud x,\ud y) \\
    &\,\,\, + \sup_{f,g \in \Cb} \int f(x)\, \mu(\ud x) + \int g(y)\, \nu(\ud y) - \int \big(f(x)+g(y)\big)\, \gamma(\ud x,\ud y)\biggr\}\,,
\end{align*}
where the infimum is taken over the set $\cM_+$ of all positive measures on $\R^\dd \times \R^\dd$ (unrestricted). Note that for $\gamma\notin \Gamma_{\mu,\nu}$ this new objective is infinite so the problem is strictly equivalent.

Next, we switch the $\inf$ and $\sup$ to get the following lower bound on the value of~\eqref{W2}:
\begin{equation}\label{EQ:kot2}
\begin{aligned}
&\sup_{f,g\in \Cb}  \biggl\{ \int f(x)\, \mu(\ud x) + \int g(y)\, \nu(\ud y) \\
&\qquad\qquad\qquad{} + \inf_{\gamma \in \cM_+}\left\{\int  \big(\|x-y\|^2- f(x)-g(y)\big) \,\gamma(\ud x,\ud y)\right\}\biggr\}
\end{aligned}
\end{equation}
Next observe that since $\gamma$ is a positive measure, it holds,
\begin{align*}
&\inf_{\gamma \in \cM_+}\left\{\int  \big(\|x-y\|^2- f(x)-g(y)\big) \,\gamma(\ud x,\ud y)\right\}\\
&\qquad =
\begin{cases}
0\,, & \text{if } f(x)+g(y)\le \|x-y\|^2\,,\ \forall\, x, y\in \R^\dd\,,\\
-\infty\,, & \text{otherwise}\,.
\end{cases}
\end{align*}
Indeed, if there exists a pair $(x,y)$ that violates the above constraint then we can take the sequence of measures $\gamma_n=n \delta_{(x,y)}$ and the integral would converge to $-\infty$. 

Hence we have shown that~\eqref{W2} is bounded below by
$$
\sup_{\substack{ f ,  g\in \Cb\\f(x) +g(y)\le \|x-y\|^2} } \biggl\{ \int f\, \ud\mu + \int g(y)\, \ud\nu\biggr\}\,.
$$

Though the choice $f, g \in \Cb$ above is motivated by reference to the weak topology on the set of probability measures, this lower bound on $W_2^2(\mu, \nu)$ holds for any pair $(f, g)$ of integrable functions satisfying $f(x) + g(y) \leq \|x - y\|^2$ almost everywhere.
\begin{lemma}\label{lem:dualK}
Let $\mu, \nu \in \cP_2(\R^\dd)$, then
\begin{align*}
    W_2^2(\mu,\nu)
    &=\inf_{\gamma \in \Gamma_{\mu, \nu}} \int \|x-y\|^2\, \gamma(\ud x,\ud y) \\
    &\ge \sup_{\substack{ f \in L^1(\mu),\, g\in L^1(\nu)\\f(x) +g(y)\le \|x-y\|^2} }\left\{\int f\, \ud \mu + \int g\, \ud \nu\right\}\,.
\end{align*}
\end{lemma}
\begin{proof}
Let $f \in L^1(\mu), g\in L^1(\nu)$ be such that $f(x) +g(y)\le \|x-y\|^2$ for $\mu$-a.e.\ $x$, $\nu$-a.e.\ $y$, and fix $\gamma \in \Gamma_{\mu, \nu}$. Then
\begin{align*}
    \int f(x)\, \mu(\ud x) + \int g(y)\, \nu(\ud y)
    &=\int\big(f(x)+g(y)\big)\,\gamma(\ud x,\ud y) \\
    &\le \int \|x-y\|^2\, \gamma(\ud x,\ud y)\,.
\end{align*}
The proof follows by taking the supremum on the left-hand side and the infimum on the right-hand side.
\end{proof}

The dual Kantorovich problem is given by
\begin{equation}
\label{DK}\tag{\textsf{D-}$\mathsf{W_2^2}$}
\sup_{\substack{ f \in L^1(\mu),\, g\in L^1(\nu)\\f(x) +g(y)\le \|x-y\|^2} }\left\{\int f\, \ud \mu + \int g\, \ud \nu\right\}\,.
\end{equation}
It is the dual problem to the \emph{primal} problem~\eqref{W2}.

In particular, Lemma~\ref{lem:dualK} describes a phenomenon known as \emph{weak} duality, in which the dual is only shown to be a lower bound on the primal problem. This terminology is to be contrasted with \emph{strong} duality, where the inequality becomes an equality so that the primal and dual objectives take the same optimal value. While strong duality is, strictly speaking, only a statement about objective values, it is often the case that the solutions to the primal and dual problems are related to each other; see~\cite[Chapter~5]{BoyVan04} for a treatment of duality in the context of convex optimization.

We show in Subsection~\ref{subsec:fundamental_ot} that strong duality in fact holds and it leads to important consequences for our problem of interest.

\subsection{The semidual}\label{subsec:semidual}\index{semidual}

 Before moving to strong duality, we make a quick detour to define the \emph{semidual} problem, a partially solved version of the dual problem~\eqref{DK}. It plays an important role in the estimation of optimal transport maps (Chapter~\ref{chap:maps}).

Consider~\eqref{DK} and suppose that we hold the first dual potential $f$ fixed; given this choice of $f$, what is the optimal choice of $g$?
Since the dual problem is a supremum, we want to make $g$ as large as possible, but we must respect the constraint $f(x) + g(y) \le \|x-y\|^2$.
The optimal function $g$ is therefore given by
\begin{align}\label{eq:c_conjugate}
    g(y) = \inf_{x\in\R^d}\{\|x-y\|^2 - f(x)\}\,.
\end{align}
The function defined in~\eqref{eq:c_conjugate} is called the \emph{$c$-conjugate} or \emph{$c$-transform} of $f$, denoted $f^c$, associated with the cost $c(x,y) = \|x-y\|^2$.
This reasoning shows that we can reformulate the dual as
\begin{align}
    \eqref{DK}
    &= \sup_{f\in L^1(\mu)}\biggl\{\int f\,\ud \mu + \int f^c\,\ud \nu\biggr\}\,.\label{eq:general_semidual}
\end{align}
This is a version of the \emph{semidual} problem, and it is applicable to optimal transport for any cost function $c$ provided that we replace $\|x-y\|^2$ with $c(x,y)$ in~\eqref{eq:c_conjugate}.

However, for the quadratic cost, we can go one step further and explicitly link the semidual with convex analysis. In this case, the semidual is given by
\begin{equation}\label{SD}\tag{\textsf{SD}}
       \inf_{{\phi \in L^1(\mu)}} {\biggl\{\int \phi\,\ud \mu + \int \phi^* \,\ud \nu\biggr\}}
    \end{equation}
where $\phi^*$ denotes the \emph{convex conjugate} of  $\phi$; see Appendix~\ref{app:convex}. 

\begin{proposition}\label{prop:semidual} 
   Let $\mu,\nu\in \cP_2(\R^\dd)$ be probability measures.
   Then, the dual problem~\eqref{DK} is equivalent to the semidual problem~\eqref{SD} in the following sense:
    \begin{enumerate}
        \item {\bf Objective values:} Write ${\sf S}$ and ${\sf D}$ for the optimal objective values of~\eqref{SD} and \eqref{DK} respectively. Then 
        $$
        {\sf D}= \int\|\cdot\|^2 \, \ud \mu + \int \|\cdot\|^2 \, \ud \nu - 2\cdot{\sf S}\,.
        $$
        \item {\bf Solutions:} A pair of functions $(f,g)$ is optimal for~\eqref{DK} if and only if $f = \|\cdot\|^2 - 2\varphi$ and $g = \|\cdot\|^2 - 2\varphi^*$ where $\varphi$ is optimal for~\eqref{SD}.
    \end{enumerate}
\end{proposition}
\begin{proof}
    Let us write $f = \|\cdot\|^2 - 2\phi$ and $g = \|\cdot\|^2 - 2\psi$; this is simply a reparametrization of the dual potentials.
    Then,
    \begin{align*}
        \int f\,\ud \mu + \int g\,\ud \nu
        &= \int \|\cdot\|^2 \,\ud \mu + \int \|\cdot\|^2\,\ud \nu - 2 \,\Bigl( \int \phi\,\ud \mu + \int \psi \,\ud \nu\Bigr)\,.
    \end{align*}
    The constraint $f(x) + g(y) \le \|x-y\|^2$ translates into
    \begin{align*}
        \|x\|^2 - 2\phi(x) + \|y\|^2 - 2\psi(y) \le \|x-y\|^2
        \Leftrightarrow \phi(x) + \psi(y) \ge \langle x,y\rangle\,.
    \end{align*}
    Hence,~\eqref{DK} is equivalent to
    \begin{align*}
        \inf_{\substack{\phi \in L^1(\mu),\, \psi\in L^1(\nu) \\ \phi(x) + \psi(y) \ge \langle x,y\rangle}} \biggl\{\int \phi\,\ud \mu + \int \psi\,\ud \nu\biggr\}\,.
    \end{align*}

    Next, let us apply the same trick as described above: for fixed $\phi$, the optimal choice of $\psi$ obeying the constraint is given by
    \begin{align*}
        \psi(y) = \sup_{x\in\R^d}\{\langle x,y\rangle - \phi(x)\}\,,
    \end{align*}
    which is precisely the definition of the convex conjugate $\phi^*$.
    Substituting this in yields the equivalence.
\end{proof}

In the preceding proof, we showed that for fixed $\phi_0$, the optimal choice of $\psi$ is $\psi = \phi_0^*$.
Due to the symmetry of the problem, for fixed $\psi = \phi_0^*$, the optimal choice of $\phi$ is then $\psi^* = \phi_0^{**}$.
One could imagine iterating this process, obtaining better and better dual potentials, but actually the process halts here.
Since $\phi_0^*$ is a closed convex function, it is self-dual, so that $\phi_0^{***} = \phi_0^*$; see Appendix~\ref{app:convex}.
In the end, this argument shows that the optimal potential $\varphi$ in~\eqref{SD} can be taken to be a closed convex function.

Thus far, we have seen two convex functions $\varphi$ arise from the optimal transport problem.
From the primal standpoint, Brenier's Theorem~\ref{thm:brenier} shows that the optimal transport plan is supported on the subdifferential of a convex function.
From the dual standpoint, a minimizer of~\eqref{SD} can be taken to be convex.
In the next section, we show that these two convex functions are one and the same.

\subsection{The fundamental theorem of optimal transport}\label{subsec:fundamental_ot}

Recall from Brenier's theorem that if a measure $\mu$ has a density then any optimal coupling between $\mu$ and $\nu$ is supported on the graph of the gradient of a convex function. It turns out that the converse holds: any coupling $\gamma \in \Gamma_{\mu, \nu}$ supported on the graph of the gradient of a convex function has to be optimal. This equivalence follows from the fundamental theorem of optimal transport stated below. In fact, this theorem contains another fundamental result about strong duality between the primal problem~\eqref{W2} and its dual~\eqref{DK} which is the key to establishing this equivalence.

\begin{theorem}[Fundamental theorem of optimal transport]\index{fundamental theorem of optimal transport}\label{thm:fundOT}
Let $\mu, \nu \in \cP_2(\R^\dd)$ be two probability measures such that $\mu$ has a density and let $X \sim \mu$. Then the following are equivalent:
\begin{enumerate}[label=(\roman*)]
\item $\bar \gamma \in \Gamma_{\mu,\nu}$ is an optimal coupling in the sense that:
$$
\int \|x-y\|^2 \,\bar \gamma(\ud x,\ud y)=W_2^2(\mu, \nu)\,.
$$
\item There exists a proper convex function $\varphi$ such that $(X, \nabla \varphi(X))\sim \bar \gamma \in \Gamma_{\mu, \nu}$.
\item Strong duality holds between~\eqref{W2} and~\eqref{DK}:
$$
\int \|x-y\|^2 \,\bar \gamma(\ud x,\ud y) = \sup_{\substack{ f \in L^1(\mu), \,g\in L^1(\nu)\\f(x) +g(y)\le \|x-y\|^2} }\left\{\int f \,\ud \mu + \int g\, \ud \nu\right\}\,.
$$
Moreover, the above supremum is achieved for 
$$
\bar f(x) \deq \|x\|^2-2\varphi(x)\qquad \text{and} \qquad \bar g(y) \deq \|y\|^2-2\varphi^*(y)\,.
$$
\end{enumerate}
\end{theorem}
\begin{proof}
We have already proved that $(i)\Rightarrow(ii)$ in Subsection~\ref{sub:brenier} so it remains to prove that $(ii) \Rightarrow (iii)$ and $(iii) \Rightarrow (i)$.

We first prove that $(ii) \Rightarrow (iii)$. To that end, observe that for $\mu$ almost every $x$
$$
\|x-\nabla \varphi(x)\|^2=\|x\|^2+\|\nabla \varphi(x)\|^2-2\,\langle x, \nabla \varphi(x)\rangle\,.
$$
Moreover, the convex conjugate $\varphi^*$ of $\varphi$ satisfies for $\mu$ almost every $x$,
$$
\varphi(x)+\varphi^*(\nabla\varphi(x))=\langle \nabla \varphi(x), x \rangle\,.
$$
This is the optimality condition for the Fenchel{--}Young inequality (Theorem~\ref{thm:fenchel_young}).
The above two displays yield
$$
\|x-\nabla \varphi(x)\|^2=\underbrace{\|x\|^2-2\varphi(x)}_{\bar f(x)} +\underbrace{\|\nabla \varphi(x)\|^2-2\varphi^*(\nabla \varphi(x))}_{\bar g(\nabla \varphi(x))}\,.
$$
Integrating with respect to $\mu$ yields
\begin{align}
\int \|x-y\|^2\,\bar \gamma(\ud x,\ud y) &=\int\|x-\nabla \varphi(x)\|^2\, \mu(\ud x)\nonumber\\
&= \int \bar f(x)\, \mu(\ud x) + \int \bar g(\nabla \varphi(x))\, \mu(\ud x)\nonumber \\
&= \int \bar f(x)\, \mu(\ud x) + \int \bar g(y)\, \nu(\ud y)\,.\label{EQ:prfundOT1}
\end{align}
We now check that the pair $(\bar f,\bar g)$ satisfies the constraints of~\eqref{DK}. It follows from the Fenchel--Young inequality (Theorem~\ref{thm:fenchel_young}) that for any $x,y \in \R^\dd$, 
\begin{align}
\bar f(x)+\bar g(y)&=\|x\|^2+\|y\|^2-2\,\big(\varphi(x)+\varphi^*(y)\big) \nonumber\\
&\le \|x\|^2+\|y\|^2-2\,\langle x,y\rangle=\|x-y\|^2\,.\label{EQ:fenchelyoung}
\end{align}
To check integrability, note that from the definition of convex conjugation, $\varphi=\varphi^{**}$ and $\varphi^*$ are both lower bounded by affine functions. Therefore, since $\mu,\nu \in \cP_2(\R^\dd)$ we have that $f_+\in L^1(\mu)$ and $g_+ \in L^1(\nu)$. Moreover,~\eqref{EQ:prfundOT1} yields $\int f\,\ud \mu+ \int g\, \ud \nu \ge 0$ 
so that $\int f\,\ud \mu>-\infty$ and $\int g\, \ud \nu >-\infty$. It yields 
$$
\int |f|\,\ud \mu=2\int f_+\,\ud \mu-\int f\,\ud \mu<\infty
$$
so that $f \in L^1(\mu)$ and similarly $g \in L^1(\nu)$. This completes the proof of $(ii) \Rightarrow (iii)$.

We now turn to the proof of  $(iii) \Rightarrow (i)$. We have by~\eqref{EQ:prfundOT1} that for any $\gamma \in \Gamma_{\mu, \nu}$
\begin{align*}
\int \|x-y\|^2\, \bar \gamma(\ud x,\ud y) &=\int \bar f\, \ud \mu + \int \bar g\, \ud \nu\\
&=\int \big(\bar f(x)+\bar g(y))\,\gamma(\ud x,\ud y) \\
&\le \int \|x-y\|^2\, \gamma(\ud x,\ud y)\,,
\end{align*}
where in the last inequality, we used~\eqref{EQ:fenchelyoung}. Therefore, $\bar \gamma$ is an optimal coupling and $(i)$ follows.
\end{proof}

Some implications of this theorem are still true even when $\mu$ does not have a density. In particular, it can be shown using these tools that strong duality still holds, and moreover that the converse of Proposition~\ref{prop:suppCM} holds: for any two measures $\mu, \nu \in \cP_2(\R^\dd)$, if $\gamma \in \Gamma_{\mu, \nu}$ is supported on a cyclically monotone set, then it must be an optimal coupling (see~\cite{AmbGig13} for example).

The optimal $f$ and $g$ arising in Theorem~\ref{thm:fundOT} play a central role in the sequel.

\begin{definition}[Kantorovich potentials]\label{def:potentials}\index{Kantorovich potentials}
The functions
$$
\bar f(x)=\|x\|^2-2\varphi(x)\qquad \text{and} \qquad \bar g(y)=\|y\|^2-2\varphi^*(y)
$$
that realize the optimum of the dual Kantorovich problem~\eqref{DK} are called \emph{Kantorovich potentials} for the pair $(\mu, \nu)$.
\end{definition}

Even though \emph{a priori} solutions to~\eqref{DK} are only defined almost everywhere, $\bar f$ and $\bar g$ are \emph{bona fide} functions defined everywhere on $\R^d$.
Note that by symmetry of~\eqref{DK} and the form of the Kantorovich potentials, it is easy to check that if $\nu$ admits a density, then $\nabla \varphi^*$ is an optimal transport map from $\nu$ to $\mu$.

\subsection{An improved Brenier theorem}

With the fundamental theorem, we can state an improved version of Brenier's theorem, which is often useful. 

\begin{theorem}[Improved Brenier]\label{thm:improvedBrenier}\index{Brenier's theorem!improved}
Let $\mu, \nu \in \cP_2(\R^\dd)$ be two probability measures such that $\mu$ has a density and let $X \sim \mu$. Then there exists a convex function $\varphi:\R^\dd \to \R$ such that $(X, \nabla \varphi(X))\sim \bar \gamma \in \Gamma_{\mu, \nu}$ and $\bar \gamma$ is an optimal coupling for~\eqref{W2}:
$$
\int \|x-y\|^2 \,\bar \gamma(\ud x,\ud y)=\min_{\gamma \in \Gamma_{\mu, \nu}}\int \|x-y\|^2\, \gamma(\ud x,\ud y)=W_2^2(\mu, \nu)\,.
$$
Moreover, $\nabla \varphi$ is unique in the sense that if there exists a convex function $\psi$ such that $\nabla\psi(X) \sim \nu$, then $\nabla \psi(X) = \nabla \varphi(X)$, almost surely.

In particular, any valid coupling $\gamma \in \Gamma_{\mu, \nu}$ of the form $(X, \nabla \psi(X))\sim \gamma$ for some convex function $\psi$, must be the \emph{unique} optimal coupling between $\mu$ and $\nu$. 
\end{theorem}
\begin{proof}
We have already proved the existence of $\varphi$ in the previous subsection and we need to prove uniqueness of $\nabla \varphi$. 

It turns out that as soon as every optimal coupling is induced by a transport map, this transport map (and hence the optimal coupling) must be unique.

To see this, let $\gamma_1$ and $\gamma_2$ be two optimal couplings induced by the transport maps $T_1$ and $T_2$ respectively:
$$
\gamma_1(Y = T_1(X))=1\,, \qquad \text{and} \qquad \gamma_2(Y = T_2(X))=1\,.
$$
Then consider the coupling $\bar \gamma=(\gamma_1 + \gamma_2)/2 \in \Gamma_{\mu,\nu}$. Then $\bar \gamma$ is also optimal since
\begin{align*}
    &\int \|x-y\|^2\, \bar \gamma(\ud x,\ud y)\\
    &\quad =\frac{1}{2}\int \|x-y\|^2\, \gamma_1(\ud x,\ud y)+\frac12\int \|x-y\|^2\, \gamma_2(\ud x,\ud y)
    =W_2^2(\mu,\nu)\,.
\end{align*}
In particular, it follows from Brenier's theorem that $\bar \gamma$ is also induced by a transport map $T$ (which happens to be the gradient of a convex function but we do not need this fact here). Therefore, if $(X,Y) \sim \bar \gamma$, it must be the case that the conditional distribution of $Y$ given $X$ is the Dirac $\delta_{T(X)}$. But by construction this conditional distribution is the mixture of two Diracs $(\delta_{T_1(X)}+\delta_{T_2(X)})/2$ and the two may only be the same when $T_1(X)=T_2(X)=T(X)$, almost surely. 

Therefore, if there exists a convex function $\psi$ such that $\nabla \psi(X)\sim \nu$, then by   Theorem~\ref{thm:fundOT}, it must be that $\gamma$ such that $(X,\nabla \psi(X))\sim \gamma \in \Gamma_{\mu, \nu}$ is an optimal coupling and therefore that $\nabla \psi(X)=\nabla \varphi(X)$, almost everywhere in light of the above discussion.

The last statement of the theorem follows by observing that the equivalence $(ii) \Leftrightarrow (i)$ in Theorem~\ref{thm:fundOT} implies that optimal couplings are supported on the graph of the gradient of a convex function, which has to be unique from the above argument.
\end{proof}

The improved Brenier theorem is very useful since it characterizes optimality of a transport map: if a transport map is the gradient of a convex function, then it is optimal and is unique! We call this map the \emph{Brenier map}. We can use Theorem~\ref{thm:improvedBrenier} to characterize optimal transport maps in two fundamental instantiations of the optimal transport problem: the one-dimensional case and the Gaussian case.

\begin{example}[One-dimensional optimal transport]\label{ex:1DOT}
    Recall that the cumulative distribution function (CDF) $F$ of a random variable $X$ is given by the map $t \mapsto \p(X\le t)$.
    
    \begin{proposition}\label{prop:w21d}
Let $\mu, \nu \in \cP_2(\R)$ be two univariate distributions with CDFs, $F_\mu$ and $F_\nu$ respectively and such that $\mu$ admits a density. Then
$$
W_2^2(\mu, \nu)=\int_0^1|F^{\dagger}_\mu(u)-F^{\dagger}_\nu(u)|^2\,\ud u
$$
and the optimal coupling between $\mu$ and $\nu$ is induced by the Brenier map $F_\nu^{\dagger}\circ F_\mu$.
\end{proposition}
\begin{proof}
Since $\mu$ has a density, $F_\mu \circ F_\mu^\dagger = \id$. Let $U\sim\unif([0,1])$ be a uniform random variable and define $X=F_\mu^{\dagger}(U), Y=F_\nu^{\dagger}(U)$ so that $X \sim \mu$ and $Y \sim \nu$. Next observe that $Y=F_\nu^{\dagger}\circ F_\mu(X)$ and that $F_\nu^{\dagger}\circ F_\mu$ is an increasing function. In light of Theorem~\ref{thm:improvedBrenier}, this defines the unique optimal coupling between $\mu$ and $\nu$.
\end{proof}

The geometric consequence of this identity, as explored further in Chapter~\ref{chap:geometry}, is that the metric space $\cP_2(\R)$ equipped with the 2-Wasserstein distance is flat. Indeed the map $\mu \mapsto F_\mu^\dagger$ is an isometric embedding of $(\cP_2(\R), W_2)$ into the (flat) Hilbert space $L^2([0,1])$.
\end{example}

\begin{example}[Gaussian optimal transport]\label{ex:ot_gaussian}
We can also used the improve Brenier theorem to derive the optimal transport map between two Gaussian measures.
Let $m_1,m_2 \in \R^\dd$ and let $\Sigma_1$, $\Sigma_2$ be positive definite $\dd\times \dd$ matrices.
Let $\mu_1 = \cN(m_1,\Sigma_1)$ and $\mu_2 = \cN(m_2,\Sigma_2)$.

Recall that affine maps preserve Gaussianity.
Namely, if $X_1\sim\mu_1$ and $T(x) = Ax+b$, where $A \in \R^{\dd\times\dd}$ and $b\in\R^\dd$, then $T(X_1)$ is also Gaussian.
To calculate the distribution of $T(X_1)$, it suffices to compute the mean and covariance, and we find that
\begin{align*}
    \E T(X_1) = Am_1+b\,, \qquad \cov T(X_1) = A \Sigma_1 A^\T\,.
\end{align*}
It is therefore a reasonable guess that the optimal transport map from $\mu_1$ to $\mu_2$ is affine, and this can be verified using Brenier's Theorem~\ref{thm:improvedBrenier}.
For this, we require that $Am_1 + b = m_2$ and $A\Sigma_1 A^\T = \Sigma_2$, representing the constraint that $T_\# \mu_1 = \mu_2$.
We also require $T$ to be the gradient of a convex function.
If we set
\begin{align*}
    \varphi(x) = \frac{1}{2}\,\langle x, A\,x\rangle + \langle b, x\rangle\,,
\end{align*}
then $\nabla \varphi(x) = \frac{1}{2}\,(A+A^\T)\,x + b$, and $\varphi$ is convex provided $A+A^\T \succeq 0$.
We conclude that $T = \nabla \varphi$ is the gradient of a convex function (and therefore optimal) provided that $A$ is symmetric and positive definite.

How do we choose $A$ and $b$ so that the pushforward constraints and the PSD constraint on $A$ are simultaneously met?
The most na\"{\i}ve choice for $A$, namely $A = \Sigma_1^{-1/2} \Sigma_2^{1/2}$, works when $\Sigma_1$ and $\Sigma_2$ commute, but in general this choice of $A$ is not even symmetric.
It takes a little ingenuity to find a PSD choice for $A$ that works, but it can be done as follows. Starting with the idea that $A$ is PSD and satisfies $A\Sigma_1 A = \Sigma_2$, then by squaring $\Sigma_1^{1/2} A \Sigma_1^{1/2}$ we find that
\begin{align*}
    {(\Sigma_1^{1/2} A \Sigma_1^{1/2})}^2
    &= \Sigma_1^{1/2} A \Sigma_1 A \Sigma_1^{1/2}
    = \Sigma_1^{1/2} \Sigma_2 \Sigma_1^{1/2}\,.
\end{align*}
Taking square roots and solving, we obtain
\begin{align*}
    A = \Sigma_1^{-1/2}\,(\Sigma_1^{1/2}\Sigma_2\Sigma_1^{1/2}){}^{1/2}\,\Sigma_1^{-1/2}\,.
\end{align*}
It is seen that this is the matrix $A$ that we are looking for.
By using the other constraint $Am_1 + b = m_2$, we find that
\begin{align*}
    T(x)
    &= \Sigma_1^{-1/2}\,(\Sigma_1^{1/2}\Sigma_2\Sigma_1^{1/2}){}^{1/2}\,\Sigma_1^{-1/2}\,(x-m_1) + m_2\,.
\end{align*}
By Theorem~\ref{thm:improvedBrenier}, this is the unique optimal transport map from $\mu_1$ to $\mu_2$.
Moreover, by substituting this into the definition of the Wasserstein distance, we find that (exercise!)
\begin{align}\label{eq:W2_Gaussians}
    W_2^2(\mu_1,\mu_2)
    &= \|m_1-m_2\|^2 + \tr\bigl[\Sigma_1 + \Sigma_2 -2\,(\Sigma_1^{1/2}\Sigma_2\Sigma_1^{1/2}){}^{1/2}]\,.
\end{align}
\end{example}

\section{Duality for \texorpdfstring{$p=1$}{p = 1}}\label{sec:duality_w1}

In the case of the 1-Wasserstein metric, the dual takes a remarkably simple form. Most textbooks derive this result as a specific instantiation of strong duality for optimal transport with a general cost $c$, which requires tools to generalize Theorem~\ref{thm:fundOT} beyond the case of the quadratic cost $c(x,y)=\|x-y\|^2$. The tools include a generalized notion of cyclical monotonicity and of Legendre transform, and the reader is invited to become familiar with them (see the discussion section for a brief overview). When specialized to the $p$-Wasserstein distance, they yield the following result; see, e.g., \cite[Section 3.1.1]{San15} for a proof.

\begin{theorem}\label{thm:Wpduality}
Fix $p \ge 1$ and let $\mu, \nu \in \cP_p(\R^\dd)$ be two probability measures. Then the following holds:
\begin{equation}
\label{eq:Wpduality}
W_p^p(\mu, \nu) =\sup_{\substack{f \in  L^1(\mu),\ g \in L^1(\nu
)\\f(x)+g(y) \le \|x-y\|^p }}\left\{\int f\, \ud \mu + \int g\, \ud \nu\right\}\,.
\end{equation}
\end{theorem}

In the case, where $p=1$, this result can be simplified to eliminate one of the dual functions.

\begin{theorem}\label{thm:W1duality}
Let $\mu, \nu \in \cP_1(\R^\dd)$ be two probability measures. Then the following holds:
\begin{equation}\label{eq:W1duality}
W_1(\mu, \nu) =\sup_{f \in \, \mathrm{Lip}_1}\left\{\int f\, \ud \mu - \int f\, \ud \nu\right\}\,,
\end{equation}
where $\mathrm{Lip}_1$ is the set of 1-Lipschitz functions.
\end{theorem}

Before proceeding to the proof of this theorem, let us see where Lipschitz functions come from. Recall that the semidual in Section~\ref{subsec:semidual} was also removing one of the dual functions.
In particular, when $p = 1$, we can replace the function $g$ in \eqref{eq:Wpduality} with the $c$-transform
$$
f^c(y)=\inf_{x \in \R^\dd}\{\|x- y\| -f(x)\}\,.
$$
The following lemma holds.

\begin{lemma}\label{lem:1lipctrans}
    For $c(x,y) = \|x-y\|$, a function $g: \R^\dd\to\R$ is a $c$-transform $g=f^c$ if and only it is 1-Lipschitz. Moreover, any 1-Lipschitz function satisfies $g^c=-g$.
\end{lemma}
\begin{proof}
    Write 
    $$
g(y)=f^c(y)= \inf_{x \in \R^\dd}\{\| x- y\| -f(x)\}\,.
    $$
For any $x$, the function $y \mapsto \| x- y\| -f(x)$ is clearly 1-Lipschitz by the reverse triangle inequality. Since the set of 1-Lipschitz functions is closed under taking infima, the function $g$ is also 1-Lipschitz. 

    To prove the converse, let $g$ be a 1-Lipschitz function so that for any $x,y\in \R^\dd$, it holds
    $$
g(y)\le \|x-y\|+g(x)\,.
    $$
    Taking the infimum over $x$ yields $g\le (-g)^c$. Moreover, 
    $$
 (-g)^c(y)=\inf_{x \in \R^\dd}\{\|x- y\| +g(x)\}\le g(y)\,,
    $$
    where we took $y=x$ in the last inequality. 
    
    We have shown that $(-g)^c=g$ and in particular that $g$ has to be a $c$-transform (of $-g$). The second statement of the lemma follows from the fact that if $g$ is 1-Lipschitz, then so is $-g$. 
\end{proof}

We are now in a position to prove Theorem~\ref{thm:W1duality}.

\begin{proof}[Proof of Theorem~\ref{thm:W1duality}]
By the argument in Subsection~\ref{subsec:semidual},~\eqref{eq:Wpduality} is equal to the following semidual
$$
\sup_{f \in L^1(\mu)}\left\{\int f\, \ud \mu + \int f^c\, \ud \nu\right\}\,.
$$
We can continue optimizing the potentials further to obtain
$$
\sup_{f \in L^1(\mu)}\left\{\int f^{cc}\, \ud \mu + \int (f^{cc})^c\, \ud \nu\right\}\,.
$$
It follows from Lemma~\ref{lem:1lipctrans} that 
$$
\{f^{cc}\,:\, f \in L^1(\mu)\}
$$
only contains 1-Lipschitz functions. Moreover, since $\mu \in \cP_1(\R^\dd)$, it holds that any 1-Lipschitz function is integrable against $\mu$:
\begin{align*}
    \int |f|\,\ud\mu
    &\le \int |f(x)-f(y)|\,\mu(\ud x) + |f(y)| \le \int \|x-y\|\,\mu (\ud x) + |f(y)| \\
    &<\infty\,,
\end{align*}
so that $f \in L^1(\mu)$. Hence, we have shown that 
\begin{equation}
    \label{eq:weakW1duality}
    W_1(\mu, \nu) =\sup_{f \in \mathrm{Lip}_1}\left\{\int f \,\ud \mu + \int f^c\, \ud \nu\right\}=\sup_{f \in \mathrm{Lip}_1}\biggl\{\int f\, \ud \mu - \int f\, \ud \nu\biggr\}\,,
\end{equation}
which concludes the proof of Theorem~\ref{thm:W1duality}.
\end{proof}

\section{Discussion}

\noindent\textbf{\S\ref{sec:ot_problem}.}
For a historical bibliography, consult~\cite{Vil03}.

\noindent\textbf{\S\ref{sec:wp_distances}.}
Villani~\cite{Vil09} remarks that the common terminology ``Wasserstein distances'' is ``very questionable''---though these distances are attributed to Leonid Vasershtein by Dobrushin~\cite{dobrushin1970prescribing}, Vasershtein was not their progenitor.
Nevertheless, the name has stuck.

Proposition~\ref{prop:wass_controls_moments} is taken from the paper~\cite{RigWee19}, which also provides a converse.
Further comparisons between ``information divergences'' (e.g., total variation, Kullback--Leibler, chi-squared divergence) and optimal transport distances are implied by so-called transport inequalities, which also connects to a larger literature on concentration of measure; see, e.g.,~\cite{BolVil05CKP} and~\cite[Chapter 22]{Vil09}.

\noindent\textbf{\S\ref{sec:ot_1d}.}
As shown in Exercise~\ref{ex:1d_ot}, the monotone coupling in Theorem~\ref{thm:W1_1D} and Proposition~\ref{prop:w21d} is also optimal for any cost function which is a strictly convex function of $x-y$; see~\cite[Section 2.2]{San15}.

\noindent\textbf{\S\ref{sub:brenier}.}
Brenier's theorem was first used to define multivariate quantiles for $\nu \in \cP_2(\R^\dd)$ by~\cite{CheGalHalHen17}, and this definition was extended to $\nu$ which may not have a second moment by~\cite{Hal17}.
In addition to its mathematical elegance, this definition of multivariate quantiles has important implications for nonparametric testing~\cite{GhoSen22,DebSen23Testing}.

Rockafellar's theorem is from~\cite{Roc1966Subdiff}.
The proof of Proposition~\ref{prop:suppCM} is from~\cite{GanMcC96}.

\noindent\textbf{\S\ref{sec:kantorovich_duality}.}
Although we deduced strong duality by explicitly exhibiting dual potentials for which the duality gap is zero (namely, the ones obtained from the characterization of optimal couplings as having cyclically monotone support), it is also possible to prove strong duality directly via appeal to an abstract min-max principle; see~\cite[Section 1.1]{Vil03}.

Many of the arguments in Section~\ref{sec:kantorovich_duality} generalize to general continuous costs $c : \cX\times \cY\to\R$: the $c$-conjugate of a function $f : \cX\to\R$ is given by $f^c(y) \deq \inf_{x\in\cX}\{c(x,y) - f(x)\}$, and likewise the $c$-conjugate of $g : \cY\to\R$ is given by $g^c(x) \deq \inf_{y\in\cY}\{c(x,y) - g(y)\}$.
We say that $f$ is $c$-concave if $f = g^c$ for some function $g$.
The equality~\eqref{eq:general_semidual} still holds and equals the optimal transport value (strong duality). Any optimal transport plan is supported on a $c$-cyclically monotone set, which is defined as in~\eqref{eq:cyclically_monotone} but replacing the quadratic cost with $c$.
Then, $c$-cyclically monotone sets are characterized as $c$-subdifferentials, where the $c$-subdifferential of a $c$-concave function $f = g^c$ is the set of $(x,y)$ pairs such that $f(x) + g(y) = c(x,y)$.
What does \emph{not} generalize as easily, however, is Brenier's theorem, which requires further conditions to ensure the single-valuedness of the $c$-subdifferential. For example if $c(x,y)=h(x-y)$ for some strictly convex function $h$, then a unique optimal transport \emph{map} exists but it need not be the gradient of a convex function; see~\cite[Theorem~1.17]{San15}.

\noindent\textbf{\S\ref{sec:duality_w1}.}
The duality formula for $W_1$ is classical and is closely related to the bounded Lipschitz metric~\cite{Dud02}, which can be extended to define a norm over signed measures.
As discussed in Subsection~\ref{subsec:ipm}, this formula expresses $W_1$ as an integral probability metric.

\section{Exercises}\label{sec:ot_exe}

\begin{enumerate}
    \item Let $A \subseteq \R\times\R$ be monotone:
    \begin{align*}
        \forall\,(x,y), (x',y')\in A\,,\; x < x' \Rightarrow y < y'\,.
    \end{align*}
    For simplicity, assume that $A$ is contained in the graph of a function.
    Show that $A$ is cyclically monotone.

\item Let $X \in \R$ be a random variable that admits a density supported on the whole real line. Let $f, g$ be two monotone increasing functions such that $f(X)$ has the same distribution as $g(X)$. Then, $f=g$ almost everywhere.

    \item\label{exe:ot_gaussians}
    Let $m_1,m_2 \in \R^\dd$ and let $\Sigma_1$, $\Sigma_2$ be positive definite $d\times d$ matrices.
    Let $\mu_1 = \cN(m_1,\Sigma_1)$ and $\mu_2 = \cN(m_2,\Sigma_2)$.
    \begin{enumerate}
        \item Verify the equation~\eqref{eq:W2_Gaussians}.
        \item Using a suboptimal coupling, prove the simple upper bound
            \begin{align*}
                W_2^2(\mu_1,\mu_2)
                &\le \|m_1-m_2\|^2 + \|\Sigma_1^{1/2} - \Sigma_2^{1/2}\|_{\rm HS}^2
            \end{align*}
            where $\|M_1 - M_2\|_{\rm HS}^2 \deq \tr({(M_1 - M_2)}^\T(M_1 - M_2))$.
            Show that this is an equality when $\Sigma_1$ and $\Sigma_2$ commute.
        \item Prove that if $\nu_1$, $\nu_2$ are probability measures with means $m_1$, $m_2$ and covariance matrices $\Sigma_1$, $\Sigma_2$ respectively, then
            \begin{align*}
                W_2(\nu_1,\nu_2)
                &\ge W_2(\mu_1,\mu_2)\,.
            \end{align*}
            \emph{Hint}: What are the optimal dual potentials for the optimal transport problem from $\mu_1$ to $\mu_2$?
    \end{enumerate}

\item 
    \begin{enumerate}
        \item Let $X$ and $Y$ be random vectors in $\R^\dd$.
            Prove that
            \begin{align*}
                &W_2^2\bigl(\law(X),\, \law(Y)\bigr) \\
                &\qquad = \|\E X - \E Y\|^2 + W_2^2\bigl(\law(X-\E X),\, \law(Y-\E Y)\bigr)\,.
            \end{align*}
            Thanks to this equality, often when we work with the $W_2$ distance, it suffices to consider centered random variables.
        \item Let $X, Y_1,\dotsc,Y_k$ be random vectors in $\R^\dd$ and $\lambda_1,\dotsc,\lambda_k\ge 0$.
            Suppose that $X$ and $Y_i$ are optimally coupled for each $i=1,\dotsc,k$.
            Show that $X$ and $\sum_{i=1}^k \lambda_i Y_i$ are optimally coupled.
    \end{enumerate}

\item 
    Let $\nu$ admit a density w.r.t.\ Lebesgue measure.
    Show that $W_2^2(\cdot, \nu)$ is strictly convex; that is, if $\mu_0, \mu_1 \in \cP_2(\R^\dd)$ are distinct and $t \in (0,1)$, then
    \begin{align}\label{eq:strict_cvxty}
        W_2^2\bigl((1-t)\,\mu_0 + t\,\mu_1,\, \nu\bigr)
        &< (1-t)\,W_2^2(\mu_0,\nu) + t\, W_2^2(\mu_1,\nu)\,.
    \end{align}
    \emph{Hint}: Start by proving that~\eqref{eq:strict_cvxty} holds with $\le$ instead of $<$.
    Next, supposing that~\eqref{eq:strict_cvxty} fails, show that there is an optimal transport plan between $(1-t)\,\mu_0 +t\,\mu_1$ and $\nu$ which is not induced by a transport map, contradicting Brenier's theorem.

\item 
    Consider the measures $\mu = \frac{1}{2}\,\cN(-m,1) + \frac{1}{2}\,\cN(+m,1)$ and $\nu = \frac{1}{4}\,\cN(-m,1) + \frac{3}{4} \, \cN(+m, 1)$, where $m > 0$.
    Prove that $W_2(\mu,\nu) \asymp m$.

    {\emph{Hint}: The point of this question is that computing the optimal transport map from $\mu$ to $\nu$ is painful, but it is not hard to obtain good lower and upper bounds.
    For the lower bound, prove and use the fact that for any $1$-Lipschitz function $f : \R\to\R$, it holds that $\E_\mu f - \E_\nu f \le W_1(\mu,\nu) \le W_2(\mu,\nu)$.
    For the upper bound, exhibit a coupling of $\mu$ and $\nu$.
    }

\item Recall that the chi-squared divergence between two probability measures $\mu$ and $\nu$ is defined as 
$$
\chi^2(\mu \mmid \nu)=\int \Bigl(\frac{\ud \mu}{\ud \nu }- 1\Bigr)^2 \,\ud \nu\,.
$$
Show that
$$
 W_1(\mu, \nu) \le \sqrt{\chi^2(\mu \mmid \nu)}\,,
$$
when $\nu$ has unit variance.

    \item\label{ex:1d_ot} Let $c : \R \to \R$ be strictly convex and consider optimal transport over $\R$ with the cost function $(x,y) \mapsto c(x-y)$.
    For example, when $c(z) = |z|^p$ for $p > 1$, we obtain $W_p^p$.
    In this exercise, we show that the coupling $\bar\gamma$ given in Proposition~\ref{prop:w21d} is \emph{universally optimal} for all costs of this form.
    See the discussion section for background.
    \begin{enumerate}
        \item Show that for any $a,b\in\R$, $\bar\gamma((-\infty, a] \times (-\infty, b]) = F_\mu(a) \wedge F_\nu(b)$.
        \item Show that $\bar\gamma$ is the \emph{unique} coupling of $\mu$ and $\nu$ such that
        \begin{align}\label{eq:characterization_1dot}
            \forall \,(x,y),(x',y') \in \supp\bar\gamma\,, \; x < x' \Rightarrow y \le y'\,.
        \end{align}
        \emph{Hint}: Consider the sets $A = (-\infty,a] \times (b,\infty)$ and $B = (a,\infty) \times (-\infty, b]$.
        Show that any coupling satisfying~\eqref{eq:characterization_1dot} must assign one of these two sets measure zero.
        Use this to show that any coupling $\gamma$ which satisfies~\eqref{eq:characterization_1dot} must agree with $\bar\gamma$.
        \item Let $(x,y), (x',y') \in \supp\gamma$, where $\gamma$ is an optimal transport plan between $\mu$ and $\nu$ with cost defined by $c$ as above.
        By $c$-cyclical monotonicity of $\supp\gamma$,
        \begin{align*}
            c(x-y) + c(x'-y') \le c(x-y') + c(x'-y)\,.
        \end{align*}
        Show that if $x < x'$, then $y\le y'$, hence $\gamma = \bar\gamma$.

        \emph{Hint}: Argue by contradiction. If the claim fails, then both $u \deq x-y'$ and $v\deq x'-y$ lie between $w \deq x-y$ and $w' \deq x'-y'$.
        Write $u$ and $v$ as convex combinations of $w$ and $w'$, and apply strict convexity of $c$.
        \item Deduce that the optimal cost equals $\E c(F_\mu^\dagger(U) - F_\nu^\dagger(U))$ where $U\sim\unif([0,1])$.
    \end{enumerate}

    \item\label{qu:tv} Given two probability measures $\mu$, $\nu$ over the same space $\msf S$, define the \emph{total variation distance} between $\mu$ and $\nu$ to be
    \begin{align*}
        d_{\msf{TV}}(\mu,\nu) = \sup_{A\subseteq \msf S}{|\mu(A) - \nu(A)|}\,,
    \end{align*}
    where the supremum ranges over all measurable subsets. Show that the total variation distance is also equal to all of the following.
    (\emph{Hint}: Consider the coupling in Theorem~\ref{thm:wpTV}.)
    \begin{enumerate}
        \item If $f$ and $g$ denote the respective densities of $\mu$ and $\nu$ with respect to some common dominating measure $\lambda$ for $\mu$ and $\nu$ (e.g., $\lambda = \mu+\nu$), then
        \begin{align*}
            d_{\msf{TV}}(\mu,\nu)
            &= \frac{1}{2}\int |f-g|\,\ud \lambda
            = 1 - \int f\wedge g \, \ud \lambda
            = \mu(f \ge g)\,.
        \end{align*}
        \item $d_{\msf{TV}}(\mu,\nu) = \inf \mbb P(X \ne Y)$ where the infimum ranges over all couplings $(X,Y)$ of $\mu$ and $\nu$.
        \item $d_{\msf{TV}}(\mu,\nu) = \sup\{\int h \, \ud (\mu-\nu) \mid h : \msf S \to [0,1]\}$. Compare with $W_1$ duality from Theorem~\ref{thm:W1duality}.
    \end{enumerate}

    \item Let $\mu, \nu \in \cP_2(\R^d)$ admit densities with respect to Lebesgue measure.
    Show that if $\nabla \varphi$ is the optimal transport map from $\mu$ to $\nu$, then $\nabla \varphi^*$ is the optimal transport map from $\nu$ to $\mu$. (See Appendix~\ref{app:convex}.)
    Apply this fact to the optimal transport map between two Gaussians and discover a non-trivial matrix identity.
\end{enumerate}

\chapter{Estimation of Wasserstein distances}
\label{chap:primal-dual}

In applications of optimal transport in statistics, it is paramount to be able to obtain good upper and lower bounds on the Wasserstein distance between probability measures.
This chapter describes tools to bound the Wasserstein distance.
To do so, we heavily employ the primal \emph{and} dual formulations of optimal transport.
As a primary application, we consider a quantitative form of the \emph{Wasserstein law of large numbers}, which is the statement that if $\mu_n$ is an empirical measure consisting of $n$ i.i.d.\ samples from a probability measure $\mu$, then $\E W_p(\mu_n, \mu) \to 0$ as $n \to \infty$.

\section{The Wasserstein law of large numbers}\label{sec:wasserstein_lln}

Suppose that $X_1, \dots, X_n \simiid \mu$, where $\mu$ is a probability measure on a compact subset of $\R^\dd$, which we assume for convenience is equal to the unit cube $[0, 1]^\dd$.
The \emph{empirical measure} is defined to be the (random) measure
\begin{align*}
    \mu_n = \frac 1n \sum_{i=1}^n \delta_{X_i}\,.
\end{align*}
The law of large numbers implies that $\mu_n \hookrightarrow \mu$ and also $\int \|\cdot\|^p \, \ud \mu_n \to \int \|\cdot\|^p \, \ud \mu$ almost surely; therefore, the discussion in Chapter~\ref{chap:OT} implies that $W_p(\mu_n, \mu) \to 0$.
Moreover, since $W_p(\mu_n, \mu)$ is bounded almost surely, we also have convergence in mean:
\begin{equation*}
	\E W_p(\mu_n, \mu) \to 0\,.
\end{equation*}
How fast does this convergence occur?
In the context of the classic law of large numbers for bounded random vectors $X_1, \dots, X_n$ in $\R^\dd$, we of course have
\begin{equation*}
	\E \biggl\|\frac 1n \sum_{i=1}^n X_i - \E X\biggr\|^2 \lesssim \frac{1}{n}\,.
\end{equation*}
Note that the rate of decay $n^{-1}$ holds irrespective of the dimension, and is true even in infinite-dimensional Hilbert spaces.

By contrast, the Wasserstein law of large numbers behaves quite differently.
In this chapter, we prove the following proposition.

\begin{proposition}\label{prop:w1_rate}
	If the support of $\mu$ lies in $[0, 1]^\dd$, then
	\begin{equation}
			\E W_1(\mu_n, \mu) \lesssim \sqrt \dd \cdot \begin{cases}
			n^{-1/2} & \text{if $\dd = 1$,} \\
			(\log n/n)^{1/2} & \text{if $\dd = 2$,} \\
			n^{-1/\dd} & \text{if $\dd \geq 3$,} \\
		\end{cases}
	\end{equation}
	and this rate is unimprovable in general.
\end{proposition}

In contrast to the standard law of large numbers, the convergence of $\mu_n$ to $\mu$ in Wasserstein distance degrades exponentially as the dimension grows, a phenomenon often known as the curse of dimensionality.

\section{The dyadic partitioning argument}\label{sec:dyadic}

The fact that the Wasserstein distance is defined by a minimization over couplings suggests a natural strategy for proving bounds: we can show an upper bound on $W_1$ by exhibiting a coupling with a small cost.
In this section, we build such a coupling, which, perhaps surprisingly, gives rise to good bounds in many situations.
The main idea is to attempt to couple $\mu$ and $\nu$ by recursively constructing candidate couplings at multiple scales.

Before stating the bound, let us describe the basic strategy.
For simplicity, let us consider proving an upper bound on $W_1(\mu, \nu)$ for $\mu$ and $\nu$ whose support lies in $[0, 1]^\dd$.
We first make a trivial observation:
\begin{equation}\label{eq-primal-dual:trivial_w1_bound}
	W_1(\mu, \nu) \leq \sqrt \dd\,.
\end{equation}
Indeed, the diameter of $[0, 1]^\dd$ is $\sqrt \dd$, so no coupling between $\mu$ and $\nu$ can move mass a greater distance than this.

Let us now imagine a slight sharpening of this bound. Let $\cQ$ be the collection of cubes of side length $1/2$ whose corners lie at points of the form $2^{-1}\,(k_1, \dots, k_\dd)$ for $k_1, \dotsc, k_\dd \in \{0, 1, 2\}$.
These cubes form a partition of $[0, 1]^\dd$ into $2^\dd$ pieces.\footnote{These cubes overlap at their boundaries, but we can easily modify these sets by removing overlaps to obtain a bona fide partition.}
Suppose for the sake of argument that $\mu(Q) = \nu(Q)$ for all $Q \in \cQ$, $j = 1, \dots, 2^\dd$, so that $\mu$ and $\nu$ assign the same mass to each of the small cubes.
Then, it would be possible to couple $\mu$ and $\nu$ by only moving mass within each small cube.
Since the diameter of each small cube is $\sqrt \dd/2$, any such coupling improves on the bound in~\eqref{eq-primal-dual:trivial_w1_bound} by a factor of $2$.

Even when $\mu$ and $\nu$ do not assign the same mass to each small cube, we can use the above idea to construct a coupling between $\mu$ and $\nu$ in two steps: first, we can match as much mass as possible between $\mu$ and $\nu$ within each cube.
This creates a partial coupling between a portion of $\mu$'s mass and a portion of $\nu$'s.
Since we only move mass within each small cube, the total cost of this partial coupling is at most $\sqrt{\dd}/2$.
We then need to extend this partial coupling to a full coupling, by transporting $\mu$'s extra mass on any cube $Q$ for which $\mu(Q) > \nu(Q)$ to $\nu$'s extra mass on some cube $Q'$ for which $\nu(Q') > \mu(Q')$.
The amount of extra mass matched in this step is $\sum_{Q \in \cQ} (\mu(Q) - \nu(Q))_+ = \frac 12 \sum_{Q \in \cQ} |\mu(Q) - \nu(Q)|$, at a total cost of at most $\frac{\sqrt \dd}2 \sum_{Q \in \cQ} |\mu(Q) - \nu(Q)|$.

Combining these bounds yields the refined estimate
\begin{equation}\label{eq-primal-dual:one_step_w1_bound}
	W_1(\mu, \nu) \leq \frac{\sqrt \dd}{2} \sum_{Q \in \cQ} |\mu(Q) - \nu(Q)| + \frac{\sqrt \dd}{2}\,.
\end{equation}
This bound improves on~\eqref{eq-primal-dual:trivial_w1_bound} when $\mu$ and $\nu$ assign similar mass to each cube.

The proof of the following bound is based on recursing the above argument $J$ times.
At the $j$-th stage, we bound the discrepancy between $\mu$ and $\nu$ on $2^{dj}$ cubes of side length $2^{-j}$.
To state this bound, let us define the set $\cQ_j$, $j \geq 0$, to consist of a set of $2^{dj}$ cubes of side length  $2^{-j}$ which form a partition of $[0,1]^\dd$.\footnote{As above, we assume that the elements of $\cQ_j$ been modified at their boundary so that $\cQ_j$ is a partition and so that $\cQ_{j+1}$ is a refinement of $\cQ_j$ for all $j \geq 0$.}

\begin{theorem}[Dyadic partitioning bound]\label{thm:dyadic}\index{dyadic partitioning}
Let $\mu, \nu \in \cP([0, 1]^\dd)$. For any $J \geq 0$,
	\begin{align*}
		W_1(\mu, \nu) \leq \sqrt \dd \sum_{j=0}^{J-1} \biggl(2^{-j} \sum_{Q \in \cQ_{j+1}} |\mu(Q) - \nu(Q)|\biggr) + \sqrt \dd\, 2^{-J}\,.
	\end{align*}
\end{theorem}
\begin{proof}
    We define a sequence of positive measures $\mu_0, \dotsc, \mu_J$ and $\nu_0, \dotsc, \nu_J$, which satisfy $\sum_{j=0}^J \mu_j = \mu$ and $\sum_{j=0}^J \nu_j = \nu$ and such that
	\begin{align*}
		\mu_j(Q) = \nu_j(Q) \quad \forall Q \in \cQ_j,\, j = 0, \dots, J\,.
	\end{align*}
 We write for simplicity $\Omega \deq [0, 1]^\dd$.
	We first claim that
	\begin{equation}\label{eq-primal-dual:sub_measure_bound}
		W_1(\mu, \nu)  \leq \sqrt \dd \sum_{j=0}^J 2^{-j} \mu_j(\Omega)\,.
	\end{equation}
	This bound is nothing but an instantiation of the strategy described above: since $\mu_j$ and $\nu_j$ assign the same mass to each element of $\cQ_j$, there exists a coupling $\gamma_j$ between $\mu_j$ and $\nu_j$ which only moves mass within each element of $\cQ_j$; for instance, we can take the piecewise independent coupling
	\begin{equation*}
		\gamma_j = \sum_{Q \in \cQ_j: \mu_j(Q) > 0} \frac{(\mu_j)|_Q \otimes (\nu_j)|_Q}{\mu_j(Q)}\,.
	\end{equation*}
	The fact that $\gamma_j \in \Gamma_{\mu_j, \nu_j}$ implies $\gamma = \sum_{j=0}^J \gamma_j \in \Gamma_{\mu, \nu}$, and
	\begin{align*}
		W_1(\mu, \nu)& \leq \int \|x - y\| \, \gamma(\ud x,\ud y) \\
  &= \sum_{j=0}^J \int \|x - y\| \, \gamma_j(\ud x,\ud y) \\
  &\leq \sqrt \dd\sum_{j=0}^J 2^{-j} \mu_j(\Omega)\,,
	\end{align*}
	where the last inequality follows from the fact if $(x, y) \in \supp(\gamma_j)$, then $x$ and $y$ lie in the same element $Q \in \cQ_j$, so that $\|x - y\| \leq \mathrm{diam}(Q) = \sqrt \dd\, 2^{-j}$.
	
	We now exhibit the measures $\mu_j$ and $\nu_j$ which give rise to the final bound.
    Define the restriction of $\mu_J$ on each $Q \in \cQ_J$ by setting
	\begin{align*}
		(\mu_J)|_Q = \frac{\mu(Q)\wedge \nu(Q)}{\mu(Q)}\, \mu|_Q\,,
	\end{align*}
	where by convention we let $\mu_J$ be zero on $Q$ if $\mu(Q) = 0$.
	Similarly, set
	\begin{equation*}
		(\nu_J)|_Q = \frac{\mu(Q)\wedge  \nu(Q)}{\nu(Q)}\, \nu|_Q\,.
	\end{equation*}
	
	For $1 \leq j < J$, let
	\begin{align*}
		\mu'_j & = \mu - \sum_{j < k \leq J} \mu_k\,, \\
		\nu'_j &= \nu - \sum_{j < k \leq J} \nu_k\,,
	\end{align*}
	and then, for each $Q \in \cQ_j$, define
	\begin{align*}
		(\mu_j)|_Q & = \frac{\mu'_j(Q)\wedge  \nu_j'(Q)}{\mu'_j(Q)}\, (\mu_j')|_Q\,, \\
		(\nu_j)|_Q & = \frac{\mu'_j(Q)\wedge  \nu_j'(Q)}{\nu'_j(Q)}\, (\nu_j')|_Q\,.
	\end{align*}
	Finally, we set 
 $$
 \mu_0 = \mu - \sum_{j = 1}^J \mu_j  \qquad \text{and}\qquad  \nu_0 = \nu - \sum_{j = 1}^J \nu_j\,,
 $$
so that
 $$
 \sum_{j=0}^J \mu_j = \mu \qquad \text{and}\qquad  \sum_{j=0}^J \nu_j = \nu\,.
 $$
	It is easy to see that $\mu_j(Q) = \nu_j(Q)$ for all $Q \in \cQ_j$ and all $j \in \{0, \dots, J\}$.
	To apply \eqref{eq-primal-dual:sub_measure_bound}, we also need to check that $\mu_j, \nu_j \geq 0$.
 
	\begin{lemma}\label{lem:prdyadpart}
		The measures $\mu_0, \dots, \mu_J$ and $\nu_0, \dots, \nu_J$ are all positive.
	\end{lemma}
\begin{proof}
		By symmetry, it suffices to verify this fact for the sequence $\mu_0, \dots, \mu_J$.
		
    We first show by backwards induction on $j$ that 
\begin{equation}\label{eq:pr:inductiondyadic1}\tag{${\sf A}_j$}
      \mu_{j+1} \geq 0 \qquad \text{and} \qquad 0 \leq \sum_{j < k \leq J} \mu_k \leq \mu
\end{equation}
  for all $j = 0, \dots, J-1$.
  
  For $j = J-1$, these bounds follow directly from the construction of $\mu_J$.
		Next assume that~\eqref{eq:pr:inductiondyadic1} holds for some $j$, then
		\begin{equation*}
			\mu_j' = \mu -  \sum_{j < k \leq J} \mu_k \geq 0\,,
		\end{equation*}
		and therefore  $\mu_j \ge 0$, since $\mu_j$ is obtained by reweighting $\mu_j'$ on each element of $\cQ_j$ by a non-negative quantity. Note also that this non-negative quantity is also bounded by one so that we also have $\mu_j \le \mu_j'$. Together these two facts yields $0\le \mu_j \le \mu_j'$ so that
		\begin{equation*}
			0 \leq \sum_{j-1 < k \leq J} \mu_k=\sum_{j < k \leq J} \mu_k + \mu_j \leq \sum_{j < k \leq J} \mu_k + \mu'_j = \mu\,.
		\end{equation*}
  We have proved that $({\sf A}_{j-1})$ holds.
		By induction, we obtain that $\mu_1, \dots, \mu_J$ are all positive.
		Finally, since we have also shown that 
  $$
  \sum_{0 < k \leq J} \mu_k \leq \mu,
  $$
  we obtain $\mu_0 \geq 0$ as well.
	\end{proof}
 
In light of~\eqref{eq-primal-dual:sub_measure_bound}, it remains to bound $\mu_j(\Omega)$ for $j = 0, \dots, J$.
	We first claim that 
	\begin{equation}\label{eq-primal-dual:q_gap}
		|\mu'_j(Q) - \nu'_j(Q)| = |\mu(Q) - \nu(Q)| \quad \forall Q \in \cQ_j,\, j = 1, \dots, J\,.
	\end{equation}
	This follows from the fact that $$\mu'_j(Q) - \nu'_j(Q) = \mu(Q) - \nu(Q) - \sum_{j < k \leq J} (\mu_k(Q) - \nu_k(Q))\,,$$
	since $\mu_k$ and $\nu_k$ assign the same mass to each element of $\cQ_k$ and since $Q$ can be written as a disjoint union of elements of $\cQ_k$, so the sum vanishes.
	We now claim that we can bound the mass that $\mu_j$ and $\nu_j$ assign to elements of $\cQ_j$ in terms of the difference between $\mu$ and $\nu$ on cubes in $\cQ_{j+1}$.
 
	\begin{lemma}\label{lem:prdyadic2}
		If $R \in \cQ_j$ for some $0 \leq j < J$, then
			\begin{equation*}
			\mu_j(R) = \nu_j(R) \leq \sum_{Q \subseteq R,\, Q \in \cQ_{j+1}} |\mu(Q) - \nu(Q)|\,.
		\end{equation*}
	\end{lemma}
	\begin{proof}
		We have already shown that $\mu_j(R) = \nu_j(R)$, so it suffices to show that expression holds for $\mu_j(R)$.
		For notational consistency, we set $\mu_0' = \mu_0$.
		Then, for any $0 \leq j < J$ and any $R \in \cQ_j$,
		\begin{align*}
			\mu_j(R) & \leq \mu'_j(R) \\
			& = \sum_{Q \subseteq R,\, Q \in \cQ_{j+1}} \mu'_j(Q) \\
			& = \sum_{Q \subseteq R,\, Q \in \cQ_{j+1}} (\mu'_{j+1}(Q) - \mu_{j+1}(Q)) \\
			& = \sum_{Q \subseteq R,\, Q \in \cQ_{j+1}} (\mu'_{j+1}(Q) - \nu'_{j+1}(Q))_+ \\
			& \le  \sum_{Q \subseteq R,\, Q \in \cQ_{j+1}} |\mu'_{j+1}(Q) - \nu'_{j+1}(Q)| \\
			& = \sum_{Q \subseteq R,\, Q \in \cQ_{j+1}} |\mu(Q) - \nu(Q)|\,,
		\end{align*}
		where the second equality comes from comparing the definitions of $\mu_j'$ and $\mu_{j+1}'$, and the last equality follows from~\eqref{eq-primal-dual:q_gap}.
	\end{proof}
	
	Putting it all together,~\eqref{eq-primal-dual:sub_measure_bound} implies
	\begin{align*}
		W_1(\mu, \nu) & \leq \sqrt{\dd} \sum_{j=0}^J 2^{-j}\mu_j(\Omega) \\
		& = \sqrt{\dd} \sum_{j=0}^{J-1} 2^{-j}\mu_j(\Omega) + \sqrt{\dd}\,2^{-J}\mu_J(\Omega)\\
		& = \sqrt{\dd} \sum_{j=0}^{J-1} \biggl(2^{-j}\sum_{R \in \cQ_j} \mu_j(R) \biggr) + \sqrt{\dd}\,2^{-J}\mu_J(\Omega) \\
		& \leq \sqrt{\dd} \sum_{j=0}^{J-1} \biggl(2^{-j}\sum_{Q \in \cQ_{j+1}} |\mu(Q) - \nu(Q)| \biggr) + \sqrt{\dd}\,2^{-J}\,.
	\end{align*}
 This concludes the proof of Theorem~\ref{thm:dyadic}.
\end{proof}

Applying Theorem~\ref{thm:dyadic} to $\mu$ and $\mu_n$, we obtain the following bound.

\begin{proposition}\label{prop:w1_dyadic_rate}
	If the support of $\mu$ lies in $[0, 1]^\dd$, then 
	\begin{equation*}
		\E W_1(\mu_n, \mu) \lesssim \sqrt \dd \cdot \begin{cases}
			n^{-1/2} & \text{if $\dd = 1$,} \\
			(\log n)\,n^{-1/2} & \text{if $\dd = 2$,} \\
			n^{-1/\dd} & \text{if $\dd \geq 3$.} \\
		\end{cases}
	\end{equation*}
\end{proposition}
\begin{proof}
	Theorem~\ref{thm:dyadic} implies that for any $J \geq 0$,
	\begin{align*}
		\E W_1(\mu_n, \mu) & \leq \sqrt{\dd} \sum_{j=0}^{J-1} 2^{-j} \sum_{Q \in \cQ_{j+1}} \E |\mu_n(Q) - \mu(Q)| +\sqrt \dd\, 2^{-J} \\
		& \leq \sqrt{\dd} \sum_{j=0}^{J-1} 2^{-j}\, 2^{\dd\,(j+1)/2} \, \biggl(\sum_{Q \in \cQ_{j+1}} \E (\mu_n(Q) - \mu(Q))^2\biggr)^{1/2} \\
      &\qquad{} + \sqrt{\dd}\,2^{-J} \\
		& \leq \sqrt{\dd} \sum_{j=0}^{J-1} 2^{-j}\, 2^{\dd\,(j+1)/2} \, n^{-1/2} + \sqrt{\dd}\,2^{-J} \\
		& \lesssim \sqrt{\dd} \cdot \begin{cases}
			2^{(J+1)\,(\dd/2-1)}\,n^{-1/2} +  2^{-J} & \text{ if $\dd \geq 3$,} \\
			J n^{-1/2} + 2^{-J} & \text{ if $\dd = 2$,} \\
			n^{-1/2} + 2^{-J} & \text{ if $\dd = 1$.}
		\end{cases}
	\end{align*}
	To balance these terms, we choose $J$ such that $2^{J} \leq n^{1/2} < 2^{J+1}$ if $\dd \leq 2$, and $J$ such that $2^{J+1} \leq n^{1/\dd} < 2^{J+2}$ if $\dd \geq 3$.
\end{proof}
Note that bound of Proposition~\ref{prop:w1_dyadic_rate} is weaker than that of Proposition~\ref{prop:w1_rate} when $\dd = 2$.
Unfortunately, the dyadic partitioning argument does not yield a sharp bound in two dimensions.
We return to this question in Section~\ref{sec:akt}.

\section{Dual chaining bounds}\label{sec:chaining}

In this section, we present a superficially different proof of Proposition~\ref{prop:w1_dyadic_rate}.
Rather than constructing a coupling in the primal, we use the dual representation of the 1-Wasserstein distance instead.
The benefit of this approach is that we can write
\begin{align}
	W_1(\mu_n, \mu) & = \sup_{f \in \mathrm{Lip}_1} \left\{\int f\, \ud \mu_n - \int f \,\ud \mu\right\} \nonumber\\
	& = \sup_{f \in \mathrm{Lip}_1} \frac 1 n \sum_{i=1}^n \{f(X_i) - \E f(X_i)\}\,. \label{eq-primal-dual:w1_empirical_process}
\end{align}
The random process $f \mapsto  \frac 1 n \sum_{i=1}^n \{f(X_i) - \E f(X_i)\}$ is known as an \emph{empirical process}\index{empirical process}, and bounding the expected suprema of such processes is a very common task in many areas of statistics.

To control this empirical process, we use a standard technique known as \emph{chaining}.\index{chaining}
Given a class $\cF$ of real-valued functions on $\Omega \subseteq \R^\dd$, we call a set $F = \{f_1, \dots, f_N\}$ an $\varepsilon$-cover of $\cF$ if, for any $f \in \cF$, there exists $f_i \in F$ such that $\|f - f_i\|_{L^\infty(\Omega)} \leq \varepsilon$.
The \emph{$\varepsilon$-covering number} of $\cF$ is
\begin{equation*}
	N(\varepsilon, \cF) = \min \{|F|: F \text{ is an $\varepsilon$-cover of $\cF$}\}\,.
\end{equation*}
The chaining argument shows that the covering number of a class $\cF$ controls the supremum of an empirical process indexed by that set.
We use the following version:

\begin{proposition}[{\cite[Theorem 5.31]{Han14}}]\label{prop:chaining}
	If $\cF$ is a set of real-valued functions on $\Omega$ such that $\|f\|_{L^\infty(\Omega)} \leq R$ for all $f \in \cF$, then
	\begin{equation*}
		\E \sup_{f \in \cF} \frac 1n \sum_{i=1}^n \{f(X_i) - \E f(X_i)\} \lesssim \inf_{\tau > 0} \left\{\tau + \frac{1}{\sqrt n} \int_\tau^R \sqrt{\log N(\varepsilon, \cF)} \, \ud \varepsilon\right\}\,.
	\end{equation*}
\end{proposition}

Proposition~\ref{prop:chaining} and~\eqref{eq-primal-dual:w1_empirical_process} imply that we can obtain an upper bound on $\E W_1(\mu_n, \mu)$ as long as we can calculate the covering numbers of the set of Lipschitz functions on $[0, 1]^\dd$.
We also notice that we can assume without loss of generality that the functions appearing in~\eqref{eq-primal-dual:w1_empirical_process} take the value $0$ at $(0, \dots, 0)$.
Indeed, a Lipschitz function on $[0, 1]^\dd$ is bounded, and since the value of $\frac 1 n \sum_{i=1}^n \{f(X_i) - \E f(X_i)\}$ is unaffected if we shift $f$ by a constant, we may fix its value at $(0, \dots, 0)$ to be $0$ without loss of generality.

\begin{lemma}\label{lem:cov_bound}
    Denote by $\Lipone([0,1]^\dd)$ the set of 1-Lipschitz functions on $[0, 1]^\dd$ satisfying $f(0) = 0$.
	Then
	\begin{equation*}
		\log N(\varepsilon,\Lipone([0,1]^\dd)) \lesssim (4 \sqrt \dd/\varepsilon)^\dd\,.
	\end{equation*}
\end{lemma}
\begin{proof}

We bound the covering number by exhibiting an $\varepsilon$-cover of $\Lipone([0,1]^\dd)$ of the specified size.
To do so, we again use the notion of a dyadic partition of $[0, 1]^\dd$ into a set $\cQ_j$ of cubes of side length $2^{-j}$.
Each element of $\cQ_j$ is of the form $2^{-j}\, ([k_1, k_1+1] \times \dotsc \times [k_\dd, k_\dd+1])$ for some integers $k_1, \dotsc, k_\dd \in [2^j-1] \deq \{0,\dotsc, 2^{j} -1\}$, and we denote such an element by $Q_{\vec k}$ for $\vec k = (k_1, \dots, k_\dd)$.\footnote{This collection of cubes overlaps at the boundaries, but as above we may remove overlaps to obtain a disjoint partition of $[0, 1]^\dd$.}

Fix an integer $j \geq 0$ and positive $\delta > 0$ to be specified.
Consider the set $\cH$ of functions $h$ satisfying the following requirements:
\begin{enumerate}
	\item $h$ is constant on each element of $\cQ_j$, i.e., there exist constants $(h_{\vec k})_{\vec k \in [2^{j}-1]^\dd}$ such that $h(x) = h_{\vec k}$ for all $x \in Q_{\vec k}$.
	\item $h_{\vec k}$ is an integer multiple of $\delta$ for all $\vec k \in [2^{j}-1]^\dd$.
	\item $h_{(0, \dots, 0)} = 0$.
	\item If $\|\vec k - \vec k'\|_{\infty} \leq 1$, then $|h_{\vec k} - h_{\vec k'}| \leq 2^{-j}\sqrt \dd + \delta$.
\end{enumerate}

We first claim that $\cH$ constitutes an $\varepsilon$-cover of $\Lipone([0,1]^\dd)$ if $2^{-j} \sqrt \dd  + \delta \leq \varepsilon$.
Given any $f \in \Lipone([0,1]^\dd)$, denote by $h_f$ the element of $\cH$ given by $(h_f)_{\vec k} = \delta\, \lfloor f(2^{-j}\,(k_1, \dots, k_\dd))/\delta \rfloor$ for all $\vec k \in [2^{j}-1]^\dd$.
To see that $h_f \in \cH$, note that it immediately satisfies the first three requirements by construction, and for the fourth, we have
\begin{align*}
	|(h_f)_{\vec k} - (h_f)_{\vec k'}| & = \delta\, \bigl\lvert\lfloor f(2^{-j}\,(k_1, \dots, k_\dd))/\delta \rfloor - \lfloor f(2^{-j}\,(k'_1, \dots, k'_\dd))/\delta \rfloor\bigr\rvert \\
	& \leq |f(2^{-j}\,(k_1, \dots, k_\dd)) - f(2^{-j}\,(k'_1, \dots, k'_\dd))| + \delta \\
	& \leq 2^{-j}\, \|\vec k - \vec k'\|_2 + \delta\,,
\end{align*}
where the last inequality follows from the fact that $f$ is Lipschitz.
Since $\|\vec k - \vec k'\|_2 \leq \sqrt \dd$ when $\|\vec k - \vec k'\|_\infty = 1$, the claim follows.
Finally, for any $x \in Q_{\vec k}$, the fact that $f$ is Lipschitz again implies
\begin{align*}
	|f(x) - (h_f)_{\vec k}| & = \bigl\lvert f(x) - \delta\, \lfloor f(2^{-j}\,(k_1, \dots, k_\dd))/\delta \rfloor\bigr\rvert \\
	& \leq |f(x) - f(2^{-j}\,(k_1, \dots, k_\dd))| + \delta \\
	& \leq \mathrm{diam}(Q_{\vec k}) + \delta \\
	& = 2^{-j} \sqrt \dd  + \delta\,.
\end{align*}
Therefore $\|f-h_f\|_\infty \leq 2^{-j} \sqrt \dd + \delta$.

We have shown that for every $f \in \Lipone([0, 1]^\dd)$, there exists $h_f \in \cH$ such that $\|f-h_f\|_\infty \leq 2^{-j} \sqrt \dd  + \delta$.
Therefore, if $2^{-j} \sqrt \dd + \delta \leq \varepsilon$, then $\cH$ is an $\varepsilon$-cover of $\Lipone([0, 1]^\dd)$.
We fix $\delta = 2^{-j} \sqrt \dd$, so that this requirement reduces to $2^{-j} \sqrt \dd \leq \varepsilon/2$.

To bound $|\cH|$, note that if we fix the value of $h_{\vec k}$ for some $\vec k$, then for any $\vec k'$ such that $\|\vec k - \vec k'\|_\infty = 1$, there are at most $5$ possible values of $h_{\vec k'}$.
This follows from the fact that $h_{\vec k'}$ must be an integer multiple of $\delta = 2^{-j} \sqrt \dd$, and there are $5$ integer multiples of $\delta$ in the interval $[h_{\vec k} - 2\delta, h_{\vec k} + 2 \delta]$.
Therefore, if we consider specifying an element $\cH$ by specifying the values of $h_{\vec k}$ sequentially by setting $h_{(0, \dots, 0)} = 0$ and proceeding in lexicographic order, then at each stage we have at most $5$ choices for the next value of $h_{\vec k}$.
This implies that $|\cH| \leq 5^{2^{dj} - 1}$.

For any $j$ for which $2^{-j} \sqrt \dd \leq \varepsilon/2$, we have therefore obtained an $\varepsilon$-cover $\cH$ of $\cF$ satisfying $\log |\cH| \lesssim 2^{\dd j}$.
Choosing $2^{j}$ to be the smallest power of two larger than $2\sqrt \dd/\varepsilon$ yields the claim.
\end{proof}

With the bound of Lemma~\ref{lem:cov_bound} in hand, we can give another proof of Proposition~\ref{prop:w1_dyadic_rate}.

\begin{proof}[Proof of Proposition~\ref{prop:w1_dyadic_rate}]
	Since $\|f\|_\infty \leq \sqrt \dd$ for all $f \in \Lipone([0, 1]^\dd)$,
	by Proposition~\ref{prop:chaining} and~\eqref{eq-primal-dual:w1_empirical_process}, for any $\tau > 0$,
	\begin{equation*}
		\E W_1(\mu_n, \mu) \lesssim \tau + \frac{1}{\sqrt n} \int_\tau^{\sqrt \dd} \sqrt{\log N(\varepsilon, \Lipone([0, 1]^\dd))} \, \ud \varepsilon\,.
	\end{equation*}
	Applying Lemma~\ref{lem:cov_bound} yields
	\begin{equation*}
		\E W_1(\mu_n, \mu) \lesssim \tau + \frac{1}{\sqrt n} \int_\tau^{\sqrt \dd} (4 \sqrt{\dd}/\varepsilon)^{\dd/2} \, \ud \varepsilon\,.
	\end{equation*}
	We now consider the bound separately for $\dd = 1$ and $\dd > 1$.
	If $\dd = 1$, then we may take $\tau = 0$ to obtain
	\begin{equation*}
		\E W_1(\mu_n, \mu) \lesssim \frac{1}{\sqrt n} \int_{0}^{1}(4 /\varepsilon)^{1/2} \, \ud \varepsilon \lesssim n^{-1/2}\,.
	\end{equation*}
	If $\dd > 1$, then $\varepsilon^{-\dd/2}$ is no longer integrable at $0$, so we take $\tau = 4\sqrt \dd\, n^{-1/\dd}$ to obtain
	\begin{equation*}
		\E W_1(\mu_n, \mu) \lesssim \sqrt \dd\, n^{-1/\dd} + \frac{1}{\sqrt n} \int_{4\sqrt \dd\, n^{-1/\dd}}^{\sqrt \dd} (4 \sqrt{\dd}/\varepsilon)^{\dd/2} \, \ud \varepsilon\,.
	\end{equation*}
	When $\dd = 2$, the integral is $O(\log n)$, and we obtain $\E W_1(\mu_n, \mu) \lesssim (\log n)/\sqrt n$.
	When $\dd > 2$, the integral is $O(n^{1/2 - 1/\dd})$, and we obtain $\E W_1(\mu_n, \mu) \lesssim \sqrt \dd\, n^{-1/\dd}$.
\end{proof}
Though these two proofs of Proposition~\ref{prop:w1_dyadic_rate} look quite different, they are in fact very similar: in both cases, we employ a multi-scale decomposition of $[0, 1]^\dd$.
The dyadic partitioning argument uses this decomposition to construct a coupling in the primal; the chaining argument uses this decomposition to control the covering numbers of Lipschitz functions in the dual.

\section{A finer analysis for \texorpdfstring{$\dd = 2$}{dimension two}}\label{sec:akt}

Both the dyadic partition argument presented in Section~\ref{sec:dyadic} and the chaining argument presented in Section~\ref{sec:chaining} suffer from the defect that they fail to obtain the correct rate for the Wasserstein law of large numbers in two dimensions.
This fact is related to the fact that $\dd = 2$ is the ``critical'' case for the behavior $\E W_1(\mu_n, \mu)$---it can be shown that in $\dd = 1$, the cost of the optimal transport between $\mu_n$ and $\mu$ is dominated by ``global'' features and that when $\dd \geq 3$, the cost of optimal transport is dominated by ``local'' irregularities. In dimension 2, by contrast, irregularities at all scales contribute simultaneously, and bounding the optimal cost requires more care.

The correct rate for $\dd = 2$ was first discovered by Ajtai, Koml\'{o}s, and Tusn\'{a}dy~\cite{AjtKomTus84} by a somewhat delicate argument.\index{Atjai--Koml\'{o}s--Tusn\'{a}dy theorem}
In this section, we present an ingenious approach due to Bobkov and Ledoux \cite{BobLed21} that obtains the correct rate by simpler means.
This proof is based on Fourier analysis, and as a first step, we show that we can focus our attention on periodic functions, to which the tools of Fourier analysis can naturally be applied.
This gives rise to the following periodic version of the Wasserstein distance: for probability measures $\mu$ and $\nu$ on $\R^\dd$, define
\begin{equation}\label{eq-primal-dual:w1tilde_def}
	\widetilde{W_1}(\mu, \nu) = \sup_{f \in \widetilde{\mathrm{Lip}}} \int f \, (\ud \mu - \ud \nu)\,,
\end{equation}
where $\widetilde{\mathrm{Lip}}$ denotes the set of $1$-Lipschitz, $2 \pi$-periodic $C^\infty$ functions on $\R^\dd$.
For measures on the cube, this definition actually agrees with standard Wasserstein distance.

\begin{lemma}\label{lem:periodic}
	If the supports of $\mu$ and $\nu$ lie in $[0, 1]^\dd$, then $\widetilde{W_1}(\mu, \nu) = W_1(\mu, \nu)$.
\end{lemma}
\begin{proof}
	The point of this lemma is that, under the restriction on the support of $\mu$ and $\nu$, we can assume that the Lipschitz functions appearing in the dual representation of $W_1$ are both periodic and smooth.
	
	We first handle the former restriction.
    Define a metric $\md_{\torus}$ on $\R^\dd$ by
	\begin{equation*}
		\md_{\torus}(x, y) = \min_{z \in \Z^\dd} \|x - y - 2 \pi z\|\,.
	\end{equation*}
	The notation $\md_{\torus}$ is used to emphasize that this is the metric that arises from identifying the opposite faces of $[0, 2 \pi]^\dd$ so that it becomes a flat torus.
	Given any $f \in \mathrm{Lip}_1([0, 1]^\dd)$, define the function $\tilde f: \R^\dd \to \R$ by
	\begin{equation}\label{eq:fbardef}
		\tilde f(y) = \sup_{x \in [0, 1]^\dd}\{f(x) - \md_{\torus}(x, y)\}\,.
	\end{equation}
	For each $x \in [0, 1]^\dd$, the function $y \mapsto f(x) - \md_{\torus}(x, y)$ is $2 \pi$-periodic and Lipschitz with respect to the Euclidean metric on $\R^\dd$ (since both facts are true of $\md_{\torus}(x, y)$).
	Both periodicity and Lipschitzness are preserved by taking pointwise suprema, so these properties are inherited by $\tilde f$ as well.
	We also note the crucial fact that $\tilde f = f$ on $[0, 1]^\dd$: indeed, for $y \in [0, 1]^\dd$, we clearly have $\tilde f(y) \geq f(y)$ by choosing $x = y$ in~\eqref{eq:fbardef}.
	On the other hand, since $d_{\torus}(x, y) = \|x - y\|$ for any $x, y \in [0, 1]^\dd$, we also have
	\begin{equation*}
		f(x) - \md_{\torus}(x, y) = f(x) - \|x - y\| \leq f(y) \quad \forall x \in [0, 1]^\dd\,,
	\end{equation*}
	where the inequality follows from the fact that $f$ is Lipschitz.
	Taking suprema on both sides yields $\tilde f(y) \leq f(y)$.
	
	Since the supports of $\mu$ and $\nu$ lie in $[0 ,1]^\dd$, we therefore have, for any $f \in \mathrm{Lip}_1([0, 1]^\dd)$
	\begin{equation*}
		\int f\, (\ud \mu - \ud \nu) = \int \tilde f\, (\ud \mu - \ud \nu)\,,
	\end{equation*}
	where the function on the right side is Lipschitz and $2 \pi$-periodic.
	This implies that we can always assume that the functions appearing in the dual representation of $W_1$ are periodic.
	
	The restriction to smooth functions is routine: since any Lipschitz function can be uniformly approximated by a smooth function, we can always assume that the functions in question are $C^\infty$.
\end{proof}

Given a probability measure $\mu$, we denote by $\phi_\mu$ its Fourier transform (or characteristic function):
\begin{equation}\label{eq-primal-dual:fourier_def}
	\phi_\mu(m) = \int e^{\ima \langle m, z \rangle} \, \mu(\ud z)\,, \quad m \in \Z^\dd\,.
\end{equation}
The basis of the Bobkov{--}Ledoux argument is the following proposition.

\begin{proposition}\label{fourier_bound}
	\begin{equation}\label{eq-primal-dual:fourier_bound}
		\widetilde{W_1}(\mu, \nu)^2 \leq \sum_{m \neq 0} \|m\|^{-2}\, |\phi_\mu(m) - \phi_\nu(m)|^2\,,
	\end{equation}
	where the sum is over all nonzero $m \in \Z^\dd$ and $\|m\|^2 = m_1^2 + \dots + m_\dd^2$.
\end{proposition}

Before giving the proof, we pause for a moment to compare Proposition~\ref{fourier_bound} to Theorem~\ref{thm:dyadic}.
Both results give a bound on $W_1(\mu, \nu)$ by comparing them at different scales: in the case of Theorem~\ref{thm:dyadic}, this is done by calculating how much they differ on smaller and smaller cubes, in the case of Proposition~\ref{fourier_bound}, this is done by calculating how much they differ at higher and higher frequencies.
In both cases, each term in the sum is weighted by the scale of the comparison ($2^{-j}$ in the case of Theorem~\ref{thm:dyadic}, $\|m\|^{-2}$ in the case of Proposition~\ref{fourier_bound}).
The key difference between these bounds is that Theorem~\ref{thm:dyadic} has an $\ell^1$ flavor, whereas Proposition~\ref{fourier_bound} has an $\ell^2$ flavor.
This different turns out to be the source of the $\sqrt{\log n}$ savings in the rate for $\dd = 2$.

\begin{proof}[Proof of Proposition~\ref{fourier_bound}]
	Given a $2 \pi$-periodic $C^\infty$ function $f$, we can expand it as a Fourier series:
	\begin{equation*}
		f(x) = \sum_{m \in \Z^\dd} \hat f_m\, e^{\ima \langle m, x \rangle}\,,
	\end{equation*}
	where the coefficients $\hat f_m$ tend to zero faster than any polynomial as $\|m\|\to\infty$.
	We may therefore differentiate term-by-term and apply Parseval's identity to obtain
	\begin{equation*}
		\frac{1}{(2 \pi)^\dd}\int_{[0, 2 \pi]^\dd} (\partial_i f(x))^2 \, \ud x = \sum_{m \in \Z^\dd} m_i^2\, |\hat f(m)|^2\,,
	\end{equation*}
	and summing over the coordinates yields
	\begin{equation*}
		\frac{1}{(2 \pi)^\dd} \int_{[0, 2 \pi]^\dd} \|\nabla f(x)\|^2 \, \ud x = \sum_{m \in \Z^\dd} \|m\|^2\, |\hat f(m)|^2\,.
	\end{equation*}
	If we assume that $f$ is $1$-Lipschitz, then $\|\nabla f(x)\|\leq 1$ for all $x \in [0, 2 \pi]^{\dd}$, so
	\begin{equation}\label{eq-primal-dual:l2-fourier-bound}
		\sum_{m \in \Z^\dd} \|m\|^2\, |\hat f(m)|^2 = \frac{1}{(2 \pi)^\dd} \int_{[0, 2 \pi]^\dd} \|\nabla f(x)\|^2 \, \ud x \leq 1\,.
	\end{equation}

	Fubini's theorem therefore implies that for any $1$-Lipschitz, $2 \pi$-periodic $C^\infty$ function $f$,
	\begin{align*}
		\int f \, (\ud \mu - \ud \nu) & = \int  \sum_{m \in \Z^\dd} \hat f_m\, e^{\ima \langle m, x \rangle}\, (\mu(\ud x) - \nu(\ud x)) \\
		& = \sum_{m \in \Z^\dd} \hat f_m\, (\phi_\mu(m) - \phi_\nu(m)) \\
		& = \sum_{m\in \Z^\dd\setminus\{0\}} \hat f_m\, (\phi_\mu(m) - \phi_\nu(m))\,,
	\end{align*}
	where the last equality follows from the fact that $\phi_\mu(0) = \phi_\nu(0) = 1$.
	The result then follows from the Cauchy--Schwarz inequality and~\eqref{eq-primal-dual:l2-fourier-bound}.
\end{proof}

Unfortunately, Proposition~\ref{fourier_bound} is often vacuous---if $\mu$ is not absolutely continuous, then $\phi_\mu$ is not integrable, and the sum in~\eqref{eq-primal-dual:fourier_bound} can diverge.
This problem is immediately apparent when attempting to apply Proposition~\ref{fourier_bound} to the singular empirical measure $\mu_n$.
The solution to this issue is to inject additional regularity into the problem by convolving with Gaussians.
For any $\varepsilon > 0$, we denote by $\nu \star \gamma_\varepsilon$ the convolution of $\nu$ with a $\cN(0, \varepsilon I)$ distribution; equivalently, $\nu \star \gamma_\varepsilon$ is the law of $X + \sqrt{\varepsilon} Z$ where $X \sim \nu$ and $Z \sim \cN(0, I)$ are independent.
We first recall the effect that this smoothing has on the Fourier transform.

\begin{lemma}
	For all $\varepsilon > 0$,
	\begin{equation*}
		\phi_{\nu \star \gamma_\varepsilon}(m) = \phi_{\nu}(m)\, e^{-\varepsilon\,\|m\|^2/2} \quad  \forall m \in \Z^\dd\,.
	\end{equation*}
\end{lemma}
\begin{proof}
	This follows directly from the representation $$\phi_{\nu \star \gamma_\varepsilon}(m) = \int e^{\ima \langle m, y \rangle} \, (\nu \star \gamma_\varepsilon)(\ud y) = \E e^{\ima \langle m, X + \sqrt{\varepsilon}Z \rangle }$$ for $X \sim \nu$ and $Z \sim \cN(0, I)$ independent.
\end{proof}

The smoothing operation is useful because it immediately ensures that the Fourier transform of the resulting measure is well-behaved.
Moreover, smoothing only changes $\widetilde{W_1}$ by a small amount.

\begin{lemma}
	For all $\varepsilon > 0$,
	\begin{equation*}
		\widetilde{W_1}(\mu, \nu) \leq \widetilde{W_1}(\mu, \nu \star \gamma_\varepsilon) + \sqrt{\dd \varepsilon}\,.
	\end{equation*}
\end{lemma}
\begin{proof}
	First, the expression $\widetilde{W_1}$ satisfies the triangle inequality; this follows directly from the definition in~\eqref{eq-primal-dual:w1tilde_def}:
	\begin{align*}
		\widetilde{W_1}(\mu, \nu') + \widetilde{W_1}(\nu', \nu) & = \sup_{f \in \widetilde{\mathrm{Lip}}} \int f\, (\ud \mu - \ud \nu') + \sup_{f \in \widetilde{\mathrm{Lip}}} \int f\, (\ud \nu' - \ud \nu) \\
		& \geq \sup_{f \in \widetilde{\mathrm{Lip}}} {\biggl\{\int f\, (\ud \mu - \ud \nu') + \int f\, (\ud \nu' - \ud \nu)\biggr\}} \\
		& = \widetilde{W_1}(\mu, \nu)\,.
	\end{align*}
	Second, $\widetilde{W_1}$ is dominated by $W_1$, since the supremum in~\eqref{eq-primal-dual:w1tilde_def} is taken over a strict subset of $\mathrm{Lip}$.
	Combining these facts yields
	\begin{equation*}
		\widetilde{W_1}(\mu, \nu) \leq \widetilde{W_1}(\mu, \nu \star \gamma_\varepsilon) + \widetilde{W_1}(\nu, \nu \star \gamma_\varepsilon) \leq \widetilde{W_1}(\mu, \nu \star \gamma_\varepsilon) + {W_1}(\nu, \nu \star \gamma_\varepsilon)\,.
	\end{equation*}
	We now use the fact that $(X, X + \sqrt{\varepsilon} Z)$ with $X \sim \nu, Z \sim \cN(0, I)$ independent is a coupling between $\nu$ and $\nu \star \gamma_\varepsilon$, so that
	\begin{equation*}
		{W_1}(\nu, \nu \star \gamma_\varepsilon) \leq \E \|X - (X + \sqrt{\varepsilon}Z)\| = \sqrt{\varepsilon}\,\E\|Z\| \leq \sqrt{\dd \varepsilon}\,.
	\end{equation*}
    This concludes the proof.
\end{proof}

Combining the preceding two lemmas yields the following corollary to Proposition~\ref{fourier_bound}.

\begin{corollary}\label{cor:fourier_with_smoothing}
	For any $\varepsilon > 0$,
	\begin{equation*}
		\widetilde{W_1}(\mu, \nu) \leq \sqrt{\sum_{m \neq 0} \|m\|^{-2}\, e^{- \varepsilon\, \|m\|^2}\, |\phi_\mu(m) - \phi_\nu(m)|^2} + 2 \sqrt{\dd\varepsilon}\,.
	\end{equation*}
\end{corollary}

We can now prove the desired bound.

\begin{theorem}
	For any probability measure $\mu$ with support in $[0, 1]^2$,
	\begin{equation*}
		\E W_1(\mu_n, \mu) \lesssim \sqrt{\log n/n}\,.
	\end{equation*}
\end{theorem}
\begin{proof}
	Since $\mu$ and $\mu_n$ have support lying in $[0, 1]^2$, we may equivalently prove an upper bound on $\E \widetilde{W_1}(\mu_n, \mu)$.
	Applying Corollary~\ref{cor:fourier_with_smoothing} and Jensen's inequality yields, for any $\varepsilon > 0$,
	\begin{equation*}
		\E \widetilde{W_1}(\mu_n, \mu) \leq \sqrt{\sum_{m \neq 0} \|m\|^{-2}\, e^{- \varepsilon\, \|m\|^2}\, \E |\phi_{\mu_n}(m) - \phi_{\mu}(m)|^2} + 2 \sqrt {2 \varepsilon}\,.
	\end{equation*}
	We can write $\phi_{\mu_n}(m) - \phi_{\mu}(m) = \frac{1}{n} \sum_{i=1}^n \{e^{\ima \langle m, X_i \rangle} - \E e^{\ima \langle m, X_i \rangle}\}$, and since $|e^{\ima \langle m, X_i \rangle}| = 1$ almost surely we conclude that
	\begin{equation}\label{eq:characteristic_1_on_n}
		\E |\phi_{\mu_n}(m) - \phi_{\mu}(m)|^2 \leq n^{-1} \quad \forall m \in \Z^\dd\,.
	\end{equation}
	Continuing, we have for any $\varepsilon > 0$,
	\begin{equation}\label{eq:fourier_trade_off}
		\E \widetilde{W_1}(\mu_n, \mu) \leq n^{-1/2} \sqrt{\sum_{m \neq 0} \|m\|^{-2}\, e^{- \varepsilon\, \|m\|^2}} + 2\sqrt{2 \varepsilon}\,.
	\end{equation}
	By comparing the sum to the integral $\int_{\|x\| \geq 1} \|x\|^{-2}\, e^{-\varepsilon\, \|x\|^2} \, \ud x$, we obtain that the sum is of order $\log(1/\varepsilon)$.
	Therefore, we obtain
	\begin{equation*}
		\E \widetilde{W_1}(\mu_n, \mu) \lesssim \sqrt{\log(1/\varepsilon)/n} + \sqrt{\varepsilon}\,.
	\end{equation*}
	Choosing $\varepsilon = n^{-1}$ gives the claim.
\end{proof}

\section{Applications}\label{sec:applications_w2}

\subsection{Estimation of Wasserstein distances}

So far, we have focused on estimating a measure $\mu$ \emph{in} Wasserstein distance using the empirical measure $\mu_n$. As the title of this chapter indicates, we are often interested in the estimation \emph{of} Wasserstein distances. Indeed, Wasserstein distances are central to many statistical tasks. For example, one of the first applications of the 1-Wasserstein distance (under the name ``earth mover's distance")\index{earth mover's distance} to machine learning was in the context of information retrieval where it was used to measure the distance between images~\cite{RubTomGui00}. Other immediate examples include  nearest neighbors~\cite{BacDonInd20,Pon23} and regression~\cite{GhoPan22,CheLinMul23} for example.

The goal of estimation is to produce an estimator $\widehat{W}$ of $W_1(\mu, \nu)$ using i.i.d.\ data $X_1, \ldots, X_m \sim \mu$ and $Y_1, \ldots, Y_n \sim \nu$. A natural candidate is the \emph{plug-in} estimator\index{plug-in estimator} $\widehat{W} \deq W_1(\mu_m, \nu_n)$ where $\mu_m$ and $\nu_n$ are the empirical measures associated to the samples above. A performance bound for this estimator can be readily obtained using the triangle inequality and Proposition~\ref{prop:w1_rate}: for $d\ge 3$,
$$
\E|W_1(\mu_m, \nu_n) - W_1(\mu, \nu)| \le \E W_1(\mu_m, \mu) + \E W_1(\nu_n, \nu) \lesssim (m\wedge n)^{-1/d}\,.
$$
This coarse bound turns out to be sharp in general. Moreover, using the more general result~\eqref{eq:general_wp_rates} presented at the end of this Chapter, we can get that for $d>2p$,
$$
\E|W_p(\mu_m, \nu_n) - W_p(\mu, \nu)| \le \E W_p(\mu_m, \mu) + \E W_p(\nu_n, \nu) \lesssim (m\wedge n)^{-1/d}\,.
$$
It turns out that when $p>1$, this bound is only sharp when $\mu$ and $\nu$ are sufficiently close. Indeed, better rates can be obtained when $p>1$ and $W_p(\mu, \nu)>c>0$. For example, when $p=2$,~\cite{ManNil24} show that
$$
\E|W_2(\mu_m, \nu_n) - W_2(\mu, \nu)| \lesssim (m\wedge n)^{-2/d}\,.
$$
That paper also shows that this rate is essentially sharp.
While significant, this improvement shows that estimation of Wasserstein distances still suffers from the curse of dimensionality.

\subsection{Hypothesis testing}

These upper bounds can be readily applied to two classical non-parametric hypothesis testing problems: goodness-of-fit and two-sample  (a.k.a.\ homogeneity)  testing. 

Consider first the goodness-of-fit test.\index{goodness-of-fit testing} Given observation $X_1, \ldots, X_n$ i.i.d.\ from some unknown distribution $\mu$, and a fixed distribution $\mu^0$, the goal is to test 
$$
H_0\,:\, \mu=\mu^0 \qtxt{vs.} H_1\,:\, \mu \neq \mu^0\,.
$$
For example, $\mu^0$ can be taken to be a standard Gaussian distribution on $\R^\dd$ or the uniform distribution on $[0,1]^\dd$. There exist many goodness-of-fit tests when $\dd=1$, for example, the Kolmogorov--Smirnov test for continuous distributions. For discrete distributions, the $\chi^2$-test is another popular choice; see, e.g., \cite[Chapter 14]{LehRom05}.

Note that the two hypotheses can be written equivalently as
$$
H^0\,:\, W_1(\mu,\mu^0)=0 \qtxt{vs.} H_1\,:\, W_1(\mu, \mu^0)>0\,.
$$

To study the theoretical limits of a such a setting, we can consider a quantitative version of this testing problem, with hypotheses
$$
H^0\,:\, W_1(\mu,\mu^0) = 0 \qtxt{vs.} H_1\,:\, W_1(\mu, \mu^0)> \eps
$$
for some $\eps > 0$.
We then ask how large the separation $\eps$ must be in order to guarantee that the combined type I and type II errors may be kept small.

Consider a simple test, which consists in rejecting the null hypothesis at level $\alpha \in (0,1)$ as soon as  $W_1(\mu_n,\mu^0)>T_n^\alpha$ for some threshold $T_n^\alpha$ such that
$$
{\mu^0}[W_1(\mu_n,\mu^0)>T_n^\alpha]=\alpha\,.
$$
If we observe $X_i=x_i$, $i=1, \ldots, n$, a $p$-value for such a test may be computed as 
$$
{\mu^0}[W_1(\mu_n,\mu^0)>W_1(\mu_n^\mathrm{obs},\mu^0)]
$$
where 
$$
\mu_n^{\rm obs}=\frac1n\sum_{i=1}^n \delta_{x_i}\,.
$$

Both the computation of $T_n^\alpha$ and of the $p$-value require understanding the actual distribution of $W_1(\mu_n,\mu^0)$ under the null hypothesis. Our results above are actually quite far from achieving this level of precision since we only know an upper bound on $\E[W_1(\mu_n,\mu^0)]$ when $\mu_0$ is supported on the unit cube. Nevertheless, these bounds are sufficient to paint a rather disappointing picture of the potential of the Wasserstein distance in multivariate goodness-of-fit tests.
Indeed, our results suggest that we must take $T_n^\alpha \gtrsim n^{-1/\dd}$ in order to control the type I error of this test, and therefore that a separation of $\eps \gg n^{-1/\dd}$ is necessary to have reasonable power.
Even in moderate dimensions, this level of separation is quite large.

It turns out that the test described above is not optimal for this problem, and that consistent testing is possible with the smaller separation $n^{-2/d}$ via a different approach.
Nevertheless, the simple fact remains that the slow convergence of $W_1$ is an impediment to the use of goodness of fit tests based on the Wasserstein distance.

An explanation for this phenomenon is that a test based on $W_1$ tries to be powerful against too many alternatives. Assume for the sake of discussion that $\mu^0$ is the standard Gaussian distribution over $\R^\dd$. The $1$-Wasserstein distance does not discriminate between distributions that are not $\mu^0$: Gaussian distributions, distributions with smooth densities, those with discontinuous densities, or even discrete distributions. Our test $\{W_1(\mu_n,\mu^0)>T_n^\alpha\}$ tries to detect all of them and spreads thin, resulting in low power against all alternatives. This behavior is to be contrasted with a simple parametric test, for example the Wald test $\{|\bar X_n|>\tau_n\}$ where $\bar X_n = \frac 1n \sum_{i=1}^n X_i$, which simply tries to detect if the mean of $\mu$ differs from that of $\mu^0$. This test is clearly unable to detect even if $\mu$ is a Rademacher distribution, which is quite far from $\mu_0$, but it focuses all of its efforts on shifts in means: when these happen, it can detect them very accurately. 

The manifestation of the curse of dimensionality also extends to two-sample testing where one observes two samples $X_1, \ldots, X_m$ from $\mu$ and $Y_1, \ldots, Y_n$ from $\nu$ and the goal is to test 
$$
H_0\,:\, \mu=\nu \qtxt{vs.} H_1\,:\, W_1(\mu, \nu) > \eps\,.
$$
Denote the corresponding empirical distributions by
$$
\mu_m=\frac1m \sum_{i=1}^m \delta_{X_i}\,, \qquad \nu_n=\frac1n \sum_{j=1}^n \delta_{Y_j}\,.
$$
In this context it is natural to reject the null hypothesis if $W_1(\mu_m, \nu_n)$ is large. Akin to the goodness-of-fit test, such tests require a sample size that is exponential in the dimension to achieve any reasonable power. 

The conclusion of this section is that Wasserstein distances suffer from the curse of dimensionality and are therefore unsuitable for statistical applications of moderate dimension. In Section~\ref{sec:regulOT} we describe various regularizations of Wasserstein distances that escape the curse of dimensionality and have been successfully applied in large-scale statistical applications.

\section{Optimality}\label{sec:estimation_lower_bds}

We have established upper bounds on the Wasserstein distance between the empirical distribution $\mu_n$ and the data generating distribution $\mu$ in two different ways: using the primal and using the dual formulation of the problem. Omitting idiosyncrasies associated to low dimensions, we found that $\mu_n$ estimates $\mu$ in $W_1$ distance at a rate of order $n^{-1/\dd}$. While this result readily yields consistency, the rate is slow even in moderate dimensions and is symptomatic of the curse of dimensionality that plagues most non-parametric methods. One could wonder then whether such rates can be improved.

Note that there are two ways to potentially improve these rates. The most obvious one would be to provide a tighter analysis than the one above and show that in fact, $\E[W_1(\mu_n, \mu)]$ is much smaller than $n^{-1/\dd}$. Another possibility would be that while this rate is tight for the empirical measures $\mu_n$, there could be another estimator $\tilde \mu_n$ of $\mu$ that enjoys much faster rates. In fact, the answer to both questions, while different in nature, is negative, as illustrated by lower bounds. 

While a negative answer to the second question implies a negative answer to the first one---if no estimator can estimate $\mu$ faster than $n^{-1/\dd}$ then certainly the empirical measure $\mu_n$ cannot---we also make the negative answer to the first question explicit since it is, in some sense stronger. Indeed, we show below that even in the case where $\mu$ is the uniform measure on $[0,1]^\dd$ then, $\E[W_1(\mu_n, \mu)]\gtrsim n^{-1/\dd}$. However, in that case, there is clearly a better estimator than $\mu_n$: simply take $\tilde \mu_n=\mu$ itself! The answer to the second question relies on the theory of minimax lower bounds as in \cite[Chapter 2]{Tsy09} and states that for any estimator, i.e., any measurable function $\tilde \mu_n=\tilde \mu_n(X_1, \ldots, X_n)$ of the data $X_1, \ldots, X_n$, there exists $\mu$ supported on $[0,1]^\dd$ such that $\E[ W_1(\tilde\mu_n, \mu)]\gtrsim n^{-1/\dd}$. Unlike the lower bound for the empirical measure $\mu_n$, in the minimax lower bounds, the unfavorable distribution $\mu$ is not explicit.

\subsection{Lower bounds for the empirical measure \texorpdfstring{$\mu_n$}{}}

The goal of this section is to show that any distribution supported on $n$ points has to be far from the uniform measure on $[0,1]^\dd$ in $W_1$ distance.

\begin{theorem}\label{thm:lbW1unif}
Fix $\dd \ge 3$ and let $\mu$ denote the uniform measure on $[0,1]^\dd$. Then for any measure $\tilde \mu_n$ supported on $n$ points $x_1, \ldots, x_n \in \R^\dd$, it holds
$$
W_1(\tilde \mu_n, \mu)\ge\frac1{108\dd}\,n^{-1/\dd}\,.
$$
\end{theorem}
\begin{proof}
We employ the dual formulation of Theorem~\ref{thm:W1duality} since proving a lower bound on $W_1$ can be done by simply exhibiting a $1$-Lipschitz function, which we define as follows. Given $x \in [0,1]^\dd$, let  $\xi(x) \in \{x_1, \ldots, x_n\}$ denote the closest point to $x$ in $\{x_1, \ldots, x_n\}$ (ties are broken arbitrarily). Next,  consider the function 
$$
f_n(x)=\|x-\xi_n(x)\|\,,
$$
which is $1$-Lipschitz thanks to the reverse triangle inequality. Moreover, for any $i=1, \ldots, n$, we have $f_n(x_i)=0$ so that
$\int f\, \ud \tilde \mu_n=0$. Hence 
$$
W_1(\tilde \mu_n, \mu) \ge \int f_n\,\ud \mu = \int \|x-\xi_n(x)\|\,\mu(\ud x)\,.
$$
To bound this quantity from below, we show that $\mu$ places significant mass on points that are far from \emph{any} $x_i$. To that end, consider a partition $\cQ$ of $[0,1]^d$ into cubes of side length $(2n)^{-1/\dd}$. Since $|\cQ|=2n$, there exist $n$ such cubes $Q_1, \ldots, Q_n$ that do not contain any of the $x_i$'s. Let $Q \in \cQ$ be one such cube with center $q$ and consider its subcube $Q' \subset Q$ also with center $q$ but with a smaller side length than $Q$ by a factor of $1-2/\dd$. Using Minkowski sum notation, we can write this as:
$$
Q'=\big(1-\frac2\dd\big)\,(Q-\{q\})+\{q\}\,.
$$
By construction, any $x \in Q'$ satisfies 
$$
\|x-\xi_n(x)\|\ge \inf_{\substack{x \in Q'\\y \in Q^\comp}} \|x-y\|=\frac1\dd\cdot(2n)^{-1/\dd}\,.
$$
Hence
$$
\int \|x-\xi_n(x)\|\,\mu (\ud x) \ge \sum_{i=1}^n \int_{Q_i'} \|x-\xi_n(x)\|\,\mu (\ud x)\ge \frac{(2n)^{-1/\dd}}{\dd}\sum_{i=1}^n \mu(Q_i')\,.
$$
We conclude by observing that 
$$
\mu(Q_i')=\big(\frac{1-2/\dd}{(2n)^{1/\dd}}\big)^\dd\ge \frac{1}{54n}\,,
$$
where we used the fact that $\dd \mapsto (1-2/\dd)^\dd$ is increasing and that $\dd \ge 3$.
\end{proof}

Theorem~\ref{thm:lbW1unif} shows that $W_1(\mu_n, \mu)$ is indeed of order $n^{-1/\dd}$ at least for $\dd\ge 3$. In fact the lower bound holds almost surely in $X_1, \ldots, X_n$ since it only exploits the fact that $\mu_n$ has a support of size at most $n$.  Note that the $\dd$ dependence in Theorem~\ref{thm:lbW1unif} is off by some polynomial factors in $\dd$. It can be shown that the $\sqrt{\dd}$ factor in Proposition~\ref{prop:w1_dyadic_rate} cannot be improved; see Exercise~\ref{exe:unif_measure_\dd_dep}.

\subsection{Minimax lower bounds}

While it is hard to think of a better estimator for $\mu$ than $\mu_n$ in general (in Section~\ref{sec:faster_rate_smoothness} we show that we can under additional assumptions on $\mu$) it could be the case that there exists another estimator $\tilde \mu_n$ for which $\E[W_1(\tilde \mu_n, \mu)]$ is smaller than $\E[W_1(\mu_n, \mu)]$ uniformly over all measures $\mu$. This possibility is ruled out by the following minimax lower bound.

\begin{theorem}\label{thm:mimiaxlbW1}
Fix $\dd\ge 3, n\ge 8$ and let $X_1, \ldots, X_n$ be $n$ i.i.d.\ observations from a distribution $\mu$ on $\R^\dd$. For any estimator $\tilde \mu_n$, i.e., any measurable function of $X_1, \ldots, X_n$, there exists a measure $\mu$ supported on $[0,1]^\dd$ such that
$$
\E_\mu [W_1(\tilde \mu_n, \mu)] \ge \frac{1}{16}\, (2n)^{-1/\dd}\,.
$$
\end{theorem}
\begin{proof}
	Our proof relies on classical techniques for minimax lower bounds. In particular, we use Theorem~2.12 in \cite{Tsy09}. According to this theorem, if we can find $2^m$ probability measures indexed by $\omega \in \{-1, 1\}^m$ each supported on $[0, 1]^\dd$ such that
	\begin{itemize}
		\item[(i)] $W_1(\mu^{(\omega)}, \mu^{(\omega')})\ge \frac{r_n}{2} \sum_{j=1}^m |\omega_j - \omega'_j|$ for any $\omega, \omega' \in  \{-1, 1\}^m$,
		\item[(ii)] for any $\omega, \in \{-1, 1\}^m$ differing in at most one coordinate,
		\begin{equation*}
			{\sf KL}(\mu^{(\omega)}\mmid \mu^{(\omega')}) \le \frac{1}{2n}\,,
		\end{equation*}
	\end{itemize}
	then for any estimator $\tilde \mu_n$ based on $n$ i.i.d.\ observations, there exists $\omega \in \{-1, 1\}^m$ such that 
	$$
	\E_{\mu^{(\omega)}}[W_1(\tilde \mu_n, \mu^{(\omega)})] \ge \frac{m r_n}{4}\,.
	$$
	
	In our construction, we take $m = n$ and define the measures $\mu^{(\omega)}$ to be supported on a discrete set as follows.
	As in the proof of Theorem~\ref{thm:lbW1unif}, let $\cQ$ denote a partition of $[0,1]^\dd$ into $2n$ cubes of side length $(2n)^{-1/\dd}$ and let $q_1, \ldots, q_{2n}$ denote their centers. Let $\mu^{(0)}$ denote the uniform measure on $\{q_1, \ldots,q_{2n}\}$:
	$$
	\mu^{(0)}=\frac1{2n} \sum_{i=1}^{2n} \delta_{q_i}\,.
	$$
	For $\omega \in \{-1, 1\}^{n}$, let $\mu^{(\omega)}$ denote a perturbation of $\mu^{(0)}$ defined as 
	$$
	\mu^{(\omega)}=\mu^{(0)}+ \frac{\alpha}{2n}\sum_{i=1}^n \omega_i\,( \delta_{q_i}- \delta_{q_{n+i}})\,,
	$$
	where $\omega=(\omega_1, \ldots, \omega_{n})$ and $\alpha\in (0,1)$ is to be defined later.
	Note that $\mu^{(\omega)}$ is a probability measure.
	
	Since $\|q_j - q_k\| \ge (2n)^{-1/\dd}$ for $j \neq k$ for we have
	$$
	W_1(\mu^{(\omega)},\mu^{(\omega')}) \ge \frac{\alpha}{2n}\, (2n)^{-1/\dd}\sum_{j=1}^n|\omega_j -\omega'_j| =: \frac{r_n}{2} \sum_{j=1}^n|\omega_j -\omega'_j|
	$$
	for any $\omega, \omega' \in \{0\}^n \cup\{-1, 1\}^n$.
	
	It remains to show that (ii) holds for a suitable choice of $\alpha$. To that end, suppose that $\omega$ and $\omega'$ differ on the $j$th coordinate.
	Observe that
	\begin{align*}
		&{\sf KL}(\mu^{(\omega)}\mmid \mu^{(\omega')})
		=\sum_{i=1}^{2n} \mu^{(\omega)}(q_i)\log\big(\frac{\mu^{(\omega)}(q_i)}{\mu^{(\omega')}(q_i)}\big)\\
		&\qquad =\frac{1}{2n}\left\{(1+\alpha \omega_j)\log\frac{1+\alpha \omega_j}{1-\alpha \omega_j}+(1-\alpha \omega_j)\log\frac{1-\alpha \omega_j}{1+\alpha \omega_j}\right\}\\
		&\qquad = \frac{\alpha}{n} \log\frac{1+\alpha}{1-\alpha}\,,
	\end{align*}
	and this quantity is smaller than $\frac{1}{2n}$ if $\alpha = \frac 14$.
	With this choice of $\alpha$, we obtain
	\begin{equation}
		r_n = \frac{1}{4n}\, (2n)^{-1/d}\,,
	\end{equation}
	which implies the desired bound.
\end{proof}

\section{Faster rates for smooth measures}\label{sec:faster_rate_smoothness}

The preceding section indicates that  no estimator can avoid the slow $n^{-1/\dd}$ rate in general.

There are multiple ways to alleviate this curse of dimensionality and the rest of this chapter illustrates two main approaches. In this section, we impose smoothness assumptions on the measure $\mu$. Such assumptions are classical in non-parametric statistics and known to partially mitigate the curse of dimensionality. In the next section, we describe how modifying/regularizing the Wasserstein distance into other distances that are similar in nature can be used to bypass the curse of dimensionality altogether.\footnote{Since we are interested in improvements to the exponent in the rate of decay, in the remainder of this chapter we ignore dimension-dependent constants in the bounds for clarity.}

The fact that imposing smoothness conditions on $\mu$ can lead to better rates is natural in light of the lower bound presented in Theorem~\ref{thm:mimiaxlbW1}.
The measures used in the proof are mixtures of Dirac masses and are therefore highly ``irregular'' in the sense that they do not even possess densities with respect to the Lebesgue measure.
By assuming that $\mu$ is smooth, we rule out these pathological examples.

For mathematical convenience, we consider smooth densities defined on the torus $\mathbb{T}^\dd \deq \R^\dd/(2 \pi \Z)^\dd$.
Concretely, this can be viewed as isomorphic to the set $[0, 2 \pi)^\dd$, equipped with the metric $\md_{\torus}(x, y)  \deq  \min_{z \in \Z^\dd} \|x - y - 2 \pi z\|$.
On this space, the Wasserstein distance coincides with $\widetilde{W_1}$ defined in Section~\ref{sec:akt}.

We focus on the torus so that we can again use the tools of Fourier analysis.
A similar but slightly more technical argument can extend the results of this section to standard Euclidean space.
Note that the $n^{-1/\dd}$ minimax lower bound proved in the previous section still holds on the torus, so in moving to this setting we have not affected the fundamental statistical difficulty of the problem.

We consider a probability measure $\mu$ on $\mathbb{T}^\dd$ with a density, which we also denote by $\mu$.
We make the assumption that the density of $\mu$ is smooth, in the sense that it lies in a \emph{Sobolev space}.

\begin{definition}
	Given a positive integer $s$, the Sobolev space\index{Sobolev space} $\Sob^s$ consists of all functions $f: \torus \to \R$ such that for every multi-index $\alpha$ with $|\alpha| \leq s$, the derivative $\dd^\alpha f$ lies in $L^2$.
	Given $f \in \Sob^s$, its Sobolev norm is defined to be
	\begin{equation*}
		\|f\|_{\Sob^s}^2 = \max_{|\alpha| \leq s} \int_{\torus} \|\dd^\alpha f\|^2 \, \ud x\,.
	\end{equation*}
\end{definition}

The importance of the Sobolev spaces lies in their close connection with Fourier series.
If $\mu \in \Sob^s$, then its Fourier transform~\eqref{eq-primal-dual:fourier_def} satisfies
\begin{equation*}
	\sum_{m \in \Z^\dd} (1+ \|m\|^{2s})\, |\phi_\mu(m)|^2 < \infty\,.
\end{equation*}
Moreover, this expression in terms of Fourier coefficients actually gives an equivalence of norms.
Indeed, from the Fourier representation
\begin{align*}
    \mu(x)
    &\propto \sum_{m\in\Z^\dd} \phi_\mu(m) \, e^{- \ima \langle m, x\rangle}
\end{align*}
we obtain, for any multi-index $\alpha$,
\begin{align*}
    \dd^\alpha \mu(x)
    &\propto \sum_{m\in\Z^\dd} m^{2\alpha}\, \phi_\mu(m) \,e^{- \ima \langle m, x\rangle}\,,
\end{align*}
where $m^\alpha = m_1^{\alpha_1} \dotsm m_\dd^{\alpha_\dd}$. By Parseval's identity,
\begin{align}
    \int_{\torus} \|\dd^\alpha \mu\|^2 \, \ud x
    &\asymp \sum_{m\in\Z^\dd} m^{2\alpha} \,|\phi_\mu(m)|^2 \label{eq:sob_equiv_1} \\
    &\lesssim \sum_{m\in\Z^\dd} (1+\|m\|^{2s}) \, |\phi_\mu(m)|^2\,. \label{eq:sob_equiv_2}
\end{align}
On the other hand, by the binomial theorem,
\begin{align*}
    \|m\|^{2s}
    = \sum_{|\alpha| = s} m^{2\alpha}\,.
\end{align*}
By~\eqref{eq:sob_equiv_1}, it holds that
\begin{align}
    \sum_{m\in\Z^\dd} \|m\|^{2s}\,|\phi_\mu(m)|^2
    &= \sum_{m\in\Z^\dd} \sum_{|\alpha| = s} m^{2\alpha}\,|\phi_\mu(m)|^2 \nonumber \\
    &\lesssim \sum_{|\alpha| = s} \int_{\torus} \|\dd^\alpha \mu\|^2\,\ud x\,. \label{eq:sob_equiv_3}
\end{align}
By~\eqref{eq:sob_equiv_2} and~\eqref{eq:sob_equiv_3} (applying the latter inequality also for $s=0$), we have shown that
\begin{align*}
    \|\mu\|_{\Sob^s}^2
    &\asymp \sum_{m\in\Z^\dd} \, (1+\|m\|^{2s}) \, |\phi_\mu(m)|^2\,.
\end{align*}

We construct an estimator $\tilde \mu_n$ obtained by estimating the Fourier coefficients of $\mu$ for all $m \in \Z^\dd$ satisfying $\|m\| \leq M$.
Concretely, we define
\begin{equation*}
	\widehat{\phi_\mu}(m) = \phi_{\mu_n}(m) = \frac 1n \sum_{j=1}^n e^{\ima \langle m, X_j \rangle}\,,
\end{equation*}
then for any $M \geq 1$ we set 
\begin{equation*}
	\tilde \mu_n(x) = \frac{1}{{(2\pi)}^\dd} \sum_{\|m\| \leq M} \widehat{\phi_\mu}(m)\, e^{- \ima \langle m, x \rangle}\,.
\end{equation*}
Note that while $\tilde \mu_n$ is always a real-valued function on $\torus$ integrating to $1$, it may not be positive everywhere; however, we ignore this issue for now.
Even when the density $\tilde \mu_n$ takes negative values, the definition of $\widetilde{W_1}$ in terms of its dual representation~\eqref{eq-primal-dual:w1tilde_def} still gives a sensible interpretation of $\widetilde{W_1}(\mu, \tilde \mu_n)$.

We have the following result.

\begin{proposition}\label{prop:smoothed_est}
	Assume $\mu \in \Sob^s(\torus)$ with $\|\mu\|_{\Sob^s} \lesssim 1$.
	For any $M \geq 1$ and $\dd \geq 3$, the estimator $\widetilde \mu_n$ defined above satisfies
	\begin{equation*}
		\E \widetilde{W_1}(\mu, \tilde \mu_n) \lesssim n^{-1/2}\, M^{\dd/2-1} + M^{-(s+1)}\,.
	\end{equation*}
\end{proposition}
\begin{proof}
We first note that
\begin{align*}
    \widetilde{W_1}(\mu, \tilde \mu_n)^2
    &\lesssim
			\sum_{m \neq 0,\, \|m\| \leq M} \|m\|^{-2}\, |\phi_\mu(m) - \widehat{\phi_\mu}(m)|^2 \\
   &\qquad{} + \sum_{\|m\| > M} \|m\|^{-2}\, |\phi_\mu(m)|^2\,.
\end{align*}
	This follows directly from Proposition~\ref{fourier_bound}, using the fact that the (signed) measure $\tilde \mu_n$ has Fourier coefficients $\widehat{\phi_\mu}(m)$ for $\|m\| \leq M$ and zero otherwise.
    As in Section~\ref{sec:akt}, we may use the fact that $|e^{\ima \langle m, X_j\rangle}| \leq 1$ to conclude that $\E |\phi_\mu(m) - \widehat{\phi_\mu}(m)|^2 \leq n^{-1}$.
Therefore,
\begin{equation*}
	\E \widetilde{W_1}(\mu, \tilde \mu_n) \leq n^{-1/2} \sqrt{\sum_{m \neq 0,\, \|m\| \leq M} \|m\|^{-2}} + \sqrt{\sum_{\|m\| > M} \|m\|^{-2}\, |\phi_\mu(m)|^2}\,.
\end{equation*}

Before proceeding, we pause to compare this bound with~\eqref{eq:fourier_trade_off}.
There are two differences: first, the smooth cut-off $e^{-\eps\|m\|^2}$ in~\eqref{eq:fourier_trade_off} has been replaced by the restriction $\|m\| \leq M$.
This change is inessential: since $e^{-\eps\|m\|^2} \ll 1$ when $\|m\|^2 \gg \eps^{-1}$, the smooth cut off term mimics a restriction to $\|m\| \lesssim \eps^{-1/2}$.
The second difference is that the term $2 \sqrt{2 \eps}$ in~\eqref{eq:fourier_trade_off} has been replaced by a term that depends on the higher Fourier coefficients of $\mu$.
This change is crucial, since, as we now show, it implies that the second term automatically becomes smaller when $\mu$ is smooth.

Since $\sum_{m\in\Z^\dd} \|m\|^{2s}\, |\phi_\mu(m)|^2 \lesssim 1$, we may write
\begin{align*}
	\sum_{\|m\| > M} \|m\|^{-2}\, |\phi_\mu(m)|^2
     &\leq M^{-2(s+1)} \sum_{\|m\| > M} \|m\|^{2s}\, |\phi_\mu(m)|^2 \\
     &\lesssim  M^{-2(s+1)}\,. 
\end{align*}

For the first term, we can compare the sum with the integral $(\int_{1 \leq \|x\| \leq M} \|x\|^{-2} \, \ud x)^{1/2}$, which is of order $M^{\dd/2-1}$ when $\dd \geq 3$.
\end{proof}
Tuning $M$ appropriately, we arrive at the following theorem.

\begin{theorem}\label{thm:smoothUB}
	If $\mu \in \Sob^s(\torus)$ with $\|\mu\|_{\Sob^s} \lesssim 1$, then there exists an estimator $\tilde \mu_n$ such that
	\begin{equation*}
		\E \widetilde{W_1}(\mu, \tilde \mu_n) \lesssim n^{-\frac{s+1}{\dd + 2s}}\,.
	\end{equation*}
\end{theorem}
\begin{proof}
	Apply Proposition~\ref{prop:smoothed_est} with $M \asymp n^{1/(\dd + 2s)}$.
\end{proof}

Theorem~\ref{thm:smoothUB} shows that, when $s > 0$, the estimator $\tilde \mu_n$ strictly improves over the empirical measure $\mu_n$.
However, we note that the estimator $\tilde \mu_n$ is a signed measure, which may be viewed as undesirable.
This is a common phenomenon in non-parametric statistics; for instance, in the design of kernel density estimators for very smooth densities, it is necessary to employ higher-order kernels which take negative values.
If a positive estimator is desired, then it is possible to show that the estimator $\bar \mu_n$ defined by
\begin{equation*}
	\bar \mu_n \deq \argmin_{\nu \in \cP(\torus)} \widetilde{W_1}(\nu, \tilde \mu_n)
\end{equation*}
also achieves the bound in Theorem~\ref{thm:smoothUB}: indeed, since $\mu \in \cP(\torus)$,
\begin{equation*}
	\widetilde{W_1}(\mu, \bar \mu_n) \leq \widetilde{W_1}(\mu, \tilde \mu_n) + \widetilde{W_1}(\bar \mu_n, \tilde \mu_n) \leq 2 \widetilde{W_1}(\mu, \tilde \mu_n)\,,
\end{equation*}
so that $\bar \mu_n$ is worse than $\tilde \mu_n$ by a factor of at most $2$.

\section{Regularization of Wasserstein distances}\label{sec:regulOT}

The curse of dimensionality that plagues statistical optimal transport has been recognized since its early days. To overcome this limitation, researchers have proposed multiple solutions which can, in retrospect, be viewed as some kind of regularization of the original optimal transport problem. In the rest of this section, we review three examples and demonstrate how they escape the curse of dimensionality.

\subsection{Integral probability metrics}\label{subsec:ipm}

Recall from the dual chaining argument of Section~\ref{sec:chaining} that the rate $n^{-1/\dd}$ came directly from the entropy number of the class of 1-Lipschitz functions. Lemma~\ref{lem:cov_bound} showed
$$
\log N(\varepsilon,\Lipone([0,1]^\dd)) \lesssim (4 \sqrt \dd/\varepsilon)^\dd\,.
$$
The polynomial scaling in $1/\eps$ is characteristic of non-parametric classes, as opposed to parametric classes where this scaling is logarithmic; see~e.g., \cite{GinNic16}. This raises the question of potentially replacing the class of 1-Lipschitz functions with a smaller, ideally parametric, class of functions. 

Take for example the class of linear functions on $\R^\dd$:
$$
\cF_{\mathrm{lin}} \deq \left\{f(x)=\langle \theta, x\rangle\,: \theta, x \in \R^\dd,\ \|\theta\|=1\right\}\,,
$$
and consider the quantity
\begin{align*}
  \delta(\mu, 
  \nu)&=\sup_{f \in \cF_{\mathrm{lin}}}\left\{ \int f \,\ud \mu - \int f\, \ud \nu \right\}\\
    &=\sup_{\theta \in \R^\dd,\, \|\theta\|=1}\left\{  \int \langle \theta , x\rangle\, \mu(\ud x) - \int \langle \theta , y\rangle\, \nu(\ud y) \right\}\\
    &=\|\E_\mu[X]-\E_\nu[Y]\|\,.
\end{align*}
In particular, $\delta(\mu, \nu)=0$ if and only if $\mu$ and $\nu$ have the same mean. This is of course not sufficient to say that the two measures are the same so the above quantity does not define a distance between probability measures like the Wasserstein distance. To do so, we need to find a class of test functions $\cF$ that is large enough to yield a distance but not as massive as 1-Lipschitz functions so as to escape the curse of dimensionality. 

\begin{definition}\label{def:IPM}
A metric $\md(\cdot, \cdot)$ between two probability measures is called an \emph{integral probability metric} (IPM)\index{integral probability metric (IPM)} if it satisfies the properties of a metric and can be written in the form
\begin{equation}
    \label{eq:IPM}
    \md(\mu, \nu)=\sup_{f \in \cF} \left| \int f\, \ud \mu -\int f\, \ud \nu\right|\,.
\end{equation}
\end{definition}
Note that both the 1-Wasserstein distance $W_1$ and the quantity $\delta$ above are of the form~\eqref{eq:IPM} with $\cF=\mathrm{Lip}_1$ and $\cF=\cF_{\rm lin}$ respectively. Indeed, the absolute value in~\eqref{eq:IPM} is implicit when $\cF$ is symmetric: $\cF=-\cF$. However, while $W_1$ is an IPM, the quantity $\delta$ is not because it fails to satisfy the properties of a metric; here: definiteness. 

Another example of a choice for $\cF$ is the set of bounded Lipschitz functions which indeed yields an IPM, but the size of this class is the same as ${\rm Lip}_1$ for the matter at hand here. To improve the sample complexity, we need much smoother functions.

\subsection{Maximum mean discrepancy}\label{subsec:mmd}

Reproducing Kernel Hilbert Spaces (RKHS) form a flexible and practical class of functions. To define this class of functions very briefly we introduce some basic definitions and key properties. We refer the reader to~\cite{MuaFukSri17} for more details on kernel methods that are particularly relevant to this section.

Consider a reproducing kernel Hilbert space $\cH$ of functions $\R^\dd \to \R$ associated to a bounded positive definite kernel $k$ on $\R^\dd$. Denote by $\langle\cdot, \cdot \rangle_{\cH}$ and $\|\cdot\|_{\cH}$ the inner product and norm on $\cH$ respectively. The reproducing property of the RKHS $\cH$ ensures that for any $f \in \cH$,
$$
\langle k(x,\cdot), f\rangle_{\cH} =f(x)\,.
$$
In particular taking $f=k(y, \cdot)$ yields
$$
\langle k(x,\cdot), k(y,\cdot)\rangle_{\cH}=k(x,y)\,.
$$

We are now in a position to define the Maximum Mean Discrepancy.

\begin{definition}\label{def:MMD}
Let $\cH$ be an RKHS. The  Maximum Mean Discrepancy (MMD) between two probability measures $\mu$ and $\nu$ on $\R^\dd$ is defined to be the quantity
$$
\MMD(\mu, \nu)=\sup_{\substack{f \in \cH\\\|f\|_{\cH}\le 1}} \left|\int f \,\ud \mu - \int f\, \ud \nu\right|\,.
$$
\end{definition}

Without further assumptions on the RHKS, MMD need not define a distance between probability measures. Indeed, observe that the set of linear functions on $\R^\dd$ equipped with the Euclidean inner product is in fact an RKHS associated to the linear kernel $k(x,y)= \langle x , y\rangle$. Moreover, if $f(x)=\langle \theta, x\rangle$, then 
$$
\|f\|_{\cH}^2= \|\langle \theta, \cdot \rangle \|_{\cH}^2=\|k(\theta,\cdot)\|_\cH^2=k(\theta,\theta)=\|\theta\|^2\,.
$$
Hence, the unit ball of $\cH$ is no other than $\cF_{\rm lin}$ and we have shown that this class is not rich enough to define an IPM. 

In fact, we have not addressed whether MMD is finite. To that end, we use the following useful proposition.

\begin{proposition}\label{prop:MMDrewrite}
Let $\cH$ be an RKHS. The  Maximum Mean Discrepancy (MMD)\index{maximum mean discrepancy (MMD)} between two probability measures $\mu$ and $\nu$ on $\R^\dd$ can be equivalently defined as
$$
\MMD(\mu, \nu)=\left\| \int k(x, \cdot)\,\mu(\ud x) - \int k(x, \cdot)\,\nu(\ud x) \right\|_{\cH}\,.
$$
\end{proposition}
\begin{proof}
For any $f\in \cH$  it holds
\begin{align*}
    \int f(x)\, (\mu-\nu)(\ud x) &= \int \langle k(x, \cdot), f\rangle_{\cH}\, (\mu-\nu)(\ud x) \\
    &= \Bigl\langle \int k(x, \cdot)\,(\mu-\nu)(\ud x) , f\Bigr\rangle_{\cH}\,.
\end{align*}
Hence, the claim follows from Cauchy--Schwarz.
\end{proof}

As a corollary of Proposition~\ref{prop:MMDrewrite}, we get that
\begin{align}
\MMD^2(\mu, \nu)
&=\iint k(x, y)\,\mu(\ud x)\,\mu(\ud y)+
\iint k(x, y)\,\nu(\ud x)\,\nu(\ud y)\nonumber\\
&\qquad -2\iint  k(x, y)\,\mu(\ud x)\,\nu(\ud y)\label{eq:MMD2}
\end{align}
and
\begin{align*}
    \MMD(\mu, \nu)
    &\le \int \|k(x,\cdot)\|_\cH\,\mu(\ud x) + \int \|k(x,\cdot)\|_\cH\,\nu(\ud x) \\
    &\le 2\sup_{x\in\R^d} \sqrt{k(x,x)}
    < \infty
\end{align*}
where we used the fact that $k$ is bounded.

The map $\mu \mapsto \int k(x, \cdot)\, \mu (\ud x)$ which embeds $\mu$ onto the RKHS $\cH$ is called \emph{kernel mean embedding}. It follows from Proposition~\ref{prop:MMDrewrite} that MMD is an IPM, meaning that it is indeed a metric, if and only if the kernel mean embedding is \emph{injective}. Kernels that ensure this property are called \emph{characteristic} and one such example is the Gaussian kernel $k(x,y) = e^{-\frac{\|x-y\|^2}{2\sigma^2}}$. To see this, observe that, up to normalizing constants, the kernel mean embedding is a convolution of $\mu$ with a Gaussian measure: for any $y \in \R^\dd$, it holds
   $$
\int k(x, y)\, \mu (\ud x)=\int e^{-\frac{\|x-y\|^2}{2\sigma^2}}\,\mu (\ud x)=(\sigma\sqrt{2\pi})^\dd\, (\mu \star \cN(0, \sigma^2 I))(y)\,.
$$
Injectivity of the convolution with a Gaussian distribution can be seen readily using characteristic functions. Indeed, the characteristic function of $\mu \star \cN(0, \sigma^2 I)$ is given by 
$$
\phi_{\mu \star \cN(0, \sigma^2 I)}(\cdot)=\phi_{\mu}(\cdot)\, \phi_{\cN(0, \sigma^2 I)}(\cdot)=\phi_{\mu}(\cdot)\,e^{-\frac{\sigma^2\,\|\cdot\|^2}{2}}\,.
$$
Hence, since the characteristic function $e^{-\frac{\sigma^2\,\|\cdot\|^2}{2}}$ of the Gaussian is everywhere positive, we get that 
$$
 \mu \star \cN(0, \sigma^2 I)= \nu \star \cN(0, \sigma^2 I)
$$
if and only if $\mu=\nu$. 

Clearly, the above argument generalizes to translation-invariant kernels that are of the form $k(x,y)=K(x-y)$ for some bounded positive definite function $K:\R^\dd \to \R$ and whose Fourier transform is everywhere positive\footnote{Note that Bochner's theorem implies that positive definite kernels have a \emph{non-negative} Fourier transform, so this is a stronger requirement.}. This includes for example the Laplace kernel $k(x,y)=e^{-\|x-y\|}$ as well as other examples; see~\cite[Table~3.1]{MuaFukSri17}.

The representation of MMD given by~\eqref{eq:MMD2} gives an easy way to estimate MMD from data. For example, assume that $X_1, \ldots, X_m$ are i.i.d.\ from $\mu$ and $Y_1, \ldots, Y_n$ are i.i.d.\ from $\mu$. Denote by $\mu_m$ and $\nu_n$ the corresponding empirical distributions.  Then,
\begin{align}
\MMD^2(\mu_m, \nu_n)=\frac1{m^2} \sum_{i,i'=1}^m &k(X_i, X_{i'})+
\frac1{n^2} \sum_{j,j'=1}^n k(Y_{j}, Y_{j'})\nonumber\\
&-\frac2{mn}\sum_{i=1}^m \sum_{j=1}^n    k(X_i, Y_j)\,.\nonumber
\end{align}
A natural question is whether this gives a good estimator of $\MMD^2(\mu, \nu)$. Using the triangle inequality, it is sufficient to control $\MMD(\mu_m,\mu)$. 
While MMD is an IPM, the closed-form representation of Proposition~\ref{prop:MMDrewrite} allows us to bypass the use of empirical process theory to control this quantity.

\begin{theorem}\label{thm:sampleMMD}
Let $k$ be a characteristic kernel such that $k(x,x)\le 1$ for any $x \in \R^\dd$. Let $X_1, \ldots, X_n$ be $n$ i.i.d.\ observations from a distribution $\mu$ on $\R^\dd$ and define the empirical measure
$$
 \mu_n=\frac1n \sum_{i=1}^n \delta_{X_i}\,.
$$
Then
$$
\E_\mu [\MMD(\mu_n, \mu)] \le \frac1{\sqrt{n}}\,.
$$
\end{theorem}
\begin{proof}
It follows from Proposition~\ref{prop:MMDrewrite} that
\begin{align*}
    \E[\MMD^2(\mu_n, \mu)]&=\E\Bigl\| \frac1n \sum_{i=1}^n \{k(X_i, \cdot) - \E k(X_i, \cdot)\} \Bigr\|_{\cH}^2\\
&=\frac1n\, \E\|  k(X_1, \cdot) - \E k(X_1, \cdot) \|_{\cH}^2\\
&=\frac1n\,\bigl( \E\|  k(X_1, \cdot) \|_{\cH}^2- \left\|\E k(X_1, \cdot) \right\|_{\cH}^2\bigr)\\
&\le \frac1n\, \E\|  k(X_1, \cdot) \|_{\cH}^2\,.
\end{align*}
Next, observe that 
$$
\E\|  k(X_1, \cdot)\|_{\cH}^2=\E [k(X_1, X_1)]\le 1\,.
$$
The claim follows from Jensen's inequality.
\end{proof}

We see that unlike Wasserstein distances, MMD does not suffer from the curse of dimensionality. This is certainly a desirable feature, but it may also be interpreted from a more cautious perspective. Indeed, while MMD does define a metric, it is less sensitive to deviations between probability measures and tends to make them small. This is why $\mu_n$, which according to the 1-Wasserstein distance is quite far from $\mu$, appears to be quite close to $\mu$ from the perspective of MMD.

\subsection{Smoothed Wasserstein distances}

We see from Proposition~\ref{prop:MMDrewrite} that when $k$ is the Gaussian kernel, $\MMD(\mu, \nu)$ is a Hilbert space norm involving the densities $\mu \star \cN(0, \sigma^2I_\dd)$ and $\nu \star \cN(0, \sigma^2I_\dd)$. We could very well measure this distance between probability measures using other distances, in particular, using Wasserstein distances.

\begin{definition}\label{def:smoothWp}\index{Wasserstein distance!smoothed}
Fix $p \ge 1$. The smoothed $p$-Wasserstein distance between two probability measures $\mu, \nu \in \cP_p(\R^\dd)$ is defined by
\begin{equation*}
W_p^{(\sigma)}(\mu, \nu) \deq W_p(\mu\star \cN(0, \sigma^2I), \nu \star \cN(0, \sigma^2I))\,.
\end{equation*}
\end{definition}
The idea of computing the Wasserstein distance between smoothed versions of the measures was already used as an analytical tool in the proof of Corollary~\ref{cor:fourier_with_smoothing}; here, we consider it as a notion of distance in its own right.

It follows readily from this definition that the smoothed Wasserstein distance is indeed a distance. Compared to MMD, which embeds distributions in a Hilbert space, the geometry induced on distributions by the smoothed Wasserstein distance is much closer to the original Wasserstein distance. Like MMD, however, smoothed Wasserstein distances enjoy faster statistical rates of convergence. For simplicity, we focus here on the case $p=1$, but parametric rates have been established for $p = 2$ as well.
\begin{theorem}\label{thm:smoothed_rate}
Fix $\sigma>0$. Let $X_1, \ldots, X_n$ be $n$ i.i.d.\ observations from a distribution $\mu$ on $[-1,1]^\dd$ and define the empirical measure
$$
 \mu_n=\frac1n \sum_{i=1}^n \delta_{X_i}\,.
$$
Then
$$
\E_\mu [W_1^{(\sigma)}(\mu_n, \mu)] \lesssim \frac{1}{\sqrt{n}}\,,
$$
where the implicit constant depends on both $\sigma^2$ and $d$.
\end{theorem}
Before turning to the proof, we note that the constant factor in this bound scales exponentially in the dimension.
This poor scaling in $d$ is, in fact, unavoidable and reflects the fundamental statistical difficulty of estimating the Wasserstein distance.
\begin{proof}
	Denote by $f$ the density of $\mu \star \cN(0, \sigma^2I)$ and by $f_n$ the density of $\mu_n \star \cN(0, \sigma^2I)$.
	Write $\varphi(z)  \deq  (2 \pi \sigma^2)^{-d/2} \exp(-\frac{1}{2 \sigma^2}\|z\|^2)$ for the density of $\cN(0, \sigma^2I)$.
	Theorem~\ref{thm:wpTV} implies
	\begin{align*}
		\E W_1^{(\sigma)}(\mu_n, \mu) & \leq \E \int \|z\|\, |f_n(z) - f(z)| \ud z \\
		& = \int \|z\| \,\E \Bigl\lvert\frac 1n \sum_{i=1}^n \varphi(z - X_i) - \E \varphi(z- X_i)\Bigr\rvert\, \ud z\\
		& \leq \frac{1}{\sqrt n} \int \|z\|\, \bigl(\E (\varphi(z - X_1) - \E \varphi(z- X_1))^2\bigr)^{1/2}\, \ud z \\
		& \leq \frac{1}{\sqrt n} \int \|z\|\, (\E \varphi(z - X_1)^2)^{1/2}\, \ud z\,.
	\end{align*}
	It suffices to show that the integral is bounded.
	If $\|z\| \le 2\sqrt d$, then we can use the crude bound $(\E \varphi(z - X_1)^2)^{1/2} \leq (2 \pi \sigma^2)^{-d/2}$.
	If $\|z\| > 2\sqrt d$, then $\|z - X_1\| \geq \|z\| - \|X_1\| \ge \|z/2\|$ almost surely, which yields $(\E \varphi(z - X_1)^2)^{1/2} \leq \varphi(z/2)$.
	We obtain
	\begin{align*}
		\E W_1^{(\sigma)}(\mu_n, \mu) \leq & \frac{(2 \pi \sigma^2)^{-d/2}}{\sqrt n} \int_{\|z\|\le 2\sqrt d} \|z\|\, \ud z + \frac{1}{\sqrt n} \int \|z\|\, \varphi(z/2)\, \ud z \\
	& \lesssim n^{-1/2}\,,
	\end{align*}
	as claimed.
\end{proof}

\subsection{Sliced Wasserstein distances}\label{subsec:sliced}\index{Wasserstein distance!sliced}

Finally, we close this chapter with yet another method to avoid the curse of dimensionality, this time based on considering the Wasserstein distance between one-dimensional projections.

Formally, let $\mbb S^{\dd-1}$ denote the unit sphere in $\R^\dd$ and for $\theta \in \mbb S^{\dd-1}$ let $\Pi^\theta : \R^\dd\to\R$ be the projection $\Pi^\theta(x) \deq \langle \theta, x \rangle$.
Define the \emph{sliced Wasserstein distance} between $\mu, \nu \in \cP_p(\R^\dd)$ to be the quantity
\begin{align}\label{eq:sliced_wass}
    \SW_p(\mu,\nu)
    &\deq \Bigl(\int W_p^p(\Pi^\theta_\# \mu, \Pi^\theta_\# \nu) \, \sigma(\ud\theta)\Bigr)^{1/p}\,,
\end{align}
where $\sigma$ is the uniform measure on $\mbb S^{\dd-1}$.

The idea of considering one-dimensional projections is rooted in applications to imaging and tomography, for which various integral transforms have been introduced. In particular, the \emph{Radon transform} of a measure $\mu$ is defined to be the collection of one-dimensional projections $(\Pi^\theta_\# \mu)_{\theta\in\mbb S^{\dd-1}}$.
It is a classical fact, known as the \emph{Cram\'er{--}Wold theorem}, that the Radon transform of $\mu$ completely characterizes $\mu$, justifying its use in defining a metric over probability measures. Let us start by checking that the axioms of a metric space are indeed satisfied.

\begin{theorem}
    For every $p\ge 1$, $\SW_p$ defines a metric over $\cP_p(\R^\dd)$.
\end{theorem}
\begin{proof}
    Symmetry and non-negativity follow from the corresponding facts about the Wasserstein distance.
    For $\Theta \sim \sigma$, the triangle inequality is verified via
    \begin{align*}
        \SW_p(\mu,\nu)
        &= \bigl(\E W_p^p(\Pi^\Theta_\# \mu, \Pi^\Theta_\# \nu)\bigr)^{1/p} \\
        &\le \bigl(\E W_p^p(\Pi^\Theta_\# \mu, \Pi^\Theta_\# \rho)\bigr)^{1/p} + \bigl(\E W_p^p(\Pi^\Theta_\# \rho, \Pi^\Theta_\# \nu)\bigr)^{1/p} \\
        &= \SW_p(\mu,\rho) + \SW_p(\rho,\nu)\,.
    \end{align*}
    Finally, we must check that $\SW_p(\mu,\nu) = 0$ implies $\mu = \nu$.
    Certainly, $\SW_p(\mu,\nu) =0$ implies that $W_p(\Pi^\theta_\#\mu,\Pi^\theta_\# \nu) = 0$ for \emph{almost every} $\theta \in \mbb S^{d-1}$, which implies $\Pi^\theta_\# \mu = \Pi^\theta_\# \nu$. To finish, we would like to upgrade ``almost every'' to ``every'' to apply the Cram\'er{--}Wold device.

    To do so, we prove a Lipschitz continuity property of the mapping $\theta \mapsto \Pi^\theta_\# \mu$.
    For $\theta' \in \mbb S^{\dd-1}$ and $X\sim \mu$,
    \begin{align*}
        W_p(\Pi^\theta_\# \mu, \Pi^{\theta'}_\# \mu)
        &\le \bigl( \E[|\langle \theta-\theta',X\rangle|^p]\bigr)^{1/p}
        \le (\E\|X\|^p)^{1/p}\, \|\theta-\theta'\|\,.
    \end{align*}
    Together with the $W_p$ triangle inequality, it shows that
    \begin{align*}
        &|W_p(\Pi_\#^\theta \mu, \Pi_\#^\theta \nu) - W_p(\Pi_\#^{\theta'} \mu, \Pi_\#^{\theta'} \nu)| \\
        &\qquad \le W_p(\Pi_\#^\theta \mu, \Pi_\#^{\theta'} \mu) + W_p(\Pi_\#^\theta \nu, \Pi_\#^{\theta'}\nu)
        \lesssim \|\theta-\theta'\|\,.
    \end{align*}
    Therefore, $\theta \mapsto W_p(\Pi^\theta_\# \mu, \Pi_\#^\theta \nu)$ is continuous, and $\SW_p(\mu,\nu) = 0$ implies that this quantity vanishes for every $\theta \in \mbb S^{\dd-1}$.\footnote{An alternative argument proceeds as follows: $\Pi_\#^\theta \mu = \Pi_\#^\theta \nu$ for almost every $\theta$ implies that the characteristic functions of $\mu$, $\nu$ are equal almost everywhere. But characteristic functions are uniformly continuous.}
\end{proof}

We can now prove that the sliced Wasserstein distance can be estimated at a parametric rate.

\begin{proposition}\label{prop:sw_rate}
    Suppose that $\mu, \nu \in \cP(B_1)$, where $B_1$ is the unit ball in $\R^\dd$, and let $\mu_n$, $\nu_n$ denote the corresponding empirical measures formed from i.i.d.\ samples $X_1,\dotsc,X_n \sim \mu$ and $Y_1,\dotsc,Y_n \sim \nu$.
    Then,
    \begin{align*}
        \E \SW_1(\mu_n, \mu)
        &\lesssim n^{-1/2}\,.
    \end{align*}
    Also,
    \begin{align*}
        \E|\SW_1(\mu_n, \nu_n) - \SW_1(\mu,\nu)| \lesssim n^{-1/2}\,.
    \end{align*}
\end{proposition}
\begin{proof}
    The second inequality follows from the first by the triangle inequality. To establish the first, we can note that $\langle \theta, X_1\rangle,\dotsc,\langle \theta, X_n\rangle$ is an i.i.d.\ sample from $\Pi^\theta_\# \mu$, and $\Pi^\theta_\# \mu_n$ is the corresponding empirical measure.
    Hence, from the one-dimensional rate in Proposition~\ref{prop:w1_dyadic_rate}, $\E W_1(\Pi^\theta_\# \mu_n, \Pi^\theta_\# \mu) \lesssim n^{-1/2}$. Then, average over $\theta$.
\end{proof}

Although we have motivated the sliced Wasserstein distance for its statistical benefits, fortuitously it also comes with substantial computational ones.
Indeed, computation of $\SW_p$ boils down to one-dimensional optimal transport, which for discrete measures can be solved via sorting; see Exercise~\ref{ex:sliced_computation}.

\section{Discussion}

\noindent\textbf{\S\ref{sec:wasserstein_lln}.}
The Wasserstein law of large numbers is discussed in more detail in~\cite[Chapter 11]{Dud02}.
The slow rate of convergence is a manifestation of the fact that $W_1$ convergence automatically implies convergence of all Lipschitz text functions (and, for $p > 1$, convergence of all higher moments, as Proposition~\ref{prop:wass_controls_moments} shows); it is therefore not surprising that a large number of samples is needed to obtain such strong control.
The implications of the Wasserstein law of large numbers for optimal quantization and $k$-means clustering were first developed in a statistical context in~\cite{pollard1982quantization}.

\noindent\textbf{\S\ref{sec:dyadic}.}
The usefulness of the dyadic partitioning argument for controlling the Wasserstein distances was first highlighted by \cite{BoiLe-14}; see also \cite{dereich2013constructive,FouGui15} extensions to the unbounded setting.
A further discussion of the history of this approach appears in~\cite{WeeBac19}, from which this version of the argument was taken.
This argument can easily be extended to show that the rate of convergence depends on the intrinsic dimension of the measure $\mu$ rather than the ambient dimension.

The dyadic partitioning argument also applies to the $p > 1$ case, and shows that
\begin{equation}\label{eq:general_wp_rates}
    \E W_p(\mu_n, \mu) \lesssim_p \sqrt \dd \cdot \begin{cases}
    n^{-1/2p}\,, & \text{if $\dd < 2p$\,,} \\
    (\log n)^{1/p}/n^{1/2p}\,, & \text{if $\dd = 2p$\,,} \\
    n^{-1/\dd}\,, & \text{if $\dd > 2p$\,.}
\end{cases}
\end{equation}
This rate is essentially sharp, apart from the logarithmic factor in the $\dd = 2p$ case.
On the other hand, the dual bounds we present in this chapter do \emph{not} easily extend to $p > 1$ since the dual formulation of $W_p(\mu_n, \mu)$ does not give rise to an empirical process when $p > 1$.

The triangle inequality implies that rates of convergence of $W_p(\mu_n, \nu)$ to $W_p(\mu, \nu)$ can be derived from the corresponding rates of convergence of $W_p(\mu_n, \mu)$; however, these rates can fail to be sharp.
To give one example, \cite{ChiRouLeg20} showed that
\begin{align}\label{eq:w2_one_sample}
    \E|W_2^2(\mu,\nu_n) - W_2^2(\mu,\nu)|
    &\lesssim \begin{cases}
        n^{-1/2}\,, &\text{if}~\dd < 4\,, \\
        (\log n)/n^{1/2}\,, &\text{if}~\dd = 4\,, \\
        n^{-2/d}\,, &\text{if}~\dd > 4\,.
    \end{cases}
\end{align}
Note that when $\mu \neq \nu$, this bound is stronger than what could be deduced from~\eqref{eq:general_wp_rates}.
A similar phenomenon exists for other $W_p$ distances~\cite{ManNil24}.

More strikingly, the rate at which $W_p(\mu_n, \nu)$ converges to $W_p(\mu, \nu)$ can be shown to depend on the \emph{smaller} of the intrinsic dimensions of $\mu$ and $\nu$; see~\cite{HunStaMun24}.
In particular, if $\nu$ is supported on a finite number of points (sometimes known as the \emph{semi-discrete}\index{semi-discrete optimal transport} optimal transport problem), then $W_p(\mu_n, \nu)$ converges to $W_p(\mu, \nu)$ at a rate that does not suffer from the curse of dimensionality.
This fact cannot be deduced from bounds on $W_p(\mu_n, \mu)$ alone.

\noindent\textbf{\S\ref{sec:chaining}.} Chaining is an idea that goes back implicitly to Kolmogorov.
In the form of Proposition~\ref{prop:chaining}, it is known as Dudley's entropy integral~\cite{Dud1967Chaining}.
For some of the many references on chaining and its applications, see~\cite{Dud99, Han14, Ver18, Wai19, Tal21UppLow, VaaWel23}.

\noindent\textbf{\S\ref{sec:akt}.}
As mentioned, the result of this section is due to~\cite{AjtKomTus84} and the argument here is taken from~\cite{BobLed21}.
The idea of using Fourier transforms to bound matching costs is originally due to~\cite{CoffSho91}, and was developed extensively by Talagrand~\cite{Tal21UppLow}.
More broadly, there is a large literature on so-called matching problems, e.g.,~\cite{Led17MatchingI, LedZhu21MatchingIII, Tal21UppLow}, and recently techniques from partial differential equations have been used to derive very sophisticated results when $\dd = 2$, see, e.g.,~\cite{AmbStrTre19}.

\noindent\textbf{\S\ref{sec:applications_w2}.} Applications of Wasserstein distances to testing can be found in~\cite{delbarriotests99,HalMoSeg21,gonzaleztests23,Neyetal23WassTesting}.
Goodness-of-fit and two-sample testing problems with Wasserstein distances were studied in~\cite{doba11}; for each problem, they showed that the optimal separation is of order $n^{-c/\dd}$ for some constant $c > 1$.

It was shown in~\cite{NilRig22} that, up to logarithmic factors, a separation of $n^{-1/\dd}$ is necessary for a ``robust'' version of goodness-of-fit problem described above, with hypotheses
$$
H^0\,:\, W_1(\mu,\mu^0) < \eps \qtxt{vs.} H_1\,:\, W_1(\mu, \mu^0)> 2 \eps\,.$$
This lower bound can be used to obtain nearly sharp minimax lower bounds for estimating the Wasserstein distance.

\noindent\textbf{\S\ref{sec:estimation_lower_bds}.} The lower bound in Theorem~\ref{thm:lbW1unif} is due to~\cite{Dud69}.
The minimax lower bound in Theorem~\ref{thm:mimiaxlbW1} was first proved by~\cite{SinPoc18}.
For expositions of minimax lower bound techniques, see~\cite{Tsy09, RigHut17, Wai19}.

\noindent\textbf{\S\ref{sec:faster_rate_smoothness}.} Minimax estimation of smooth densities in the Wasserstein distance was studied in~\cite{Singhetal18DensityEst, Lia21} for $W_1$, and in~\cite{NilBer22MinimaxDensity} for $W_p$, $p > 1$.
The case of $p > 1$ evinces different behavior from the $p = 1$ case: 
\cite{NilBer22MinimaxDensity} showed that the rate in Theorem~\ref{thm:smoothUB} is achievable for $p > 1$ only under the additional assumption that the density of $\mu$ is bounded below; without this assumption, rates of estimation are strictly worse.
The results for the $p > 1$ case are confined to densities lying in Besov classes; extending the arguments of \cite{NilBer22MinimaxDensity} to other classes of densities is an open question.
A version of this problem on manifolds has been studied in~\cite{Div22}.

Non-parametric density estimation is itself a classical topic in statistics, albeit usually studied in other distance metrics~\cite{Tsy09}.

\noindent\textbf{\S\ref{sec:regulOT}.}
It is worth mentioning that IPMs have received significant attention in the context of Generative Adversarial Networks (GANs), and in particular, Wasserstein GANs~\cite{ArjChiBot17} where $\cF$ is chosen to be a family of deep neural networks.
For statistical analyses of IPMs and GANs, see, e.g.~\cite{UppSinPoc19,Lia21}.

Maximum mean discrepancy was first developed in~\cite{BorGreRas06}, and is now the subject of a large literature, see,~\cite{MuaFukSri17}. The rate given in Theorem~\ref{thm:sampleMMD} is folklore.
A notable special case of MMD is the class of energy distances, for which we recommend~\cite[Subsection 1.2.4]{Ger24Thesis} for an introduction and references.

The favorable statistical properties of the smoothed Wasserstein distances were first recorded in~\cite{GolGreNil20}.
The simple proof of Theorem~\ref{thm:smoothed_rate} is taken from~\cite{Wee18}.

Sliced Wasserstein distances were introduced in~\cite{Rabin+12Barycenter}. They arose as a device to understand the ``iterative distribution transfer'' algorithm~\cite{PitKokDah07DistTransfer}; see the PhD thesis of Bonnotte~\cite{Bon13Thesis} for history.
For further discussion, consult~\cite[Section 5.5.4]{San15}, and see~\cite{NadDurChi20} for generalizations with similar metric and statistical properties.
Extensions of Proposition~\ref{prop:sw_rate} appear in~\cite{ManBalWas22}.
The sharp condition for the fast rate of estimation of sliced Wasserstein distances to hold was obtained in~\cite{BobLed19OneDim}.
Other results in this vein, including distributional limits, can be found in~\cite{ManBalWas22, NilRig22, OkaIma22ProjWass, XiNil22Sliced, XuHua22CLTSliced, ParSle23Sliced, Gol+24RegOT}.

\section{Exercises}

\begin{enumerate}
	\item\label{exe:unif_measure_\dd_dep} This exercise shows that the $\sqrt \dd$ factor in Proposition~\ref{prop:w1_rate} cannot be improved.
	\begin{enumerate}
		\item Show that there exists a positive universal constant $c$ such that for any $n \geq 1$, the Lebesgue measure of a ball in $\R^\dd$ with radius $c \sqrt{\dd}\, n^{-1/\dd}$ is at most $(2n)^{-1}$. (Hint: recall that the unit ball in $\R^\dd$ has Lebesgue measure $\pi^{\dd/2}/\Gamma(\frac \dd 2 + 1)$.)
		\item Let $\mu$ be the uniform measure on $[0, 1]^\dd$, and let $\tilde \mu_n$ be any measure supported on $n$ points $x_1, \dots, x_n$.
		If we denote by $B(x_i, \epsilon)$ a ball of radius $\epsilon$ around $x_i$, show that $\mu (\bigcup_{i=1}^n B(x_i, \epsilon)) \leq \frac 12$ if $\epsilon = c \sqrt{\dd}\, n^{-1/\dd}$, where $c$ is the constant from part (a).
		\item Conclude that if $\gamma \in \Gamma(\mu, \tilde \mu_n)$, then $\int \|x - y\| \, \ud \gamma(x, y) \geq \frac 12 \cdot c \sqrt{\dd}\, n^{-1/\dd}$.
	\end{enumerate}
 
	\item Adapt the proof of Proposition~\ref{prop:w1_dyadic_rate} to establish~\eqref{eq:general_wp_rates}.
 
    \item Recall the bounded differences inequality (e.g.,~\cite[Theorem 6.2]{BouLugMas13}). Use it to prove a concentration inequality for $W_1(\mu_n,\mu)$ around its expectation, where $X_1,\dotsc,X_n$ are i.i.d.\ from a distribution $\mu$ supported on a ball of radius $R$ and $\mu_n$ is the empirical measure $\mu_n = \frac{1}{n} \sum_{i=1}^n \delta_{X_i}$.

    \item Prove that $\SW_p \le W_p$ for any $p\ge 1$.
    Is this inequality tight?
    Similarly, show that $W_1^{(\sigma)} \leq W_1$.
    Is this inequality tight?

    \item\label{ex:sliced_computation} We consider the computational aspects of the sliced Wasserstein distance, defined in Subsection~\ref{subsec:sliced}.
    \begin{enumerate}
    \item Show that if $\mu, \nu \in \cP_p(\R)$ ($p\ge 1$), and $\mu$ and $\nu$ are each uniformly distributed on $n$ points, then $W_p(\mu,\nu)$ can be computed in $O(n \log n)$ time (where we treat arithmetic and comparison operations as constant time).
        Assume that the measures $\mu$, $\nu$ are given as (unordered) lists of points.
    \item Now suppose that $\mu,\nu \in \cP_p(\R^\dd)$ are each uniformly distributed on $n$ points, presented as lists of points in $\R^\dd$.
        Argue that we can compute $W_p(\Pi^\theta_\# \mu, \Pi^\theta_\# \nu)$ in $O(\dd n + n \log n)$ time.
    \item This is still insufficient for algorithmic purposes, since computing $\SW_1(\mu,\nu)$ exactly requires computing $W_1(\Pi^\theta_\# \mu, \Pi^\theta_\# \nu)$ for uncountably many values of $\theta$ and integrating.
        Argue instead that if $\mu, \nu \in \cP(B_1)$ and we draw $m$ i.i.d.\ points $\theta_1,\dotsc,\theta_m$ from the uniform measure $\sigma$ on $\mbb S^{\dd-1}$, then the Monte Carlo average
        \begin{align*}
            \widehat{\SW_1}(\mu, \nu)
            &\deq \frac{1}{m}\sum_{i=1}^m W_1(\Pi^{\theta_i}_\# \mu, \Pi^{\theta_i}_\# \nu)
        \end{align*}
        approximates $\SW_1(\mu,\nu)$ to an additive error of size $O(m^{-1/2})$.
    \end{enumerate}
    
\item Theorem~\ref{thm:W1_1D} implies that if $\mu \in \cP_1(\R)$, then
\begin{align*}
    W_1(\mu_n, \mu) = \int_{-\infty}^{\infty} |F_{\mu_n}(t) - F_\mu(t)| \,\ud t\,,
\end{align*}
where $F_{\mu_n}$ and $F_\mu$ are the cumulative distribution functions of $\mu_n$ and $\mu$, respectively.
Use this fact to show Proposition~\ref{prop:w1_rate} for $d = 1$ directly.
(Hint: $\E |F_{\mu_n}(t) - F_\mu(t)|^2 = F_\mu(t)\,(1-F_\mu(t))$.)

\item The minimax lower bound proved in Theorem~\ref{thm:mimiaxlbW1} is suboptimal when $d = 1$.
This exercise proves an optimal bound based on testing between two hypotheses
(Theorem 2.2 in \cite{Tsy09}).
To use this approach, it suffices to construct two measures $\mu^{(0)}$ and $\mu^{(1)}$ with support in $[0, 1]$ such that 
\begin{align*}
W_1(\mu^{(0)}, \mu^{(1)}) & \geq 2 r_n\,, \\
{\sf KL}(\mu^{(1)}\mmid \mu^{(0)}) & \leq \tfrac{1}{2n}\,.
\end{align*}
The existence of such measures implies that for any estimator $\tilde \mu_n$ based on $n$ i.i.d.\ observations, there exists $j \in \{0, 1\}$ such that $\E_{\mu^{(j)}}[W_1(\tilde \mu_n, \mu^{(j)})] \geq \frac{r_n}{4}$.
\begin{enumerate}
    \item Fix $\eps \in (0, 1/2)$, and consider $\mu^{(0)} = (\tfrac 12 + \eps) \delta_0 + (\tfrac 12 - \eps) \delta_1$ and $\mu^{(1)} = (\tfrac 12 - \eps) \delta_0 + (\tfrac 12 + \eps) \delta_1$.
    Show that $W_1(\mu^{(0)}, \mu^{(1)}) = 2 \eps$.
    \item Show that this pair of measures satisfies ${\sf KL}(\mu^{(1)}, \mu^{(0)}) \leq  \frac{16  \eps^2}{1 - 4 \eps^2}$.
    \item Conclude that there exists a positive universal constant $c$ such that for any estimator $\tilde \mu_n$ there exists $j \in \{0, 1\}$ such that $\E_{\mu^{(j)}}[W_1(\tilde \mu_n, \mu^{(j)})] \geq c n^{-1/2}$.
\end{enumerate}

    \item The regularization strategies discussed in Section~\ref{sec:regulOT} can be applied more broadly.
    For example, given probability measures $\mu$, $\nu$, define the following quantity and call it ``smoothed $L_2$'':
    \begin{align*}
        d_\sigma^2(\mu,\nu) \deq \int \bigl(\mu\star \cN(0,\sigma^2 I) - \nu \star \cN(0,\sigma^2 I)\bigr)^2\,.
    \end{align*}
    Show that it can be estimated at a parametric rate: if $\mu_n$ denotes the empirical measure formed from $n$ i.i.d.\ samples from $\mu$, then
    \begin{align*}
        \E d_\sigma^2(\mu_n,\mu)
        &\le \frac{1}{{(2\pi\sigma^2)}^{d/2}\,n}\,.
    \end{align*}
    
    \item For $\mu, \nu \in \cP_p(\R^d)$, the ``max-sliced Wasserstein distance''~\cite{DesHuSun19,KolNadSim19}\footnote{Also known as the ``low-dimensional Wasserstein distance''~\cite{NilRig22} or the ``subspace robust Wasserstein distance''~\cite{PatCut19}.} between them is
    \begin{equation*}
    	\mathsf{MSW}_p(\mu, \nu) \deq \max_{\theta \in \mbb S} W_p(\Pi^\theta_\# \mu, \Pi^\theta_\# \nu)\,.
    \end{equation*}
    The goal of this exercise is to show that $\mathsf{MSW}_1$ can be estimated at the parametric rate.
    \begin{enumerate}
    	\item Let $\mu \in \cP(B_1)$, and let $\mu_n$ be the empirical measure corresponding to $X_1, \dots, X_n \simiid \mu$.
    	Show that
    	\begin{equation*}
    		\mathsf{MSW}_1(\mu_n, \mu) = \sup_{f \in \cF} \frac 1n \sum_{i=1}^n \{f(X_i) - \E f(X_i)\}\,,
    	\end{equation*}
    	where $\cF$ is the class of functions of the form $x \mapsto h(\theta^\T x)$ where $\theta \in \mbb S^{d-1}$ and $h$ is a $1$-Lipschitz function on $[-1, 1]$ satisfying $h(0) = 0$.
    	\item Prove that
    	\begin{equation*}
    		\log N(\eps, \cF) \lesssim 1/\eps + d \log(1+1/\eps)\,.
    	\end{equation*}
    	\emph{Hint:} Consult Exercise~\ref{ex:rd_covering_numbers} in Chapter~\ref{chap:maps}.
    	\item Using Proposition~\ref{prop:chaining}, conclude that
    	\begin{equation*}
    		\E \, \mathsf{MSW}_1(\mu_n, \mu) \lesssim \sqrt{d/n}\,.
    	\end{equation*}
    	In fact, a more sophisticated argument shows that the correct rate is dimension free~\cite{boedihardjo2024sharp}.
    \end{enumerate}
\end{enumerate}

\chapter{Estimation of transport maps}
\label{chap:maps}

Thus far, the statistical questions we have investigated center around the estimation of optimal transport distances (and their variants), but the gamut of diverse applications of optimal transport (to name but a few: data fusion~\cite{CouFlaTui16} adaptation/transfer learning~\cite{CouFlaTui17}, and computational biology~\cite{SchShuTab19,BunStaGut23}), it is the optimal transport \emph{map} which is the object of primary interest.
In this chapter, we address the question of estimating this map on the basis of finitely many samples.

\section{Problem formulation}\label{sec:maps_intro}

Recall from Brenier's theorem that if $\mu$ has a density,
\begin{align*}
    W_2^2(\mu,\nu)
    &= \min_{\gamma \in \Gamma_{\mu,\nu}} \int \|x-y\|^2\,\gamma(\ud x,\ud y) \\
    &= \min_{T : T_\# \mu = \nu} \int \|x-T(x)\|^2 \, \mu(\ud x)\,.
\end{align*}
Moreover, the optimal transport map takes the form $T = \nabla \varphi$, where $\varphi$ is convex.
We can also write this as $(X, \nabla \varphi(X)) \sim \gamma$, or $\gamma(\ud x, \ud y) = \mu(\ud x) \, \delta_{T(x)}(\ud y)$.

The statistical question under investigation is formulated as follows.
Given samples $X_1,\dotsc,X_n\simiid \mu$ and $Y_1,\dotsc,Y_n\simiid \nu$, how can we estimate the optimal transport map $T$ from $\mu$ to $\nu$ via an estimator $\widehat T$ constructed on the basis of the samples?

We take as our measure of performance the integrated error
\begin{align*}
    \int \|\widehat T(x) - T(x)\|^2 \, \mu(\ud x)\,.
\end{align*}
The $L^2$ integrated error is a natural measure of distance that is commonly employed in non-parametric statistics.
In this context, however, it takes on an additional interpretation of controlling the Wasserstein distance between the pushforwards of $\mu$ under the two maps.
More precisely, by definition we have $\nu = T_\# \mu$.
If we define the measure $\widehat \nu$ via $\widehat \nu = \widehat T_\# \mu$, then $(\widehat T(X), T(X))$ for $X \sim \mu$ is a (suboptimal) coupling of $\widehat \nu$ and $\nu$, and hence
\begin{align*}
    W_2^2(\widehat \nu,\nu)
    &\le \E\|\widehat T(X) - T(X)\|^2
    = \int \|\widehat T(x) - T(x)\|^2 \, \mu(\ud x)\,.
\end{align*}
See Figure~\ref{fig:ot_est} for an illustration.

\begin{figure}[ht]
    \centering
    \begin{tikzpicture}[
    dot/.style={circle, fill, minimum size=3pt, inner sep=0pt, outer sep=0pt},
    arrow/.style={-{Stealth[length=5pt]}, bend left=20}
    ]
    \node[dot, label=below:{$\mu$}] (mu) {};
    \node[dot, right=3cm of mu, label=below:{$\nu$}] (nu) {};
    \node[dot, above=1.5cm of nu, label=above:{$\widehat{\nu} = \widehat{T}_\# \mu$}] (nuhat) {};

    \draw[arrow] (mu) to node[midway, sloped, anchor=south] {$T$} (nu);
    \draw[arrow] (mu) to node[midway, sloped, anchor=south] {$\widehat{T}$} (nuhat);
    \end{tikzpicture}
    \caption{The $L^2$ error between the transport maps controls the $W_2^2$ distance between $\widehat \nu$ and $\nu$.}\label{fig:ot_est}
\end{figure}
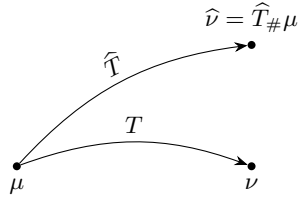

A first approach to estimation might be to compute the optimal coupling between the empirical measures $\mu_n$ and $\nu_n$, i.e., solve $\min_{\gamma \in \Gamma_{\mu_n,\nu_n}} \int \|x-y\|^2 \, \gamma(\ud x,\ud y)$, but we rapidly recognize an untenable hole in this na\"{\i}ve plan.
Namely, even if the optimal transport plan $\gamma_n$ is induced by a transport map $T_n$, so that $\gamma_n(\ud x, \ud y) = \mu_n(\ud x) \, \delta_{T_n(x)}(\ud y)$, the mapping $T_n$ is only well-defined on the sample $\{X_1,\dotsc,X_n\}$ and it is not clear how to extend it in a principled manner to a mapping over all of $\R^\dd$ (Figure~\ref{fig:out_of_sample}).
To remedy this, several approaches have been proposed in the literature aimed at building an interpolation $\widehat T_n$ of $T_n$ to out-of-sample points.
For example, we can take $\widehat T_n(x)$ to equal $T_n(X_i)$, where $X_i$ is the closest sample point to $x$. This is a 1-nearest neighbor estimator and it can be shown to be minimax optimal without further smoothness assumptions~\cite{ManBalNil21}; see Exercise~\ref{ex:1d_ot_map_est} for the one-dimensional case. Such an approach, however, cannot take advantage of additional regularity of $\mu$ and $\nu$ and we do not pursue it any further here.

\begin{figure}[ht]
    \centering
    \begin{tikzpicture}[
    dot/.style={circle, fill, minimum size=3pt, inner sep=0pt, outer sep=0pt},
    arrow/.style={-{Stealth[length=5pt]}, bend left=20}
    ]
    \node[dot] (left1) {};
    \node[dot, below left=0.65cm and 0.2cm of left1] (left2) {};
    \node[dot, below right=0.75cm and 0.15cm of left2] (left3) {};
    \node[dot, below right=0.6cm and 0.25cm of left3] (left4) {};
    \node[dot, below left=0.2cm and 0.6cm of left4, label=below:{$x$}, label=above:{\textbf{?}}] (left5) {};

    \node[dot, right=3cm of left1] (right1) {};
    \node[dot, below=0.75cm of right1] (right2) {};
    \node[dot, below=0.6cm of right2] (right3) {};
    \node[dot, below=0.7cm of right3] (right4) {};

    \draw[arrow] (left1) to (right1);
    \draw[arrow] (left2) to (right2);
    \draw[arrow] (left3) to (right3);
    \draw[arrow] (left4) to (right4);
    \end{tikzpicture}
    \caption{How do we interpolate the empirical optimal transport map at the out-of-sample point $x$?}\label{fig:out_of_sample}
\end{figure}
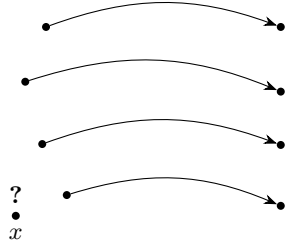

Instead, in the next section, we devise an estimator based on the semidual formulation of optimal transport.
A benefit of this estimation strategy is that it can be used to flexibly incorporate additional assumptions---e.g., smoothness---on the population-level transport map $T$.
Adopting sufficiently strong assumptions gives rise to map estimators that avoid the curse of dimensionality.

\section{The semidual problem and its stability}\label{sec:map_semidual}\index{semidual!stability}

We recall the semidual problem: if $\nabla \varphi$ is the optimal transport map, then $\varphi$ minimizes
\begin{align*}
    \cS(\phi)
    &\deq \int \phi \, \ud \mu + \int \phi^* \, \ud \nu
\end{align*}
where $\phi^*(y) = \sup_{x\in\R^\dd}\{\langle x, y \rangle - \phi(x)\}$ is the convex conjugate of $\phi$.
Crucially, the semidual problem readily lends itself to replacing the population measures $\mu$, $\nu$ with their empirical counterparts $\mu_n$, $\nu_n$, leading to a natural estimator for $\varphi$: namely, we set
\begin{align}\label{eq:empirical_semidual}
    \widehat \varphi
    &= \argmin_{\phi \in \cF} \cS_n(\phi)
    \deq \argmin_{\phi \in \cF}{\Bigl\{\int \phi \, \ud \mu_n + \int \phi^* \, \ud \nu_n\Bigr\}}
\end{align}
where $\cF$ is a suitable class of functions to be chosen later.
We then obtain an estimator for the optimal transport map by setting $\widehat T = \nabla \widehat \varphi$.
Through~\eqref{eq:empirical_semidual}, we have placed the problem of transport map estimation within the well-studied framework of empirical risk minimization (ERM) which is a cornerstone of modern statistical theory---see, e.g.,~\cite{Wai19} for a modern overview of these techniques.
Akin to many other estimators defined via ERM, it is unclear whether the estimator $\widehat T$ can be computed efficiently; however, our focus here is on the statistical, rather than computational, aspects of transport map estimation.

Through the statistician's lens, the uniqueness assertion in Brenier's theorem ensures that the optimal transport map $T$ is identifiable, and hence our statistical question is well-posed. In other words, if $\cS(\phi)=\cS(\varphi)$ then $\nabla \phi =\nabla \varphi$, $\mu$-a.s.
However, in order to obtain rates of estimation, this qualitative assertion needs to be upgraded into a stability statement, which is given as the following theorem.

\begin{theorem}\label{thm:stability_semidual}
    Assume that $\phi$ is strongly convex and smooth,
    \begin{align*}
        \frac{1}{2}\,I \preceq \nabla^2 \phi \preceq 2 \, I\,.
    \end{align*}
    Then,
    \begin{align}\label{eq:stability_semidual}
        \frac{1}{4} \,\|\nabla \phi - \nabla \varphi\|_{L^2(\mu)}^2
        &\le \cS(\phi) - \cS(\varphi)
        \le \|\nabla \phi - \nabla \varphi\|_{L^2(\mu)}^2\,.
    \end{align}
\end{theorem}

Before proving Theorem~\ref{thm:stability_semidual}, we first describe how the stability result feeds into the overall statistical analysis.
The proof is prototypical of analysis of ERM estimators.
By definition, $\widehat \varphi$ minimizes $\cS_n$.
Applying the machinery of empirical process theory, we control the fluctuations of the random functional $\cS_n$ from its mean $\cS$, thereby concluding that $\cS(\widehat \varphi) - \cS(\varphi)$ is small.
The first inequality in~\eqref{eq:stability_semidual} then implies that the estimation error $\|\widehat T - T\|_{L^2(\mu)}^2$ is small.

Actually, to obtain faster rates of estimation, we improve upon this argument by incorporating another ingredient: the \emph{fixed-point} or \emph{localization} technique.
Briefly, the estimation rates depend on a uniform bound on the deviations of $\cS_n$ from $\cS$ over a set of functions that contains the estimator $\widehat \varphi$.
Once we know through the stability inequality~\eqref{eq:stability_semidual} that $\widehat \varphi$  lies close to $\varphi$, we can repeat the argument but restricting to a smaller class of functions, thereby improving our estimation rates further.
Seeking the fixed point of this iterative process in which we refine our bounds by localizing the estimator $\widehat \varphi$, we arrive at our final rates of estimation.

We now turn towards the proof of Theorem~\ref{thm:stability_semidual}.
We repeatedly use the Fenchel{--}Young inequality (Theorem~\ref{thm:fenchel_young}), as well as the fact that $\alpha$-convexity of $f$ is equivalent to $\alpha^{-1}$-smoothness of $f^*$ (Lemma~\ref{thm:strcvx_smooth_dual}).

\begin{proof}[Proof of Theorem~\ref{thm:stability_semidual}]
    Since ${(\nabla \varphi)}_\# \mu = \nu$,
    \begin{align*}
        \cS(\phi)
        &= \int \phi(x) \, \mu(\ud x)+ \int \phi^*(y) \, \nu(\ud y)
        = \int \bigl(\phi(x) + \phi^*(\nabla \varphi(x))\bigr) \, \mu(\ud x)\,.
    \end{align*}
    By strong convexity of $\phi^*$,
    \begin{align*}
        \phi^*(\nabla \varphi(x))
        &\ge \phi^*(\nabla \phi(x)) + \langle \underbrace{\nabla \phi^*(\nabla \phi(x))}_{=x}, \nabla \varphi(x) - \nabla \phi(x)\rangle \\
        &\qquad{} + \frac{1}{4} \, \|\nabla \varphi(x) - \nabla \phi(x)\|^2
    \end{align*}
    hence
    \begin{align*}
        \phi(x) + \phi^*(\nabla \varphi(x))
        &\ge \underbrace{\phi(x) + \phi^*(\nabla \phi(x))}_{=\langle x, \nabla \phi(x)\rangle} + \langle x, \nabla \varphi(x) - \nabla \phi(x) \rangle \\
        &\qquad{} + \frac{1}{4} \,\|\nabla \varphi(x) - \nabla \phi(x)\|^2 \\
        &= \langle x, \nabla \varphi(x) \rangle + \frac{1}{4} \,\|\nabla \varphi(x) - \nabla \phi(x)\|^2\,.
    \end{align*}
    However,
    \begin{align*}
        \cS(\varphi)
        &= \int \bigl(\varphi(x) + \varphi^*(\nabla \varphi(x)) \bigr) \, \mu(\ud x)= \int \langle x, \nabla \varphi(x)\rangle \, \mu(\ud x)\,.
    \end{align*}
    Therefore, we obtain
    \begin{align*}
        \cS(\phi)
        &\ge \cS(\varphi) + \frac{1}{4} \,\|\nabla \varphi - \nabla \phi\|_{L^2(\mu)}^2\,.
    \end{align*}

    Similarly, by smoothness,
    \begin{align*}
        \phi^*(\nabla \varphi(x))
        &\le \phi^*(\nabla \phi(x)) + \langle \underbrace{\nabla \phi^*(\nabla \phi(x))}_{=x}, \nabla \varphi(x) - \nabla \phi(x) \rangle \\
        &\qquad{} + \|\nabla \varphi(x) - \nabla \phi(x)\|^2
    \end{align*}
    hence
    \begin{align*}
        \phi(x) & + \phi^*(\nabla \varphi(x))
       \le\phi(x) + \phi^*(\nabla \phi(x)) +\langle x, \nabla \varphi(x) \rangle - \langle x, \nabla \phi(x) \rangle \\
       & \qquad + \|\nabla \varphi(x)- \nabla \phi(x)\|^2\\
        &=\langle x, \nabla \phi(x) \rangle +\varphi(x) + \varphi^*(\nabla \varphi(x))- \langle x, \nabla \phi(x) \rangle \\
       & \qquad + \|\nabla \varphi(x)- \nabla \phi(x)\|^2
    \end{align*}
    and the result follows from integration.
\end{proof}

Note that we have proved something even stronger: for
\begin{align*}
    \cS(\phi)
    &= \int \underbrace{\bigl(\phi(x) + \phi^*(\nabla \varphi(x))\bigr)}_{s_\phi(x)} \, \mu(\ud x)
\end{align*}
we have the pointwise bounds
\begin{align*}
    \frac{1}{4}\,\|\nabla \varphi(x) - \nabla \phi(x)\|^2
    \le s_\phi(x) - s_\varphi(x)
    \le \|\nabla \varphi(x) - \nabla \phi(x)\|^2\,.
\end{align*}

\section{A special case: affine transport maps}\label{sec:affine_transport}

As a sanity check, we first show that the semidual estimation technique is reasonable for a very simple problem.
Consider the one-sample setting, where $\mu = \cN(0, I)$ is known, and we obtain samples from $\nu = (\nabla \varphi)_\sharp \mu$ for some $\varphi \in \cF$, where  $\cF$ consists of all convex quadratic functions $x \mapsto \frac 12\, x^\top A x + b^\top x$ with $A \succeq 0$.
If $\varphi$ is of this form, then the transport map $\nabla \varphi$ is the affine map $Ax + b$, and $\nu = \cN(b, A^2)$.

In this setting, it is natural to estimate the transport map by first computing the empirical mean $\widehat m$ and covariance $\widehat \Sigma$ of $\nu$, and setting
\begin{equation}\label{eq:linmongemap}
    \widehat T(x) = \widehat \Sigma^{1/2} x + \widehat m\,.
\end{equation}
This estimator is studied in more generality in~\cite{FlaLouFer20} by leveraging techniques to derive rates of estimation for covariance matrices.

The next result shows that $\widehat T$ defined in~\eqref{eq:linmongemap} is precisely the estimator computed by minimizing the empirical semidual functional.

\begin{proposition}\label{prop:semidual_quadratic}
	Let $\cF$ be the set of all convex quadratic functions on $\R^d$.
	Let $\mu = \cN(0, I)$, and write $\nu_n$ for an empirical measure consisting of i.i.d.\ samples from a probability measure $\nu$.
	If
	\begin{equation*}
		\widehat \varphi = \argmin_{\phi \in \cF}{\biggl\{\int \phi \, \ud \mu + \int \phi^* \, \ud \nu_n\biggr\}}\,,
	\end{equation*}
	then
	\begin{equation*}
		\nabla \widehat \varphi(x) = \widehat \Sigma^{1/2} x + \widehat m\,,
	\end{equation*}
	where $\widehat m$ and $\widehat \Sigma$ are the mean and covariance of $\nu_n$, respectively.
\end{proposition}
\begin{proof}
	By definition, $\widehat \varphi = \frac 12\, x^\top \widehat A x + \widehat b^\top x$, where $(\widehat A, \widehat b)$ solve
	\begin{align*}
		\min_{A \succeq 0,\, b\in\R^d}\biggl[\int \Bigl(\frac 1 2\, x^\top A x + b^\top x\Bigr) \, \mu(\ud x) + \int \Bigl(\frac 1 2\, y^\top A y + b^\top y\Bigr)^* \, \nu_n(\ud y)\biggr]\,.
	\end{align*}
    By Lemma~\ref{lem:quadratic_conjugate}, the convex conjugate of a quadratic is also a quadratic, so the second integral only depends on moments of $\nu_n$ of order at most $2$. We can therefore replace the integration over $\nu_n$ with any other measure that matches the first two moments, in particular $\cN(\widehat m,\widehat \Sigma)$. Then, since the function class contains all Kantorovich potentials between Gaussians (see Example~\ref{ex:ot_gaussian}), it follows that the minimizer is the one which corresponds to the optimal transport from $\cN(0, I)$ to $\cN(\widehat m,\widehat \Sigma)$.
\end{proof}
    
In the setting of Proposition~\ref{prop:semidual_quadratic}, it is easy to analyze the performance of $\widehat T$ directly, since it is defined explicitly in terms of the sample mean and covariance.
However, for more general families $\cF$, we typically cannot solve the semidual problem explicitly, and we need to use more abstract arguments to analyze $\widehat \varphi$.

\section{Obtaining the slow rate}\label{sec:map_slow}

In this and the following section, we focus on estimating maps arising from potentials that lie in a suitable class $\Phi$ whose covering numbers---in the sense of Section~\ref{sec:chaining}---are suitably bounded.
The size of these covering numbers directly affects the quality of the estimator obtained by minimizing the empirical semidual, as in~\eqref{eq:empirical_semidual}.

To begin our analysis, we make several technical assumptions on the measures $\mu$ and $\nu$ and the family $\Phi$.

\begin{assumption}\label{map_assumptions}
	There exists $\varphi \in \Phi$ such that $\nu = (\nabla \varphi)_{\sharp} \mu$, where $\mu$, $\nu$, and $\Phi$ satisfy:
	\begin{itemize}
		\item The supports of $\mu$ and $\nu$ lie in $\Omega = B_{1}(0)$.
		\item The set $\Phi$ is bounded in $L^\infty$ on $\Omega$, i.e., $\sup_{\phi \in \Phi} \|\phi\|_{L^\infty(\Omega)} \lesssim 1$.
		\item The potentials satisfy $\phi(0) = 0$ and $\phi(x) = + \infty$ if $x \notin \Omega$.
		\item The potentials are lower-semicontinuous everywhere, and smooth and strongly convex on $\Omega$: $\frac 12 I \preceq \nabla^2 \phi(x) \preceq 2 I$ if $\|x\| < 1$.
	\end{itemize}
\end{assumption}

The first and second of these assumptions can be weakened under suitably strong moment assumptions, but we do not pursue this avenue here.
The third is without loss of generality: since subtracting a constant from $\phi$ does not affect the semidual objective or the gradient $\nabla \phi$, we can always assume that $\phi(0) = 0$, and since the supports of $\mu$ and $\nu$ lie in $B_1(0)$, we may define $\phi$ to be infinity outside of this set without affecting the semidual problem.
The fact that $\phi = +\infty$ identically outside of $\Omega$ simplifies several arguments involving the conjugate function, since it implies that
\begin{equation*}
	\phi^*(y) = \sup_{x \in \R^d}{\{\langle x, y \rangle - \phi(x)\}} = \sup_{x \in \Omega}{\{\langle x, y \rangle - \phi(x)\}}
\end{equation*} 
for all $y \in \R^d$ and $\phi \in \Phi$.
The fourth assumption is the most important, because it guarantees the stability of the semidual problem via Theorem~\ref{thm:stability_semidual}.

With these assumptions in hand, we can carry out the first step of the analysis.

\begin{lemma}\label{lem:semidual-erm}
	Adopt Assumption~\ref{map_assumptions}, and assume that $\nu = \nabla \varphi_\sharp \mu$ for some $\varphi \in \Phi$.
	Let $\widehat \varphi$ be given by
	\begin{equation*}
		\widehat \varphi = \argmin_{\phi \in \Phi} \cS_n(\phi)\,.
	\end{equation*}
	Then
	\begin{equation}\label{eq:erm_bound}
		\|\nabla \widehat \varphi - \nabla \varphi\|_{L^2(\mu)}^2 \lesssim \sup_{\phi \in \Phi} {\{|(\mu_n - \mu)(\phi)| + |(\nu_n - \nu)(\phi^*)|\}}\,.
	\end{equation}
\end{lemma}
\begin{proof}
	The proof is an application of a standard argument in empirical risk minimization.
	Theorem~\ref{thm:stability_semidual} implies
	\begin{align*}
		\|\nabla \widehat \varphi - \nabla \varphi\|_{L^2(\mu)}^2 & \lesssim \cS(\widehat \varphi) - \cS(\varphi) \\
		& = \cS(\widehat \varphi) - \cS_n(\widehat \varphi) + \cS_n(\widehat \varphi) - \cS_n(\varphi) + \cS_n(\varphi) - \cS(\varphi) \\
		& \leq 2 \sup_{\phi \in \Phi} |\cS_n(\phi) - \cS(\phi)|\,,
	\end{align*}
	where the last inequality uses that $\cS_n(\widehat \varphi) - \cS_n(\varphi) \leq 0$ by definition of $\widehat \varphi$.
	Expanding the definitions of $\cS$ and $\cS_n$ yields the claim.
\end{proof}

To bound the right side of~\eqref{eq:erm_bound}, we employ the chaining technique of Proposition~\ref{prop:chaining}.
For simplicity, we focus on the case where the class of functions is small enough that the covering numbers satisfy
\begin{equation}\label{eq:integrability}
	\log N(\eps, \Phi) \lesssim \eps^{-\gamma} \log(1 + \eps^{-1}) \,, \quad \gamma \in [0, 1)
\end{equation}
for all sufficiently small $\eps$.

A paradigmatic example of such classes are parametric classes, where the set $\Phi$ is finite dimensional, that is, where $\Phi = \{\phi_{\theta}\}_{\theta \in \Theta}$ is indexed by a parameter $\theta \in \Theta \subseteq \R^M$.
Indeed, in this case, we have the following bound.

\begin{lemma}\label{lem:phi-covering}
	Assume that $\Phi = \{\phi_{\theta}\}_{\theta \in \Theta}$,  where $\Theta \subseteq \R^M$ is bounded, and the potentials satisfy $\|\phi_\theta - \phi_{\theta'}\|_{L^\infty(\Omega)} \lesssim \|\theta - \theta'\|$.
	Then there exists a positive constant $C$ such that the covering numbers of $\Phi$ satisfy $\log N(\eps, \Phi) = 0$ if $\eps \geq C$ and 
	\begin{equation*}
		\log N(\eps, \Phi) \lesssim \log(1 + \eps^{-1})
	\end{equation*}
	otherwise.
\end{lemma}
\begin{proof}
	By assumption, $\Theta \subseteq B_R(0)$ for some $R > 0$, so $\|\phi - \psi\|_{L^\infty(\Omega)} \lesssim R$ for any $\phi, \psi \in \Phi$.
	Therefore, if $\eps$ is larger than a sufficiently large constant, any element of $\Phi$ constitutes a one-element $\eps$-cover of $\Phi$.
	
	For any $\delta > 0$, Exercise~\ref{ex:rd_covering_numbers} shows that there exists $\theta_1, \dots, \theta_N$ with $N \leq (1 + 2R \delta^{-1})^M$ such that $\bigcup_{i=1}^N B_\delta(\theta_i) \supseteq \Theta$.
	Then $\phi_{\theta_1},\dots, \phi_{\theta_N}$ is an $O(\delta)$-cover of $\Phi$.
	Indeed, for any $\theta \in \Theta$, we may choose $i \in [N]$ such that $\|\phi_\theta - \phi_{\theta_i}\|_{L^\infty(\Omega)} \lesssim \|\theta - \theta_i\| \leq \delta$.
	Taking $\delta = c \eps$ for a sufficiently small positive constant $c$ yields the claim.
\end{proof}

To give some examples of parametric classes, $\{\phi_\theta\}_{\theta \in \Theta}$ could be a set of convex quadratic functions, as in Section~\ref{sec:affine_transport}, or it could consist of linear combinations of a fixed dictionary $\{\phi_1, \dots, \phi_M\}$, with
\begin{equation*}
	\phi_\theta = \sum_{i=1}^M \theta_i \phi_i\,.
\end{equation*}
Note that it is common in non-parametric statistics to choose the dictionary carefully to balance approximation and estimation errors, but we do not delve into such questions here in order to focus on the core statistical content.

More generally, condition~\eqref{eq:integrability} allows for infinite-dimensional function classes which are nevertheless not ``too large''.
By contrast, it excludes classes whose complexity grows with the ambient dimension, such as the class of Lipschitz functions studied in Lemma~\ref{lem:cov_bound}.

What is the optimal rate of estimating $T = \nabla \varphi$ under this assumption?
If $\Phi$ is a parametric class, we expect the minimax rate to be $n^{-1}$---in particular, we expect that the map estimation problem avoids the curse of dimensionality.
Indeed, rates avoiding the curse of dimensionality are achievable whenever a bound such as~\eqref{eq:integrability} is satisfied.

Lemma~\ref{lem:semidual-erm} involves both the potential $\phi$ and its conjugate $\phi^*$.
Unfortunately, even if the set $\Phi$ has a simple form, the set $\Phi^* = \{\phi^*: \phi \in \Phi\}$ may defy easy description.
However, we make the following crucial observation: the covering numbers of $\Phi$ control those of $\Phi^*$.

\begin{lemma}\label{lem:phistar-covering}
	For any $\eps > 0$,
	\begin{equation*}
		N(\eps, \Phi^*) \leq N(\eps, \Phi)\,.
	\end{equation*}
\end{lemma}
\begin{proof}
	The result follows from the fact that the conjugation operation is a contraction in $L^\infty$.
    Indeed, given any pair of functions $\phi, \psi \in \Phi$,
	\begin{align*}
		|\phi^*(y) - \psi^*(y)| & = |\sup_{x \in \Omega}{\{\langle x, y \rangle - \phi(x)\}} - \sup_{x' \in \Omega}{\{\langle x', y \rangle - \psi(x')\}}| \\
		& \leq \sup_{x \in \Omega} |\phi(x) - \psi(x)| = \|\phi - \psi\|_{L^\infty(\Omega)}\,.
	\end{align*}
	In particular, if $\phi_1, \dots, \phi_N$ is an $\eps$-net for $\Phi$, then $\phi_1^*, \dots, \phi_N^*$ is an $\eps$-net for $\Phi^*$. 
\end{proof}

We can now prove our first convergence rate for map estimation.

\begin{theorem}\label{thm:slow_rate}
	Adopt Assumption~\ref{map_assumptions} and assume~\eqref{eq:integrability} holds.
	The semidual estimator $\widehat \varphi$ satisfies the bound
	\begin{equation*}
		\E \|\nabla \widehat \varphi - \nabla \varphi\|_{L^2(\mu)}^2 \lesssim n^{-1/2}\,.
	\end{equation*}
\end{theorem}
\begin{proof}
	Lemma~\ref{lem:semidual-erm} implies
	\begin{equation*}
		\E \|\nabla \widehat \varphi - \nabla \varphi\|_{L^2(\mu)}^2 \lesssim \E \sup_{\phi \in \Phi} |(\mu_n - \mu)(\phi)| + \E \sup_{\phi \in \Phi} |(\nu_n - \nu)(\phi^*)|\,.
	\end{equation*}
	By Assumption~\ref{map_assumptions}, there exists a positive constant $R$ such that $\|\phi\|_{L^\infty(\Omega)} \leq R$ and $\|\phi^*\| \leq R$ for all $\phi \in \Phi$.
	Applying Proposition~\ref{prop:chaining} with $\tau = 0$ yields
	\begin{equation*}
		\E \|\nabla \widehat \varphi - \nabla \varphi\|_{L^2(\mu)}^2 \lesssim \frac{1}{\sqrt n} \int_0^R \bigl(\sqrt{\log N(\eps, \Phi)} + \sqrt{\log N(\eps, \Phi^*)}\bigr) \, \ud \eps\,.
	\end{equation*}
	Applying~\eqref{eq:integrability} and Lemma~\ref{lem:phistar-covering}, we obtain
	\begin{equation*}
		\E \|\nabla \widehat \varphi - \nabla \varphi\|_{L^2(\mu)}^2 \lesssim \frac{1}{\sqrt n} \int_0^R \eps^{-\gamma/2} \sqrt{ \log(1+\eps^{-1})} \, \ud \eps \lesssim n^{-1/2}\,,
	\end{equation*}
	as desired.
\end{proof}

As anticipated, the parametric assumption on the class $\Phi$ translates to a rate of convergence that avoids the curse of dimensionality.
However, the ``slow rate'' $n^{-1/2}$ is not quite what we hoped to prove.
To obtain the ``fast rate'' $n^{-1}$, we need to \emph{localize} and exploiting this localization step requires imposing additional assumptions on $\mu$.

\section{The fixed point argument}\label{sec:map_fixed_pt}

As mentioned above, the chaining argument in Theorem~\ref{thm:slow_rate} fails to give the correct rate of convergence because it is based on bounding the deviations of $\cS_n$ from $\cS$ uniformly over the set $\Phi$.
However, Theorem~\ref{thm:slow_rate} shows that $\widehat \varphi$ is close to $\varphi$ when $n$ is large, which suggests that it is not necessary to bound the deviations of $\cS_n$ from $\cS$ uniformly over the set $\Phi$, but only over that subset near $\varphi$.

The argument below, due to van de Geer, formalizes this process in the context of map estimation.
The main idea is to apply the reasoning of Lemma~\ref{lem:semidual-erm} not to $\widehat \varphi$ but to a convex combination of $\widehat \varphi$ and $\varphi$ itself.
For simplicity, to apply this argument, we make one more assumption on $\Phi$.
\begin{assumption}\label{assume:phi_convex}
	The set $\Phi$ is convex, that is, if $\phi, \psi \in \Phi$, then $(1-\lambda) \phi + \lambda \psi \in \Phi$ for all $\lambda \in [0, 1]$.
\end{assumption}
Define
\begin{align*}
    \varphi_\varepsilon
    &= (1-\lambda)\,\varphi + \lambda \,\widehat \varphi\,, \qquad \lambda = \frac{\varepsilon}{\varepsilon + \|\nabla \widehat \varphi -\nabla \varphi\|_{L^2(\mu)}}\,.
\end{align*}
Then,
\begin{align*}
    \|\nabla \varphi_\varepsilon - \nabla \varphi\|_{L^2(\mu)}
    &= \lambda \, \|\nabla \widehat \varphi - \nabla \varphi\|_{L^2(\mu)} \\
    &= \varepsilon\, \bigl(\frac{\|\nabla \widehat \varphi - \nabla \varphi\|_{L^2(\mu)}}{\varepsilon + \|\nabla \widehat \varphi - \nabla \varphi\|_{L^2(\mu)}}\bigr)
    \le \varepsilon\,.
\end{align*}
Moreover:
\begin{align*}
    \|\nabla \varphi_\varepsilon - \nabla \varphi\|_{L^2(\mu)} \le \frac{\varepsilon}{2}
    &\iff \frac{\varepsilon\,\|\nabla \widehat \varphi - \nabla \varphi\|_{L^2(\mu)}}{\varepsilon + \|\nabla \widehat \varphi - \nabla \varphi\|_{L^2(\mu)}} \le \frac{\varepsilon}{2} \\
    &\iff \|\nabla \widehat \varphi - \nabla \varphi\|_{L^2(\mu)} \le \varepsilon\,,
\end{align*}
so in considering $\|\nabla \varphi_\varepsilon - \nabla \varphi\|_{L^2(\mu)}$ instead of $\|\nabla \widehat \varphi - \nabla \varphi\|_{L^2(\mu)}$ we only lose a factor of $2$.
Finally, Assumption~\ref{assume:phi_convex} guarantees that $\varphi_\eps \in \Phi$, since it is a convex combination of elements of $\Phi$.

Also, note that for $\lambda \in [0,1]$, pointwise we have $((1-\lambda)\,\phi_0 + \lambda\,\phi_1)^* \le (1-\lambda) \, \phi_0^* + \lambda \, \phi_1^*$, from which it follows that $\cS_n$ is a convex functional.
If we set $\overline{\cS} = \cS - \cS(\varphi)$ and $\overline{\cS_n} = \cS_n - \cS_n(\varphi)$, then by convexity,
\begin{align*}
    \overline{\cS_n}(\varphi_\varepsilon)
    &\le (1-\lambda) \, \underbrace{\overline{\cS_n}(\varphi)}_{=0} + \lambda\,\underbrace{\overline{\cS_n}(\widehat \varphi)}_{\le 0}
    \le 0
\end{align*}
by the definition of $\widehat \varphi$.
Therefore, by Theorem~\ref{thm:stability_semidual}, if $\varphi_\varepsilon$ is $\frac{1}{2}$-strongly convex and $2$-smooth,
\begin{align*}
    \frac{1}{4}\,\|\nabla \varphi_\varepsilon - \nabla \varphi\|_{L^2(\mu)}^2
    &\le \overline{\cS}(\varphi_\varepsilon)
    \le (\overline{\cS} -\overline{\cS_n})(\varphi_\varepsilon) \\
    &= (\cS - \cS_n)(\varphi_\varepsilon) - (\cS - \cS_n)(\varphi) \\
    &\le \sup_{\phi \in \Phi_\varepsilon}{(\cS-\cS_n)(\phi)} - (\cS - \cS_n)(\varphi) \\
    &\le \sup_{\phi \in \Phi_\varepsilon}{\{|(\mu_n - \mu)(\phi - \varphi)| + |(\nu_n - \nu)(\phi^* - \varphi^*)|\}}\,,
\end{align*}
where $\Phi_\varepsilon = \{\phi \in \Phi : \|\nabla \phi - \nabla \varphi\|_{L^2(\mu)} \le \varepsilon\}$.

If the right side of the above inequality is less than $\eps^2/16$, then $\|\nabla \varphi_\eps - \nabla \varphi\|_{L^2(\mu)} \leq \eps/2$, which in turn implies that $\|\nabla \widehat \varphi - \nabla \varphi\|_{L^2(\mu)} \leq \eps.$
We therefore can control the risk of $\widehat \varphi$ if we can find an $\eps$ for which the supremum of the empirical process over $\cF_\eps$ is of order $\eps^2$.
We formalize the above considerations in the following proposition.

\begin{proposition}\label{prop:local_erm}
	Adopt Assumptions~\ref{map_assumptions} and~\ref{assume:phi_convex}.
	For $\eps > 0$, let
	\begin{equation*}
		r(\eps) = \sup_{\phi \in \Phi_\varepsilon}{\{|(\mu_n - \mu)(\phi - \varphi)| + |(\nu_n - \nu)(\phi^* - \varphi^*)|\}}\,.
	\end{equation*}
	Then on the event $\{r(\eps) \leq \eps^2/16\}$, the semidual estimator $\widehat \varphi$ satisfies $\|\nabla \widehat \varphi - \nabla \varphi\|_{L^2(\mu)} \leq \eps$.

	In particular, if $\eps_n$ is a deterministic quantity such that $r(\eps_n) \leq \eps_n^2/16$ with high probability, then the risk of $\widehat \varphi$ is bounded by $\eps_n$ with high probability.
\end{proposition}

Comparing Lemma~\ref{lem:semidual-erm} with Proposition~\ref{prop:local_erm} shows that we have replaced the task of bounding the deviations of an empirical process uniformly over $\Phi$ by the task of bounding them over the smaller set $\Phi_\eps$.

\section{Obtaining the fast rate}\label{sec:map_fast}

In order to exploit the fact that we now seek to bound the empirical process only over $\Phi_\eps$, we need to formalize the notion that $\Phi_\eps$ is much smaller than $\Phi$.
A complicating factor is that the chaining technique given in Proposition~\ref{prop:chaining} measures the ``size'' of $\Phi$ by its $\eps$-covering numbers, which are defined in terms of $L^\infty$ covers.
By contrast, the restriction of $\Phi$ to $\Phi_\eps$ is based on the additional restriction on the \emph{$L^2$ norm} of the gradients of $\phi$.
We therefore need a version of the chaining bound which is able to exploit the size of a function class with respect to both $L^\infty$ and $L^2$.

The following modified chaining bound addresses this deficit.
\begin{proposition}[{\cite[Theorem 2.14.21]{VaaWel23}}]\label{prop:improved_chaining}
	Let $P$ be a probability measure on a set $\Omega \subseteq \R^d$.
	Let $X_1, \dots, X_n \simiid P$.
	If $\cF$ is a set of real-valued functions such that $\|f\|_{L^2(P)} \leq \sigma$ and $\|f\|_{L^\infty(\Omega)} \leq R$ for all $f \in \cF$,
	then
	\begin{multline}\label{eq:improved_chaining_bound}
	\E \sup_{f \in \cF} \frac 1n \sum_{i=1}^n \{f(X_i) - \E f(X_i)\} \lesssim \frac{1}{\sqrt n} \int_0^{\sigma} \sqrt{\log N(\eps, \cF)}\, \ud \eps \\
	+ \frac 1 n \int_0^{R} \log N(\eps, \cF)\, \ud \eps\,.
	\end{multline}
\end{proposition}

Note that in the first term in~\eqref{eq:improved_chaining_bound}, the upper limit of the integral is $\sigma$ rather than $R$.
The second integral incurs a worse dependence on the covering number, but appears with a prefactor of $\frac 1n$ rather than $\frac{1}{\sqrt n}$.
We may therefore hope that when $n$ is large enough, the first term dominates.
If the $L^2$ size of $\cF$ is small, as captured by $\sigma$, then the first term may be substantially smaller than the bound obtained by applying Proposition~\ref{prop:chaining} directly.

Proposition~\ref{prop:improved_chaining} also comes with a tail bound, showing that the quantity on the right-hand side of~\eqref{eq:improved_chaining_bound} also bounds the empirical process with high probability.

\begin{proposition}\label{prop:chain_tail}
	Let $J_n(\cF)$ denote the right side of~\eqref{eq:improved_chaining_bound}.
	Under the same assumptions as Proposition~\ref{prop:improved_chaining}, there exists a positive universal constant $C$ such that for any $t \geq 0$,
	\begin{equation*}
		\p\biggl(\sup_{f \in \cF} \frac 1n \sum_{i=1}^n \{f(X_i) - \E f(X_i)\}  \geq C\,\bigl(J_n(\cF) + \sigma\sqrt{\frac{t}{n}} +  R\, \frac{t}{n}\bigr)\biggr) \leq \exp(-t)\,.
	\end{equation*}
\end{proposition}

Proposition~\ref{prop:improved_chaining} requires us to control the $L^2$ norm of the elements of our function class; however, $\Phi_\eps$ is defined using the $L^2$ norms of the \emph{gradients} of the elements of $\Phi$.
We therefore adopt the final assumption on $\mu$, which allows us to move back and forth between these notions.
\begin{definition}
	A  measure $P$ satisfies a \emph{Poincar\'e inequality}\index{Poincar\'e inequality} (with constant $C$) if for all $f \in L^2(P)$,
	\begin{equation*}
		\int \Bigl(f - \int f \, \ud P\Bigr)^2 \, \ud P \leq C \int \|\nabla f\|^2 \, \ud P\,.
	\end{equation*}
\end{definition}
\begin{assumption}\label{assume:poincare}
	The measure $\mu$ satisfies a Poincar\'e inequality.
\end{assumption}

The Poincar\'e inequality is a quantitative form of the statement that the support of $\mu$ is connected.
Indeed, a Poincar\'e inequality holds for any measure having a density bounded away from zero and infinity on a bounded Lipschitz domain.

Under this new assumption, we obtain $L^2$ bounds on $\Phi_\eps$ and $\Phi^*_\eps$.
\begin{proposition}\label{prop:map_l2_bound}
	If Assumptions~\ref{map_assumptions} and~\ref{assume:poincare} hold, then
	\begin{align*}
		\|\phi - \varphi - \mu(\phi - \varphi)\|_{L^2(\mu)} & \lesssim \eps \\
		\|\phi^* - \varphi^* - \nu(\phi^* - \varphi^*)\|_{L^2(\nu)} & \lesssim \eps
	\end{align*}
	for all $\phi \in \Phi_\eps$.
\end{proposition}
\begin{proof}
    The first bound follows directly from the Poincar\'e inequality: for $\phi \in \Phi_\eps$,
	\begin{equation*}
		\|\phi - \varphi - \mu(\phi - \varphi)\|_{L^2(\mu)}^2 \leq C\, \|\nabla (\phi - \varphi)\|_{L^2(\mu)}^2 \lesssim \eps^2\,.
	\end{equation*}
	
	To prove the second bound, we first use the strong convexity of $\phi$ and $\varphi$.
	Consider the functional $\cT$ defined by
	\begin{equation*}
		\cT(\psi)  \deq  \int \psi \, \ud \nu + \int \psi^* \, \ud \mu\,.
	\end{equation*}
	That is, $\cT$ is the semidual functional obtained by exchanging the roles of $\mu$ and $\nu$.
	Since $(\phi^*)^* = \phi$ for all convex and lower semicontinuous $\phi$, we have that $\cS(\phi) = \cT(\phi^*)$ for all such $\phi$.
	In particular, the minimizer of $\cT$ is $\varphi^*$, and Theorem~\ref{thm:stability_semidual} implies that
	\begin{equation*}
		\frac 14\, \|\nabla \phi^* - \nabla \varphi^*\|_{L^2(\nu)}^2 \leq \cT(\phi^*) - \cT(\varphi^*) = \cS(\phi) - \cS(\varphi) \leq \|\nabla \phi - \nabla \varphi\|_{L^2(\mu)}^2\,.
	\end{equation*}
	Therefore $\|\nabla \phi^* - \nabla \varphi^*\|_{L^2(\nu)}^2 \lesssim \eps^2$ for all $\phi \in \Phi_\eps$.
	
	To obtain the bound, all that is left is to show that $\nu$ also satisfies a Poincar\'e inequality, since we can then conclude as in the proof of the first inequality.
	To see this, we use the fact that $\nu = (\nabla \varphi)_\sharp \mu$.
	The fact that $\varphi$ is smooth means that for any $f: \R^\dd \to \R^\dd$, $$\|\nabla (f \circ \nabla \varphi)(x)\| = \|\nabla^2 \varphi(x) \, \nabla f(\nabla \varphi(x))\| \leq 2\, \|\nabla f(\nabla \varphi(x))\|\,.$$
	The Poincar\'e inequality for $\mu$ implies that for any $f \in L^2(\nu)$,
	\begin{align*}
		\int \Bigl(f - \int f \, \ud \nu\Bigr)^2 \, \ud \nu
              & = \int \Bigl(f\circ\nabla \varphi - \int f\circ\nabla \varphi \, \ud \mu\Bigr)^2 \, \ud \mu \\
		& \leq C \int \|\nabla (f \circ \nabla \varphi)\|^2 \, \ud \mu \\
		& \leq 4 C \int \|\nabla f(\nabla \varphi(x))\|^2 \, \ud \mu \\
		& = 4 C \int \|\nabla f\|^2 \, \ud \nu\,.
	\end{align*}
	Therefore $\nu$ also satisfies a Poincar\'e inequality, with constant $4C$.
	Hence we may conclude as in the first case.
\end{proof}

We are finally ready to prove the desired rate.
Note that in the finite-dimensional setting, when Lemma~\ref{lem:phi-covering} holds, this theorem shows that the map can be estimated at nearly the parametric rate.
\begin{theorem}
	Adopt Assumptions~\ref{map_assumptions},~\ref{assume:phi_convex}, and~\ref{assume:poincare}.
	If~\eqref{eq:integrability} holds, then the semidual estimator $\widehat \varphi$ satisfies
	\begin{equation}\label{eq:map_hp_bound}
		\| \nabla \widehat \varphi - \nabla \varphi\|_{L^2(\mu)}^2 \lesssim \Bigl(\frac{\log n}{n}\Bigr)^{\frac{2}{2 + \gamma}} + \frac{t+1}n
	\end{equation}
	with probability at least $1-e^{-t}$.
	In particular,
	\begin{equation*}
		\E \| \nabla \widehat \varphi - \nabla \varphi\|_{L^2(\mu)}^2 \lesssim\Bigl(\frac{\log n}{n}\Bigr)^{\frac{2}{2 + \gamma}}\,.
	\end{equation*}
\end{theorem}
\begin{proof}
	By Proposition~\ref{prop:local_erm}, it suffices to show that $r(\eps_n) \leq \eps_n^2/16$ with probability at least $1- e^{-t}$ for $\eps_n \asymp \bigl(\frac{\log n}{n}\bigr)^{\frac{1}{2+\gamma}} + \sqrt{\frac {t+1} n}$.
	Let us first bound $\sup_{\phi \in \Phi_\eps} |(\mu_n- \mu) (\phi-\varphi)|$.
	Since $(\mu_n - \mu)h = 0$ if $h$ is a constant function, we have
	\begin{equation*}
		\sup_{\phi \in \Phi_\eps} |(\mu_n- \mu)( \phi-\varphi)| = \sup_{\phi \in \Phi_\eps} |(\mu_n- \mu) (\phi - \varphi - \mu(\phi - \varphi))| \,.
	\end{equation*}
	We can apply Proposition~\ref{prop:improved_chaining}, Proposition~\ref{prop:chain_tail}, and Proposition~\ref{prop:map_l2_bound} along with the fact that $\phi - \varphi - \mu(\phi - \varphi)$ is bounded to obtain that there exists a constant $C_2$ such that
	\begin{align*}
		&\sup_{\phi \in \Phi_\eps}|(\mu_n- \mu) (\phi - \varphi - \mu(\phi - \varphi))| \\
            &\qquad  \lesssim \frac{1}{\sqrt n} \int_0^{C_2 \eps} \delta^{-\gamma/2}\sqrt{\log(1 + \delta^{-1})} \, \ud \delta \\
            &\qquad\qquad{} + \frac 1 n \int_0^{C_2} \delta^{-\gamma} \log(1+\delta^{-1})  \, \ud \delta + \eps\sqrt{\frac tn} + \frac tn\\
		&\qquad \lesssim  \eps^{1-\gamma/2}\sqrt{\frac{\log(1+\eps^{-1})}{n}} + \eps \sqrt{\frac tn}+ \frac{t+1}{n}
	\end{align*}
	with probability at least $1-e^{-t}$.
	Therefore, taking
         \begin{align}\label{eq:choice_of_epsn}
             \varepsilon_n = C\,\Bigl(\Bigl(\frac{\log n}{n}\Bigr)^{\frac{1}{2 + \gamma}} + \sqrt{\frac{t+1}n}\Bigr)
         \end{align}
         for a sufficiently large constant $C$, we can ensure that
	\begin{equation*}
		\sup_{\phi \in \Phi_{\eps_n}}|(\mu_n- \mu) (\phi - \varphi - \mu(\phi - \varphi))| \leq \eps_n^2/32
	\end{equation*}
	with probability at least $1 - e^{-t}/2$.
	An analogous argument yields
	\begin{equation*}
		\sup_{\phi \in \Phi_{\eps_n}}|(\nu_n- \nu) (\phi^* - \varphi^* - \nu(\phi^* - \varphi^*))| \leq \eps_n^2/32
	\end{equation*}
	with the same probability.
	By a union bound, we obtain that $r(\eps_n) \leq \eps_n^2/16$ for $\eps_n$ as in~\eqref{eq:choice_of_epsn} as claimed.
	
	The second bound following from integrating the tail.
\end{proof}

\section{Discussion}

\noindent\textbf{\S\ref{sec:maps_intro}.}
The empirical ``plug-in'' approach based on nearest neighbors was developed as a simple alternative to the semidual approach in~\cite{ManBalNil21, DebGhoSen21, GhoSen22}.
Although the nearest neighbors estimator does not adapt to the smoothness of $\mu$ and $\nu$, one recovers minimax rates via the optimal transport map between density estimators~\cite{ManBalNil21}, and even central limit theorems~\cite{Man+23CLT}.
However, compared to the plug-in approach, the semidual approach developed here is overall more flexible and can be combined with other tools such as kernel SoS~\cite{Vac+24KernelSoS}.

\noindent\textbf{\S\ref{sec:map_semidual}.} The semidual approach to map estimation was introduced in the paper~\cite{HutRig21}, which also proved the semidual stability estimates and minimax lower bounds.
That paper showed that, if the map between $\mu$ and $\nu$ is assumed to be $s$-smooth (i.e., to possess $s$ bounded derivatives), then a suitable semidual estimator achieves the minimax-optimal rate
\begin{equation}\label{eq:map_smooth_minimax}
	\E \|\nabla \widehat \varphi - \nabla \varphi\|_{L^2(\mu)}^2 \lesssim n^{-\frac{2 \alpha}{2 \alpha - 2 +d}}\, (\log n)^2 + \frac 1 n\,.
\end{equation}
This approach was then explored in great generality in~\cite{DivNilPoo22OTGeneral}, and the arguments in that paper are closely related to those in this chapter.
However, the tools we describe here are not strong enough to prove~\eqref{eq:map_smooth_minimax}, since the class of $s$-smooth functions does not satisfy~\eqref{eq:integrability}.
More information about how to obtain~\eqref{eq:map_smooth_minimax} along the lines of the arguments we have presented in this chapter can be found in~\cite[Section 4.4]{DivNilPoo22OTGeneral}.

The alternative semidual stability estimates in Exercises~\ref{exe:alternative_semidual_stab} and~\ref{exe:another_stability} are taken from~\cite{ManBalNil21}.

\noindent\textbf{\S\ref{sec:affine_transport}.} Estimating the transport map between Gaussians was given as an example in~\cite{DivNilPoo22OTGeneral} in which the semidual approach yields parametric rates; see the paper for other function classes of interest.

\noindent\textbf{\S\ref{sec:map_slow}.} Standard references for empirical risk minimization (or M-estimation) include~\cite{Vaart98Asymptotic, Wai19}.
The ``slow rate'' is characteristic of M-estimation problems in the absence of strong convexity; the Poincar\'e inequality assumption adopted in \S\ref{sec:map_fast} can be viewed as the appropriate strong convexity condition for the semi-dual functional $\cS$.

\noindent\textbf{\S\ref{sec:map_fixed_pt}.} The one-shot localization we use is due to van de Geer~\cite{Gee87, Gee02} and provides an alternative to the usual localization arguments (e.g.,~\cite[Chapter 14]{Wai19}).

\noindent\textbf{\S\ref{sec:map_fast}.}
The improved chaining bound of Proposition~\ref{prop:improved_chaining} is obtained by the ``generic chaining'' technique developed by Talagrand~\cite{Tal96a}.
This technique is at the heart of the study of Gaussian processes, see~\cite{Tal21UppLow}.
The tail bound in Proposition~\ref{prop:chain_tail} follows from a more general result for generic chaining bounds~\cite[Theorem 2.2.19]{VaaWel23}.

The Poincar\'e inequality is a standard tool in high-dimensional probability, see~\cite{BouLugMas13, BakGenLed14, Han14}.
It is an open question whether the rates presented in this chapter are achievable without making this assumption.

\section{Exercises}

\begin{enumerate}
    \item \label{ex:1d_ot_map_est}Let $\mu, \nu \in \cP([0,1])$ and let $X_1,\dotsc,X_n \sim \mu$, $Y_1,\dotsc,Y_n\sim \nu$ be i.i.d.\ samples independent from each other.
    Assume that $\mu$, $\nu$ have differentiable CDFs $F_\mu$, $F_\nu$ respectively, such that
    \begin{align*}
        0 < c \le F_\mu', \, F_\nu' \le C < \infty \qquad \text{on}~[0,1]\,.
    \end{align*}
    This is equivalent to $\mu$, $\nu$ having densities on $[0,1]$ which are bounded away from zero and infinity.

    Let us show that the following estimator $\widehat T_n$ obeys a parametric rate of convergence.
    Let $X_{(1)} < \cdots < X_{(n)}$, $Y_{(1)} < \cdots < Y_{(n)}$ denote the samples in sorted order, and given $x \in [0,1]$ let $X_{(i)}$ denote the largest sample from $\mu$ such that $X_{(i)} \le x$.
    We then set $\widehat T_n(x) \deq Y_{(i)}$ (if no such $X_{(i)}$ exists, then output $\widehat T_n(x) \deq 0$).
    This estimator can be viewed as a piecewise constant interpolation of the empirical optimal coupling, or as a 1-nearest neighbor estimator.
    For simplicity, we fix $x \in [0,1]$ and prove
    \begin{align*}
        \E[|\widehat T_n(x) - T(x)|^2] \lesssim 1/n
    \end{align*}
    where $T$ is the true optimal transport map $F_\nu^{-1} \circ F_\mu$ from $\mu$ to $\nu$, although it is a straightforward exercise to extend the results of this problem to the integrated risk $\E \int |\widehat T_n - T|^2 \, \D \mu$.
    \begin{enumerate}
        \item Let $N_x \deq \sum_{i=1}^n \one\{X_i \le x\}$ and $N_y' \deq \sum_{i=1}^n \one\{Y_i \le y\}$.
            Argue that if $N_y' < N_x$, then $\widehat T_n(x) \ge y$; if $N_y' > N_x$, then $\widehat T_n(x) \le y$.
        \item Using the Dvoretzky--Kiefer--Wolfowitz inequality, argue that for any $\delta \in (0,1)$, the following hold simultaneously with probability at least $1-\delta$:
            \begin{align*}
                |N_x - nF_\mu(x)| \lesssim \sqrt{n\log(1/\delta)}
            \end{align*}
            and
            \begin{align*}
                |N_y' - nF_\nu(y)| \lesssim\sqrt{n\log(1/\delta)} \qquad\text{for all}~y\in [0, 1].
            \end{align*}
        \item Use the previous two parts to conclude.
    \end{enumerate}
    
    \item\label{ex:rd_covering_numbers}
	This exercise shows that the ball $B_R(0)$ in $\R^\dd$ can be covered by $(1 + 2 R \delta^{-1})^\dd$ balls of radius $\delta$.
	\begin{enumerate}
		\item Argue by rescaling that it suffices to consider the case $R = 1$.
		\item Let $\cX = \{x_1, \dots, x_N\}$ be any set of elements of $B_1(0)$ such that $\|x_i - x_j\| > \delta$ for all $i \neq j$.
		Such a set is called a $\delta$-packing of $B_1(0)$.
		Show that if $\cX$ is a $\delta$-packing of $B_1(0)$, then the sets $\{B_{\frac \delta 2}(x)\}_{x \in \cX}$ are disjoint subsets of $B_{1+\frac \delta 2}(0)$.
		Conclude that $|\cX| \leq (1+2\delta^{-1})^d$.
		\item Suppose that $\cX$ is a maximal $\delta$-packing of $B_R(0)$, i.e., there does not exist a strict superset of $\cX$ which is also a $\delta$-packing.
		Argue (via the contrapositive) that $B_R(0) \subseteq \bigcup_{x \in \cX} B_\delta(x)$.
		\item Use the previous two parts to conclude.
	\end{enumerate}

    \item\label{exe:alternative_semidual_stab} Let $\mu$, $\nu$ be probability measures over $\R^d$ and let $\nabla \varphi$ denote the optimal transport map from $\mu$ to $\nu$. Assume that $\nabla\varphi$ is $L$-Lipschitz.
    In this exercise, we establish the following estimate: for any other convex function $\phi$, if $\hat\nu \deq {(\nabla\phi)}_\# \mu$, then
    \begin{align*}
        &\|\nabla \phi - \nabla\varphi\|_{L^2(\mu)}^2 \\
        &\qquad \le L\,\Bigl(W_2^2(\mu, \hat\nu) - W_2^2(\mu, \nu) - \int \bigl(\,\|\cdot\|^2 - 2\varphi^*\bigr)\,\ud (\hat\nu-\nu)\Bigr)\,.
    \end{align*}
    \emph{Hint}: First, argue that by strong convexity of $\varphi^*$, it holds that
    \begin{align*}
        &\frac{1}{2L}\,\|\nabla \phi - \nabla \varphi\|_{L^2(\mu)}^2 \\
        &\qquad \le \int \varphi^* \,\ud (\hat\nu-\nu) - \int \langle x, \nabla \phi(x) - \nabla \varphi(x) \rangle \, \mu(\ud x)\,.
    \end{align*}
    Then, expand out the quantity $W_2^2(\mu,\hat\nu) - W_2^2(\mu,\nu)$ and substitute this into the above inequality.

    \item We now use the stability estimate from the previous exercise to deduce rates for map estimation in the one-sample setting.
    Let $\mu$, $\nu$ be probability measures with densities supported on the ball $B_1(0)$ of radius $1$ and assume that the optimal transport map $\nabla\varphi$ from $\mu$ to $\nu$ is Lipschitz.
    Assume that we have access to $\mu$, and to $n$ i.i.d.\ samples from $\nu$. Let $\nabla \widehat \varphi$ denote the optimal transport map from $\mu$ to the empirical measure $\nu_n$. Using~\eqref{eq:w2_one_sample}, prove that
    \begin{align*}
        \E\|\nabla\widehat\varphi-\nabla\varphi\|_{L^2(\mu)}^2
        &\lesssim \begin{cases}
            n^{-1/2}\,, & d < 4\,, \\
            n^{-1/2} \log n\,, & d = 4\,, \\
            n^{-2/d}\,, & d > 4\,.
        \end{cases}
    \end{align*}

    \item\label{exe:another_stability} Starting with Exercise~\ref{exe:alternative_semidual_stab}, assume additionally that $\varphi$ is $\ell$-strongly convex.
    Let $\gamma$ denote the optimal coupling between $\nu$ and $\hat\nu$, and let $(Y, \hat Y) \sim \gamma$.
    Then, $(\nabla \varphi^*(Y), \hat Y)$ is a (suboptimal) coupling between $\mu$ and $\hat\nu$, hence $W_2^2(\mu,\hat \nu) \le \E\|\nabla \varphi^*(Y) - \hat Y\|^2$. Expand this out and use the smoothness of $\varphi^*$ to deduce the stronger stability estimate
    \begin{align*}
        \|\nabla \phi - \nabla \varphi\|_{L^2(\mu)}
        &\le \sqrt{L/\ell}\, W_2(\nu,\hat \nu)\,.
    \end{align*}
\end{enumerate}

\chapter{Entropic optimal transport}
\label{chap:entropic}

Entropic regularization is one of the most active research areas in modern optimal transport.
As a regularization technique, it technically falls under the scope of Section~\ref{sec:regulOT}.
Indeed, we show in this chapter that it yields parametric rates, like many of the other regularization approaches we have discussed.
But entropic optimal transport is, in fact, much more.

Since the groundbreaking work of Cuturi~\cite{Cut13}, it has been primarily used as a computational device that enables fast computation of optimal transport distances using the Sinkhorn algorithm. However, our focus remains statistical and we refer the reader to the excellent manuscript~\cite{PeyCut19} of Gabriel Peyr\'e and Marco Cuturi for more details on the computational benefits of entropic regularization.

The basic principle of entropic optimal transport is to modify the definition of optimal transport to include a penalization term based on the entropy of the coupling, that is, to consider the optimization problem\index{entropic optimal transport}
\begin{equation}\label{eq:eot_cont}
	\inf_{\gamma \in \Gamma_{\mu, \nu}} \biggl\{\int \|x - y\|^2 \, \gamma(\ud x, \ud y) - \eps \operatorname{Ent}(\gamma)\biggr\}\,,
\end{equation}
where $\operatorname{Ent}(\gamma)$ denotes the differential entropy $\int \gamma(x) \log \frac{1}{\gamma(x)} \, \ud x$ for an absolutely continuous probability measure $\gamma$.
In fact, Cuturi originally considered a discrete version of this problem, where $\mu$ and $\nu$ are finitely supported and the coupling $\gamma$ can therefore be identified with a matrix.
He considered the problem
\begin{equation}\label{eq:eot_disc}
	\inf_{\gamma \in \Gamma_{\mu, \nu}} \biggl\{\sum_{i, j} \|x_i - y_j\|^2 \gamma_{i,j} - \eps H(\gamma)\biggr\}\,,
\end{equation}
where $H(\gamma)$ denotes the Shannon entropy $\sum_{i,j} \gamma_{i,j} \log \frac{1}{\gamma_{i,j}}$.

The role of the penalty terms in both~\eqref{eq:eot_cont} and~\eqref{eq:eot_disc} is to encourage the coupling to be \emph{more spread out} than the solution to the unregularized optimal transport problem.
Informally, the entropy of a measure increases when its mass is more evenly spread.
Indeed, Exercise~\ref{exe:max_entropy} shows that uniform distributions (over a subset of $\R^d$ in continuous case, or over a finite set in the discrete case) have the maximum possible entropy.
The entropic penalty biases the solutions to~\eqref{eq:eot_cont} and~\eqref{eq:eot_disc} towards that extreme.

To put~\eqref{eq:eot_cont} and~\eqref{eq:eot_disc} on a common footing, we introduce the KL divergence between probability measures:
\begin{align*}
	\KL(P \mmid Q)
	&= \begin{cases}
		\int \frac{\ud P}{\ud Q}\log \frac{\ud P}{\ud Q} \, \ud Q & \text{if}~P \ll Q, \\
		+\infty & \text{otherwise.}
	\end{cases}
\end{align*}
Exercise~\ref{exe:kl_to_ent} shows that both~\eqref{eq:eot_cont} and~\eqref{eq:eot_disc} are equivalent to
\begin{equation}\label{eq:eot_def}
		\inf_{\gamma \in \Gamma_{\mu, \nu}} \biggl\{\int \|x - y\|^2 \, \gamma(\ud x, \ud y) + \eps \KL(\gamma \mmid \mu \otimes \nu)\biggr\}\,,
\end{equation}
in the sense of yielding the same optimal $\gamma$, and we take~\eqref{eq:eot_def} as the main definition of entropic OT.

In the next section, we give a non-rigorous motivation for this regularization approach from the perspective of convex duality.
We analyze the resulting dual problem in Section~\ref{sec:eot_duality}.

\section{Derivation of entropic optimal transport}\label{sec:eot_deriv}

In this section, we attempt to motivate the definition of entropic OT from basic optimization and duality principles.
For simplicity, we assume for now that we work on a compact set $\Omega \subseteq \R^d$.
Given $\mu, \nu \in \cP(\Omega)$, recall from Theorem~\ref{thm:fundOT} that $W_2^2(\mu, \nu)$ can be written
\begin{equation*}
W_2^2(\mu, \nu) = \sup_{\substack{ f, g \in \Cb(\Omega) \\f(x) +g(y)\le \|x-y\|^2} }\left\{\int f \,\ud \mu + \int g\, \ud \nu\right\}\,.
\end{equation*}
Formally, we can rewrite this as an unconstrained maximization problem by introducing a penalization term that enforces the constraint.
Indeed, if we define
\begin{equation*}
	\iota(f, g) = \begin{cases}
	0 & \text{if $f(x) + g(y) \leq \|x - y\|^2$ $\mu \otimes \nu$-a.e.,} \\
	+ \infty & \text{otherwise,}
	\end{cases}
\end{equation*}
then we obtain
\begin{equation*}
	W_2^2(\mu, \nu) = \sup_{f, g \in \Cb(\Omega) }{\biggl\{\int f \,\ud \mu + \int g\, \ud \nu - \iota(f, g)\biggr\}}\,.
\end{equation*}
This is a concave maximization problem, so it is (in principle) benign; however, from a computational and statistical perspective, the fact that $\iota$ fails to be continuous, much less smooth, is a source of difficulty.
To remedy this, we can consider a relaxed version of this problem obtained by replacing $\iota$ by a smoothed version.\footnote{The smoothing we employ is reminiscent of the ``softmax'' function in machine learning.}
Define
\begin{equation*}
	\iota^\eps(f, g) = \eps \iint \Bigl(e^{(f(x) + g(y) - \|x - y\|^2)/\eps}  - 1\Bigr) \, \mu(\ud x)\, \nu(\ud y)\,.
\end{equation*}
The function $\iota^\eps$ is convex and continuous on the space $\Cb(\Omega)$ of bounded, continuous functions on $\Omega$.
Moreover, it is easy to see that as $\eps \searrow 0$, we recover the original hard constraint.

\begin{lemma}
    For any measurable $f, g$,
	\begin{equation*}
		\lim_{\eps \searrow 0} \iota^\eps (f, g) = \iota(f, g)\,.
	\end{equation*}
\end{lemma}
\begin{proof}
	Suppose first that $\iota(f, g) = 0$, so that $f(x) + g(y) \leq \|x - y\|^2$ $\mu \otimes \nu$-almost everywhere.
	Then the integral $$\iint \Bigl(e^{(f(x) + g(y) - \|x - y\|^2)/\eps}  - 1\Bigr) \, \mu(\ud x)\, \nu(\ud y)$$ is bounded as $\eps \searrow 0$, and hence $\iota^\eps(f, g) \to 0$.
	
	On the other hand, if $\iota(f, g) = +\infty$, then there exists $\delta > 0$ and a set $U$ of positive $\mu \otimes \nu$ measure such that $e^{(f(x) + g(y) - \|x - y\|^2)/\eps} \geq e^{\delta/\eps}$ for all $(x, y) \in U$.
    We obtain
    \begin{align*}
        \eps \iint \Bigl(e^{(f(x) + g(y) - \|x - y\|^2)/\eps}  - 1\Bigr) \, \mu(\ud x)\, \nu(\ud y)
        &\geq \eps e^{\delta/ \eps}\, (\mu \otimes \nu)(U) - \eps \\
        &\to \infty\,.
    \end{align*}
    This concludes the proof.
\end{proof}

We are therefore led to consider the following ``$\eps$-smoothed'' dual version of the $W_2^2$ distance:
\begin{equation}\label{eta-D}\tag{$\mathsf{\eps}$\textsf{-D-}$\mathsf{W_2^2}$}
    \sup_{f, g \in \Cb(\Omega)}{\biggl\{\int f \,\ud \mu + \int g\, \ud \nu - \iota^\eps(f, g)\biggr\}}
\end{equation}

Now that we have derived a relaxation of the dual problem, we can ask what this corresponds to in the primal problem.
It turns out that the relaxation we have proposed in the dual corresponds to an \emph{entropic penalty} in the primal problem.

To obtain this connection, let us define a version of the Kullback--Leibler divergence over the space $\cM_+(\Omega)$ of all positive (not necessarily probability) Borel measures on $\Omega$:
\begin{align*}
\KL(P \mmid Q)
&= \begin{cases}
\int \Bigl(\frac{\ud P}{\ud Q}\log \frac{\ud P}{\ud Q} - \frac{\ud P}{\ud Q} + 1\Bigr) \, \ud Q & \text{if}~P \ll Q, \\
+\infty & \text{otherwise.}
\end{cases}
\end{align*}
Note that the integrand is non-negative, so that the integral is always well defined, and $\KL$ is always non-negative.
When $P$ and $Q$ are probability measures, the terms $- \frac{\ud P}{\ud Q} + 1$ cancel out and we obtain the usual definition.

The importance of this definition is that the convex conjugate of the functional $\KL(\cdot \mmid Q)$ (see Appendix \ref{sec:cvx_sub_dual}) is precisely the exponential term appearing in $\iota^\eps$.
This fact is a variant of what is commonly known as the \emph{Gibbs variational principle}.\index{Gibbs variational principle}

\begin{proposition}\label{prop:gibbs}
	For any bounded measurable function $h$,
	\begin{equation*}
    \sup_{P\in \cM_+(\Omega)}{\Bigl\{\int h \, \ud P - \KL(P \mmid Q)\Bigr\}}
= \int (\exp h - 1) \, \ud Q\,.
	\end{equation*}
	Moreover the supremum is achieved at $P_h$ satisfying $\frac{\ud P_h}{\ud Q} = \exp h$.
\end{proposition}
\begin{proof}
    We show that, for any Borel measure $P$, the difference
	\begin{equation*}
		\Delta  \deq  \int (\exp h - 1) \, \ud Q - \int h \, \ud P + \KL(P \mmid Q)
	\end{equation*}
	is non-negative, and equals $0$ when $P = P_h$.
	We may assume without loss of generality that $\KL(P \mmid Q) < +\infty$, since otherwise the claim is vacuous.
Expanding the definition of $\KL(P \mmid Q)$, we obtain
	\begin{multline}\label{rl_diff}
		\Delta = \int \Big(e^h - 1 - h\, \frac{\ud P}{\ud Q} + \frac{\ud P}{\ud Q}\log \frac{\ud P}{\ud Q} - \frac{\ud P}{\ud Q} + 1\Big) \, \ud Q \\
		= \int \Big(\frac{\ud P}{\ud Q}\log \Bigl(e^{-h}\,\frac{\ud P}{\ud Q}\Bigr) - \frac{\ud P}{\ud Q}  + e^h\Big) \, \ud Q\,.
	\end{multline}
	By change of measure, we have
	\begin{equation*}
		\frac{\ud P}{\ud Q} = \frac{\ud P_h}{\ud Q}\, \frac{\ud P}{\ud P_h} = e^h\, \frac{\ud P}{\ud P_h}\,.
	\end{equation*}
	Therefore~\eqref{rl_diff} can be written
	\begin{equation*}
		\Delta = \int e^h\, \Big(\frac{\ud P}{\ud P_h}\log \frac{\ud P}{\ud P_h} - \frac{\ud P}{\ud P_h} + 1 \Big) \, \ud Q = \KL(P \mmid P_h)\,.
	\end{equation*}
	Since $\KL(P \mmid P_h) \geq 0$ and $\KL(P_h \mmid P_h) = 0$, this proves the claim.
\end{proof}

We can therefore rewrite~\eqref{eta-D} as
\begin{multline*}
	\sup_{f, g \in \Cb(\Omega)}{\Bigl\{\int f \, \ud \mu + \int g \, \ud \nu\Bigr\}} \\
	- \eps \sup_{\gamma \in \cM_+(\Omega)}{\Bigl\{ \int \frac{f(x) + g(y) - \|x - y\|^2}{\eps} \, \gamma(\ud x,\ud y) 
	- \KL(\gamma \mmid \mu \otimes \nu)\Bigr\}}\,,
\end{multline*}
and, rearranging,
\begin{align*}
	\sup_{f, g \in \Cb(\Omega)} \inf_{\gamma \in \cM_+(\Omega)}\,\Bigl\{\int \|x - y\|^2\, &\gamma(\ud x,\ud y) + \eps \KL(\gamma \mmid \mu \otimes \nu) \\
 &{}+ \int f \, \ud \mu + \int g \, \ud \nu - \int f \oplus g \, \ud \gamma\Bigr\}\,,
\end{align*}
where we define $(f \oplus g)(x, y) = f(x) + g(y)$.

As in Subsection~\ref{subsec:dualK}, we can swap the inf and sup to get an upper bound on the value of~\eqref{eta-D}:
\begin{align}\label{eta-P-1}
\begin{aligned}
	\inf_{\gamma\in\cM(\Omega)}\,&\Bigl\{\int \|x - y\|^2\, \gamma(\ud x,\ud y) + \eps \KL(\gamma \mmid \mu \otimes \nu) \\
 &\qquad{} + \sup_{f, g\in\Cb(\Omega)}{\Bigl\{\int f \, \ud \mu + \int g \, \ud \nu - \int f \oplus g \, \ud \gamma\Bigr\}}\Bigr\}\,.
 \end{aligned}
\end{align}
We have already seen that
\begin{equation*}
	\sup_{f, g\in\Cb(\Omega)}{\Bigl\{\int f \, \ud \mu + \int g \, \ud \nu - \int f \oplus g \, \ud \gamma\Bigr\}} = \begin{cases}
	0, & \text{if $\gamma \in \Gamma_{\mu, \nu}$}, \\
	\infty, & \text{otherwise}.
	\end{cases}
\end{equation*}
Therefore, \eqref{eta-P-1} is equivalent to
\begin{equation}
\label{eta}\tag{$\mathsf{\eps}$-$\mathsf{W_2^2}$}
	\inf_{\gamma \in \Gamma_{\mu, \nu}}{\Bigl\{\int \|x - y\|^2\, \gamma(\ud x,\ud y) + \eps \KL(\gamma \mmid \mu \otimes \nu)\Bigr\}}\,.
\end{equation}
This is the \emph{primal} version of the entropic OT problem, and it is the version that is usually taken as the definition of entropic regularization.
This choice of regularization is usually justified \emph{a posteriori} by the existence of Sinkhorn's algorithm (see Section~\ref{sec:eot_duality}), but as we have seen it also arises naturally from a simple relaxation of the dual Kantorovich problem.
The argument in this section establishes a form of weak duality, showing that the value of \eqref{eta} is lower bounded by the value of \eqref{eta-D}.
The next section establishes a tight connection between the primal and dual problems, both in terms of their optimal value and their optimal solutions.
This connection has been heavily exploited in the statistical analysis of entropic OT.

\section{Duality}\label{sec:eot_duality}

In this section, we show that the values of the primal problem~\eqref{eta} and dual problem~\eqref{eta-D} actually agree, and that an optimal solution to one problem can be extracted from the optimal solution to the other.

\begin{proposition}\label{prop:eot_strong_duality}
	Let $f^\star$, $g^\star$ solve~\eqref{eta-D}. Then
	\begin{equation}\label{eq:gamma_star_def}
		\gamma^\star( \ud x, \ud y) = \exp\Bigl(\frac{f^\star(x) + g^\star(y) - \|x - y\|^2}{\eps}\Bigr)\, \mu(\ud x)\, \nu(\ud y)
	\end{equation}
	is the unique solution to~\eqref{eta}, and
	\begin{align*}
		&\int \|x - y\|^2 \,\gamma^\star(\ud x,\ud y) + \eps \KL(\gamma^\star \mmid \mu \otimes \nu) \\
            &\qquad  = \int f^\star\, \ud \mu + \int g^\star\, \ud \nu - \iota^\eps(f^\star, g^\star)
		= \int f^\star\, \ud \mu + \int g^\star\, \ud \nu\,.
	\end{align*}
\end{proposition}
\begin{proof}
	It suffices to show that $\gamma^\star$ is a solution to~\eqref{eta}, since uniqueness follows immediately from the strict convexity of $\KL(\gamma \mmid \mu \otimes \nu)$.
	
	We need to show that $\gamma^\star \in \Gamma_{\mu, \nu}$.
	Clearly $\gamma^\star$ is positive, so it suffices to show that it has the correct marginals.
	To that end, we need to verify that for any Borel set $A$,
	\begin{equation*}
		\nu(A) = \int_A \int \exp\Bigl(\frac{f^\star(x) + g^\star(y) - \|x - y\|^2}{\eps}\Bigr)\, \mu(\ud x) \,\nu(\ud y)\,.
	\end{equation*}
	Equivalently, we need to show that
	\begin{equation}\label{eq:marg_constraint}
		\int \exp\Bigl(\frac{f^\star(x) + g^\star(y) - \|x - y\|^2}{\eps}\Bigr)\, \mu(\ud x) = 1 \qquad \text{$\nu$-a.e.}
	\end{equation}
	
	Let us define the function
	\begin{equation}\label{eq:bar_g_def}
		\bar g(y) = - \eps \log \int \exp\Bigl(\frac{f^\star(x)  - \|x - y\|^2}{\eps}\Bigr)\, \mu(\ud x)\,.
	\end{equation}
	We show that $\bar g = g^\star$, $\nu$-almost everywhere.
	Since
	\begin{equation}\label{eq:bar_marg_constraint}
		\int \exp\Bigl(\frac{f^\star(x) + \bar g(y) - \|x - y\|^2}{\eps}\Bigr)\, \mu(\ud x) = 1\qquad \forall y\in\R^d
	\end{equation}
	holds by definition, this establishes that~\eqref{eq:marg_constraint} holds as well.
	
	Let $\bar \gamma( \ud x, \ud y) = \exp((f^\star(x) + \bar g(y) - \|x - y\|^2)/\eps) \,\mu(\ud x)\, \nu(\ud y)$.
	Then,
	\begin{align*}
		0 \leq \KL(\bar \gamma \mmid \gamma^\star) & = \int \Bigl(\frac{\ud \bar \gamma}{\ud \gamma^\star} \log \frac{\ud \bar \gamma}{\ud \gamma^\star} - \frac{\ud \bar \gamma}{\ud \gamma^\star} + 1\Bigr) \, \ud \bar \gamma \\
		& = \int \frac{\bar g(y) - g^\star(y)}{\eps}\, \bar \gamma(\ud x,\ud y) - \int \ud \bar \gamma + \int \ud \gamma^\star \\
		& = \frac{1}{\eps} \,\Bigl(\int (\bar g - g^\star) \, \ud \nu - \iota^\eps(f^\star, \bar g) + \iota^\eps(f^\star, g^\star) \Bigr)\,,
		\intertext{where the last step uses that the second marginal of $\bar \gamma$ is $\nu$, by~\eqref{eq:bar_marg_constraint}.
		We obtain that}
	\eps \KL(\bar \gamma \mmid \gamma^\star) & = \int \bar g \, \ud \nu + \int f^\star \, \ud \mu - \iota^\eps(f^\star, \bar g) \\
	& \quad \quad  - \left(\int g^\star \, \ud \nu + \int f^\star \, \ud \mu - \iota^\eps(f^\star, g^\star) \right) \\
	& \leq 0\,,
	\end{align*}
	since $(f^\star, g^\star)$ are optimal for~\eqref{eta-D}.
	
	Therefore $\KL(\bar \gamma \mmid \gamma^\star) = 0$, so $\bar \gamma = \gamma^\star$, and $\bar g = g^\star$, $\nu$-almost everywhere.
	This establishes that the second marginal of $\gamma^\star$ is $\nu$, and an analogous argument shows that the first marginal of $\gamma^\star$ is $\mu$.
	We obtain that $\gamma^\star$ is feasible in~\eqref{eta}.
	
	To conclude, we compute the value that $\gamma^\star$ achieves in the primal problem:
	\begin{align*}
		&\int \|x - y\|^2\, \gamma^\star(\ud x,\ud y) + \eps \KL(\gamma^\star \mmid \mu \otimes \nu) \\
  &\qquad  = 	\int \|x - y\|^2 \,\gamma^\star(\ud x,\ud y) \\
		&\qquad\qquad{} + \int (f^\star(x) + g^\star(y) - \|x - y\|^2)\, \gamma^\star(\ud x,\ud y) \\
		&\qquad  = \int (f^\star(x) + g^\star(y))\, \gamma^\star(\ud x,\ud y) \\
		&\qquad = \int f^\star \,\ud \mu + \int g^\star\, \ud \nu - \iota^\eps(f^\star, g^\star)\,,
	\end{align*}
	where the last step uses that $\gamma^\star \in \Gamma_{\mu, \nu}$ and that $\iota^\eps(f^\star, g^\star) = 0$ since $\gamma^\star$ is a probability measure.
	Therefore, the primal objective evaluated at $\gamma^\star$ and the dual objective evaluated at $(f^\star, g^\star)$ have the same value, and weak duality (see Section~\ref{sec:eot_deriv}) shows that $\gamma^\star$ and $(f^\star, g^\star)$ are an optimal pair of primal/dual solutions.
\end{proof}

Proposition~\ref{prop:eot_strong_duality} deserves several remarks.
First, note that the hypothesis of the proposition is the existence of optimal solutions to \eqref{eta-D}.
We do not justify the existence of such solutions here, but it can be shown that that if $\mu$ and $\nu$ are compactly supported, then there exist $f^\star, g^\star \in \Cb(\Omega)$.
More generally, if $\mu$ and $\nu$ have finite second moments, then optima exist in $L^1(\mu)$ and $L^1(\nu)$, respectively, and Proposition~\ref{prop:eot_strong_duality} continues to hold. 

The proof of Proposition~\ref{prop:eot_strong_duality} actually shows that if $f$, $g$ are such that $\gamma(\ud x, \ud y) = \exp((f(x) + g(y) - \|x-y\|^2)/\eps)\, \mu(\ud x) \,\nu(\ud y)$ is a valid coupling between $\mu$ and $\nu$, then $\gamma$ is optimal for~\eqref{eta} and $f$, $g$ are optimal for~\eqref{eta-D}.
This fact can be viewed as an entropic variant of the complementary slackness condition $\bar f(x) + \bar g(y) = \|x - y\|^2$, $\bar \gamma$-almost everywhere, which holds for the optimal solutions of~\eqref{W2} and~\eqref{DK}.
(See Theorem~\ref{thm:fundOT}.)
We can therefore conclude that $f^\star$ and $g^\star$ are optimal for~\eqref{eta-D} if and only if they satisfy the marginal constraint~\eqref{eq:marg_constraint} and the analogous constraint for the other marginal:
\begin{equation}\label{eq:marg_constraint_2}
			\int \exp\Bigl(\frac{f^\star(x) + g^\star(y) - \|x - y\|^2}{\eps}\Bigr)\, \nu(\ud y) = 1 \quad \text{$\mu$-a.e.}
\end{equation}
The marginal constraints~\eqref{eq:marg_constraint} and~\eqref{eq:marg_constraint_2}, sometimes known as the \emph{Schr\"odinger system}, are at the core of the theory of entropic OT.
Even though these equations \emph{a priori} only hold $\nu$- and $\mu$-almost everywhere, respectively, the construction in~\eqref{eq:bar_g_def} shows that we can construct canonical extensions of $f^\star$ and $g^\star$ so that the marginal constraints hold everywhere in $\R^d$.
Moreover, the dominated convergence theorem shows that if $\mu$ and $\nu$ are compactly supported, then these extensions are continuous (even $C^\infty$) functions on $\R^d$.
In what follows, we may therefore always assume that $f^\star$ and $g^\star$ are defined everywhere on $\R^d$, and that \eqref{eq:marg_constraint} and~\eqref{eq:marg_constraint_2} hold for all $y$ and $x \in \R^d$, respectively.

Finally, we note that Proposition~\ref{prop:eot_strong_duality} is the basis for the celebrated \emph{Sinkhorn algorithm}\index{Sinkhorn's algorithm} for entropic optimal transport.
This algorithm is defined by initializing $f_0 \equiv 0$, and for $t \geq 1$ performing the updates
\begin{align}
    g_t(y) & = - \eps \log \int \exp\Bigl(\frac{f_{t-1}(x)  - \|x - y\|^2}{\eps}\Bigr) \,\mu(\ud x)\,,\label{eq:sinkhorn1} \\
    f_t(y) & = - \eps \log \int \exp\Bigl(\frac{g_{t}(y)  - \|x - y\|^2}{\eps}\Bigr) \,\nu(\ud y)\,.\label{eq:sinkhorn2}
\end{align}
Proposition~\ref{prop:eot_strong_duality} shows that a fixed point of this algorithm yields an optimal solution to~\eqref{eta-D}, and therefore an optimal solution to~\eqref{eta}.

\section{Statistical rates for dual solutions}\label{sec:eot_sample}

In this and the following section, we consider the statistical behavior of empirical versions of the entropic OT problem.
In contrast to the results of Chapter~\ref{chap:primal-dual}, the rates of convergence (as a function of $n$) no longer suffer from the curse of dimensionality.
However, the price to pay for this improvement is a steep dependence on $1/\eps$.

The strategy for proving statistical bounds is to analyze the dual problem~\eqref{eta-D}.
We then transfer these bounds to the primal solution using the connection between primal and dual solutions given by Proposition~\ref{prop:eot_strong_duality}.

Let us denote by $S(\mu, \nu)$ the value of the primal problem~\eqref{eta}.
Given i.i.d.\ samples from $\mu$ and $\nu$, we are chiefly interested in estimating two quantities:
\begin{itemize}
	\item The cost $S(\mu, \nu)$,
	\item The \emph{entropic map} or \emph{entropic regression function}
 \begin{align*}
 b^\star(x) = \E_{(X, Y) \sim \gamma^\star}[Y \mid X = x]\,.
 \end{align*}
\end{itemize}
Estimating the first quantity is the entropic analogue of the question we considered in Chapter~\ref{chap:primal-dual}.
Estimating the second quantity is the entropic analogue of the map estimation task described in Chapter~\ref{chap:maps}. 
Indeed, $b^\star$ is a projection of $\gamma^\star$, in the sense of $L^2$, onto the space of maps; however, we stress that $b^\star$ is not a valid transport map between $\mu$ and $\nu$, since $(b^\star)_\# \mu \neq \nu$.

As in Chapter~\ref{chap:primal-dual}, we analyze \emph{plug-in estimators} for these quantities: $S(\mu_n, \nu_n)$ for the cost, and $\hat b(x) = \E_{(X, Y) \sim \hat \gamma}[Y \mid X = x]$, where $\hat \gamma$ is the optimal solution to the empirical entropic OT problem between $\mu_n$ and $\nu_n$.\footnote{Though the formulas for the entropic maps $b^\star$ and $b_n$ define them as elements of $L^1(\mu)$ and $L^1(\mu_n)$, respectively, the canonical extensions described in Section~\ref{sec:eot_duality} can be used to define continuous versions of $b^\star$ and $b_n$.}

As emphasized above, our main tool to analyze these quantities is the duality relationship established in Proposition~\ref{prop:eot_strong_duality}.
Denote by $\Cb^{\oplus}$ the subspace of $\Cb(\Omega \times \Omega)$ consisting of functions of the form $f \oplus g$ for $f, g \in \Cb(\Omega)$.
The dual problem~\eqref{eta-D} depends on $f$ and $g$ only through their sum $h = f \oplus g \in \Cb^{\oplus}$. In particular, if $(f,g)$ is a dual solution, then so is $(f+\lambda,g-\lambda)$ for any $\lambda \in \R$.
Inspired by this fact, let us define the \emph{dual functional} $\Phi: \Cb^{\oplus}  \to \R$ given by
\begin{equation}\label{eq:Phi_def}
	h \mapsto \int \bigl(h   - \eps\, (e^{(h - c)/\eps} - 1)\bigr) \, \ud(\mu \otimes \nu)\,,
\end{equation}
where $c(x, y) = \|x - y\|^2$ is the squared Euclidean cost.
The dual problem can then be written succinctly as $\sup_{h \in \Cb^{\oplus}} \Phi(h)$.

Suppose we wish to compare $S(\mu, \nu)$ to $S(\mu_n, \nu_n)$, where, as in Chapter~\ref{chap:primal-dual}, $\mu_n$ and $\nu_n$ denote empirical measures corresponding to i.i.d.\ samples from $\mu$ and $\nu$.
We can define an empirical version of the dual functional by
\begin{equation*}
	\widehat \Phi(h) = \int \bigl(h   - \eps\, (e^{(h - c)/\eps} - 1)\bigr) \, \ud(\mu_n \otimes \nu_n)\,.
\end{equation*}
Then
\begin{equation*}
	S(\mu_n, \nu_n) - S(\mu, \nu) = \sup_{h \in \Cb^{\oplus}} \widehat \Phi(h) - \sup_{h \in \Cb^{\oplus}} \Phi(h) = \widehat \Phi(\hat h) - \Phi(h^\star)\,,
\end{equation*}
where $\hat h$ and $h^\star$ are maximizers of $\widehat \Phi$ and $\Phi$, respectively.

Exercise~\ref{exe:eot_empirical_process} sketches a direct approach to obtain an upper bound on $\widehat \Phi(\hat h) - \Phi(h^\star)$ based on empirical process theory, analogous to the one developed in Section~\ref{sec:chaining} for the unregularized optimal transport problem.
However, we pursue a different path, which leverages the strong concavity of the dual functional. 

We begin with a non-rigorous sketch of the argument. Strong concavity of the functional $\widehat \Phi$ should imply there exists $\delta>0$ such that
\begin{equation*}
		\widehat \Phi(\hat h) \leq \widehat \Phi(h^\star) + \langle \nabla \widehat \Phi(h^\star), \hat h - h^\star \rangle_{L^2(\mu_n \otimes \nu_n)} - \frac{\delta}{2}\, \|\hat h -  h^\star\|_{L^2(\mu_n \otimes \nu_n)}^2\,.
	\end{equation*}
While it is possible to define a suitable notion of gradient for $\nabla \widehat \Phi$, it is sufficient for our purposes to interpret the above inner product as a directional (G\^ateaux) derivative. 
In contrast to the empirical process theory approach, this inequality implies that we can obtain a bound by controlling $\widehat \Phi$ and $\nabla \widehat \Phi$ at the \emph{fixed} function $h^\star$ rather than the random function $\hat h$. In particular, there is no need to ``sup-out" $\hat h$  which allows us to circumvent the use of empirical process theory.

We make the following assumption.
\begin{assumption}\label{assume:bounded}
	The supports of $\mu$ and $\nu$ lie in $\Omega \subseteq B_{1/2}(0)$.
\end{assumption}
In particular, under Assumption~\ref{assume:bounded}, $\mathrm{diam}(\Omega) \leq 1$.
This assumption implies simple \emph{a priori} bounds on $\hat h$ and $h^\star$.

\begin{proposition}\label{prop:duals_bounded}
	Under Assumption~\ref{assume:bounded}, it holds that
 $$
 \|\hat h\|_{L^\infty}, \|h^\star\|_{L^\infty} \leq 2\,.
 $$
\end{proposition}
\begin{proof}
	We first prove the claim for $h^\star$.
	Recall that $h^\star = f^\star \oplus g^\star$ for $f^\star, g^\star \in \Cb(\Omega)$
	and that thanks to canonical extensions, we may assume that the marginal constraints~\eqref{eq:marg_constraint} and \eqref{eq:marg_constraint_2} hold for all $x, y \in \Omega$.
Since $c(x, y) \leq 1$ for all $x, y \in \Omega$, we get
	\begin{align*}
		1 & = \int e^{(f^\star \oplus g^\star - c)/\eps} \, \mu(\ud x) \geq e^{(g^\star(y) - 1)/\eps}\int e^{f^\star(x)/\eps}\, \mu(\ud x)\,, \quad \forall y \in \Omega\,, \\
		1 & = \int e^{(f^\star \oplus g^\star - c)/\eps} \, \nu(\ud y) \geq e^{(f^\star(x) - 1)/\eps}\int e^{g^\star(y)/\eps}\, \nu(\ud y)\,,\quad \forall x \in \Omega\,.
	\end{align*}
 Multiplying these two inequalities yields
 $$
 e^{(h^\star - 2)/\eps}\int e^{h^\star/\eps}\,\ud(\mu \otimes \nu) \le 1 \,.
 $$
 Next note that by Jensen's inequality, we get
 $$
\int e^{h^\star/\eps}\,\ud(\mu \otimes \nu)\ge e^{S(\mu, \nu)/\eps} \ge 1\,,
 $$
 where we used Proposition~\ref{prop:eot_strong_duality}. From the above two displays, we get that $h^\star \le 2$ for all $x,y \in \Omega$.

 Next, since $c\ge 0$ on $\Omega\times \Omega$, we get from the same argument that
 $$
 1\le e^{h^\star/\eps}\int e^{h^\star/\eps}\,\ud(\mu \otimes \nu) \le e^{(h^\star+2)/\eps} \,,
 $$
 where we used the bound $h^\star \le 2$ that we just proved. Hence $h^\star \ge -2$ for all $x,y \in \Omega$.
	
	Since the only fact that was used about $h^\star$ is that it maximizes the dual functional $\Phi$ corresponding to measures whose supports lie in $\Omega$, the claim also holds for $\hat h$ when replacing $(\mu, \nu)$ with $(\mu_n, \nu_n)$ in Proposition~\ref{prop:eot_strong_duality}. Again, canonical extensions play a crucial role here.
\end{proof}

We require some fundamental differentiability and concavity properties of the empirical dual functional.
If we let $\mu_n = \frac 1n \sum_{i=1}^n \delta_{X_i}$ and $\nu_n = \frac 1n \sum_{j=1}^n \delta_{Y_j}$, for $X_1, \dots, X_n \simiid \mu$ and $Y_1, \dots, Y_n \simiid \nu$, then we can write $\widehat \Phi$ explicitly as
\begin{equation}\label{eq:emp_dual_def}
	\widehat \Phi(h) = \frac{1}{n^2} \sum_{i, j = 1}^n\Bigl( h(X_i, Y_j) - \eps \,\bigl(e^{(h(X_i, Y_j) - \|X_i - Y_j\|^2)/\eps}-1\bigr)\Bigr)\,.
\end{equation}
Rather than appealing to functional analysis to study differentiability of the functional $\widehat \Phi$, it is sufficient to study the function $\varphi$ defined on $[0,1]$ by
\begin{equation}\label{eq:def_varphi(t)}
    \varphi(t)=\hat \Phi(h_t)\,, \quad \text{where} \ h_t \deq (1-t)\hat h + th^\star\,.
\end{equation}
In particular, it is twice differentiable with derivatives given by
\begin{equation}
    \label{eq:phiprime}
\varphi'(t)=     \frac{1}{n^2} \sum_{i, j = 1}^n\Bigl( \big(h^\star(X_i, Y_j)-\hat h(X_i, Y_j)\big)\, \big(1- e^{(h_t(X_i, Y_j) - \|X_i - Y_j\|^2)/\eps}\big)\Bigr)\,,
\end{equation}
and
\begin{align}
\varphi''(t)&=    -\frac{1}{\eps n^2} \sum_{i, j = 1}^n\Bigl( \big(h^\star(X_i, Y_j)-\hat h(X_i, Y_j)\big)^2\, e^{(h_t(X_i, Y_j) - \|X_i - Y_j\|^2)/\eps}\Bigr)\nonumber\\
&\le-\frac{e^{-3/\eps}}{\eps} \,\|\hat h - h^\star\|^2_{L^2(\mu_n \otimes \nu_n)} \,,\label{eq:phi2prime}
\end{align}
where we used Proposition~\ref{prop:duals_bounded} and Assumption~\ref{assume:bounded} in the above inequality. 
We readily get that $\varphi$ is strongly concave on $[0,1]$. 

The expression~\eqref{eq:phiprime} reveals that the derivative of $\varphi$ is an $L^2(\mu_n \otimes \nu_n)$ inner product with a function in the space $\Cb^{\oplus}$.
This inner product can be well understood using the Hoeffding (a.k.a.\ Efron--Stein, a.k.a.\ ANOVA) decomposition~\cite{Hoe48}\index{Hoeffding's decomposition}.

\begin{definition}\label{def:efron-stein}
    Let $X,Y$ be two independent random variables with distributions $P$ and $Q$ respectively. Given $k \in L^2(P\otimes Q)$, the Hoeffding decomposition of $k(X,Y)$ in $ L^2(P\otimes Q)$ is given by
    $$
k(X,Y)= \bar k_1(X) + \bar k_2(Y) + \bbar k + \mathfrak{r}(X,Y)
    $$
    where
\begin{align*}
    \bbar k &= \E[k(X,Y)] \in \R\,,\\
    \bar k_1(x) &=\E[k(X,Y)\mid X=x]-\bbar k\,,\\
     \bar k_2(y) &=\E[k(X,Y)\mid Y=y] - \bbar k\,,
\end{align*}
and 
$$
\mathfrak{r}(X,Y) = k(X,Y) - \bar k_1(X) - \bar k_2(Y) - \bbar k\,.
$$
\end{definition}

It is easy to check (exercise!) that the Hoeffding decomposition is orthogonal in $L^2(P\otimes Q)$. In fact, the same calculations reveal the   relevance of this decomposition to our problem: for any $h=f\oplus g \in \Cb^{\oplus}$, it holds
\begin{align*}
\langle h, k\rangle_{L^2(P\otimes Q)}&= \langle f, \bar k_1\rangle_{L^2(P)}+ \langle g, \bar k_2\rangle_{L^2(Q)}+ \langle  h, \bbar k\rangle_{L^2(P\otimes Q)}\\
&=\langle h, \bar k_1 +\bar k_2 + \bbar k\rangle_{L^2(P\otimes Q)}\,.
\end{align*}
Using Cauchy--Schwarz and orthogonality of the Hoeffding decomposition, we get
\begin{align}
\langle h, k\rangle_{L^2(P\otimes Q)}^2&
\le \|h\|^2_{L^2(P\otimes Q)}\|\, \bar k_1 +\bar k_2 + \bbar k\|^2_{L^2(P\otimes Q)}\nonumber\\
&=\|h\|^2_{L^2(P\otimes Q)}\,\big(\| \bar k_1\|^2_{L^2(P\otimes Q)}+\|\bar k_2\|^2_{L^2(P\otimes Q)} + \bbar k^2\big)\,.\label{eq:cs-hoeffding}
\end{align}
Using orthogonality again implies
\begin{align*}
	\E\bigl[\bigl(\E[k(X,Y)\mid X]\bigr)^2\bigr] & = \|\bar k_1 + \bbar k\|_{L^2(P \otimes Q)}^2\\
	& =  \|\bar k_1\|_{L^2(P \otimes Q)}^2 + \bbar k^2
\end{align*}
and similarly
$$
\E\bigl[\bigl(\E[k(X,Y)\mid Y]\bigr)^2\bigr] = \| \bar k_2\|^2_{L^2(P\otimes Q)}+ \bbar k^2\,.
$$
These two identities together with~\eqref{eq:cs-hoeffding} yield
$$
\frac{\langle h, k\rangle_{L^2(P\otimes Q)}^2}{\|h\|_{L^2(P\otimes Q)}^2} \le \E\bigl[\bigl(\E[k(X,Y)\mid X]\bigr)^2\bigr]+ \E\bigl[\bigl(\E[k(X,Y)\mid Y]\bigr)^2\bigr] - \bbar k^2\,.
$$
Applying this result with $P=\mu_n$, $Q=\nu_n$ we get the following lemma.

\begin{lemma}\label{lem:cs-projection}
	For any $k \in L^2(\mu_n \otimes \nu_n)$ and any $h \in \Cb^{\oplus}$, we have
	\begin{equation*}
		\langle h, k \rangle_{L^2(\mu_n \otimes \nu_n)}^2 \leq \|h\|_{L^2(\mu_n \otimes \nu_n)}^2\,(\|\nu_n(k)\|_{L^2(\mu_n)}^2 + \|\mu_n(k)\|_{L^2(\nu_n)}^2)\,,
	\end{equation*}
 where
 $$
\mu_n(k)(y)  = \frac 1n \sum_{i=1}^n k(X_i, y)\,, \,\qquad 
	\nu_n(k)(x) = \frac 1n \sum_{j=1}^n k(x, Y_j)\,.
 $$
\end{lemma}

We are now in a position to obtain an important lemma showing that $\hat h$ is a good estimator of $h^\star$. In turn, rates of convergence for the cost and the entropic map follow from this lemma.

\begin{lemma}\label{lem:entropic_potential_estimation}
Let Assumption~\ref{assume:bounded} hold. Then
$$
\E\|\hat h - h^\star\|^2_{L^2(\mu_n \otimes \nu_n)} \le \frac{2\eps^2 e^{10/\eps}}{n}\,.
$$
\end{lemma}
\begin{proof}
Since $\varphi$ is strongly concave, using respectively~\eqref{eq:phi2prime}, the optimality condition $\varphi'(0)=0$, and~\eqref{eq:phiprime}, we get
\begin{align}
    &\frac{e^{-3/\eps}}{\eps}\,\|\hat h - h^\star\|^2_{L^2(\mu_n \otimes \nu_n)} \le \varphi'(0) -\varphi'(1)= -\varphi'(1) \nonumber \\
    &\qquad = \frac{1}{n^2} \sum_{i, j = 1}^n\Bigl( \big(h^\star(X_i, Y_j)-\hat h(X_i, Y_j)\big)\, \big(e^{(h^\star(X_i, Y_j) - \|X_i - Y_j\|^2)/\eps}-1\big)\Bigr) \nonumber\\
    & \qquad= \langle h^\star - \hat h,p^\star - 1 \rangle_{L^2(\mu_n \otimes \nu_n)} \label{eq:inter_1}\,,
\end{align}
where 
$$
p^\star(x,y)=e^{(h^\star(x, y) - \|x - y\|^2)/\eps}\,.
$$

Since $h^\star - \hat h \in \Cb^{\oplus}$, Lemma~\ref{lem:cs-projection} implies
\begin{equation*}
	\langle h^\star - \hat h,p^\star - 1 \rangle_{L^2(\mu_n \otimes \nu_n)} \leq \|h^\star - \hat h\|_{L^2(\mu_n \otimes \nu_n)}\,\delta_n\,,
\end{equation*}
where
\begin{equation*}
	\delta_n = \bigl(\|\nu_n(p^\star - 1)\|_{L^2(\mu_n)}^2 + \|\mu_n(p^\star - 1)\|_{L^2(\nu_n)}^2\bigr)^{1/2}\,.
\end{equation*}

Combining this with~\eqref{eq:inter_1} yields
\begin{equation}
	\label{eq:pr_entropic_potential_estimation_1}
	\|\hat h - h^\star\|^2_{L^2(\mu_n \otimes \nu_n)}\le \eps^2 e^{6/\eps}\,\delta_n^2\,.
\end{equation}

Recall from~\eqref{eq:marg_constraint} that
$$
\E[ p^\star(X_i, Y_j) \mid Y_j]=1\,,
$$
so that
\begin{align*}
	\E\|\mu_n(p^\star - 1)\|_{L^2(\nu_n)}^2 & =\frac{1}{n}\, \E \sum_{j=1}^n\Bigl(\frac 1n \sum_{i=1}^n p^\star(X_i, Y_j) - \E[p^\star (X_i, Y_j)\mid Y_j]\Bigr)^2 \\
	& = \frac 1 n\, \E \var\bigl(p^\star(X_1, Y_1) \bigm\vert Y_1\bigr) \\
	& \leq \frac{\E[p^\star(X_1, Y_1)^2]}{n} \leq \frac{e^{4/\eps}}{n}\,.
\end{align*}
An analogous bound holds for $\E\|\nu_n(p^\star - 1)\|_{L^2(\mu_n)}^2$, which implies  that
\begin{equation}\label{eq:delta_n_bound}
	\E \delta_n^2 \leq \frac{2 e^{4/\eps}}{n}\,,
\end{equation}
proving the claim.
\end{proof}

\section{Statistical rates for primal solutions}\label{sec:eot_rates_primal}

Lemma~\ref{lem:entropic_potential_estimation} shows that solutions to the dual problem~\eqref{eta-D} converge at the parametric rate.
In this section, we use this result to give bounds for the primal problem as well.

We now turn to our first quantity of interest, the cost $S(\mu, \nu)$. The following result shows that the mean squared error and bias of the estimator $S(\mu_n, \nu_n)$ are both of order $n^{-1}$, albeit with constants that scale exponentially in $1/\eps$.
Strikingly, the $n^{-1}$ rate we have obtained for the variance and bias is characteristic of \emph{parametric} estimation problems, despite the non-parametric setting. It is of course to be contrasted with the slow, non-parametric rates of Chapter~\ref{chap:primal-dual}.

\begin{theorem}\label{thm:rate_Sn}
	If $\mu$ and $\nu$ satisfy Assumption~\ref{assume:bounded}, then the mean squared error and bias of $S(\mu_n, \nu_n)$ satisfy
	\begin{align}
		\E (S(\mu_n, \nu_n) - S(\mu, \nu))^2 & \lesssim \frac{1}{n}\,, \nonumber\\
				|\E S(\mu_n, \nu_n) - S(\mu, \nu)| & \lesssim \frac{1}{n}\,,\label{eq:ent_bias}
	\end{align}
	where the implicit constants depends exponentially on $1/\eps$.
\end{theorem}
\begin{proof}
	We begin with establishing~\eqref{eq:ent_bias} by studying the bias. Jensen's inequality implies
	\begin{equation*}
		\E S(\mu_n, \nu_n) = \E \sup_{h \in \Cb^{\oplus}} \widehat \Phi(h) \geq \sup_{h \in \Cb^{\oplus}} \Phi(h) = S(\mu, \nu)\,.
	\end{equation*}
Hence
\begin{align*}
    0\le b_n & \deq \E S(\mu_n, \nu_n) - S(\mu, \nu)\\
    &=\E\big[ \widehat \Phi(\hat h) - \widehat \Phi(h^*)\big]\\
    &=\E\big[\varphi(0)-\varphi(1)\big]\le -\E\varphi'(1)\,,
\end{align*}
where we recall that the concave function $\varphi$ is defined in~\eqref{eq:def_varphi(t)}. It follows from~\eqref{eq:phiprime} that $-\varphi'(1)$ is given by
\begin{align*}
    &\frac{1}{n^2} \sum_{i, j = 1}^n\Bigl( \big(h^\star(X_i, Y_j)-\hat h(X_i, Y_j)\big)\, \big(e^{(h^\star(X_i, Y_j) - \|X_i - Y_j\|^2)/\eps}-1\big)\Bigr) \\
    &\qquad \le \|\hat h -h^\star\|_{L^2(\mu_n \otimes \nu_n)}\, \delta_n\,,
\end{align*}
where $\delta_n$ is defined as in the proof of Lemma~\ref{lem:entropic_potential_estimation} and we have applied Lemma~\ref{lem:cs-projection}.
By the Cauchy--Schwarz inequality,
\begin{align*}
    0 \le b_n
    &\le \sqrt{\E\|\hat h -h^\star\|_{L^2(\mu_n \otimes \nu_n)}^2 \,\E\delta_n^2} \\
    &\le \frac{\sqrt 2 \eps e^{5/\eps}}{\sqrt n}\,\frac{\sqrt 2 e^{2/\eps}}{\sqrt n}
    = \frac{2\eps e^{7/\eps}}{n}\,,
\end{align*}
where we used Lemma~\ref{lem:entropic_potential_estimation} and~\eqref{eq:delta_n_bound}.

	To prove the bound on the mean squared error, we first use a bias--variance decomposition to write
	\begin{align*}
		\E (S(\mu_n, \nu_n) - S(\mu, \nu))^2 & = \var(S(\mu_n, \nu_n)) + |\E S(\mu_n, \nu_n) - S(\mu, \nu)|^2 \\
		& = \var(S(\mu_n, \nu_n))  + O(n^{-2})\,.
	\end{align*}
	It therefore suffices to show that the variance of $S(\mu_n, \nu_n)$ is $O(n^{-1})$.
	For this purpose, we employ the \emph{Efron--Stein inequality}. More precisely,~\cite[Corollary~3.2]{BouLugMas13} is sufficient for our purposes. It says that if $f = f(Z_1, \dots, Z_m)$ is a function of independent random variables that satisfies the bounded differences inequality:
 \begin{equation}\label{eq:bounded_differences}
		|f(z_1, \dots, z_m)-f(z_1, \dots, z_{i-1}, z_i', z_{i+1}, \dots, z_m)| \leq 2c
	\end{equation}
        for all $z_1, \dots, z_m, z_i'$ and all $i \in [m]$,
	then $\var(f) \leq c^2 m$.
	
	Let us view $S(\mu_n, \nu_n)$ as a function of the $m=2n$ independent random variables $X_1, \dots, X_n, Y_1, \dots, Y_n$.
Fix $(X_2, \dots, X_{n}) = (x_2, \dots, x_{n})$ and $(Y_1, \dots, Y_n) = (y_1, \dots, y_n)$, and view the dual functional $\widehat \Phi = \widehat \Phi_{x_1}$ as a function of the value of $X_1 = x_1$ alone.
	Then for any $x_1 \in \Omega$, under Assumption~\ref{assume:bounded}, Proposition~\ref{prop:duals_bounded} implies that the maximizer of the dual functional $\widehat \Phi_{x_1}$ over $\Cb^{\oplus}$ is achieved at an $h$ satisfying $\|h\|_\infty \leq 2$.
	Therefore
	\begin{align*}
		|\sup_{h \in \Cb^{\oplus}} \widehat \Phi_{x_1}(h) - \sup_{h' \in \Cb^{\oplus}} \widehat \Phi_{x'_1}(h')| &\leq \sup_{h \in \Cb^{\oplus},\, \|h\|_\infty \leq 2} |\widehat \Phi_{x_1}(h) - \widehat \Phi_{x_1'}(h)| \\
        & \le  \frac{2\eps e^{2/\eps}+4}{n} \eqqcolon 2c\,,
	\end{align*}
	where we have used the fact that each term in~\eqref{eq:emp_dual_def} is bounded.
	Repeating this argument for $X_2, \dots, X_n, Y_1, \dots, Y_n$, we obtain that $S(\mu_n, \nu_n)$ satisfies~\eqref{eq:bounded_differences} with 
 $$
 c =  \frac{\eps e^{2/\eps}+2}{n} \le \frac{2 \eps e^{2/\eps}}{n}\,,
 $$
 since $\eps e^{2/\eps} > 2$ for all $\eps > 0$.
	
Applying the Efron--Stein inequality, we obtain 
 $$
 \var(S(\mu_n, \nu_n)) \le 2c^2 n\le \frac{8\eps^2 e^{4/\eps}}{n}\,,
 $$
 as claimed.
\end{proof}

\medskip

We now conclude with an analogous sample complexity result for the entropic map $b^\star$.

Recall from~\eqref{eq:gamma_star_def} that the density of $\gamma^\star$ with respect to $\mu \otimes \nu$ is
$$
p^\star = e^{(h^\star - c)/\eps}\,.
$$
Similarly, let $\hat p = e^{(\hat h - c)/\eps}$ denote the density of $\hat \gamma$ with respect to $\mu_n \otimes \nu_n$. Note that thanks to canonical extensions, these two functions may be defined on the whole space.

In the sequel, we use the fact that these densities are uniformly bounded. Indeed, from Proposition~\ref{prop:duals_bounded} and Assumption~\ref{assume:bounded},
\begin{equation}\label{eq:bdd_densities}
    \|\hat p\|_{L^\infty}, \|p^\star\|_{L^\infty} \le e^{2/\eps}\,.
\end{equation}
With these definitions, we have the identities
\begin{align*}
	b^\star(x) & = \int y\, p^\star(x, y) \, \nu(\ud y)\,,\\
	\hat b(x) & = \int y\, \hat p(x, y) \, \nu_n(\ud y)\,.
\end{align*}

The following bound holds.
\begin{theorem}\label{map_rate}
	Adopt Assumption~\ref{assume:bounded}.
	The empirical entropic map satisfies
	\begin{equation*}
		\E \|b^\star - \hat b\|^2_{L^2(\mu_n)} \lesssim \frac{1}{n}\,,
	\end{equation*}
	where the implicit constant depends exponentially on $1/\eps$.
\end{theorem}
\begin{proof}
	Fix $x \in \Omega$.
	Young's inequality and Jensen's inequality imply
	\begin{align*}
		&\|b^\star(x) - \hat b(x)\|^2 \\
            &\quad  \le 2\,\big\|\int y\, p^\star(x, y) \, (\nu-\nu_n)(\ud y)\big\|^2 + 2\,\big\|\int y\, (p^\star - \hat p)(x, y) \, \nu_n(\ud y)\big\|^2  \\
		&\quad \leq 2\,\big\|\int y\, p^\star(x, y) \, (\nu-\nu_n)(\ud y)\big\|^2 + 2\,\int \|y\|^2\, |(p^\star - \hat p)(x, y)|^2 \, \nu_n(\ud y) \\
		&\quad \leq 2\,\big\|\int y \, p^\star(x, y) \, (\nu-\nu_n)(\ud y)\big\|^2 + 2\,\|p^\star(x, \cdot) - \hat p(x, \cdot)\|_{L^2(\nu_n)}^2\,,
	\end{align*}
	where in the last inequality we use the fact that $\|y\| \leq 1$ on the support of $\nu_n$, by Assumption~\ref{assume:bounded}.
	We therefore obtain
	\begin{align*}
		&\|b^\star - \hat b\|^2_{L^2(\mu_n)} \\
  &\qquad \le \frac 2n \sum_{i=1}^n \big\|\int y\,  p^\star(X_i, y) \, (\nu-\nu_n)(\ud y)\big\|^2 + 2\,\|p^\star - \hat p\|_{L^2(\mu_n \otimes \nu_n)}^2\,.
	\end{align*}
 To control the first term, observe that
 \begin{align*}
   \E\big[&\big\|\int y\,  p^\star(X_i, y) \, (\nu-\nu_n)(\ud y)\big\|^2\bigm\vert X_i\big]\\
   &=\frac1n\, \E\big[\bigl\|Y_1  p^\star(X_i, Y_1)- \E[Y_1  p^\star(X_i, Y_1)] \bigr\|^2\bigm\vert X_i\big] \le \frac{e^{4/\eps}}{n}\,,
 \end{align*}
 where we used Assumption~\ref{assume:bounded} and~\eqref{eq:bdd_densities}.
 
To control the second term, we use the fact that the exponential function $e^{x}$ is $e^M$-Lipschitz on $(-\infty, M]$ for any~$M$. Hence, using Assumption~\ref{assume:bounded} and Proposition~\ref{prop:duals_bounded}, we get also that
	\begin{equation*}
		|p^\star(x, y) - \hat p(x, y)| \leq e^{2/\eps}\, |h^\star(x, y) - \hat h(x, y)| \quad \quad \forall x, y \in \Omega\,.
	\end{equation*} 
	Therefore 
 $$
 \E  \|p^\star - \hat p\|_{L^2(\mu_n \otimes \nu_n)}^2 \le e^{4/\eps}\, \E  \|h^\star - \hat h\|_{L^2(\mu_n \otimes \nu_n)}^2 \le \frac{\eps^2 e^{14/ \eps}}{n}\,,
 $$
 by Lemma~\ref{lem:entropic_potential_estimation}. We have proved that
 $$
\E\|b^\star - \hat b\|^2_{L^2(\mu_n)}  \le \frac{2e^{4/\eps}}{n} + \frac{2\eps^2 e^{14/\eps}}{n} \lesssim \frac 1n\,.
 $$
\end{proof}

The preceding theorem gives a bound on the empirical entropic map in expected $L^2(\mu_n)$ norm.
At the price of a larger constant factor, it is also possible to obtain a bound in $L^2(\mu)$.

\begin{theorem}\label{thm:ent_map_estimation_pop}
	Adopt Assumption~\ref{assume:bounded}.
	The empirical entropic map satisfies
	\begin{equation*}
		\E \|b^\star - \hat b\|^2_{L^2(\mu)} \lesssim \frac{1}{n}\,,
	\end{equation*}
	where the implicit constant depends exponentially on $1/\eps$.
\end{theorem}
\begin{proof}
	As in the proof of Theorem~\ref{map_rate}, we have the pointwise bound
	\begin{align*}
		&\|b^\star(x) - \hat b (x)\|^2 \\
            &\qquad \lesssim \Bigl\|\int y\,  p^\star(x, y) \, (\nu-\nu_n)(\ud y)\Bigr\|^2 + \|h^\star(x, \cdot) - \hat h(x, \cdot)\|_{L^2(\nu_n)}^2\,.
	\end{align*}
	Integrating with respect to $\mu$ and taking expectation, we obtain
	\begin{equation*}
		\E \|b^\star - \hat b\|^2_{L^2(\mu)} \lesssim \frac 1n + \E \|h^\star - \hat h\|_{L^2(\mu \otimes \nu_n)}^2\,.
	\end{equation*}
It is thus sufficient to establish that
\begin{equation}\label{eq:pointwise_to_l2}
     \E\|h^\star - \hat h\|_{L^2(\mu \otimes \nu_n)}^2 \lesssim \frac1n\,.
\end{equation}
To that end, we use the fact that the logarithm and exponential functions are locally Lipschitz, with Lipschitz constant depending on the magnitude of the arguments.
	In particular, we recall the elementary inequalities
	\begin{align*}
	e^{\min\{a, b\}}\,|a-b| \leq	|e^{a} - e^{b}| \leq e^{\max\{a, b\}}\,|a - b|\,,
	\end{align*}
	which also imply the bound $|\log a - \log b| \leq \frac{|a - b|}{\min\{a, b\}}$ for $a,b > 0$.
 
Recall that 
$$
\hat h(x,y) =\hat f(x)  + \hat g(y)= - \eps \log \int e^{(\hat g(y) - \|x - y\|^2)/\eps} \, \nu_n(\ud y)+ \hat g(y)\,,
$$
and 
\begin{align*}
    h^\star(x,y) &=f^\star(x)  + g^\star(y)= - \eps \log \int e^{(g^\star(y) - \|x - y\|^2)/\eps} \, \nu(\ud y)+ g^\star(y)\\
    &=- \eps \log \int e^{(g^\star(y) - \|x - y\|^2)/\eps} \, \nu_n(\ud y)+g^\star(y) + \Delta(x)\,,
\end{align*}
where
\begin{equation*}
    \Delta(x) = \eps \log \int e^{(g^\star(y) - \|x - y\|^2)/\eps} \, \nu_n(\ud y) - \eps \log \int e^{(g^\star(y) - \|x - y\|^2)/\eps} \, \nu(\ud y)\,.
\end{equation*}
Moreover, we may assume that $\int \hat g\,\ud \nu_n=\int g^\star\, \ud \nu_n$ without loss of generality since dual solutions are defined up to an additive constant. Using the Lipschitz properties listed above together with Assumption~\ref{assume:bounded} and Proposition~\ref{prop:duals_bounded}, we get
\begin{align*}
  |\hat h(x,y)-h^\star(x,y)| \lesssim \int &|\hat g -g^\star|\, \ud \nu_n  + |\hat g(y) - g^\star(y)|\\
  &+  \Bigl|\int e^{(g^\star(y) - \|x - y\|^2)/\eps} \, (\nu_n - \nu)(\ud y)\Bigr|\,.  
\end{align*}
Using Jensen's inequality and a trivial variance bound for the average of independent and bounded random variables, we finally obtain
$$
\E\|\hat h-h^\star\|_{L^2(\mu \otimes \nu_n)}^2  \lesssim \E \|\hat g - g^\star\|_{L^2(\nu_n)}^2 +\frac1n\,.
$$
Finally, note that
\begin{align*}
  \|\hat h-h^\star\|_{L^2(\mu_n \otimes \nu_n)}^2 &= \int \bigl[(\hat f - f^\star) \oplus (\hat g - g^\star)\bigr]^2 \, \ud (\mu_n\otimes \nu_n) \\
  &=\|\hat f - f^\star\|_{L^2(\mu_n)}^2 +  \|\hat g - g^\star\|_{L^2(\nu_n)}^2 \\
  &\qquad{} +2\int (\hat f -f^\star)\, \ud \mu_n \int (\hat g -g^\star)\, \ud \nu_n\\
  &\ge  \|\hat g - g^\star\|_{L^2(\nu_n)}^2 \,,
\end{align*}
where in the last inequality we used the fact that $\int \hat g\,\ud \nu_n=\int g^\star\, \ud \nu_n$.

Hence we have proved that 
$$
\E\|\hat h-h^\star\|_{L^2(\mu \otimes \nu_n)}^2 \lesssim \E\|\hat h-h^\star\|_{L^2(\mu_n \otimes \nu_n)}^2 + \frac1n\,.
$$
Together with Lemma~\ref{lem:entropic_potential_estimation}, it completes the proof of~\eqref{eq:pointwise_to_l2}, and hence of the theorem.
\end{proof}

The conclusion of this section is quite striking: non-parameteric quantities can be estimated at a parametric rate. An inspection of the proofs of these results indicates that strong convexity is key to achieve such a result. In retrospect it is not surprising that the empirical risk minimizer of a strongly convex functional should enjoy such dimension-free rates. Indeed, stochastic gradient descent on such an objective does (see, e.g.,~\cite[Theorem~4]{KarNutSch16}). This phenomenon is not new: it is known for specific losses such as the ones employed in Chapter~\ref{chap:barycenters} and was observed in~\cite{EHanLi18} for example.

\section{Discussion}

\noindent\textbf{\S\ref{sec:eot_deriv}.}
Entropic optimal transport was first popularized for computational purposes in~\cite{Cut13}. See~\cite{PeyCut19} for an introduction to computational optimal transport which nicely complements our treatment of statistical optimal transport.
Entropic optimal transport between Gaussians (Exercise~\ref{exe:eot_gaussians}) was computed in~\cite{Janetal20EOTGaussians, MalGerMin22EOTGaussians}.

Besides statistical applications, entropic optimal transport has also been used to establish mathematical results for unregularized optimal transport~\cite{FatGozPro20, Gen+20EntropicHWI, ChePoo23Caffarelli}.

\noindent\textbf{\S\ref{sec:eot_duality}.}
The convergence of Sinkhorn's algorithm is discussed in many places, see~\cite{Kal+08RAS, AltWeeRig17, DvuGasKro18, DeB+21DiffSB, Leg21Sinkhorn, AubKorLeg22MD, GhoNut22Sinkhorn, BalBer23Sinkhorn, Gre+23Sinkhorn, ChiDelVas24Sinkhorn}.

\noindent\textbf{\S\ref{sec:eot_sample}.}
The proofs in this section are based on~\cite{RigStr22EOT}.
Note that the bounds obtained here are \emph{dimension-free}, but scale exponentially w.r.t.\ $1/\eps$.
The sample complexity of entropic optimal transport was first established by~\cite{genevay2019sample}, who proved bounds that scaled as $e^{O(1/\eps)} \eps^{-O(d)}$.
Later, \cite{MenNil19EOT} observed that it was possible to slightly modify their argument to remove the exponential factor.
These bounds can be better in low dimension, but provide poor control when the dimension is large.

Recent works have also focused on obtaining bounds which instead scale as $\eps^{-O(d^\star)}$, where the parameter $d^\star$ denotes the intrinsic dimensionality of the measures (in fact, the minimum intrinsic dimension among $\mu$ and $\nu$). See Exercise~\ref{qu:eot_intrinsic} and~\cite{Str23MID} for an approach close to the one taken here, and~\cite{GroHun23AdaptEOT} for an empirical process argument.

The techniques used in this section can be used not only to prove sample complexity bounds, but also to obtain distributional limits~\cite{MenNil19EOT, delSanLou+23CLTEntropic, GonLouNil24EOT}.
Such bounds were originally obtained in the discrete case by~\cite{BigCazPap19,KlaTamMun20}.

One notable quirk about entropic optimal transport is that in general, $S(\mu,\mu) > 0$ due to the presence of the entropic term in the objective.
In light of this, Genevay et al.~\cite{GenPeyCut18} proposed the ``debiased'' quantity $D(\mu, \nu)  \deq  S(\mu,\nu) - \frac{1}{2}\,(S(\mu,\mu) + S(\nu,\nu))$, called the \emph{Sinkhorn divergence}\index{Sinkhorn divergence} between $\mu$ and $\nu$.
It can be shown that the Sinkhorn divergence is convex in each of its variables, non-negative, and vanishes if and only if $\mu = \nu$~\cite{feydy2019interpolating}.
Like entropic optimal transport, the Sinkhorn divergence can be estimated at a parametric rate~\cite{genevay2019sample,delSanLou+23CLTEntropic}.
The Sinkhorn divergence has been advocated as tool for estimating Wasserstein distances~\cite{ChiRouLeg20}, although there are caveats when using it for map estimation~\cite{PooCutNil22Debiaser}.

\noindent\textbf{\S\ref{sec:eot_rates_primal}.}
Theorem~\ref{thm:ent_map_estimation_pop} provides a rate for estimating the entropic map $b^\star$, but combined with an approximation result quantifying the distance between $b^\star$ and the true optimal transport map $T$, one can use $\hat b$ as a computationally efficient estimator for $T$ (c.f.~\cite{PooNil22EntropicMaps}).
Although it is not minimax in general, it is in the semi-discrete case~\cite{PooDivNil23SemiDiscrete}\index{semi-discrete optimal transport}.

As the alternative nomenclature ``entropic regression function'' indicates, the entropic map also solves a regression problem with respect to the entropic coupling $\gamma^\star$; indeed,  $b^\star = \argmin_{f: \R^d \to \R^d} \E_{\gamma^\star} \|Y - f(X)\|^2$.
Analyzing this regression problem when the minimization is taken over a smaller class of candidate regression functions rather than all maps from $\R^d \to \R^d$ is an open problem.

\section{Exercises}

\begin{enumerate}
	\item\label{exe:max_entropy}
	\begin{enumerate}
		\item Let $\Omega$ be a compact subset of $\R^d$ with positive Lebesgue measure.
		Show that the uniform measure on $\Omega$ has the largest differential entropy of any probability measure on $\Omega$.
		\item Let $m$ be a positive integer. Show that the uniform measure on $[m]$ has the largest Shannon entropy of any probability measure on $[m]$.
	\end{enumerate}
	
	\item \label{exe:kl_to_ent} 
 \begin{enumerate}
     \item Show that if $\mu$, $\nu$, and $\gamma$ are absolutely continuous, and $\gamma \in \Gamma_{\mu, \nu}$, then $\KL(\gamma \mmid \mu \otimes \nu) = \operatorname{Ent}(\mu) + \operatorname{Ent}(\nu) - \operatorname{Ent}(\gamma)$.
     Conclude that if $\mu$ and $\nu$ are absolutely continuous, then the optimization problem~\eqref{eq:eot_cont} is equivalent to~\eqref{eq:eot_def}.
     \item Show by an analogous calculation that if $\mu$ and $\nu$ are discrete, then~\eqref{eq:eot_disc} is equivalent to~\eqref{eq:eot_def}.
 \end{enumerate}
	
    \item \label{exe:eot_smoothness} Let $\gamma_\eps$ denote the entropic optimal transport plan between $\mu$ and $\nu$, with corresponding potentials $f_\eps$, $g_\eps$.
    Define $\varphi_\eps \deq \frac{1}{2}\,\|\cdot\|^2 - f_\eps$.
    Compute the derivatives of $\varphi_\eps$ and conclude that
    \begin{align}
        \nabla \varphi_\eps(x)
        &= \E_{\gamma_\eps}[Y\mid X=x]\,, \\
        \nabla^2 \varphi_\eps(x)
        &= \eps^{-1} \cov_{\gamma_\eps}(Y \mid X=x)\,.\label{eq:entropic_hessian}
    \end{align}
    In particular, since we expect that $\varphi_\eps$ converges to the unregularized Brenier potential $\varphi$ as $\eps\to0$ (proven rigorously in~\cite{NutWie22EOTPot}), and $\varphi_\eps$ is convex by~\eqref{eq:entropic_hessian}, this gives another explanation for Brenier's Theorem~\ref{thm:improvedBrenier}.
    
    \item\label{exe:eot_gaussians} Compute the entropic optimal transport solution (i.e., the potentials, the plan, the cost) between two Gaussians. \emph{Hint}: as you might expect, the entropic potentials are quadratic functions.

    \item In this exercise, we present another view on the Sinkhorn iterations~\eqref{eq:sinkhorn1},~\eqref{eq:sinkhorn2}. Consider the joint distributions
    \begin{align*}
        \gamma_{t-\frac{1}{2}}(\ud x,\ud y)
        &\propto \exp\bigl( (f_{t-1}(x)+ g_t(y) - \|x-y\|^2)/\eps\bigr)\,\mu(\ud x)\,\nu(\ud y)\,, \\
        \gamma_t(\ud x, \ud y)
        &\propto \exp\bigl( (f_t(x)+ g_t(y) - \|x-y\|^2)/\eps\bigr)\,\mu(\ud x)\,\nu(\ud y)\,.
    \end{align*}
    Here, we take $\gamma_0(\ud x, \ud y) \propto \exp(-\|x-y\|^2/\eps)\,\mu(\ud x)\,\nu(\ud y)$.
    Show that the Sinkhorn updates correspond to iteratively ``fixing the marginals''; i.e., $\gamma_{t-\frac{1}{2}}$ is obtained from $\gamma_{t-1}$ by keeping the conditional distribution of $X\mid Y$ fixed but setting the $Y$-marginal to $\nu$, and $\gamma_t$ is obtained from $\gamma_{t-\frac{1}{2}}$ by keeping the conditional distribution of $Y\mid X$ fixed but setting the $X$-marginal to $\mu$.

    \item In the discrete setting where $\mu$, $\nu$ are finitely on $\{x_1,\dotsc,x_m\}$ and $\{y_1,\dotsc,y_n\}$ respectively, Sinkhorn's algorithm shows that there are positive \emph{scalings} $\kappa \in \R_+^m$, $\lambda \in \R_+^n$ of the rows and columns of the matrix $M$ with entries $M_{i,j} = \exp(-\|x_i-y_j\|^2/\eps)$, such that the scaled matrix $\diag(\kappa)\,M\diag(\lambda)$ has marginals $\mu$ and $\nu$ respectively; see~\cite{PeyCut19} for details.

    As a special case, suppose that $\mu$, $\nu$ are \emph{uniformly} distributed and $m=n$.
    Then, the scaled matrix $\tilde M \deq \diag(\kappa)\,M\diag(\lambda)$ has marginals $\mu$ and $\nu$ if and only if $n\tilde M$ is doubly stochastic, i.e., it belongs to the Birkhoff polytope~\eqref{eq:birkhoff}.
    In this case, prove the existence of these scalings for any matrix $M$ with positive entries by considering the KL minimization problem
    \begin{align*}
        \operatorname*{minimize}_{\gamma \in \msf{Birk}}\quad \sum_{i,j=1}^n \bigl( \gamma_{i,j} \log \frac{\gamma_{i,j}}{M_{i,j}} - \gamma_{i,j} + M_{i,j}\bigr)\,.
    \end{align*}
    Namely, show that a solution to this problem exists, and using Lagrange multipliers, show that $\gamma$ is of the form $\diag(\kappa)\,M\diag(\lambda)$ for positive scalings $\kappa$, $\lambda$.
    
	\item\label{exe:eot_empirical_process} The strong convexity arguments in Section~\ref{sec:eot_sample} are designed to avoid the use of empirical process theory.
	However, this exercise shows how to use empirical process theory to prove sample complexity bounds using techniques analogous to those in Section~\ref{sec:chaining}.
	Unlike the approach in Section~\ref{sec:eot_sample}, these bounds depend polynomially on $1/\eps$, but with exponent scaling with $d$. 
	For simplicity, we focus on the \emph{one-sample problem}, and prove bounds on the quantity $S(\mu_n, \nu) - S(\mu, \nu)$.
	\begin{enumerate}
		\item Let $f$ and $g$ be solutions to~\eqref{eta-D} for any pair of measures supported on $\Omega = B_{1/2}(0)$.
		Let $s$ be a positive integer.
		Arguing as in Exercise~\ref{exe:eot_smoothness}, show that there exists a positive constant $C_s$ such that for all multi-indices $\alpha = (\alpha_1, \dots, \alpha_d)$ with $|\alpha| = s$, we have the bound
		\begin{equation*}
			\sup_{x \in \Omega} |\partial_\alpha f(x)| \leq C_s \eps^{1-s}\,.
		\end{equation*}
		Argue that we can assume that $f(0) = 0$ without loss of generality, and thereby obtain the bound $\sup_{x \in \Omega} |f(x)| \leq C_0$ for some positive constant $C_0$. 
		\item For $L > 0$, define
		\begin{equation*}
			\cF_s(L)  \deq  \Big\{f: \Omega \to \R^d: \sum_{k=0}^s \sum_{\alpha: |\alpha| = k} \|\partial_\alpha f\|_{L^\infty(\Omega)} \leq L\Big\}\,.
		\end{equation*}
		Fix a positive integer $s$. Argue that there exists a constant $C = C(s)$ such that for $L = C\,(1+\eps^{1-s})$, we have
		\begin{equation*}
			|S(\mu_n, \nu) - S(\mu, \nu)| \leq \sup_{f \in \cF_s(L)} \left|\int f \, \ud \mu_n - \int f \, \ud \mu\right|\,.
		\end{equation*}
		\emph{Hint:} let $f^\star$ and $g^\star$ be solutions to~\eqref{eta-D} for $\mu$ and $\nu$, and let $\hat f$ and $\hat g$ be the solutions for $\mu_n$ and $\nu$.
		Argue that $h^\star = f^\star \oplus g^\star$ and $\hat h = \hat f \oplus \hat g$ satisfy
		\begin{equation*}
			S(\mu_n, \nu) - S(\mu, \nu) = \widehat \Phi(\hat h) - \Phi(h^\star) \leq \int \hat f \, \ud \mu_n - \int \hat f \, \ud \mu\,,
		\end{equation*}
		and analogously
		\begin{equation*}
			S(\mu, \nu) - S(\mu_n, \nu) \leq \int f^\star \, \ud \mu - \int f^\star \, \ud \mu_n\,.
		\end{equation*}
		Then apply part~(a).
		\item It can be shown that $\log N(\delta, \cF_s(L)) \lesssim (L/\delta)^{d/s}$, and moreover that this bound holds also for fractional $s$, where $\cF_s$ is interpreted as a suitable H\"older space.
		Taking $s = d/2 + 1$, use Proposition~\ref{prop:chaining} to conclude
		\begin{equation*}
			\E |S(\mu_n, \nu) - S(\mu, \nu)| \lesssim (1+\eps^{-d/2})\, n^{-1/2}\,.
		\end{equation*}
	\end{enumerate}
	
    \item\label{qu:eot_intrinsic} The statistical results in Section~\ref{sec:eot_sample} rely on pointwise bounds on the density $p^\star$.
    Here, we show a different way to control $\|p^\star\|_{L^2(\mu\otimes \nu)}$, which provides an entry point into~\cite{Str23MID}.
    \begin{enumerate}
        \item Argue that the dual potentials $f^\star$, $g^\star$ are $O(1)$-Lipschitz, and that $\log p^\star$ is $O(\eps^{-1})$-Lipschitz. (Here, we are still working over a bounded domain.)
        \item Prove that for all $\delta > 0$, $\int {\nu(B(z, \delta))}^{-1}\,\nu(\ud z) \le N(\delta/4, \supp \nu)$, where $N(\delta/4, \supp\nu)$ is the covering number of $\supp\nu$ at scale $\delta/4$. (Let $z_1,\dotsc,z_K \in \supp\nu$ be a $\delta/2$-covering of $\supp\nu$ with $K \le N(\delta/4, \supp\nu)$. Bound the integral by summing over the integals over $B(z_k, \delta/2)$ for $k=1,\dotsc,K$.)
        \item Using the fact that
        \begin{align*}
            1 = p^\star(x,y) \int \frac{p^\star(x,y')}{p^\star(x,y)}\,\nu(\ud y')
            \ge p^\star(x,y) \int_{B(y, r)} \frac{p^\star(x,y')}{p^\star(x,y)}\,\nu(\ud y')
        \end{align*}
        and the log-Lipschitz property from (a), show that $p^\star(x,y) \lesssim {\nu(B(y, r))}^{-1}$, where $r \asymp \eps$.
        \item Combining this with the estimate in (b), prove that $\|p^\star\|_{L^2(\mu\otimes \nu)} \lesssim \sqrt{N(r', \supp \nu)}$ where $r' \asymp \eps$.
        Explain why this implies that if $\supp\mu$ is $d_\mu$-dimensional and $\supp \nu$ is $d_\nu$-dimensional, then $\|p^\star\|_{L^2(\mu\otimes \nu)} \lesssim \eps^{-(d_\mu \wedge d_\nu)/2}$.
    \end{enumerate}
\end{enumerate}

\chapter{Wasserstein gradient flows: theory}
\label{chap:WGF}

We have seen in Proposition~\ref{prop:wp_is_metric} that $\cP_2(\R^d)$, once endowed with the $W_2$ distance, has the structure of a metric space. It has in fact a much richer geometric structure, as it resembles a Riemannian manifold.
Consequently, we can bring to bear the calculation rules of Riemannian geometry, known in this context as \emph{Otto calculus}, on the design and interpretation of algorithms over the space of probability measures.

The identification of $\cP_2(\R^d)$ with a Riemannian manifold is purely ``formal'' (that is, heuristic).
For instance, $\cP_2(\R^d)$ is not locally homeomorphic to a Hilbert space.
However, the Riemannian view of $\cP_2(\R^d)$ is nevertheless a powerful tool for understanding the geometric properties of the Wasserstein space.

To elucidate this Riemannian viewpoint, we work with absolutely continuous measures in this chapter, for which the Riemannian formalism can be put on a more rigorous footing.
Our main goal in constructing this formalism is to define interesting dynamics on the Wasserstein space, given by \emph{gradient flows}.
Having defined these dynamics, we shall see that they often make sense even for non-absolutely continuous measures---in particular, they give rise to well-defined dynamics for discrete measures, viewed as particle systems.
Once derived, these non-trivial dynamics can be studied directly for discrete measures without the need for making rigorous sense of the Riemannian calculations in the discrete case.
The reader interested in seeing a fully rigorous derivation of gradient flows for general measures should consult the seminal monograph of Ambrosio, Gigli, and Savar\'e~\cite{AmbGigSav08}, or~\cite{San17GradFlows} for a quick overview.

\section{Metric derivative and the continuity equation}\label{sec:continuity_eq}

The utility of optimal transport lies in its endowment of the space of probability measures with a geometric structure which respects that of the underlying space. For example, we saw that the mapping $x\mapsto \delta_x$ is an isometric embedding of $(\R^d,\|\cdot\|)$ into $(\cP_2(\R^d), W_2)$.
Accordingly, as we now seek to understand dynamics on $\cP_2(\R^d)$, our approach is to ``lift'' the corresponding dynamics of particles on $\R^d$.

The general way to prescribe dynamics on $\R^d$ using differential calculus is via ordinary differential equations (ODEs). Namely, given a time-dependent family of vector fields ${(v_t)}_{t\ge 0}$, consider the ODE
\begin{align}\label{eq:particle_ode}
    \dot X_t
    &= v_t(X_t)\,.
\end{align}
Under standard assumptions on ${(v_t)}_{t\ge 0}$,\footnote{For example, if the vector fields are Lipschitz uniformly in time, then well-posedness follows from the Cauchy--Lipschitz theorem.} there is a unique solution to the ODE for any given initial condition $X_0$.
Suppose now that $X_0$ is drawn \emph{randomly} from a measure $\mu_0 \in \cP_2(\R^d)$, and similarly let $\mu_t$ denote the law of $X_t$ for all $t\ge 0$.
We think of the curve of measures ${(\mu_t)}_{t\ge 0}$ as describing the evolution of a \emph{collection} of particles.
Then, the dynamics of ${(\mu_t)}_{t\ge 0}$ is described by a partial differential equation (PDE), known as the \emph{continuity equation}.

\begin{proposition}[Continuity equation]\index{continuity equation}
    Suppose that $X_0 \sim \mu_0$, and that ${(X_t)}_{t\ge 0}$ evolves according to the dynamics~\eqref{eq:particle_ode}, which we assume is well-posed.
    Let $\mu_t$ denote the law of $X_t$ for all $t\ge 0$.
    Then, ${(\mu_t)}_{t\ge 0}$ satisfies the following equation in the weak sense,
    \begin{align}\label{eq:continuity}
        \partial_t \mu_t + \divergence(\mu_t v_t) = 0\,,
    \end{align}
    i.e., for all compactly supported and smooth test functions $\varphi : \R^d\to\R$, it holds that
    \begin{align}\label{eq:continuity_weak}
        \partial_t \int \varphi \, \ud \mu_t = \int \langle \nabla \varphi, v_t\rangle \, \ud \mu_t\,.
    \end{align}
\end{proposition}

The equation~\eqref{eq:continuity_weak}, when written in probabilistic language, reads $\partial_t \E \varphi(X_t) = \E\langle \nabla \varphi(X_t), v_t(X_t)\rangle$, and it simply follows from~\eqref{eq:particle_ode} and the chain rule.
The real content of the proposition actually lies in~\eqref{eq:continuity}: when $\mu_t$ admits a smooth density w.r.t.\ Lebesgue measure, which by an abuse of notation we denote also by $\mu_t$, then integration by parts yields the equation
\begin{align*}
    \int \varphi \, \partial_t \mu_t
    &= \partial_t \int \varphi \, \ud \mu_t
    = \int \langle \nabla \varphi, v_t\rangle \, \mu_t
    = -\int\varphi \divergence(\mu_t v_t)\,.
\end{align*}
Since this equality is supposed to hold for all suitable test functions $\varphi$, it follows that~\eqref{eq:continuity} holds pointwise.
To summarize, we see that ${(\mu_t)}_{t\ge 0}$ solves the PDE~\eqref{eq:continuity}, at least when $\mu_t$ admits a smooth density for all $t\ge 0$.
In general, it is more convenient to make statements that hold for curves ${(\mu_t)}_{t\ge 0}$ without knowing in advance the regularity of $\mu_t$, in which case~\eqref{eq:continuity} should be interpreted to hold in the weak sense~\eqref{eq:continuity_weak}.
However, for the sake of developing the framework of Otto calculus unencumbered by technical distractions, from now on we ignore such issues of regularity and pretend that we are working with  curves of smooth densities. See~\cite{AmbGigSav08} for a more rigorous treatment.

The equations~\eqref{eq:particle_ode} and~\eqref{eq:continuity} provide us with dual perspectives on the same dynamics; in the field of fluid dynamics, these perspectives are known as \emph{Lagrangian} and \emph{Eulerian} respectively.
The Lagrangian perspective describes the evolution of individual particle trajectories, whereas the Eulerian perspective tracks the evolution of aggregate quantities through the notions of mass density $\mu_t$ and velocity field $v_t$.

To foreshadow the development of geometry over $\cP_2(\R^d)$, let us first examine how to develop geometry over $\R^d$; for now, we refer to concepts from Riemannian geometry loosely, but we return to the subject in Section~\ref{sec:riem}.
For a single particle trajectory $t\mapsto X_t$ evolving according to~\eqref{eq:particle_ode}, the kinetic energy at time $t$ (assuming the particle has unit mass) is $\|\dot X_t\|^2 = \|v_t(X_t)\|^2$.
The total energy of the curve over the time interval $[0,1]$ is $\int_0^1 \|v_t(X_t)\|^2\,\ud t$, and if we minimize this energy over all curves ${(X_t)}_{t\in [0,1]}$ evolving according to~\eqref{eq:particle_ode} with endpoints fixed at $X_0$ and $X_1$, we obtain the constant-speed geodesic $t\mapsto X_t \deq (1-t)\,X_0 + t\,X_1$.
Geometrically, we think of $v_t(X_t)$ as the \emph{tangent vector} to the curve ${(X_t)}_{t\in [0,1]}$ at time $t$, and we measure its length using the Euclidean norm $\|\cdot\|$.

We now try to lift this picture to $\cP_2(\R^d)$.
For a curve ${(\mu_t)}_{t\ge 0}$ evolving according to~\eqref{eq:continuity}, it is natural to think of the velocity vector field $v_t : \R^d\to\R^d$ as an abstract ``tangent vector'' to ${(\mu_t)}_{t\ge 0}$ at time $t$, and to measure its squared ``length'' via the kinetic energy\footnote{Since $\mu_t$ plays the role of a mass density, then $\|v_t\|^2 \,\mu_t$ is the kinetic energy density, and integrating this over all of space yields the kinetic energy.}
\begin{align}\label{eq:kinetic_energy}
    \|v_t\|_{\mu_t}^2 \deq \int \|v_t\|^2 \, \ud \mu_t\,.
\end{align}
However, we immediately arrive at an obstacle in doing so: given a curve ${(\mu_t)}_{t\ge 0}$, there is not a \emph{unique} choice of vector fields ${(v_t)}_{t\ge 0}$ for which~\eqref{eq:continuity} holds, and hence it is unclear what vector field $v_t$ to choose as our tangent vector.
Indeed, if we start with any family of vector fields ${(v_t)}_{t\ge 0}$ for which~\eqref{eq:continuity} holds, and if $w_t$ satisfies $\divergence(\mu_t w_t) = 0$ for all $t\ge 0$, then ${(v_t + w_t)}_{t\ge 0}$ is another family of vector fields for which~\eqref{eq:continuity} holds by linearity of the divergence operator.
To see a concrete example of non-uniqueness, suppose that $\mu_t$ is the standard Gaussian distribution on $\R^2$ for all $t\ge 0$.
Then, one natural choice of vector fields is to take $v_t = 0$ for all $t\ge 0$; however, due to the rotational invariance of the standard Gaussian, another choice is to choose vector fields inducing a rotation (see Figure~\ref{fig:gaussian_vector_fields}).

\begin{figure}[ht]
    \centering
\begin{tikzpicture}[scale=0.6]

\begin{scope}
    \foreach \x in {-3,...,3} {
        \foreach \y in {-3,...,3} {
            \fill (\x,\y) circle (1pt);
        }
    }
    \draw (0,-4.5) node {Zero vector field};
\end{scope}

\begin{scope}[xshift=10cm]
    \foreach \x in {-3,...,3} {
        \foreach \y in {-3,...,3} {
            \ifnum\y=0
                \ifnum\x=0 
                    \fill (\x,\y) circle (1pt);
                \else
                    \pgfmathsetmacro{\r}{sqrt(\x*\x + \y*\y)/5.0}
                    \draw[-{Stealth[length=3pt]}] (\x,\y) -- ++({atan2(\y,\x)+90}:\r);
                \fi
            \else
                \pgfmathsetmacro{\r}{sqrt(\x*\x + \y*\y)/5.0}
                \draw[-{Stealth[length=3pt]}] (\x,\y) -- ++({atan2(\y,\x)+90}:\r);
            \fi
        }
    }
    \draw (0,-4.5) node {Rotation vector field};
\end{scope}

\end{tikzpicture}
    \caption{Two vector fields which preserve the standard Gaussian on $\R^2$.}\label{fig:gaussian_vector_fields}
\end{figure}
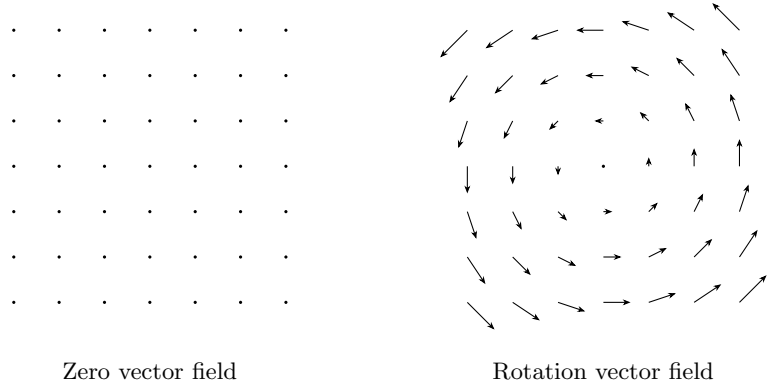

We see that the vector field on the right of Figure~\ref{fig:gaussian_vector_fields} induces extraneous motion for the particles and is therefore not the most parsimonious explanation for the dynamics of ${(\mu_t)}_{t\ge 0}$.
To resolve the ambiguity in the choice of vector fields, we can elect to declare as our tangent vector the vector field which minimizes the kinetic energy~\eqref{eq:kinetic_energy} while still explaining the dynamics of ${(\mu_t)}_{t\ge 0}$.
As discussed below, the minimization of kinetic energy falls naturally in line with the philosophy of ``optimal'' transport of mass.

To further motivate this choice, we introduce the notion of the metric derivative of a curve ${(x_t)}_{t\ge 0}$ in a metric space $(S, \msf d)$.
Although in general we cannot make sense of the notion of a tangent vector (or the ``derivative'') of a curve in a general metric space, it turns out that we can make sense of its \emph{speed}.

\begin{definition}[Metric derivative]\index{metric derivative}
    Let $(S,\msf d)$ be a metric space and let ${(x_t)}_{t\ge 0}$ be a curve in $S$.
    Then, the \emph{metric derivative} of the curve at time $t$ is given by
    \begin{align*}
        |\dot x|(t)
        &\deq \lim_{s\to t, \; s\ne t} \frac{\msf d(x_s, x_t)}{|s-t|}\,,
    \end{align*}
    provided that the limit exists.
\end{definition}

The next theorem shows that our selection principle produces a well-defined choice of tangent vector $v_t$ which can be characterized in any one of three ways: (1) as the vector field $v_t$ with minimal kinetic energy subject to the constraint~\eqref{eq:continuity}; (2) as the unique choice of vector field solving~\eqref{eq:continuity} with length $\|v_t\|_{\mu_t}$ equal to the metric derivative $|\dot \mu|(t)$; (3) as a limit of Brenier maps.

However, let us first introduce the concept of the \emph{flow map} associated with the ODE~\eqref{eq:particle_ode} (or equivalently, with the family of vector fields ${(v_t)}_{t\ge 0}$).
The map $F_{0,t} : \R^d\to\R^d$ is defined as the map which, given $X_0\in\R^d$, outputs the solution $X_t$ to the ODE~\eqref{eq:particle_ode} at time $t$ when started at $X_0$.
In an analogous manner, we can define the flow map $F_{s,t} :\R^d\to\R^d$ for any pair of times $0 \le s \le t$.
The significance of this definition is that it shifts our attention away from thinking about the ODE as describing a single trajectory, and instead views the effect of the ODE as a deformation of the entire space $\R^d$.

Given a pair of probability measures $\mu, \nu \in \cP_2(\R^d)$ such that $\mu$ admits a density, we write $T_{\mu \to \nu}$ for the Brenier map from $\mu$ to $\nu$, and denote by $|\dot \mu|$ the metric derivative of a curve in $\cP_2(\R^d)$ with respect to the metric $\msf d = W_2$.

\begin{theorem}\label{thm:fundamental_otto}
    Let ${(\mu_t)}_{t\ge 0}$ be a regular\footnote{Here, ``regular'' can be taken to mean that $\mu_t \in \cP_2(\R^d)$ and admits a density, and that the metric derivative $|\dot\mu|(t)$ exists for all $t\ge 0$. The qualifier ``for all $t\ge 0$'' in the assumptions and conclusions can be replaced by ``for almost every $t\ge 0$'', and the assumed existence of a density can also be relaxed.} curve of probability measures.
    Then, for every family of vector fields ${(v_t)}_{t\ge 0}$ for which~\eqref{eq:continuity} holds, we have $|\dot \mu|(t) \le \|v_t\|_{\mu_t}$ for all $t\ge 0$.

    Conversely, there exists a unique family ${(v_t)}_{t\ge 0}$ 
 such that~\eqref{eq:continuity} holds and for which $|\dot\mu|(t) = \|v_t\|_{\mu_t}$ for every $t\ge 0$. This family is such that
    \begin{align}\label{eq:vector_field_as_limit}
        v_t = \lim_{h\searrow 0} \frac{T_{\mu_t\to\mu_{t+h}} - {\id}}{h} \qquad \text{in}~L^2(\mu_t)\,,
    \end{align}
    where $\id$ is the identity function.
\end{theorem}
\begin{proof}[Proof sketch]
    We start with the first statement.
    Let ${(X_t)}_{t\ge 0}$ be the curve of random variables with $X_0 \sim \mu_0$ solving $\dot X_t = v_t(X_t)$. Since the continuity equation~\eqref{eq:continuity} holds by assumption, then $X_t \sim \mu_t$ for all $t\ge 0$, and in particular, $\mu_{t+h} = {(F_{t,t+h})}_\# \mu_t$.
    We can upper bound $W_2(\mu_t, \mu_{t+h})$ using this suboptimal coupling:
    \begin{align*}
        \frac{W_2^2(\mu_t, \mu_{t+h})}{h^2}
        &\le \E\Bigl[\frac{\|F_{t,t+h}(X_t) - X_t\|^2}{h^2} \Bigr]\,.
    \end{align*}
Observe that $F_{t,t+h}(X_t) = X_t + hv_t(X_t) + o(h)$ so that
letting $h\to 0$, we obtain $|\dot\mu|(t) \le \sqrt{\E[\|v_t(X_t)\|^2]} = \|v_t\|_{\mu_t}$.

    Note that the inequality above arises from the use of a suboptimal coupling. Intuitively, if we take $X_t$ and $X_{t+h}$ to be optimally coupled and thereby define $v_t$ according to~\eqref{eq:vector_field_as_limit}, then we ought to obtain an equality. This is in fact the case but we omit the proof.

	Finally, by combining the two statements, we obtain that if $(v_t)_{t \geq 0}$ is such that~\eqref{eq:continuity} holds and $|\dot\mu|(t) = \|v_t\|_{\mu_t}$, then
    \begin{align*}
        0  \in \argmin_{w_t : \R^d\to\R^d}{\|v_t + w_t\|_{\mu_t}}\quad\text{s.t.}\quad \divergence(\mu_t w_t) = 0\,.
    \end{align*}
    Moreover, this is a strictly convex problem, so the minimizer is unique.
    We obtain that if $(v_t)_{t \geq 0}$ and $(u_t)_{t \geq 0}$ are two vector fields satisfying the continuity equation and $\|v_t\|_{\mu_t} = \|u_t\|_{\mu_t} = |\dot\mu|(t)$, then $w_t = u_t - v_t = 0$, as claimed.
\end{proof}

In the above theorem, we used the fact that $\mu_t$ admits a density in order to write~\eqref{eq:vector_field_as_limit}, i.e., to assert that the optimal transport map exists.
In order to facilitate the discussion, we restrict to this class of measures from now on, although we return to the subject of particle methods at the end of the chapter.

\begin{definition}
    $\cP_{2,\rm ac}(\R^d)$ is the class of probability measures over $\R^d$ with finite second moment and which are absolutely continuous (i.e., admit a density w.r.t.\ Lebesgue measure).
\end{definition}

\section{Elements of Riemannian geometry}\label{sec:riem}

Before proceeding further, we provide a brief and informal exposition to the concepts from Riemannian geometry that we need.

A \emph{manifold} $\cM$ is a set which is locally homeomorphic to a Euclidean space $\R^d$. At each point $p\in \cM$, we can associate a \emph{tangent space} $T_p\cM$, which is a $d$-dimensional vector space and represents all possible velocity vectors of curves passing through $p$.
A \emph{Riemannian metric} is a choice of an inner product $\mg_p$ on each tangent space $T_p\cM$, and once endowed with a Riemannian metric, the manifold is then called a \emph{Riemannian manifold}.
To emphasize the Hilbertian structure, we usually simply write $\langle \cdot,\cdot\rangle_p$ for the metric $\mg_p$.
Usually, one imposes additional smoothness assumptions for these objects in order to properly build up a theory of differential calculus, but here we focus on introducing the basic language without delving into details.
In the case of the Wasserstein space, note that we have already identified a natural norm for a ``tangent vector'' (velocity vector field) $v_t$ at $\mu_t$---the $L^2$ norm, $\sqrt{\mg_{\mu_t}(v_t,v_t)} = \|v_t\|_{\mu_t}$---indicating the possibility of identifying further Wasserstein analogues of Riemannian theory.

The next important concept is that of a geodesic.
For a curve ${(p_t)}_{t\in [0,1]}$ in $\cM$, let $\dot p_t \in T_{p_t} \cM$ denote the tangent vector at time $t$.
Given $p_0,p_1\in \cM$, \emph{geodesics} or shortest paths\footnote{In Riemannian geometry, it is more customary to define geodesics to only be \emph{locally} length-minimizing, but here we always use the word ``geodesic'' to refer to shortest paths.} between $p_0$ and $p_1$ are obtained by solving either of the following variational problems,
\begin{align*}
    \min_{{(p_t)}_{t\in [0,1]}}\int \|\dot p_t\|_{p_t} \, \ud t \qquad\text{or}\qquad
    \min_{{(p_t)}_{t\in [0,1]}}\int \|\dot p_t\|_{p_t}^2 \, \ud t 
\end{align*}
over curves ${(p_t)}_{t\in [0,1]}$ joining $p_0$ to $p_1$.
In the first problem, the objective functional is the \emph{arc length} of the curve; in the second problem, the objective functional is called the \emph{energy}.
The second variational problem is technically more convenient.
Indeed, the arc length is invariant under reparametrization (i.e., replacing ${(p_t)}_{t\in [0,1]}$ by ${(p_{f(t)})}_{t\in [0,1]}$ for any continuous and strictly increasing function $f : [0,1]\to [0,1]$), so solutions to the first variational problem can only be unique up to reparametrization.
In contrast, the second variational problem singles out a specific parametrization of the optimal curves, namely, curves with \emph{constant speed} (i.e., $t\mapsto \|\dot p_t\|_{p_t}$ is constant).
Henceforth, we only consider constant-speed geodesics and therefore omit the adjective ``constant-speed''.

The values of the variational problems are $\msf d(p_0,p_1)$ and $\msf d^2(p_0,p_1)$ respectively, where $\msf d$ defines a metric (in the sense of metric spaces) on $\cM$ induced by the Riemannian metric $\langle \cdot,\cdot\rangle$.

The \emph{exponential map}\footnote{The name is motivated by a classical example of a manifold, the set of orthogonal matrices, in which the tangent space at the identity matrix is the set of anti-symmetric matrices and the exponential map $\exp_I(A) = \exp(A)$ coincides with the matrix exponential.} is a mapping $\exp_p : T_p\cM \to \cM$ which maps a tangent vector $v$ to $p_1$, where ${(p_t)}_{t\in [0,1]}$ is the constant-speed geodesic such that $p_0 = p$ and $\dot p_0 = v$.
The inverse map is called the \emph{logarithmic map} $\log_p : \cM \to T_p\cM$, which maps $q\mapsto \exp_p^{-1}(q)$.
Actually, in general, the exponential map may not be defined over all of $T_p\cM$ because a geodesic, once extended too far, may no longer remain a shortest path between its endpoints; think, for instance, of extending the geodesic from the north pole to the south pole of the sphere.

With the idea of a geodesic in hand, we can then define the concepts of convexity, gradients, and gradient flows, which form the building blocks of optimization over curved spaces.
We say that a set $C \subseteq \R^d$ is \emph{convex} if for all $p_0,p_1 \in C$ and all $t\in [0,1]$, it holds that $(1-t)\,p_0 + t\,p_1 \in C$.
In this definition, $t\mapsto (1-t)\,p_0 + t\,p_1$ is the Euclidean geodesic joining $p_0$ to $p_1$.
We can generalize this definition to Riemannian manifolds: we say that $C \subseteq \cM$ is \emph{geodesically convex}\index{Geodesic convexity} if for all $p_0,p_1 \in C$, all geodesics joining $p_0$ to $p_1$ also lie in $C$.

We can also define convexity for functions: given $\alpha \in\R$, a function $f : \cM \to\R$ is called \emph{$\alpha$-geodesically convex} if
\begin{align}\label{eq:geod_cvx_def}
    f(p_t) \le (1-t)\,f(p_0)+t\,f(p_1) - \frac{\alpha\,t\,(1-t)}{2}\,\msf d^2(p_0,p_1)
\end{align}
for all $t \in [0,1]$ and all geodesics ${(p_t)}_{t\in [0,1]}$.
Equivalently, we have the first-order condition
\begin{align*}
    f(q)
    &\ge f(p) + \langle \nabla f(p), \log_p(q)\rangle_p + \frac{\alpha}{2}\, \msf d^2(p,q)\qquad\forall p,q\in \cM
\end{align*}
where $\nabla f$, the Riemannian gradient, is defined so that for all curves ${(p_t)}_{t\ge 0}$, $\nabla f(p_t) \in T_{p_t}\cM$ satisfies $\partial_t f(p_t) = \langle \nabla f(p_t), \dot p_t\rangle_{p_t}$.
Equivalently, we also have the second-order condition
\begin{align*}
    \nabla^2 f(p)[v,v] \ge \alpha\,\|v\|_p^2\qquad\forall p\in \cM, \; \forall v \in T_p\cM\,,
\end{align*}
where $\nabla^2 f$, the Riemannian Hessian, can be defined via $\nabla^2 f(p)[v,v] \deq \partial_t^2 f(p_t)|_{t=0}$, where ${(p_t)}_{t\in [0,1]}$ is the geodesic with $p_0 = p$ and $\dot p_0 = v$.

In the next section, we return to $\cP_{2,\rm ac}(\R^d)$ which, despite not being a \emph{bona fide} Riemannian manifold, carries enough structure to apply calculation rules from Riemannian geometry (and indeed, the formidable book~\cite{AmbGigSav08} is devoted to the task of placing this endeavor on rigorous footing).
It leads to a toolbox, known as \emph{Otto calculus} after Felix Otto, for the study of gradient flows over the space of probability measures.

\section{The Riemannian structure of Wasserstein space}\label{sec:riem_wass}

We are now in a position to define a formal Riemannian structure over $\cP_2(\R^d)$.
Recall from Brenier's theorem (Theorem~\ref{thm:improvedBrenier}) that optimal transport maps for the quadratic cost are gradients of convex functions.
From~\eqref{eq:vector_field_as_limit}, it follows that optimal velocity vector fields lie in the $L^2$ closure of the space of gradients of functions (which are not necessarily convex, since we have subtracted the identity map).

\begin{definition}\label{def:w2_tangent}
    Let $\mu \in \cP_{2,\rm ac}(\R^d)$.
    We define the \emph{tangent space} to $\cP_{2,\rm ac}(\R^d)$ at $\mu$ to be
    \begin{align*}
        T_\mu \cP_{2,\rm ac}(\R^d)
        &\deq \overline{\{\nabla \psi \mid \psi : \R^d\to\R~\text{compactly supported, smooth}\}}^{L^2(\mu)}
    \end{align*}
    where $\overline{\{\cdot\}}^{L^2(\mu)}$ denotes the $L^2(\mu)$ closure.
    We endow $T_\mu \cP_{2,\rm ac}(\R^d)$ with the $L^2(\mu)$ inner product,
    \begin{align*}
        \langle \nabla \psi_1,\nabla \psi_2\rangle_\mu
        &\deq \int \langle \nabla \psi_1,\nabla \psi_2\rangle \, \ud \mu\,.
    \end{align*}
\end{definition}

One can show that requiring $v_t$ to be the gradient of a function, $v_t = \nabla \psi_t$, in fact furnishes a fourth characterization of the optimal vector field $v_t$ in Theorem~\ref{thm:fundamental_otto}, thus justifying Definition~\ref{def:w2_tangent}, but we do not prove this here.

\begin{remark}
    We pause to describe a common alternative convention: instead of defining the tangent vector at $\mu_t$ to be the driving velocity vector field $\nabla \psi$, we could take it to be the ordinary time derivative $\partial_t\mu_t$ which is given by the continuity equation:
    $\partial_t \mu_t = -\divergence(\mu_t \nabla \psi) \eqqcolon \chi$. In this case, the tangent space becomes the space of signed measures with zero total mass, and the metric becomes $\langle \chi, \chi'\rangle_\mu = \int\langle \nabla \psi, \nabla \psi' \rangle\,\ud\mu$, where $\psi$, $\psi'$ solve the equations $\chi=-\divergence(\mu \nabla \psi)$, $\chi'=-\divergence(\mu\nabla \psi')$.
    This just amounts to a change of notation: $\nabla \psi \mapsto \chi$ is an isometry between our convention and the alternative convention.
\end{remark}

Although our logical development thus far has strongly hinted at a connection between this Riemannian structure and the theory of optimal transport, we have not yet stated any result to this effect. The following theorem computes the constant-speed geodesics in the metric defined above.
In the language of Section~\ref{sec:riem}, it asserts that the metric induced by the Riemannian structure we defined over $\cP_{2,\rm ac}(\R^d)$ is indeed the Wasserstein distance.

\begin{theorem}[Benamou--Brenier]\label{thm:benamou_brenier}\index{Benamou--Brenier formula}
    Let $\mu_0,\mu_1 \in \cP_{2,\rm ac}(\R^d)$.
    Then,
    \begin{align}
        W_2^2(\mu_0,\mu_1)
        &= \inf\Bigl\{\int_0^1 \|v_t\|_{\mu_t}^2 \, \ud t \Bigm\vert {(\mu_t, v_t)}_{t\in [0,1]}~\text{solves}~\eqref{eq:continuity}\Bigr\}\,.\label{eq:benamou_brenier}
    \end{align}
    The optimal curve ${(\mu_t)}_{t\in [0,1]}$ is unique and is described by $X_t\sim\mu_t$, where $X_t = (1-t)\,X_0 +t\,X_1$ and $(X_0,X_1)\sim\bar\gamma \in \Gamma_{\mu_0,\mu_1}$ with $\bar\gamma$ being an optimal coupling.
\end{theorem}
\begin{proof}
    Let ${(\mu_t,v_t)}_{t\in [0,1]}$ solve~\eqref{eq:continuity}, and let $\dot X_t = v_t(X_t)$ with $X_0 \sim \mu_0$.
    Then, it holds that
    \begin{align*}
        W_2^2(\mu_0,\mu_1)
        &\le \E[\|X_0 - X_1\|^2]
        = \E\Bigl[\Bigl\lVert \int_0^1 \dot X_t \, \ud t\Bigr\rVert^2\Bigr]
        \le \int_0^1 \E[\|\dot X_t\|^2] \, \ud t \\
        &= \int_0^1 \|v_t\|_{\mu_t}^2 \, \ud t\,.
    \end{align*}

    To study the equality case, note that in the above calculations we employed two inequalities.
    The first inequality is an equality if and only if $(X_0, X_1)$ are optimally coupled.
    The second inequality is an equality if and only if $t\mapsto \dot X_t$ is constant, which forces $\dot X_t = X_1 - X_0$ for all $t\in [0,1]$.
    
    It therefore suffices to show that there exists a family of vector fields $(v_t)_{t \geq 0}$ such that $X_1 - X_0 = v_t((1-t)X_0 + tX_1)$ for all $t \in [0, 1)$.
    Since $(X_0, X_1)$ are optimally coupled, there exists a convex function $\varphi$ such that $\nabla \varphi(X_0) = X_1$.
    This implies that $(1-t)X_0 + tX_1 = \nabla \varphi_t(X_0)$, where $\varphi_t(x) \deq (1-t) \tfrac{\|x\|^2}{2} +  t \varphi(x)$.
    In particular, $\varphi_t$ is strongly convex for all $t \in [0, 1)$, so $\nabla \varphi_t$ is invertible.
    We may therefore define $v_t \deq (\nabla \varphi - \id) \circ \nabla \varphi_t^{-1}$, which satisfies $v_t((1-t)X_0 + tX_1) = v_t(\nabla \varphi_t(X_0)) = \nabla \varphi(X_0) - X_0 = X_1 - X_0$, as claimed.
\end{proof}

Note that we can also write $\mu_t = [(1-t)\,{\id} + t\,T]_\#\mu_0$, where $T$ is the optimal transport map from $\mu_0$ to $\mu_1$.
Hence, we formulate the following definition.

\begin{definition}\label{def:w2_geod}
    Let $\mu_0,\mu_1 \in \cP_{2,\rm ac}(\R^d)$ and let $T$ denote the optimal transport map from $\mu_0$ to $\mu_1$.
    The \emph{constant-speed geodesic} joining $\mu_0$ to $\mu_1$ is the curve ${(\mu_t)}_{t\in [0,1]}$, where
    \begin{align}\label{eq:displ_interp}
        \mu_t = [(1-t)\,{\id} + t\,T]_\# \mu_0\,.
    \end{align}
    This curve is known as the \emph{displacement interpolation}, \emph{McCann's interpolation}, or simply the \emph{Wasserstein geodesic} joining $\mu_0$ to $\mu_1$.
\end{definition}

From~\eqref{eq:displ_interp}, we can identify $\log_{\mu}(\nu) = T_{\mu\to\nu} - {\id}$, and hence $\exp_\mu(\nabla \psi) = ({\id} + \nabla \psi)_\# \mu$.
Note that the exponential map is not well-defined if $\nabla^2\psi$ has an eigenvalue smaller than $-1$, since then ${\id} + \nabla \psi$ is not the gradient of a convex function and thus not an optimal transport map.
This reflects our earlier discussion that the exponential map is not necessarily defined on the full tangent space of a Riemannian manifold.\footnote{However, the domain of the exponential map for a Riemannian manifold always contains a neighborhood of the origin, whereas this is not true for the Wasserstein space. This is one of the reasons why the Wasserstein space is not truly a Riemannian manifold, even an infinite-dimensional one.}

\section{Otto calculus}\label{sec:otto_calc}

The next step is to identify the Wasserstein gradient, which, in turn, allow us to define Wasserstein gradient flows.
After obtaining criteria for functionals to be geodesically convex, we can then obtain rates of convergence thereof.

It turns out that the Wasserstein gradient can be expressed in terms of the first variation.

\begin{definition}[First variation]\label{defn:first_variation}\index{first variation}
    Let $\cF : \cP_{2,\rm ac}(\R^d)\to\R$ be a functional.
    The \emph{first variation} of $\cF$ at $\mu$, denoted $\delta \cF(\mu) : \R^d\to\R$, is the function defined by
    \begin{align}\label{eq:first_var_def}
        \lim_{\varepsilon\searrow 0}\frac{\cF(\mu+\varepsilon\chi) - \cF(\mu)}{\varepsilon} = \int \delta\cF(\mu)\,\ud\chi\,,
    \end{align}
    for all signed measures $\chi$ such that $\mu+\varepsilon\chi \in \cP_{2,\rm ac}(\R^d)$ for all sufficiently small $\varepsilon$.
\end{definition}

If ${(\mu_t)}_{t\ge 0}$ is a curve of densities, then we can write the linear approximation $\mu_{t+\varepsilon} \approx \mu_t + \varepsilon\,\partial_t\mu_t$ for $\varepsilon$ small, where $\partial_t\mu_t$ denotes the usual time derivative.
We can take $\chi = \partial_t \mu_t$, in which case~\eqref{eq:first_var_def} reads $\partial_t \cF(\mu_t) = \int \delta \cF(\mu_t)\,\partial_t\mu_t$.
Note also that the first variation is only defined up to an additive constant, since the perturbations $\chi$ always satisfy $\int \ud \chi = 0$.

We can now define Wasserstein gradients.
By definition, given a curve of measures ${(\mu_t)}_{t\ge 0}$ with tangent vectors ${(v_t)}_{t\ge 0}$, the gradient of a functional $\cF : \cP_{2,\rm ac}(\R^d) \to\R$ is the element of  $T_{\mu_t}\cP_{2,\rm ac}(\R^d)$ such that $\partial_t \cF(\mu_t) = \langle \gradW \cF(\mu_t),v_t\rangle_{\mu_t}$.
The following proposition shows that this definition can be written directly in terms of the first variation.

\begin{proposition}\label{prop:w2_grad}\index{Wasserstein gradient}
    Let $\cF : \cP_{2,\rm ac}(\R^d) \to\R$ be a functional with first variation $\delta \cF$.
    Then, the Wasserstein gradient of $\cF$ is the vector field $ \gradW \cF(\mu): \R^d \to \R^d$ defined by
    \begin{align*}
        \gradW \cF(\mu) = \nabla \delta \cF(\mu)\,,
    \end{align*}
    where $\nabla$ on the right-hand side denotes the usual Euclidean gradient.
\end{proposition}
\begin{proof}
    Let ${(\mu_t)}_{t\ge 0}$ be a curve of measures with tangent vectors ${(v_t)}_{t\ge 0}$.
    The fact that $v_t$ is the tangent vector at time $t$ means that it solves the continuity equation~\eqref{eq:continuity}.
    From the above discussion of the first variation,
    \begin{align*}
        \partial_t \cF(\mu_t)
        &= \int \delta \cF(\mu_t)\,\partial_t \mu_t
        = -\int \delta \cF(\mu_t) \divergence(\mu_t v_t)\\
        &= \int \langle \nabla \delta \cF(\mu_t), v_t\rangle\,\ud \mu_t
        = \langle \nabla \delta \cF(\mu_t), v_t\rangle_{\mu_t}\,.
    \end{align*}
    Moreover, since $\nabla \delta \cF(\mu_t)$ is the gradient of a function, from Definition~\ref{def:w2_tangent} we have $\nabla \delta \cF(\mu_t) \in T_{\mu_t} \cP_{2,\rm ac}(\R^d)$.
    From this, we conclude that $\nabla \delta \cF(\mu_t)$ is indeed the Wasserstein gradient of $\cF$ at $\mu_t$.
\end{proof}

To compute the Wasserstein gradient, we therefore have to compute the first variation and then take its gradient.
We illustrate this on three canonical examples of functionals over $\cP_{2,\rm ac}(\R^d)$.

\begin{example}[Potential energy]\label{ex:potential}\index{potential energy}
    Let $\cF(\mu) \deq \int V \, \ud \mu$ for some (potential) function $V : \R^d\to\R$.
    Then, $\partial_t \cF(\mu_t) = \int V \,\partial_t \mu_t$ and we can identify $\delta \cF(\mu) = V$.
    Thus, $\gradW \cF(\mu) = \nabla V$.
\end{example}

\begin{example}[Internal energy]\label{ex:internal}\index{internal energy}
    Let $\cF(\mu) \deq \int U(\mu(x)) \, \ud x$ for some function $U : \R_+\to\R$.
    For example, $U(x) = x\log x$ gives rise to the entropy\footnote{This is the negative of the thermodynamic entropy.} functional.
    Then, $\partial_t \cF(\mu_t) = \int U'(\mu_t) \,\partial_t \mu_t$, so we can identify $\delta \cF(\mu) = U'\circ \mu$ and therefore $\gradW\cF(\mu) = \nabla(U'\circ \mu)$.
    In the case of entropy, $\delta \cF(\mu) = \log \mu + 1$, and $\gradW\cF(\mu) = \nabla \log \mu$.
\end{example}

\begin{example}[Interaction energy]\label{ex:interaction}\index{interaction energy}
    Take a symmetric kernel $K : \R^d\to\R$, i.e., $K(-z) = K(z)$. Set $\cF(\mu) \deq \frac{1}{2}\iint K(x-y)\,\mu(\ud x)\,\mu(\ud y)$.
    For example, we could consider a Gaussian kernel $K(x) = \exp(-\frac{\|x\|^2}{2\sigma^2})$.
    Then, $\partial_t \cF(\mu_t) = \iint K(x - y)\,\mu_t(\ud y)\,\partial_t \mu_t(\ud x)$, so we can identify $\delta \cF(\mu) = \int K(\cdot - y)\,\mu(\ud y)$, and $\gradW\cF(\mu) = \int \nabla K(\cdot-y)\,\mu(\ud y)$.
\end{example}

We can now define the Wasserstein gradient flow of a functional $\cF$ over $\cP_{2,\rm ac}(\R^d)$. The gradient flow is a curve of measures ${(\mu_t)}_{t\ge 0}$ such that the tangent vector to the curve at time $t$ equals $-\gradW\cF(\mu_t)$.
Recalling that the tangent vectors governs the evolution of ${(\mu_t)}_{t\ge 0}$ through the continuity equation~\eqref{eq:continuity}, we arrive at the following definition.

\begin{definition}[Wasserstein gradient flow]\label{def:w2_grad_flow}\index{Wasserstein gradient flow}
    Let $\cF : \cP_{2,\rm ac}(\R^d)\to\R$ be a functional.
    Then, ${(\mu_t)}_{t\ge 0}$ is called the \emph{Wasserstein gradient flow} of $\cF$ if it solves the PDE
    \begin{align*}
        \partial_t \mu_t
        &= \divergence\bigl(\mu_t \gradW \cF(\mu_t)\bigr)\,.
    \end{align*}
\end{definition}

As is well-understood in optimization, gradient flows are natural dynamics for minimizing the objective functional $\cF$ because, as discussed shortly, they always reduce the value of the objective.
Wasserstein gradient flows therefore constitute a principled approach for designing dynamics over the space of probability measures aimed at minimizing some criterion, a task which is ubiquitous in applied mathematics, statistics, and beyond; see Chapter~\ref{chap:appWGF}.

A quick calculation using the definition of the Wasserstein gradient flow ${(\mu_t)}_{t\ge 0}$ of $\cF$ yields
\begin{align}
\label{eq:grad_flow_dissipate}
    \partial_t \cF(\mu_t)
    &= \langle \gradW \cF(\mu_t), v_t\rangle_{\mu_t}
    = -\|\gradW\cF(\mu_t)\|_{\mu_t}^2
\end{align}
where $v_t = -\gradW\cF(\mu_t)$ is the tangent vector at time $t$.
This equality, which states that the objective functional is dissipated at a rate equal to the squared norm of the gradient, is a generic fact about gradient flows.
In particular, if $\cF$ is bounded below, it implies that any limit point of the gradient flow must be a stationary point of $\cF$.

However, we can say more once we have a quantitative lower bound on the rate of dissipation.
The simplest such condition is the \emph{Polyak--\L{}ojasiewicz (P\L{}) inequality}.

\begin{definition}[Polyak--\L{}ojasiewicz (P\L{}) 
inequality]\index{Polyak--\L{}ojasiewicz (P\L{}) inequality}
    We say that $\cF : \cP_{2,\rm ac}(\R^d)\to\R$ satisfies a \emph{P\L{} inequality} with constant $\alpha > 0$ if for all $\mu \in \cP_{2,\rm ac}(\R^d)$,
    \begin{align*}
        \|\gradW \cF(\mu)\|_\mu^2
        &\ge 2\alpha\,(\cF(\mu) - \inf\cF)\,.
    \end{align*}
\end{definition}

The P\L{} inequality over $\R^d$ is discussed in Appendix~\ref{sec:strcvx_smooth}; the above definition adapts this concept to the present setting.
From~\eqref{eq:grad_flow_dissipate}, the P\L{} inequality yields
\begin{align*}
    \partial_t (\cF(\mu_t)-\inf \cF)
    &\le -2\alpha\,(\cF(\mu_t) - \inf\cF)\,.
\end{align*}
Let $\phi(t) \deq \cF(\mu_t) - \inf \cF$, so that $\dot \phi(t) \le -2\alpha \phi(t)$.
If this inequality were an equality, then we could solve the differential equation to obtain $\phi(t) = \phi(0) \exp(-2\alpha t)$.
In general, when we have a differential \emph{inequality}, we can bound $\phi$ by the solution to the differential equation; this is formalized as Gr\"onwall's inequality.

\begin{lemma}[Gr\"onwall's inequality]\index{Gr\"onwall's inequality}
    Let $c\in\R$.
    Let $\phi : [0,T] \to \R$ be differentiable, satisfying $\dot\phi(t) \le c\phi(t)$ for all $t\in [0,T]$.
    Then,
    \begin{align*}
        \phi(t) \le \phi(0)\exp(ct) \qquad\forall\,t\in [0,T]\,.
    \end{align*}
\end{lemma}
\begin{proof}
    It holds that
    \begin{align*}
        \partial_t [\exp(-ct)\, \phi(t)]
        &= \exp(-ct)\,[-c\phi(t) + \dot \phi(t)]
        \le 0\,.
    \end{align*}
    This implies $\exp(-ct)\,\phi(t) \le \exp(-c\cdot 0)\,\phi(0)=\phi(0)$.
\end{proof}

On the other hand, applying the same argument as in Lemma~\ref{lem:strcvx_implies_pl}, one can show that $\cF$ satisfies the P\L{} inequality with constant $\alpha$ as soon as $\cF$ is $\alpha$-geodesically convex.
We deduce the following useful corollary.

\begin{corollary}\label{cor:conv_grad_flow}
    Let $\cF : \cP_{2,\rm ac}(\R^d)\to\R$ be $\alpha$-geodesically convex.
    Then, along the Wasserstein gradient flow ${(\mu_t)}_{t\ge 0}$ for $\cF$, it holds
    \begin{align*}
        \cF(\mu_t) - \inf \cF
        &\le e^{-2\alpha t}\,(\cF(\mu_0) - \inf \cF)\,.
    \end{align*}
\end{corollary}

\section{Bures--Wasserstein}\label{sec:bw}

It is illuminating to specialize the concepts in the previous section to the submanifold of Wasserstein space consisting of Gaussian measures.

\begin{definition}\index{Bures--Wasserstein}
    The \emph{Bures--Wasserstein space} $\BW(\R^d)$ is the space of non-degenerate Gaussians on $\R^d$, equipped with the Wasserstein metric.
\end{definition}

Concretely, since Gaussians are parameterized by the mean and covariance matrix, we can think of $\BW(\R^d) \cong \R^d\times \mb S_{++}^d$, where $\mb S_{++}^d$ is cone of symmetric positive definite $d \times d$ matrices.
Recall from Example~\ref{ex:ot_gaussian} that for any $\mu_0,\mu_1 \in \BW(\R^d)$, the optimal transport map $T$ from $\mu_0$ to $\mu_1$ is an affine map, and the Wasserstein geodesic joining $\mu_0$ to $\mu_1$ is
\begin{align*}
    \mu_t = \underbrace{[(1-t)\,{\id} + t\,T]}_{\text{affine}}{}_\# \mu_0\,, \qquad t\in [0,1]\,.
\end{align*}
Since the pushforward of a non-degenerate Gaussian by a non-singular affine map is also a non-degenerate Gaussian, the Wasserstein geodesic from $\mu_0$ to $\mu_1$ lies entirely in $\BW(\R^d)$, or in other words:

\begin{proposition}\label{prop:bw_geod_cvx}
    $\BW(\R^d)\subseteq \cP_{2,\rm ac}(\R^d)$ is geodesically convex.
\end{proposition}

Recall that a functional $\cF$ on a Riemannian manifold $\cM$ is $\alpha$-geodesically convex if the mapping $[0,1] \to \cM$, $t\mapsto \cF(p_t)$ is $\alpha$-convex for all geodesics ${(p_t)}_{t\in [0,1]}$ on $\cM$.
The geodesic convexity of $\BW(\R^d)$ means that the intrinsic geodesics of $\BW(\R^d)$ coincide with the Wasserstein geodesics, which immediately furnishes the following corollary.

\begin{corollary}\label{cor:bw_cvxty}
    Let $\cF : \cP_{2,\rm ac}(\R^d)\to\R$ be an $\alpha$-geodesically convex functional. Then, $\cF$ is also $\alpha$-geodesically convex when viewed as a functional over $\BW(\R^d)$.
\end{corollary}

We make use of this fact in Subsection~\ref{sec:gvi}.

The Riemannian structure of $\cP_{2,\rm ac}(\R^d)$ descends to $\BW(\R^d)$ and endows the Bures{--}Wasserstein space with the structure of a bona fide finite-dimensional Riemannian manifold.
The tangent space at $\mu\in \BW(\R^d)$ is
\begin{align*}
    T_\mu \BW(\R^d)
    &= \{\lambda \,(T_{\mu\to \nu}-{\id}) \mid \lambda > 0,\, \nu \in \BW(\R^d)\} \\
    &= \{x\mapsto Sx+a \mid a\in\R^d,\,S \in \mb S^d\}\,,
\end{align*}
where $\mb S^d$ is the space of symmetric $d\times d$ matrices.
Actually, it is convenient to reparametrize the tangent space as
\begin{align*}
    T_\mu \BW(\R^d)
    &= \{x\mapsto S\,(x-m)+a \mid a\in\R^d,\,S \in \mb S^d\}\,,
\end{align*}
where $m = \E_{X\sim \mu}[X]$.

By definition, the Riemannian structure induced on $\BW(\R^d)$ is the restriction of the inner product on $T_\mu \cP_{2,\rm ac}(\R^d)$ to the subspace $T_\mu \BW(\R^d) \subseteq T_\mu \cP_{2,\rm ac}(\R^d)$, i.e., the $L^2(\mu)$ inner product.
Using this, we could compute the BW gradient of a functional $\cF$ from scratch.
However, since we have already computed the Wasserstein gradient of $\cF$ at $\mu$ to be the vector field $\gradW \cF(\mu) = \nabla \delta \cF(\mu)$ (Proposition~\ref{prop:w2_grad}), a more expedient approach is to now compute the orthogonal projection of $\gradW \cF(\mu)$ onto $T_\mu\BW(\R^d)$.

\begin{theorem}\label{thm:bw_grad}
    Let $\cF : \cP_{2,\rm ac}(\R^d)\to\R$ be a functional with first variation $\delta \cF(\mu)$ at $\mu$. Then, the Bures{--}Wasserstein gradient\index{Bures--Wasserstein!gradient} of $\cF$ at $\mu \in \BW(\R^d)$ is the affine mapping
    \begin{align*}
        x\mapsto  \bigl(\int \nabla^2 \delta \cF(\mu)\, \ud \mu\bigr) (x- m) + \int \nabla \delta \cF(\mu)\,\ud \mu\,,
    \end{align*}
    where $m=\int x \,\mu(\ud x)$ is the mean of $\mu$.
\end{theorem}
\begin{proof}
    Recall that $\gradW\cF(\mu) = \nabla \delta \cF(\mu)$ (Proposition~\ref{prop:w2_grad}).
    The BW gradient at $\mu$ is the orthogonal projection of $\nabla \delta \cF(\mu)$ onto $T_\mu\BW(\R^d)$; by definition, this is the element $\nabla_{\BW}\cF(\mu) \in T_\mu \BW(\R^d)$ which satisfies
    \begin{align}\label{eq:ortho_proj}
        \langle \nabla_{\BW} \cF(\mu), v\rangle_\mu = \langle \nabla \delta \cF(\mu), v\rangle_\mu \qquad\forall v \in T_\mu \BW(\R^d)\,.
    \end{align}
    We can write out this condition more explicitly.
    Since $\nabla_{\BW}\cF(\mu), v \in T_\mu \BW(\R^d)$, they are of the form
    \begin{align*}
        \nabla_{\BW} \cF(\mu)
        &= S\,(\cdot - m) + a\,, \\
        v
        &= \tilde S\,(\cdot - m) + \tilde a\,.
    \end{align*}
    On one hand,
    \begin{align}
        \langle \nabla_{\BW} \cF(\mu), v\rangle_\mu
        &= \E_{X\sim \mu}\langle S\,(X-m) + a, \tilde S \,(X-m) + \tilde a \rangle \nonumber\\
        &= \langle S,\Sigma\,\tilde S\rangle + \langle a, \tilde a\rangle\,,\label{eq:bw_grad_1}
    \end{align}
    where $\Sigma$ is the covariance matrix of $\mu$.
    On the other hand,
    \begin{align}
        \langle \nabla \delta \cF(\mu), v\rangle_\mu
        &= \E_{X\sim \mu}\langle [\nabla \delta \cF(\mu)](X), \tilde S\,(X-m) + \tilde a\rangle \nonumber \\
        &= \langle \E_{X\sim \mu}[\nabla \delta \cF(\mu)(X)\, {(X-m)}^\T], \tilde S\rangle \label{eq:bw_grad_2}\\
        &\qquad{} + \langle \E_{X\sim \mu} \nabla \delta \cF(\mu)(X), \tilde a\rangle\,.\label{eq:bw_grad_3}
    \end{align}
    Also, integration by parts yields
    \begin{align*}
        \E_{X\sim \mu}[\nabla \delta \cF(\mu)(X)\, {(X-m)}^\T]
        &= \int \nabla \delta \cF(\mu)\,{(\Sigma \Sigma^{-1}\,(\cdot-m))}^\T\, \ud \mu \\
        &= -\int \nabla \delta \cF(\mu)\,{(\nabla \log \mu)}^\T \, \ud \mu\,\Sigma \\
        &= -\int \nabla \delta \cF(\mu)\, {(\nabla \mu)}^\T \, \Sigma \\
        &= \int \nabla^2 \delta \cF(\mu)\, \ud \mu\,\Sigma\,.
    \end{align*}
    Hence,
    \begin{align}
        \langle \E_{X\sim \mu}[\nabla \delta \cF(\mu)(X)\, {(X-m)}^\T], \tilde S\rangle
        &= \Bigl\langle \int \nabla^2 \delta \cF(\mu)\, \ud \mu\,\Sigma, \tilde S\Bigr\rangle \nonumber\\
        &= \Bigl\langle \Sigma\int \nabla^2 \delta\cF(\mu)\, \ud \mu, \tilde S\Bigr\rangle \nonumber\\
        &= \Bigl\langle \int \nabla^2 \delta\cF(\mu)\, \ud \mu,\Sigma\, \tilde S\Bigr\rangle\,.\label{eq:bw_grad_4}
    \end{align}
    Since~\eqref{eq:ortho_proj} is supposed to hold for all $\tilde a \in\R^d$ and all $\tilde S\in \mb S^d$, by comparing~\eqref{eq:bw_grad_1},~\eqref{eq:bw_grad_2},~\eqref{eq:bw_grad_3}, and~\eqref{eq:bw_grad_4}, we can identify 
    $$
    S = \int \nabla^2 \delta\cF(\mu)\, \ud \mu\quad \text{ and} \quad  a = \int \nabla \delta \cF(\mu)\, \ud \mu\,.
    $$
    This completes the derivation.
\end{proof}

Once we have identified the BW gradient, we can use the Lagrangian interpretation of the continuity equation to implement the gradient flow via the dynamics
\begin{align*}
    \dot X_t
    &= -\nabla_{\BW}\cF(\mu_t)(X_t) \\
    &= -\bigl(\int\nabla^2 \delta \cF(\mu_t)\,\ud \mu_t\bigr) \,(X_t - m_t) - \int\nabla \delta \cF(\mu_t)\,\ud \mu_t\,,
\end{align*}
where $X_t \sim \mu_t$ and we denote the mean and covariance of $\mu_t$ by $m_t$ and $\Sigma_t$ respectively.
However, since $\mu_t$ is a Gaussian for each $t\ge 0$, it is expedient to instead track $\mu_t$ \emph{exactly} through the mean $m_t$ and covariance $\Sigma_t$. They follow the dynamics:
\begin{align*}
    \dot m_t
    &= \E \dot X_t
    = -\int \nabla \delta \cF(\mu_t)\,\ud \mu_t\,,
\end{align*}
and
\begin{align*}
    \dot \Sigma_t
    &= \partial_t \E[(X_t-m_t)\,{(X_t-m_t)}^\T] \\
    &= \E[\dot X_t\,{(X_t - m_t)}^\T] + \E[(X_t - m_t)\,\dot X_t^\T] \\
    &= -\E\Bigl[\Bigl(\bigl(\int \nabla^2 \delta \cF(\mu_t)\,\ud \mu_t\bigr) \,(X_t - m_t) \\
    &\qquad\qquad\qquad\qquad\qquad{} + \int\nabla \delta \cF(\mu_t) \, \ud \mu_t\Bigr)\,{(X_t - m_t)}^\T\Bigr] + \cdots \\
    &= -\bigl(\int\nabla^2 \delta\cF(\mu_t)\,\ud \mu_t\bigr) \,\E[(X_t - m_t)\,{(X_t - m_t)}^\T] + \cdots \\
    &= -\bigl(\int \nabla^2\delta \cF(\mu_t)\,\ud \mu_t\bigr)\,\Sigma_t -\Sigma_t\,\bigl(\int \nabla^2\delta \cF(\mu_t)\,\ud \mu_t\bigr)\,,
\end{align*}
where above, $A + \cdots$ is shorthand for the expression $A + A^\T$.
Finally, we have arrived at an explicit system of equations.

\begin{theorem}\label{thm:bw_gf}
    The BW gradient flow of the functional $\cF$ is the curve ${(\mu_t = \cN(m_t, \Sigma_t))}_{t\ge 0}$, where
    \begin{align}\label{eq:bw_gf}
        \boxed{
        \begin{aligned}
            \dot m_t &= -\E\nabla \delta \cF(\mu_t)(X_t)\,, \\
            \dot \Sigma_t &= - \E\nabla^2\delta \cF(\mu_t)(X_t)\,\Sigma_t - \Sigma_t\,\E\nabla^2\delta \cF(\mu_t)(X_t)\,,
        \end{aligned}
        }
    \end{align}
    and $X_t \sim \mu_t$.
\end{theorem}

\section{Gaussian mixtures}\label{sec:gaussian_mixtures}

Building on top of the ideas introduced in Section~\ref{sec:bw}, we now consider gradient flows over the space of Gaussian mixtures, which is a far richer space. In fact, as explained below, \emph{any} measure over $\R^d$ can be viewed as a Gaussian mixture when viewed through the right lens.

Before doing so, we first note that simply constraining the Wasserstein gradient flow to lie on the space of Gaussian mixtures, similarly to how we proceeded in Section~\ref{sec:bw}, does not work.
The problem is that we cannot explicitly compute the optimal transport map between two Gaussian mixtures, even infinitesimally, and so we cannot identify the tangent space---unless each Gaussian mixture has one component, or one dimension. 
Nevertheless, following~\cite{CheGeoTan19, DelDes20GMM}, there is a natural geometric structure we can consider: Wasserstein over Bures--Wasserstein.

Gaussian mixtures are typically introduced as distributions of the form $\sum_{k=1}^K w_k \cN(m_k,\Sigma_k)$ for \emph{mixing weights} $w_k \ge 0$, $\sum_{k=1}^K w_k = 1$, but we can define a Gaussian mixture more broadly as a measure of the form $\int \cN(m, \Sigma)\,\nu(\ud m,\ud \Sigma)$.
The finite Gaussian mixture above corresponds to a discrete \emph{mixing measure} $\nu$: $\nu = \sum_{k=1}^K w_k \delta_{(m_k,\Sigma_k)}$.
This new, broader definition of a Gaussian mixture, is nearly useless, since any measure $\mu$ can be represented thus: $\mu = \int \delta_x \, \mu(\D x)$, where $\delta_x = \cN(x, 0)$ is a degenerate Gaussian. Also, the representation of a Gaussian mixture by a mixing measure $\nu$ is ``overparametrized'', i.e., $\nu$ is highly non-unique: consider the equality $\cN(0, I) = \int \cN(x,\tau I)\, \nu(\ud x)$ where $\nu = \cN(0, (1-\tau)I)$, valid for any $\tau \in [0,1]$.
Nevertheless, the utility of this perspective is that it leads to a natural interpretation: \emph{a mixing measure for a Gaussian mixture is simply a probability measure over the Bures--Wasserstein space}.
Let us see how this leads to the definition of a geometric structure.

\emph{The Bures--Wasserstein space is a Riemannian manifold.}
As noted earlier, the space $\BW(\R^d)$ is isometric to the manifold $\R^d\times \mb S_{++}^d$ equipped with a certain Riemannian metric. Hereafter we consider the metric arising from Otto calculus but any metric, including the Euclidean one could be used here.

\emph{We can consider the space of probability measures (with finite second moment) over any metric space, and endow it with the $2$-Wasserstein distance.}
Indeed, recall from Section~\ref{sec:ot_problem} that the optimal transport problem can be defined with more general costs, so we can take the squared distance function over the metric space as our cost.

\emph{When the metric space in question is a Riemannian manifold, the results from Sections~\ref{sec:continuity_eq}--\ref{sec:otto_calc} continue to hold with appropriate modifications.}
We do not justify this in detail here, but we invite the reader to revisit these sections with a fresh perspective. For example, the ODE~\eqref{eq:particle_ode} still makes sense, keeping in mind that a vector field $v$ on a manifold $\cM$ is an assignment $x\mapsto v(x)$ of a tangent vector $v(x) \in T_x \cM$ at each point $x\in \cM$.
The continuity equation still makes sense in its weak form~\eqref{eq:continuity_weak}, where $\nabla$ now refers to the Riemannian gradient, and even~\eqref{eq:continuity} makes sense if we interpret $\mu_t$ as a density with respect to the volume measure, etc.
In fact, this geometric setting is the source of some of the deepest developments in optimal transport theory; see~\cite{Vil09}.

Crucially, for our purposes, the formula for the Wasserstein gradient given in Proposition~\ref{prop:w2_grad} still holds, where we again interpret $\nabla$ as the Riemannian gradient.

Putting this discussion together, we can derive the Wasserstein gradient flow which lives in the space $(\cP_2(\BW(\R^d)), W_2)$.

\begin{theorem}\label{thm:gm_flow}
    Given $\nu \in \cP_2(\BW(\R^d))$, let $G_\nu$ denote the corresponding Gaussian mixture $G_\nu = \int \cN(m,\Sigma)\,\nu(\ud m, \ud \Sigma)$.
    Let $\cF$ be a functional over $\cP_2(\R^d)$, and let $\cG$ be the corresponding functional over $\cP_2(\BW(\R^d))$ given by $\nu \mapsto \cF(G_\nu)$.
    Then, the Wasserstein gradient flow of $\cG$ is the curve ${(\nu_t)}_{t\ge 0}$ described as follows: $\nu_t = \law(m_t,\Sigma_t)$, where
    \begin{align*}
        \boxed{\begin{aligned}
        \dot m_t
        &= -\E\nabla \delta \cF(G_{\nu_t})(X_t)\,, \\
        \dot \Sigma_t
        &= -\E\nabla^2 \delta \cF(G_{\nu_t})(X_t) \,\Sigma_t - \Sigma_t\,\E\nabla^2\delta \cF(G_{\nu_t})(X_t)\,,
        \end{aligned}}
    \end{align*}
    and $X_t \sim \cN(m_t,\Sigma_t)$.
\end{theorem}
\begin{proof}
    The first variation of $\cG$ is computed as follows.
    If ${(\nu_t)}_{t\in\R}$ is a curve in $\cP_2(\BW(\R^d))$, then $\partial_t G_{\nu_t} = \int \cN(m,\Sigma)\,\partial_t \nu_t(\ud m, \ud \Sigma)$.
    Hence,
    \begin{align*}
        \partial_t \cG(\nu_t)
        &= \partial_t \cF(G_{\nu_t})
        = \int \delta \cF(G_{\nu_t})\,\partial_t G_{\nu_t}\\
        &= \iint \delta \cF(G_{\nu_t})\, \ud \cN(m,\Sigma)\, \partial_t \nu_t(\ud m, \ud \Sigma)\,.
    \end{align*}
    This implies that the first variation is
    \begin{align*}
        \delta \cG(\nu) : (m,\Sigma) \mapsto \int \delta \cF(G_\nu)\, \D \cN(m,\Sigma)\,.
    \end{align*}
    Based on our identification of $\BW(\R^d)$ with $\R^d\times \mb S_{++}^d$, the Riemannian gradient of $\delta \cG(\nu)$, evaluated at $(m,\Sigma)$, is the same as the Bures--Wasserstein gradient of $\mu \mapsto \int \delta \cF(G_\nu)\,\ud \mu$, evaluated at $\mu = \cN(m,\Sigma)$, and we computed the latter in Theorem~\ref{thm:bw_grad}.
    Therefore, the result follows from Theorem~\ref{thm:bw_gf}.
\end{proof}

\section{Wasserstein--Fisher--Rao}\label{sec:wfr}

We now describe a variation of the Wasserstein geometry that is often useful in applications as illustrated in Chapter~\ref{chap:appWGF}.
This variation, known as the \emph{Wasserstein--Fisher--Rao} (WFR) or \emph{Hellinger--Kantorovich}\index{Hellinger--Kantorovich|see{Wasserstein--Fisher--Rao}} distance, was originally proposed and studied as a model of \emph{unbalanced optimal transport}\index{unbalanced optimal transport}, that is, optimal transport between positive measures not necessarily containing the same total mass.
A notable example is the cellular trajectory reconstruction application mentioned in the Preface, in which measures are used to represent snapshots of the cell population at different times, and for which the total mass indeed changes due to the birth and death of individual cells.

From the static perspective, WFR defines a variant of the optimal transport problem, and its various properties such as the cost, metric properties, duality, etc.\ can all be investigated in a similar vein as we did in Chapter~\ref{chap:OT}.
We refer to the references~\cite{LieMieSav16HK, KonMonVor16,
 Chi+18FR, LieMieSav18HK} for detailed investigations in this direction.
In this section, we follow~\cite[Appendix H]{Lametal22GVI} and focus on the Riemannian structure of the resulting metric space for the purpose of deriving gradient flows.

\paragraph{Fisher--Rao}\index{Fisher--Rao}
Before turning toward Wasserstein--Fisher--Rao, we first describe one of its key components: the \emph{Fisher--Rao} metric.
This is a metric over the space $\cM_+(\R^d)$ of positive measures over $\R^d$, defined via
\begin{align*}
    \dFR^2(\mu_0, \mu_1)
    \deq \int {(\sqrt{\mu_0} - \sqrt{\mu_1})}^2\,,
\end{align*}
where as usual we identify measures with their Lebesgue densities, assuming that they exist.\footnote{When the densities do not exist, we can define the distance via $\dFR^2(\mu_0,\mu_1) \deq \int (\sqrt{\frac{\D\mu_0}{\D\lambda}} - \sqrt{\frac{\D\mu_1}{\D\lambda}})^2 \, \D \lambda$ with respect to any common dominating measure $\lambda$.}
When $\mu_0$, $\mu_1$ are probability measures, then $\dFR$ coincides with the statistician's \emph{Hellinger distance}.\footnote{This explains the competing naming conventions for the WFR metric; note that $\{\text{Wasserstein}, \text{Fisher--Rao}\} \cong \{\text{Hellinger}, \text{Kantorovich}\}$.}

Geometrically, $\dFR$ is the metric over (say) non-negative densities obtained by demanding that $\mu \mapsto \sqrt\mu$ be an isometry into the Hilbert space $L^2(\R^d)$.
Therefore, the geometry of $(\cM_+(\R^d), \dFR)$ is flat, and we can obtain a Riemannian structure through the isometry.
Namely, if $\dot\mu$ denotes the derivative in time of a curve of densities, then the derivative of the square root is $\dot{\sqrt\mu} = \dot \mu/(2\sqrt\mu)$.
If we measure the ``length'' of the latter in the $L^2(\R^d)$ norm, we arrive at the induced Riemannian metric
\begin{align*}
    \mg_\mu(\dot \mu,\dot\mu)
    &\deq  \|\dot{\sqrt{\mu}}\|_{L^2(\R^d)}^2=\int \frac{\dot\mu^2}{4\mu}
\end{align*}
over the tangent space $T_\mu \cM_+(\R^\dd)$ of functions $\dot\mu : \R^\dd\to\R$.

This geometry is set over the space of all positive measures, including measures with differing amounts of mass, and indeed this geometry proves useful for modeling physical situations in which change of mass naturally occurs.
A canonical example is when $\mu$ represents the concentration of a chemical substance, and the concentration changes over time due to chemical reactions.
This is modelled by the \emph{reaction equation} $\partial_t \mu_t = \alpha_t \mu_t$, where $\alpha_t : \R^d\to\R$ dictates the rate of reaction at each point in space.
Note that in this notation, $\alpha = \dot \mu/\mu$, which amounts to a reparametrization of the tangent space.
In other words, we can equivalently think of the tangent space as consisting of functions $\alpha : \R^\dd \to\R$ equipped with the metric
\begin{align}\label{eq:fr_metric}
    \widetilde{\mg}_\mu(\alpha,\alpha)
    \deq \mg_\mu(\alpha\mu,\alpha\mu)
    = \frac{1}{4} \int \alpha^2\,\D \mu\,.
\end{align}
Going forward, we adopt this as our definition of the metric, and hence we write $\|\alpha\|_\mu^2 \deq \widetilde{\mg}_\mu(\alpha,\alpha)$.

We can draw comparisons with the definition of the Wasserstein geometry: at each measure $\mu$, the ``tangent space'' $T_\mu \cM_+(\R^d)$ at $\mu$ is now defined to be the space of all functions $\alpha : \R^d\to\R$, equipped with the metric~\eqref{eq:fr_metric}, and the continuity equation is replaced by the reaction equation $\partial_t \mu_t = \alpha_t \mu_t$.
Compared to the Wasserstein metric, the Fisher{--}Rao metric is based on an entirely different intuition: rather than \emph{transportation} of mass, the reaction equation now describes spontaneous \emph{creation and destruction} of mass.

Despite motivating the Fisher{--}Rao geometry for problems involving change of mass, we may also wish to apply it to problems in which we want to maintain a flow on the space of probability measures.
To do so, we consider the induced geometry over $\cP(\R^d)$.
The equation $\partial_t \mu_t = \alpha_t \mu_t$ preserves the total mass if and only if $\int \alpha_t \,\D \mu_t = 0$ for all $t\ge 0$, so we restrict the tangent space to $T_\mu \cP(\R^d) = \{\alpha : \R^d \to \R \mid \int \alpha \,\D \mu = 0\}$, equipped with the metric~\eqref{eq:fr_metric}.\footnote{A discrete analogy: endow the space of probability measures on $\{1,\dotsc,d\}$ (i.e., the simplex in $\R^d$) with a geometry via an isometry $p\mapsto \sqrt p$, where $\sqrt p$ is an element of the unit sphere $\cS^{d-1}$. See Exercise~\ref{ex:fr_isometry}.}
The preservation of mass ensures that any mass that is destroyed is also instantly created elsewhere. To adhere to the lexicon of transport, this phenomenon is sometimes referred to as \emph{teleportation}; however, it should be noted that it merely corresponds to a reweighting.

What about gradient flows? Given a functional $\cF : \cM_+(\R^d)\to\R$ or $\cF : \cP(\R^d)\to\R$, the gradient by definition satisfies $\partial_t \cF(\mu_t) = \langle \gradFR \cF(\mu_t), \alpha_t\rangle_{\mu_t}$ along every curve $\partial_t \mu_t = \alpha_t \mu_t$.
By unpacking the definitions, one checks (see Exercise~\ref{ex:fr_gradient}) that
\begin{align}\label{eq:fr_gradient}
    \gradFR \cF(\mu) = \delta \cF(\mu) \quad \text{or}\quad \gradFR \cF(\mu) =\delta \cF(\mu) - \int \delta \cF(\mu)\,\ud \mu
\end{align}
depending on whether we are working over $\cM_+(\R^d)$ or $\cP(\R^d)$ respectively; here, $\delta \cF$ denotes the first variation of $\cF$ (Definition~\ref{defn:first_variation}).
To disambiguate the two cases and to emphasize the original motivation of the WFR metric from unbalanced optimal transport, we refer to the former case as \emph{unbalanced Fisher--Rao} and the latter as simply \emph{Fisher--Rao}.
The Fisher-Rao gradient flow of $\cF$ follows the tangent vector $ -\gradFR \cF(\mu_t)$ at time $t$ and is  given by
\begin{align}\label{eq:fr_gf}
    \partial_t \mu_t = -\gradFR \cF(\mu_t)\,\mu_t\,.
\end{align}

\paragraph{Wasserstein--Fisher--Rao}\index{Wasserstein--Fisher--Rao (WFR)}
We now combine both the Wasserstein and Fisher--Rao geometries into a hybrid geometry that incorporates both mass transport and creation/destruction (a.k.a.\ reweighting, a.k.a.\ teleportation).
The idea is to simply consider the continuity equation with reaction, $\partial_t \mu_t + \divergence(\mu_t v_t) = \alpha_t \mu_t$. The governing equation is parameterized by a function $\alpha$ and a vector field $v$, which together form a tangent vector. It is then natural to consider the metric\footnote{Strictly speaking, to add the Wasserstein and Fisher--Rao geometries, we should add a factor of $\frac{1}{4}$ in front of the $\alpha^2$ term, and this is indeed the convention adopted in some works. We omit this factor for parsimony.}
\begin{align}\label{eq:wfr_pre_metric}
    \|(\alpha,v)\|_\mu^2
    &= \int (\alpha^2 + \|v\|^2)\,\ud \mu\,.
\end{align}
However, similarly to our discussion in Section~\ref{sec:continuity_eq}, this does not uniquely define a tangent vector because there is too much freedom to choose the pair $(\alpha,v)$ while maintaining the same evolution of measures ${(\mu_t)}_{t\ge 0}$.
It can be shown that the \emph{optimal} pair $(\alpha, v)$, in the sense of minimizing the norm~\eqref{eq:wfr_pre_metric} (c.f.\ the discussion in Section~\ref{sec:continuity_eq}) can be characterized as follows: $\alpha = \psi$ and $v = \nabla \psi$ for some function $\psi : \R^d\to\R$.
Hence, we can define the WFR tangent space at $\mu$ to be
\begin{align*}
    T_\mu \cM_+(\R^d) = \overline{\{(\psi,\nabla\psi) \mid \psi : \R^d\to\R~\text{compact supp., smooth}\}}^{L^2(\mu)}
\end{align*}
equipped with the norm
\begin{align}\label{eq:wfr_metric}
    \|(\psi,\nabla\psi)\|_\mu^2
    &\deq \int (\psi^2 + \|\nabla \psi\|^2)\,\ud \mu\,.
\end{align}
This has the pleasing interpretation of ``completing'' the Wasserstein metric $\|\nabla\psi\|_{L^2(\mu)}^2$ to the full Sobolev norm of $\psi$.
Note also that the governing equation becomes
\begin{align}\label{eq:reaction_transport}
    \partial_t \mu_t +\divergence(\mu_t\nabla \psi_t) = \psi_t \mu_t\,.
\end{align}

As before, we can also restrict to the space of probability measures $\cP(\R^d)$, in which case we restrict to pairs $(\psi,\nabla \psi)$ such that $\int \psi\,\ud \mu = 0$, endowed with the same metric~\eqref{eq:wfr_metric}.
We refer to WFR over the full space $\cM_+(\R^\dd)$ as \emph{unbalanced WFR}, henceforth reserving the use of WFR for the restriction to $\cP(\R^\dd)$.

The following theorem computes the  WFR gradient.

\begin{theorem}
    Let $\cF$ be a functional over $\cM_+(\R^d)$ or $\cP(\R^d)$.
    Then, unbalanced WFR gradient of $\cF$, denoted $\gradWFR \cF$, is given by
    \begin{align*}
        \gradWFR\cF(\mu) = \bigl(\delta \cF(\mu), \nabla \delta \cF(\mu)\bigr)
    \end{align*}
    and the WFR gradient by
    \begin{align*}
        \gradWFR\cF(\mu) = \Bigl(\delta \cF(\mu) - \int \delta \cF(\mu)\,\ud \mu, \nabla \delta \cF(\mu)\Bigr)\,.
    \end{align*}
\end{theorem}
\begin{proof}
    Let ${(\mu_t)}_{t\ge 0}$ satisfy~\eqref{eq:reaction_transport}.
    Then, by integration by parts,
    \begin{align*}
        \partial_t \cF(\mu_t)
        &= \int \langle\nabla \delta \cF(\mu_t), \nabla \psi_t \rangle \,\ud \mu_t + \int \delta \cF(\mu_t)\,\psi_t\,\ud \mu_t\,.
    \end{align*}
    In the unbalanced case, we can identify this as
    \begin{align*}
        \langle (\delta \cF(\mu), \nabla \delta \cF(\mu)), (\psi_t, \nabla \psi_t)\rangle_{\mu_t}
    \end{align*}
    according to the definition of the metric~\eqref{eq:wfr_metric}.
    In the balanced case, we have $\int \psi_t \,\ud \mu_t = 0$, and since the WFR gradient is by definition an element of the tangent space its first component must also have mean zero, so the claim follows.
\end{proof}

The WFR gradient flow is therefore given by
\begin{align}\label{eq:wfr_gf}
    \partial_t \mu_t
    &= \divergence\bigl(\mu_t\,\nabla \delta\cF(\mu_t)\bigr) - \Bigl(\delta \cF(\mu_t) - \int \delta \cF(\mu_t)\,\D \mu_t\Bigr)\,\mu_t\,.
\end{align}

\section{Mean-field particle systems}\label{sec:mean_field_particle_systems}

We conclude this chapter by describing Wasserstein and WFR gradient flows from a particle systems perspective. These arise naturally when these gradient flows are initialized at finite measures. Indeed, a key observation is that since both gradient flows can be implemented using ordinary differential equations, if $\mu_0$ is a finite measure, $\mu_t$ remains a finite measure at all times $t$ along these gradient flows.

\subsection{Particle Wasserstein gradient flow}\label{sec:mf_WGF}

To illustrate this point,  let $\cF$ be a function over $\cP(\R^d)$ and recall from Definition~\ref{def:w2_grad_flow} that the Wasserstein gradient flow of $\cF$ is the continuity equation associated with the ODE
\begin{equation}\label{eq:mean_field_wgf}
            \dot X_t
    = -\gradW \cF(\mu_t)(X_t)\,,
\end{equation}
where $\mu_t$ denotes the law of $X_t$. In particular, we only need to describe these dynamics on the support of $\mu_t$.

Assume now that the Wasserstein gradient flow is initialized at 
$$
\mu_0 \deq  \frac1N \sum_{j=1}^N \delta_{X_0^{j}}\,,
$$
for a given collection of points $X_0^{1}, \ldots, X_0^N \in \R^d$. We get that
$$
\mu_t \deq  \frac1N \sum_{j=1}^N \delta_{X_t^{j}}\,,
$$
where for $i \in [N]$,
\begin{equation}\label{eq:part_wgf}
    \boxed{\dot X_t^i = -\gradW \cF(\mu_t)(X_t^i)\,.}
\end{equation}
These dynamics describe an interacting particle system where particles $(X_t^1, \ldots, X_t^N)$ are subject to dynamics of the form
\begin{equation}\label{eq:gen_int_part}
    \dot X_t^i = V_t^i(X_t^1, \ldots, X_t^N)\,, \quad i \in [N]\,. 
\end{equation}
Note that in the case of Wasserstein gradient flows, we further have that:
\begin{itemize}
    \item[(a)] each particle $X_t^i$ interacts with the others only through the effect of their distribution $\mu_t$, and
    \item[(b)] these interactions have the same form for all the particles. 
\end{itemize} 
Slightly overloading notation, this means that the general dynamics in~\eqref{eq:gen_int_part} simplify to
$$
  \dot X_t^i=V_t^i(X_t^1, \ldots, X_t^N)\stackrel{\text{(a)}}{=}  V_t^i(X_t^i, \mu_t) \stackrel{\text{(b)}}{=} V_t(X_t^i, \mu_t)\,, \qquad i \in [N]\,.
$$
These two properties are precisely captured by~\eqref{eq:mean_field_wgf}: the first one is obvious and the second one is manifest due to the absence of a superscript $i$, which indicates that each particle is subject to the same vector field. Such a system is said to exhibit \emph{mean-field interactions}. Both the Wasserstein and WFR gradient flows are of this form.

Mean-field interaction systems are convenient because it is strictly equivalent to describe the dynamics of each particle and that of their distribution. The latter takes the form of a PDE given by the continuity equation~\eqref{eq:continuity}.

Since the Wasserstein gradient flow only moves particles, the weights in $\mu_0$ do not change over time: if the Wasserstein gradient flow is initialized at 
\begin{equation}\label{eq:finite_init}
    \mu_0 \deq  \sum_{j=1}^N w_0^{j}\delta_{X_0^{j}}\,,
\end{equation}
where $w_0^{j}\ge 0,\, j \in [N]$ and $\sum_{j=1}^N w_0^{j}=1$,
then
$$
\mu_t \deq  \sum_{j=1}^N w_0^{j}\delta_{X_t^{j}}\,,
$$
where $X_t^1, \ldots, X_t^N$ evolve according to~\eqref{eq:part_wgf}. 
To also impose dynamics on the weights, we employ instead a WFR gradient flow.

\subsection{Particle WFR gradient flow}\label{subsec:particle_wfr}

Recall that tangent vector fields for the Wasserstein space are displacement maps of the form $\nabla \psi$. The Wasserstein--Fisher--Rao (WFR) geometry
reinterprets the tangent space by replacing the governing continuity equation~\eqref{eq:continuity} with the reaction-transport equation~\eqref{eq:reaction_transport}.
In particular, it offer the possibility of traversing the space of probability measures, say from initial distribution to target distribution, more efficiently by reweighting particles rather than having to move them across the space in a continuous fashion. When initialized at a finite measure of the form~\eqref{eq:finite_init}, this effect manifests itself in the form of time-varying weights:
$$
\mu_t \deq  \sum_{j=1}^N w_t^{j}\delta_{X_t^{j}}\,,
$$
where $w_t^{j}\ge 0,\, j \in [N]$ and $\sum_{j=1}^N w_t^{j}=1$. 

The particle updates follow the Wasserstein geometry, and the weight updates follow the Fisher{--}Rao geometry: for $i\in [N]$,
\begin{align}\label{eq:wfr_particle}
    \boxed{\begin{aligned}
        \dot X^i_t
        &= -\nabla\delta\cF(\mu_t)(X_t^i)\,, \\
        \dot w^i_t
        &= -\Bigl(\delta \cF(\mu_t)(X_t^i) - \int \delta \cF(\mu_t)\,\ud \mu_t\Bigr)\,w^i_t\,.
    \end{aligned}}
\end{align}

\subsection{Gaussian particles}\label{ssec:gaussian_particles}

Following Section~\ref{sec:gaussian_mixtures}, we can also take a finite Gaussian mixture with mixing measure
\begin{align*}
    \nu_t
    &= \frac{1}{K} \sum_{k=1}^K \delta_{(m^k,\Sigma^k)}
\end{align*}
that evolves according to the Wasserstein gradient flow for the functional $\nu \mapsto \cG(\nu) = \cF(G_\nu)$.
By Theorem~\ref{thm:gm_flow}, this flow takes the following form: for each $k\in [K]$,
\begin{align}\label{eq:gm_particle_flow}
    \boxed{\begin{aligned}
        \dot m_t^k
        &= -\E\nabla \delta \cF(G_{\nu_t})(X_t^k)\,, \\
        \dot \Sigma_t^k
        &= -\E\nabla^2 \delta \cF(G_{\nu_t})(X_t^k) \,\Sigma_t^k - \Sigma_t^k\,\E\nabla^2 \delta \cF(G_{\nu_t})(X_t^k)\,,
    \end{aligned}}
\end{align}
where $X_t^k \sim \cN(m_t^k, \Sigma_t^k)$.
Note that this is an interacting system of ``particles'' $(m_t^k, \Sigma_t^k)$, but each particle corresponds to a Gaussian component $\cN(m_t^k, \Sigma_t^k)$, and the collection thereof to the Gaussian mixture $\nu_t$.
We therefore refer to $\cN(m_t^k,\Sigma_t^k)$ as a \emph{Gaussian particle}.

We emphasize that these dynamics do \emph{not} implement the Wasserstein gradient flow for $\cF$. Nevertheless, these dynamics are perfectly valid for minimizing $\cF$ over the space of $K$-component Gaussian mixtures.

Recall that in Section~\ref{sec:gaussian_mixtures}, we equipped the space of probability measures over $\BW(\R^d)$---i.e., the space of mixing measures---with the Wasserstein geometry.
But we could have equally well considered equipping this space with the WFR geometry.
The corresponding particle dynamics evolves the finite Gaussian mixture
\begin{align*}
    \nu_t = \sum_{k=1}^K w_t^k \delta_{(m_t^k, \Sigma_t^k)}
\end{align*}
with changing weights, governed by
\begin{align*}
    \boxed{\begin{aligned}
        \dot m_t^k
        &= -\E\nabla \delta \cF(G_{\nu_t})(X_t^k)\,, \\
        \dot \Sigma_t^k
        &= -\E\nabla^2\delta \cF(G_{\nu_t})(X_t^k) \,\Sigma_t^k - \Sigma_t^k\,\E\nabla^2\delta \cF(G_{\nu_t})(X_t^k)\,, \\
        \dot w_t^k
        &= -\Bigl(\E\delta \cF(G_{\nu_t})(X_t^k) - \frac{1}{K} \sum_{k'=1}^K \E\delta \cF(G_{\nu_t})(X_t^{k'})\Bigr)\, w_t^k\,,
    \end{aligned}}
\end{align*}
where $X_t^k \sim \cN(m_t^k, \Sigma_t^k)$.

\subsection{Implementation strategies for gradient flows}

For both the Wasserstein gradient flow and the WFR gradient flow, one needs to compute the Wasserstein gradient $\gradW \cF\bigl(\mu_t\bigr)=\nabla\delta\cF(\mu_t)$ on the support of $\mu_t$. When $\mu_t$ is a discrete measure, this quantity may not be well-defined. This is the case for example when $\cF$ is the entropy functional which is itself not defined on discrete measures, let alone its Wasserstein gradient.
(Note, however, that it may be well-defined when we use Gaussian particles.)

In practice, the particle implementations discussed here typically needs to be combined with other tricks (e.g., ``kernelization'' as in Subsection~\ref{sec:svgd}). These implementation strategies are described in the next chapter.

\section{Discussion}

\noindent\textbf{\S\ref{sec:continuity_eq}.} Detailed treatments of the metric derivative and the continuity equation can be found in~\cite{AmbGigSav08, Vil09, San15}.
In particular, a rigorous version of Theorem~\ref{thm:fundamental_otto} can be found in~\cite[Chapter 8]{AmbGigSav08} or~\cite[Chapter 5]{San15}.

\noindent\textbf{\S\ref{sec:riem}.} There are many excellent textbooks covering Riemannian geometry, e.g.,~\cite{DoC1992Riem}.

\noindent\textbf{\S\ref{sec:riem_wass}.} The Benamou--Brenier formula is often called the ``dynamical'' formulation of optimal transport (as opposed to Chapter~\ref{chap:OT}, which describes the ``static'' picture).
There is also a dynamical version of the dual problem, in which the dual potentials evolve according to the \emph{Hamilton--Jacobi equation}\@; see~\cite[Section 8.1]{Vil03}. The dynamical version of entropic optimal transport, introduced in Chapter~\ref{chap:entropic}, is closely tied to the well-known \emph{Schr\"odinger bridge} problem~\cite{Leo14Schrodinger, CheGeoPav21}.

\noindent\textbf{\S\ref{sec:otto_calc}.} The formal calculation rules described in this section were first laid out by Otto~\cite{Ott01}, although some of the ideas were already anticipated in the earlier work of~\cite{Laf1988DensityManifold}.

The tangent space at a measure $\mu$, together with its metric, can be viewed as a linearization of the geometric structure at $\mu$. This gives rise to ``linearized optimal transport'' which has formed the basis for numerous applications~\cite{Wan+13LinearOT, BasKolRoh14Microscopy, KolRoh15SuperRes, SegCut15, Kol+16Pattern, Big+17GeoPCA, ParTho18OT, Cai+20LinOT}.

Besides gradient flows, there have also been proposals for adaptations of other optimization algorithms to the Wasserstein space, e.g.,~\cite{ChoLiZho20Ham, WanLi20Newton, WanLi22Accel, Tan23Accel, Che+24Accel}.

\noindent\textbf{\S\ref{sec:bw}.} The Bures{--}Wasserstein geometry is named after Donald Bures~\cite{Bures1969Kakutani}, who introduced this metric over the PSD cone in his work on quantum information theory.
BW geometry is further explored in~\cite{Modin17OT, BhaJaiLim19BW, Han+21BW, Van22BW}.
The connection with the Burer--Monteiro factorization, as described in Exercise~\ref{ex:burer_monteiro}, has been explored in the context of low-rank matrix recovery~\cite{LuoGar22Matrix, MauLeGRig23LowRank}.

\noindent\textbf{\S\ref{sec:gaussian_mixtures}.}
As mentioned in the main text, the geometry described in this section was first considered in~\cite{CheGeoTan19, DelDes20GMM}, although the gradient flow equations were obtained in~\cite{Lametal22GVI}.

\noindent\textbf{\S\ref{sec:wfr}.} The Fisher{--}Rao geometry is well-studied in information geometry~\cite{AmaNag00, Ay+17InfoGeo}.

\noindent\textbf{\S\ref{sec:mean_field_particle_systems}.}
In some sources, such as~\cite{LieMieSav16HK}, WFR gradient flows are written in terms of the square root of the weights, i.e., in terms of $r \deq \sqrt{w}$.
This convention is motivated by the fact that the FR distance corresponds to the Euclidean distance between the square roots of the weights, and the WFR distance can therefore be interpreted as a coupling cost on the space of $(r, x)$ pairs equipped with a ``cone'' metric (c.f.\ Subsection~\ref{subsec:cone}). Since we do not cover this perspective here, we adopt the more straightforward parametrization in terms of $w$.

The use of Gaussian particles was first advocated in~\cite{Lametal22GVI}.

\section{Exercises}

In the following exercises, you may use the following formula for the Wasserstein gradient of the squared Wasserstein distance:
\begin{align}\label{eq:grad_of_w2}
    [\gradW W_2^2(\cdot, \nu)](\mu)
    &= 2\,({\id} - T_{\mu\to\nu})\,.
\end{align}
We do not give the full proof of~\eqref{eq:grad_of_w2} here, but it is straightforward to establish the upper bound (Exercise~\ref{ex:w2_grad_upper}).

\begin{enumerate}
    \item Let $X_0\sim \mu_0$, where $\mu_0$ is the standard Gaussian over $\R^2$. For $t\ge 0$, let $X_t=R_tX_0$ where $R_t$ is a rotation by $\theta(t)$ radians. Compute the vector field $v_t$ such that $\dot X_t = v_t(X_t)$ and show that $\divergence(\mu_0 v_t)=0$ for all $t\ge 0$ (and hence that $X_t \sim \mu_0$ for all $t\ge 0$).
    
    \item\label{ex:prod_measures_cvx} Show that the set of product measures over $\R^d$ is a geodesically convex subset of $\cP_2(\R^d)$.

    \item\label{ex:w2_grad_upper} Suppose that ${(X_t)}_{t\in\R}$ follows the ODE $\dot X_t = v_t(X_t)$, so that $\mu_t = \law(X_t)$ evolves according to the continuity equation $\partial_t \mu_t + \divergence(\mu_t v_t) = 0$, and suppose that $\mu_0 = \mu$.
    Prove that
    \begin{align*}
        \limsup_{h\searrow 0}\frac{W_2^2(\mu_h, \nu) - W_2^2(\mu_0,\nu)}{h}
        \le 2\,\langle {\id} - T_{\mu\to\nu}, v_0 \rangle_\mu\,.
    \end{align*}
    \emph{Hint:} Let $X_0\sim \mu$ and $Y\sim \nu$ be optimally coupled.
    
    \item\label{ex:weakly_cvx_rate} Let $\cF : \cP_{2,\rm ac}(\R^d)\to \R\cup \{\infty\}$ be a geodesically convex functional which is minimized at $\pi$.
    Let ${(\mu_t)}_{t\ge 0}$ denote the Wasserstein gradient flow for $\cF$.
    By differentiating $t\mapsto 2t\,\cF(\mu_t) + W_2^2(\mu_t, \pi)$, prove
    \begin{align*}
        \cF(\mu_t) - \inf \cF
        &\le \frac{W_2^2(\mu_0,\pi)}{2t}\,.
    \end{align*}
    
    \item \label{ex:otto_villani} Let $\cF : \cP_{2,\rm ac}\to\R\cup\{\infty\}$ be a functional.
    We saw that $\alpha$-strong convexity of $\cF$ implies the Polyak{--}\L{}ojasiewicz (P\L{}) inequality
    \begin{align}\label{eq:pl_ineq}
        \|\gradW \cF(\mu)\|_\mu^2
        &\ge 2\alpha\,(\cF(\mu) - \inf \cF)\,, \qquad \forall \mu \in \cP_{2,\rm ac}(\R^d)\,.
    \end{align}
    \begin{enumerate}
        \item Show that~\eqref{eq:pl_ineq} implies the quadratic growth inequality
            \begin{align}\label{eq:quadratic_growth}
                \cF(\mu) - \inf \cF
                &\ge \frac{\alpha}{2}\,W_2^2(\mu,\pi)\,, \qquad \forall \mu \in \cP_{2,\rm ac}(\R^d)\,,
            \end{align}
            where $\pi$ is the minimizer of $\cF$. This is known as the \emph{Otto{--}Villani theorem} after~\cite{OttVil00LSI}.

            \emph{Hint}: Differentiate $t\mapsto \sqrt{\frac{\alpha}{2}}\, W_2(\mu_t,\mu_0) + \sqrt{\cF(\mu_t) - \inf \cF}$ along the Wasserstein gradient flow of $\cF$.
            You may assume that the gradient flow converges to $\pi$, which is a consequence of~\eqref{eq:quadratic_growth} if $\cF$ is uniquely minimized.
        \item In general,~\eqref{eq:quadratic_growth} does not imply~\eqref{eq:pl_ineq}.
            However, prove that when $\cF$ is geodesically convex, then~\eqref{eq:quadratic_growth} implies~\eqref{eq:pl_ineq} but with $\alpha$ replaced by $\alpha/4$.
    \end{enumerate}

    \item Suppose that instead of the PL inequality~\eqref{eq:pl_ineq}, we instead have the inequality
    \begin{align*}
        \|\gradW \cF(\mu)\|_\mu^p \ge c\,(\cF(\mu) - \inf \cF)\,, \qquad\forall \mu \in \cP_{2,\rm ac}(\R^\dd)\,,
    \end{align*}
    for some power $0 < p < 2$.
    Show that the Wasserstein gradient flow dissipates $\cF$ at a polynomial rate: $\cF(\mu_t) - \inf \cF = O(1/t^{\frac{2}{p}-1})$.
    What happens in the case $p > 2$?
    
    \item Consider $\cF : \cP_{2,\rm ac}(\R^d)\to\R$ which is the operator norm of the second moment matrix:
            \begin{align*}
                \cF(\mu) \deq \Bigl\lVert \int xx^\T \, \mu(\ud x) \Bigr\rVert_{\rm op}\,.
            \end{align*}
            Prove that $\cF$ is geodesically convex.
            
    \item \begin{enumerate}
        \item Compute the Wasserstein gradient of the chi-squared divergence $\chi^2(\cdot \mmid \pi)$ at $\mu$.
            Recall that $\chi^2(\mu \mmid \pi) \deq \int \frac{\D\mu}{\D\pi} \, \D \mu - 1$.
            Also, write down the equation for the Wasserstein gradient flow of the chi-squared divergence.
        \item Prove that when $\pi$ is log-concave, then $\chi^2(\cdot \mmid \pi)$ is geodesically convex.
    \end{enumerate}
    
    \item\label{ex:gen_geod_cvx} We say that a functional $\cF : \cP_{2,\rm ac}(\R^d)\to\R\cup\{\infty\}$ is $\alpha$-convex along \emph{generalized geodesics} if for all triples $\mu_0,\mu_1,\nu \in \cP_{2,\rm ac}(\R^d)$, if we define the generalized geodesic joining $\mu_0$ to $\mu_1$ with base $\nu$ via
    \begin{align*}
        \mu_t^\nu
        &\deq {[(1-t)\,T_{\mu_0\to\nu} + t\,T_{\mu_1\to\nu}]}_\# \nu\,,
    \end{align*}
    then it holds:
    \begin{align*}
        \cF(\mu_t^\nu)
        &\le (1-t)\,\cF(\mu_0) + t\,\cF(\mu_1) - \frac{\alpha\,t\,(1-t)}{2}\,W_2^2(\mu_0,\mu_1)\,.
    \end{align*}
    \begin{enumerate}
        \item Explain why, if $\cF$ is $\alpha$-convex along generalized geodesics, then it is $\alpha$-strongly convex.
            Also, explain why being $\alpha$-convex along generalized geodesics is equivalent to the mapping $\cF \circ \exp_\nu$ being $\alpha$-strongly convex on the tangent space $T_\nu \cP_{2,\rm ac}(\R^d)$.
        \item Show that for $\pi \propto \exp(-V)$ where $V$ is $\alpha$-strongly convex, then $\KL(\cdot \mmid \pi)$ is $\alpha$-convex along generalized geodesics (this strengthens the convexity result of Corollary~\ref{cor:kl_cvx}).
        \item Show that for any $\mu_0,\mu_1,\nu \in \cP_{2,\rm ac}(\R^d)$, there \emph{exists} at least one generalized geodesic joining $\mu_0$ to $\mu_1$, along which $\frac{1}{2}\,W_2^2(\cdot, \nu)$ is $1$-strongly convex.
    \end{enumerate}
    \emph{Remark}: Generalized geodesics play an important role in studying the geometry of the Wasserstein space. The result of the third part of this question was used to show existence of the minimizing movements scheme, which in turn is used to rigorously construct Wasserstein gradient flows; see~\cite{AmbGigSav08} for further reading.

    \item Compute the Wasserstein geodesic joining two Gaussians.

    \item\label{ex:wgf_of_w2} Use~\eqref{eq:grad_of_w2} to show that if ${(\mu_t)}_{t\ge 0}$ is the Wasserstein gradient flow of $\frac{1}{2}\,W_2^2(\cdot,\nu)$, then $t\mapsto \mu_{1-\exp(-t)}$ is the constant-speed Wasserstein geodesic joining $\mu_0$ to $\nu$.

    \item\label{ex:burer_monteiro} Let $\cF$ be a functional over $\cP_2(\R^d)$, and consider the functional $(m, U) \mapsto F(m,U) \deq \cF(\cN(m,UU^\T))$.
    Show that the \emph{Euclidean gradient flow} for $F$ over $\R^d \times \R^{d\times d}$ yields the same dynamics (up to rescaling time) as the Bures--Wasserstein gradient flow~\eqref{eq:bw_gf} where $\Sigma = UU^\T$.
    Similarly, show that the Euclidean gradient flow of $(m^1,\dotsc,m^K, U^1,\dotsc,U^K) \mapsto \cF(\frac{1}{K}\sum_{k=1}^K \cN(m^k, U^k (U^k)^\T))$ recovers the Gaussian mixture flow~\eqref{eq:gm_particle_flow} (up to rescaling time).

    \emph{Remark}: The parametrization $\Sigma = UU^\T$ is often referred to as the \emph{Burer--Monteiro parametrization}\index{Burer--Monteiro}, especially when it is used to constrain $\Sigma$ to have low rank~\cite{BurMon03LowRank, BurMon05LocalMin}.

    \item\label{ex:fr_isometry} Let $p$ be an element in the interior of the simplex, i.e., a strictly positive probability distribution over the finite alphabet $\{1,\dotsc,d\}$.
    Consider the isometry $f : p\mapsto \sqrt p$ that maps $p$ to an element of the sphere: $\sqrt p \in \cS^{d-1}$.
    Show that under this isometry, a tangent vector $\dot p$ on the simplex is mapped to $v = \dot p/(2\sqrt p)$ and conclude that $\dot p$ is tangent to the simplex (i.e., $\sum_{i\in [d]} \dot p_i = 0$) if and only if $v$ is tangent to the sphere (i.e., $\sqrt p \perp v$).

    \item\label{ex:fr_gradient} Verify the expressions~\eqref{eq:fr_gradient} for the (unbalanced) Fisher--Rao gradient of a functional.
    
    \item Compute the FR and WFR gradients of the functionals listed in Examples~\ref{ex:potential},~\ref{ex:internal}, and~\ref{ex:interaction}.
    
    \item Show that the FR gradient flow~\eqref{eq:fr_gf} and the WFR gradient flow~\eqref{eq:wfr_gf} maintain the property that $\mu_t$ is a probability measure for all $t\ge 0$.
    
    \item Show that~\eqref{eq:wfr_particle} indeed follows the WFR gradient flow~\eqref{eq:wfr_gf}.
\end{enumerate}

\chapter{Wasserstein gradient flows: applications}
\label{chap:appWGF}

In the previous chapter, we developed a Riemannian structure on the space $(\cP_2(\R^d), W_2)$ in order to define Wasserstein and WFR gradient flows. In this chapter, we use these gradient flows as optimization algorithms over the space of probability measures for various tasks arising in statistics and machine learning. Each task corresponds to choosing a specific functional $\cF$ over this space. In particular, akin to the notion of convexity in classical optimization (see, e.g.,~\cite{Bub15}), the notion of geodesic convexity is instrumental in deriving rates of convergence.

\section{Variational inference}\label{sec:vi}\index{variational inference (VI)}

As our first application of gradient flow theory, we consider a rich source of optimization problems over the space of measures arising from the burgeoning field of \emph{variational inference} (VI)~\cite{JorGhaJaa99, WaiJor08, BleKucMcA17}.
In VI\@, we posit access to a probability measure $\pi$ over $\R^d$ via an expression for its density, and our goal is to perform inference.
A typical example arises when $\pi$ is the posterior distribution from a Bayesian inference problem, in which case VI is also known as \emph{variational Bayes}, and it has gradually emerged as an appealing computational counterpoint to traditional Markov chain Monte Carlo (MCMC) methods, which we study in Section~\ref{sec:sampling}.

The idea of VI is to approximate $\pi$ with an element of a simpler class $\cQ$ of probability measures by solving the optimization problem
\begin{align}\label{eq:VI}\tag{$\msf{VI}$}
    q_\star = \argmin_{q\in \cQ} \KL(q \mmid \pi)\,.
\end{align}
Although a plethora of variants have been proposed which replace the KL divergence with other objectives\footnote{Including the KL divergence with the order of arguments swapped, which is closer to the statistician's concept of maximum likelihood.}, the one we present here is particularly popular in practice.
This is because typically we do not have access directly to $\pi$ but rather to an unnormalized density $\tilde \pi$, and the unknown normalization constant $Z = \int \tilde \pi$ does not affect the optimization objective in~\eqref{eq:VI}.

This modelling choice also has fortuitous consequences for the convexity of the VI problem over the Wasserstein space, leading to the development of flow-based algorithms.

\subsection{Convexity of the VI problem}

In order to apply the gradient flow machinery to~\eqref{eq:VI}, we are inexorably led to our next undertaking: studying the geodesic convexity of the KL divergence over the Wasserstein space.

Henceforth, we always assume that $\pi$ admits a density of the form $\pi \propto \exp(-V)$, where $V : \R^d\to\R$ is called the potential function.
The first observation is that the KL divergence decomposes into a sum of two functionals:
\begin{align}\label{eq:kl_decomp}
    \cF(\mu) \deq \KL(\mu\mmid \pi)
    &= \int \mu \log \frac{\mu}{\pi}
    = {\underbrace{\int V\,\ud \mu}_{\eqqcolon \cV(\mu)}} + {\underbrace{\int \mu \log \mu}_{\eqqcolon \cH(\mu)}} + {\text{const.}}
\end{align}
where $\cV$ is the \emph{potential energy}, $\cH$ is the \emph{entropy}, and ``const.'' denotes an additive constant that does not depend on $\mu$ (and hence is irrelevant for studying properties of the gradient flow).

In Examples~\ref{ex:potential} and~\ref{ex:internal}, we have already computed the Wasserstein gradients $\gradW\cV(\mu) = \nabla V$ and $\gradW \cH(\mu) = \nabla \log \mu$. Adding these together, we obtain $\gradW\cF(\mu) = \nabla \log \mu + \nabla V = \nabla \log(\mu/\pi)$.
From Definition~\ref{def:w2_grad_flow}, the Wasserstein gradient flow of $\cF$ solves
\begin{align}\label{eq:wgf_kl}
    \partial_t \mu_t
    &= \divergence\bigl(\mu_t\,\nabla \log \frac{\mu_t}{\pi}\bigr)\,.
\end{align}

Leveraging~\eqref{eq:kl_decomp}, we can study the convexity of the two functionals $\cV$ and $\cH$ separately.
The potential energy is straightforward.

\begin{theorem}\label{thm:potential_cvx}
    Suppose that $V : \R^d\to\R$ is $\alpha$-convex on $\R^d$.
    Then, the corresponding potential energy functional $\cV$ defined by $\cV(\mu) \deq \int V\,\ud \mu$ is $\alpha$-geodesically convex on $\cP_{2,\rm ac}(\R^d)$.
\end{theorem}
\begin{proof}
    We use the second-order condition from Section~\ref{sec:riem}.
    Namely, let $X_t\sim \mu_t$ for $t\in [0,1]$, where ${(\mu_t)}_{t\in [0,1]}$ is a Wasserstein geodesic; thus, $X_t = (1-t)\,X_0 + t\,T(X_0)$ where $T$ is the optimal transport map from $\mu_0$ to $\mu_1$.
    We compute
    \begin{align*}
        \gradW^2\cV(\mu_0)[T-{\id}, T-{\id}]
        &= \partial_t^2 \cV(\mu_t)\big|_{t=0}
        = \partial_t^2 \E V(X_t)\big|_{t=0} \\
        &= \E\langle T(X_0) - X_0, \nabla^2 V(X_0)\,(T(X_0) - X_0)\rangle \\
        &\ge \alpha\, \E[\|T(X_0) - X_0\|^2] \\
        &= \alpha\,\|T-{\id}\|_{\mu_0}^2\,,
    \end{align*}
    where we used the assumption $\nabla^2 V \succeq \alpha I$.
\end{proof}

Next, we show that the entropy $\cH$ is geodesically convex.
For this, we invoke the change of variables formula.

\begin{lemma}[Change of variables]\label{lem:change_of_var}\index{change of variables}
    Let $\mu$ be a density on $\R^d$, let $T : \R^d\to\R^d$ be a diffeomorphism, and let $\nu \deq T_\# \mu$.
    Then, $\nu$ has density given by
    \begin{align*}
        \nu(T(x)) = \frac{\mu(x)}{|\det \nabla T(x)|}\,.
    \end{align*}
\end{lemma}

Recall the mnemonic for memorizing this rule: under the change of variables $y = T(x)$, one has $\ud y = |\det \nabla T(x)|\,\ud x$, since the Jacobian determinant $|\det \nabla T(x)|$ measures the volume distortion of the map $T$.
The pushforward satisfies, by definition, $\int \varphi\circ T \, \ud \mu = \int \varphi\,\ud \nu$ for all test functions $\varphi$.
We can write this as
\begin{align*}
    \int \varphi(T(x)) \,\mu(x)\,\ud x
    &= \int \varphi(y) \, \nu (y) \, \ud y \\
    &= \int \varphi(T(x))\,\nu(T(x))\,|\det \nabla T(x)|\,\ud x
\end{align*}
and Lemma~\ref{lem:change_of_var} follows.
In applications, we do not always know that optimal transport maps are diffeomorphisms, but nevertheless a variant of Lemma~\ref{lem:change_of_var} still holds, and we refer to~\cite[Theorem 4.8]{Vil03} for the technical details.

The change of variables formula furnishes the quickest proof of geodesic convexity of $\cH$.

\begin{theorem}\label{thm:entropy_cvx}
    The entropy functional $\cH$, given by $\cH(\mu) \deq \int \mu \log \mu$, is geodesically convex on $\cP_{2,\rm ac}(\R^d)$.
\end{theorem}
\begin{proof}
    Again let ${(\mu_t)}_{t\in [0,1]}$ be a Wasserstein geodesic and let $T_t \deq (1-t)\,{\id} + t\,T$, so that $\mu_t = (T_t)_\# \mu_0$.
    Then, by Lemma~\ref{lem:change_of_var} (which we apply blithely despite not knowing that $T_t$ is a diffeomorphism),
    \begin{align*}
        \cH(\mu_t)
        &= \int (\log \mu_t) \, \ud \mu_t
        = \int \log(\mu_t \circ T_t)\,\ud \mu_0
        = \int \log \frac{\mu_0}{\det \nabla T_t} \, \ud \mu_0 \\
        &= \cH(\mu_0) - \int \log \det \nabla T_t \, \ud \mu_0\,.
    \end{align*}
    It is a standard exercise to show that $-\log \det$ is convex over the positive definite cone, and $t\mapsto \nabla T_t$ is affine; therefore, the composition $t\mapsto - \log \det \nabla T_t$ is convex.
    In turn, it shows that $t\mapsto \cH(\mu_t)$ is convex, which is what we wanted to show.
\end{proof}

\begin{corollary}\label{cor:kl_cvx}
    Let $\pi \propto\exp(-V)$ be a density, where $V : \R^d\to\R$ is $\alpha$-convex.
    Then, the functional $\cF\deq \KL(\cdot \mmid \pi)$ is $\alpha$-geodesically convex on $\cP_{2,\rm ac}(\R^d)$.
\end{corollary}

Distributions $\pi \propto \exp(-V)$ for which $V$ is strongly convex are known as \emph{strongly log-concave distributions}. Therefore, we have shown that the KL divergence w.r.t.\ a (strongly) log-concave measure is (strongly) geodesically convex.

For~\eqref{eq:VI}, our goal is to minimize the KL divergence over a subset $\cQ \subseteq \cP_{2,\rm ac}(\R^d)$.
If $\cQ$ is geodesically convex (see Section~\ref{sec:riem}), then we immediately obtain the following corollary.

\begin{corollary}\label{cor:gconvexQ}
    Let $\pi \propto \exp(-V)$ be a density on $\R^d$, where $V$ is $\alpha$-convex.
    Let $\cQ \subseteq \cP_{2,\rm ac}(\R^d)$ be geodesically convex.
    Then, $\KL(\cdot \mmid \pi)$ is $\alpha$-geodesically convex over $\cQ$.

    In particular, the solution $q_\star$ to~\eqref{eq:VI} is unique.
\end{corollary}

Recall from Lemma~\ref{lem:strcvx_implies_pl} and Section~\ref{sec:otto_calc} that $\alpha$-convexity implies a PL inequality, which in turn implies rapid convergence for the gradient flow:

\begin{corollary}\label{cor:vi_convergence}
    Let $\pi \propto \exp(-V)$ be a density on $\R^d$, where $V$ is $\alpha$-convex.
    Let $\cQ \subseteq \cP_{2,\rm ac}(\R^d)$ be geodesically convex.
    Then, the Wasserstein gradient flow ${(q_t)}_{t\ge 0}$ of $\KL(\cdot\mmid \pi)$ constrained to lie in $\cQ$ satisfies
    \begin{align*}
        \KL(q_t \mmid \pi) - \KL(q_\star \mmid \pi)
        &\le e^{-2\alpha t}\,\{\KL(q_0 \mmid \pi) - \KL(q_\star \mmid \pi)\}\,.
    \end{align*}
\end{corollary}

In the sequel, our aim is show how the constrained Wasserstein gradient flow can be implemented in several important cases.

\subsection{Gaussian VI}\label{sec:gvi}\index{variational inference (VI)!Gaussian}

In this section, we study the problem of \emph{Gaussian VI}\@, in which the variational family $\cQ$ consists of all non-degenerate Gaussian measures over $\R^d$.
This family is simple yet abundantly motivated: if we can approximate $\pi \propto\exp(-V)$ by a Gaussian $\cN(m, \Sigma)$, then the parameters $(m,\Sigma)$ of the Gaussian are a reasonable guess for the mean and covariance matrix of $\pi$, which already suffice to construct credible regions.
The \emph{Laplace approximation} takes $m = \theta^\star$ and $\Sigma = {[\nabla^2 V(\theta^\star)]}^{-1}$ where $\theta^\star = \argmin V$ is the mode of $\pi$.
When $\pi$ is a Bayesian posterior, the validity of this approximation can be justified in the large-sample limit by the Bernstein--von Mises theorem.

To go beyond the Laplace approximation, we can ask for the \emph{optimal} Gaussian approximation, which is formulated as the VI problem
\begin{align}\label{eq:gvi}\tag{$\msf{GVI}$}
    q_\star = \argmin_{q\in \BW(\R^d)}\KL(q\mmid \pi)
\end{align}
where $\BW(\R^d)$ is the Bures--Wasserstein space introduced in Section~\ref{sec:bw}.
Note that if we had considered the KL divergence with the arguments swapped, $q \mapsto \KL(\pi \mmid q)$, then the optimal solution is the one that matches the mean and covariance of $\pi$, which defeats the purpose of VI since they are precisely the parameters we are trying to compute.

Following~\cite{Lametal22GVI}, our approach to solve~\eqref{eq:gvi} is to follow the Wasserstein gradient flow constrained to lie in the Bures--Wasserstein space, see Section~\ref{sec:bw}.
By Corollary~\ref{cor:bw_cvxty}, $\BW(\R^d)$ is geodesically convex, and hence the guarantee of Corollary~\ref{cor:vi_convergence} applies.
It remains to derive the form of the BW gradient flow using Theorem~\ref{thm:bw_grad} in order to arrive at an implementable algorithm.

For the KL divergence $\cF = \KL(\cdot\mmid \pi)$, $\nabla \delta \cF(q) = \nabla V + \nabla \log q$ and $\nabla^2\delta \cF(q) = \nabla^2 V + \nabla^2 \log q$.
Hence,
\begin{align*}
    \nabla_{\BW}\cF(q)(x)
    &= \Bigl(\int (\nabla^2 V + \nabla^2 \log q)\,\ud q\Bigr)\,(x-m_q) \\
    &\qquad{} + \int\nabla V\,\ud q + \underbrace{\int\nabla \log q\,\ud q}_{=0} \\
    &= \Bigl(\int\nabla^2 V\,\ud q - \Sigma_q^{-1}\Bigr)\,(x-m_q) + \int\nabla V\,\ud q\,.
\end{align*}

By setting the BW gradient equal to zero, we also deduce the first-order stationarity conditions, which are both necessary and sufficient by convexity.

\begin{proposition}
    Suppose that $\pi\propto\exp(-V)$, where $V$ is $\alpha$-convex for some $\alpha > 0$.
    Then, the unique minimizer $q_\star$ in~\eqref{eq:gvi} is characterized by the conditions
    \begin{align*}
        \int\nabla V\,\ud q_\star = 0 \qquad\text{and}\qquad \int\nabla^2 V\,\ud q_\star = \Sigma_{q_\star}^{-1}\,.
    \end{align*}
\end{proposition}

We can now write down the BW gradient flow using Theorem~\ref{thm:bw_gf}.
Note that there is a slight simplification since the covariance matrix $\Sigma_t$ cancels with the Hessian of the first variation of the entropy.

\begin{theorem}
    The BW gradient flow of the functional $\KL(\cdot\mmid \pi)$, where $\pi\propto\exp(-V)$, is the curve ${(q_t = \cN(m_t, \Sigma_t))}_{t\ge 0}$, where
    \begin{align}\label{eq:bw_gvi}
        \boxed{
        \begin{aligned}
            \dot m_t &= -\E\nabla V(X_t)\,, \\
            \dot \Sigma_t &= - \E\nabla^2 V(X_t)\,\Sigma_t - \Sigma_t\,\E\nabla^2 V(X_t) + 2I\,,
        \end{aligned}
        }
    \end{align}
    and $X_t \sim q_t$.
\end{theorem}

To implement the gradient flow, the system of ODEs~\eqref{eq:bw_gvi} can be discretized in time. At each iteration $t$, since we keep track of the mean $m_t$ and covariance $\Sigma_t$, the expectations $\E\nabla V(X_t)$ and $\E\nabla^2 V(X_t)$ can be approximated via Monte Carlo averages by drawing samples from $q_t = \cN(m_t,\Sigma_t)$, or via quadrature rules.
Furthermore,~\cite{Diaoetal23FBGVI} observed that the splitting~\eqref{eq:kl_decomp} of the KL divergence naturally suggests a proximal gradient method for~\eqref{eq:gvi}\@.

We mention two appealing features of the gradient flow perspective.
First, it comes with principled guarantees: by mimicking optimization proofs over the Bures--Wasserstein space, the papers~\cite{Lametal22GVI, Diaoetal23FBGVI} translate Corollary~\ref{cor:vi_convergence} into non-asymptotic convergence rates for the stochastic gradient-based implementations of~\eqref{eq:bw_gvi}.
Second, it readily leads to an extension to variational inference over the richer class of mixtures of Gaussians by applying either of the Gaussian particle methods from Subsection~\ref{ssec:gaussian_particles}; see~\cite{Lametal22GVI} for details.

\subsection{Mean-field VI}\index{variational inference (VI)!mean-field}\label{ssec:mfvi}

Another important variational family is the class $\cQ = {\cP_{2,\rm ac}(\R)}^{\otimes d}$ of product measures over $\R^d$, in which case the problem~\eqref{eq:VI} is known as \emph{mean-field VI}:
\begin{align}\label{eq:mfvi}\tag{$\msf{MFVI}$}
    q_\star = \argmin_{q\in {\cP_{2,\rm ac}(\R)}^{\otimes d}} \KL(q \mmid \pi)\,.
\end{align}
This form of VI has its roots in product measure approximations of spin systems from statistical physics and can be motivated statistically by the desire to compute integrals of separable test functions $\phi(x) = \sum_{i=1}^d \phi_i(x_i)$ against the posterior.

Since the family of product measures is geodesically convex (Exercise~\ref{ex:prod_measures_cvx} in Chapter~\ref{chap:WGF}), Corollary~\ref{cor:vi_convergence} once again shows that the constrained Wasserstein gradient flow converges rapidly.
One can also show that for a functional $\cF$, the component of the Wasserstein gradient $\gradW\cF(\mu)$ which is tangential to the space of product measures takes the form
\begin{align}\label{eq:mf_gradient}
    x\mapsto \sum_{i=1}^d \Bigl(\int\gradW\cF(\mu)(x)\,\mu(\ud x_1,\dotsc,\ud x_{i-1}, \ud x_{i+1},\dotsc,\ud x_d)\Bigr)\,e_i\,,
\end{align}
where $e_i$ is the $i$-th standard basis vector; see Exercise~\ref{ex:mf_gradient}.

Note that the particle approach of Subsection~\ref{sec:mf_WGF} does not apply because the Wasserstein gradient of the KL divergence is not defined for discrete measures.
One way to circumvent this issue is via stochastic dynamics, see Subsection~\ref{sec:interacting}.
In this section, we instead describe the approach of~\cite{JiaChePoo24MFVI} which, similarly to the previous subsection, is based on finite-dimensional parameterization.
Key to this approach is that for mean-field VI, the measure $q_\star$ is not intrinsically high-dimensional due to the product structure, which allows for efficient parameterization.

The idea is to parameterize an element $q \in {\cP_{2,\rm ac}(\R)}^{\otimes d}$ via the Brenier map $T_{\rho\to q}$ from the standard Gaussian measure $\rho$.
Since both $\rho$ and $q$ are product measures, one sees that the transport map is separable, i.e., it is of the form $T_{\rho\to q}(x) = (T_1(x_1),\dotsc,T_d(x_d))$ for some \emph{univariate} (and increasing, by Theorem~\ref{thm:fundOT}) maps $T_i : \R\to\R$.
However, the class of such maps is still infinite-dimensional and needs to be further restricted for implementation purposes.

To do so, we take a finite family $\cM$ of optimal transport maps---called the \emph{dictionary}---and take as our eventual family of maps the set $\cone(\cM)$ of all conic combinations of elements of $\cM$:
\begin{align*}
    \cone(\cM)
    &= \Bigl\{\sum_{T\in\cM}\lambda_T T \Bigm\vert \lambda \in \R_+^{\cM}\Bigr\}\,.
\end{align*}
We denote $T^\lambda \deq \sum_{T\in \cM} \lambda_T T$ and $\rho^\lambda \deq (T^\lambda)_\# \rho$.
This set is now finite-dimensional, and in fact is parameterized by the positive orthant $\R_+^{\cM}$.
Therefore, we can now optimize the functional $\lambda \mapsto \KL(\rho^\lambda \mmid \pi)$ over $\R_+^{\cM}$.
Note that we have replaced our original variational family ${\cP_{2,\rm ac}(\R)}^{\otimes d}$ with the smaller set $\cone(\cM)_\# \rho$, but the hope is that for an appropriate choice of $\cM$, the family $\cone(\cM)_\# \rho$ is expressive enough to approximately capture all of ${\cP_{2,\rm ac}(\R)}^{\otimes d}$.
In~\cite{JiaChePoo24MFVI}, this is indeed shown to be the case, e.g., when $\cM$ consists of increasing and piecewise linear functions which act on a single coordinate, and that the total size of $\cM$ is polynomially bounded in the problem parameters.

We have the following lemma, whose proof we leave as Exercise~\ref{ex:mf_isometry}.

\begin{lemma}\label{lem:mf_isometry}
    Assume that $\cM$ consists of maps $T$ which are separable, in the sense that $T(x) = (T_1(x_1),\dotsc,T_d(x_d))$ for increasing univariate maps $T_1,\dotsc,T_d : \R\to\R$.
    Then, $\cone(\cM)_\#\rho$ is geodesically convex.
    
    Moreover, the map $(\R_+^{\cM}, \|\cdot\|_Q) \to (\cone(\cM)_\#\rho, W_2)$, $\lambda \mapsto \rho^\lambda$ is an isometry, where $\|x\|_Q^2 \deq \langle x, Q\,x\rangle$ and $Q$ is the $|\cM|\times |\cM|$ matrix with entries
    \begin{align}\label{eq:mfvi_Q}
        Q_{T,T'}
        &\deq \langle T, T'\rangle_\rho\,, \qquad T,T'\in \cM\,.
    \end{align}
\end{lemma}

The first statement of Lemma~\ref{lem:mf_isometry} shows that under strong log-concavity for $\pi$, the problem of minimizing $\KL(\cdot \mmid \pi)$ over $\cone(\cM)_\# \rho$ is a strongly convex problem in the Wasserstein geometry, and in particular, that Corollary~\ref{cor:vi_convergence} applies.
The second statement shows that implementing the Wasserstein gradient flow in this case amounts to implementing a \emph{Euclidean} gradient flow up to preconditioning by $Q^{-1}$.
Indeed, the isometry implies that the Wasserstein gradient flow of the KL divergence over $\cone(\cM)_\# \rho$ is equivalent to the gradient flow of $\lambda \mapsto \KL(\rho^\lambda \mmid \pi)$ over $(\R_+^\cM, \|\cdot\|_Q)$, and the gradient operator in the $\|\cdot\|_Q$ norm is simply the Euclidean gradient operator premultiplied by the matrix $Q^{-1}$.
In particular, we can write down the resulting algorithm explicitly.

\begin{theorem}
    The Wasserstein gradient flow of $\KL(\cdot \mmid \pi)$ restricted to ${\cone(\cM)}_\# \rho$ is given by ${(\rho^{\lambda(t)})}_{t\ge 0}$, where
    \begin{align*}
        \boxed{\dot \lambda_T = -\sum_{T'\in \cM} {(Q^{-1})}_{T,T'}\, \int\bigl[\langle \nabla V\circ T^{\lambda}, T'\rangle - \langle (\nabla T^{\lambda})^{-1}, \nabla T'\rangle\bigr]\,\ud \rho}
    \end{align*}
    and the matrix $Q$ is given in~\eqref{eq:mfvi_Q}.
\end{theorem}
\begin{proof}
    Lemma~\ref{lem:mf_isometry} implies that the Wasserstein gradient flow is given by $\dot\lambda_t = -Q^{-1}\,\nabla_\lambda \KL(\rho^{\lambda_t} \mmid \pi)$.
    We compute
    \begin{align*}
        \partial_{\lambda_T} \cV(\rho^\lambda)
        &= \partial_{\lambda_T} \int V\circ T^\lambda \, \D \rho
        = \int \langle \nabla V\circ T^\lambda, T\rangle\,\D \rho
    \end{align*}
    and
    \begin{align*}
        \partial_{\lambda_T} \cH(\rho^\lambda)
        &= \partial_{\lambda_T} \int \rho^\lambda \log \rho^\lambda\\
        &= -\partial_{\lambda_T} \int \log \det \nabla T^\lambda \, \D \rho \\
        &= - \int \langle (\nabla T^\lambda)^{-1}, \nabla T\rangle \,\D \rho\,,
    \end{align*}
    where we used respectively Lemma~\ref{lem:change_of_var} and a classical result of matrix calculus to compute the gradient of the $\log \det$ function.
\end{proof}

As in the previous subsection, the expectations can be estimated using Monte Carlo averages.

\subsection{Stein variational gradient descent}\label{sec:svgd}\index{Stein variational gradient descent (SVGD)}

We now ask whether we can minimize the KL divergence over the family $\cQ$ of empirical measures---measures of the form $\frac{1}{N} \sum_{i=1}^N \delta_{x^i}$.
Unlike the preceding two subsections, this class is arbitrarily expressive: as $N\to\infty$, we can always find a sequence of empirical measures that converges weakly to $\pi$, as a consequence of the law of large numbers; see Chapter~\ref{chap:primal-dual}.

Recall from Section~\ref{sec:mean_field_particle_systems} that the Wasserstein gradient flow of a functional $\cF$ over the full Wasserstein space can be represented via
\begin{align}\label{eq:wgf_general_func}
    \dot X_t = -\gradW \cF(\mu_t)(X_t)\,, \qquad X_t\sim \mu_t\,,
\end{align}
provided that $\gradW \cF(\mu_t)$ makes sense.
Note also that unlike the previous two subsections, we do not have to do anything special to ensure that $\mu_t$ remains an empirical measure for all $t\ge 0$: the gradient flow~\eqref{eq:wgf_general_func} \emph{automatically} preserves the space of empirical measures.

If we specialize this to $\cF = \KL(\cdot \mmid \pi)$, then~\eqref{eq:wgf_kl} leads to
\begin{align}\tag{$\msf{WGF}$}\label{eq:wgf_kl_particle}
    \dot X_t = -\nabla \log \frac{\mu_t}{\pi}(X_t)\,, \qquad X_t\sim \mu_t\,.
\end{align}
Unfortunately, the expression $\nabla \log(\mu_t/\pi)$ does \emph{not} make sense when $\mu_t$ is an empirical measure.

The next idea that springs to mind is to replace $\mu_t$ in~\eqref{eq:wgf_kl_particle} with a smoothed version via a kernel density estimator (KDE).
More precisely, let us initialize $N$ particles $X_0^1,\dotsc,X_0^N$, and let $k : \R^d\to\R_+$ be a symmetric kernel with $\int k = 1$.
At time $t$, we replace $\mu_t$ by the KDE $\widehat \mu_t = \frac{1}{N}\sum_{i=1}^N k(\cdot - X_t^i)$, which leads to an interacting system of particles:
\begin{align*}
    \dot X_t^i
    &= -\nabla V(X_t^i) - \Bigl[\nabla \log \frac{1}{N} \sum_{j=1}^N k(\cdot - X_t^j)\Bigr](X_t^i)\,, \qquad i\in [N]\,.
\end{align*}
In the mean-field limit $N\to\infty$, we expect that $\frac{1}{N} \sum_{j=1}^N k(\cdot - X_t^j) \to \int k(\cdot - y) \, \mu_t(\ud y) = k \star \mu_t$, where $\mu_t = \law(X_t)$, leading to the dynamics
\begin{align*}
    \dot X_t
    &= -\nabla V(X_t) - \nabla \log(k\star \mu_t)(X_t)
\end{align*}
or equivalently
\begin{align*}
    \partial_t \mu_t
    &= \divergence\bigl(\mu_t\,(\nabla V + \nabla \log(k \star \mu_t))\bigr)\,.
\end{align*}
However, these dynamics do not necessarily converge to the target $\pi\propto \exp(-V)$.
This issue can be fixed by introducing a \emph{bandwidth} parameter to the kernel $k$ which tends to zero as $N\to\infty$, but there is an alternative which stems from the RKHS literature (recall the discussion in Subsection~\ref{subsec:mmd}) which we describe next.
The \emph{Stein variational gradient descent} (SVGD) algorithm, due to~\cite{LiuWan16SVGD}, manages to use a \emph{fixed} kernel $k$ but still admits $\pi$ as a stationary solution.

We define the integral operator
\begin{align*}
    \cK_\mu : f \mapsto \int f(y)\,k(\cdot- y) \, \mu(\ud y)\,,
\end{align*}
and we follow the dynamics
\begin{align}\label{eq:svgd}\tag{$\msf{SVGD}$}
    \partial_t \mu_t
    &= \divergence\Bigl(\mu_t\, \cK_{\mu_t} \nabla \log \frac{\mu_t}{\pi}\Bigr)\,,
\end{align}
where the integral operator acts on vector fields coordinate-wise.
Clearly these dynamics leave $\pi$ stationary.
Let us calculate the effect of the integral operator above. First,
\begin{align}\label{eq:svgd1}
    \cK_\mu \nabla \log \frac{1}{\pi}
    &= \cK_\mu \nabla V
    = \int \nabla V(y) \, k(\cdot - y)\,\mu(\ud y)\,.
\end{align}
For the other term, we use integration by parts:
\begin{align}
    \cK_\mu \nabla \log \mu
    &= \int \frac{\nabla \mu(y)}{\mu(y)}\,k(\cdot- y)\,\mu(\ud y)
    = \int \nabla \mu(y)\,k(\cdot- y)\,\ud y \nonumber\\
    &= \int \nabla k(\cdot - y)\,\mu(\ud y)\,.\label{eq:svgd2}
\end{align}
Both~\eqref{eq:svgd1} and~\eqref{eq:svgd2} are expectations w.r.t.\ $\mu$, so they can be replaced by empirical averages over $N$ particles.
This leads to the algorithm
\begin{align*}
    \boxed{\dot X_t^i = -\frac{1}{N} \sum_{j=1}^N [k(X_t^i - X_t^j)\,\nabla V(X_t^j) + \nabla k(X_t^i - X_t^j)]\,.}
\end{align*}
Although SVGD has been an active subject of research, many theoretical questions regarding its convergence remain open.

Recall that we motivated SVGD with the idea of approximating $\pi$ by an empirical measure, noting that the class of empirical measures over $N$ atoms is arbitrarily expressive as $N\to\infty$.
But if our ultimate goal is fidelity with respect to $\pi$, we may as well ask if we can directly output samples from $\pi$ itself.
In the next section, we discuss the problem of sampling via MCMC methods, which can be viewed as \emph{stochastic} implementations of the Wasserstein gradient flow.

\section{Sampling}\label{sec:sampling}

One of the most compelling applications of the theory of Wasserstein gradient flows is to provide a geometric interpretation of the Langevin diffusion, as put forth in the seminal work of Jordan, Kinderlehrer, and Otto~\cite{JorKinOtt98}.
As before, let $V : \R^d\to\R$ be a smooth potential with $\int \exp(-V) < \infty$ and let $\pi$ denote the probability measure over $\R^d$ with density $\pi\propto\exp(-V)$.

Suppose that we wish to \emph{sample} from the distribution $\pi$. In other words, we want to design an algorithm for producing a random variable whose law is close to $\pi$.
For example, $\pi$ could be the posterior distribution in a Bayesian inference problem, in which case basic downstream tasks such as constructing credible regions or point estimates are often intractable in non-conjugate models.
Nevertheless, we can usually solve these tasks approximately and efficiently, given a subroutine for drawing approximate samples from the posterior.
Beyond the application to computational Bayesian statistics, sampling also plays an important role in scientific computing through Monte Carlo integration and for the design of randomized algorithms.

The predominant approach to this problem, dubbed Markov chain Monte Carlo (MCMC), is to design a Markov chain whose unique stationary distribution is, or at least is close to, the target $\pi$.
When $\pi$ admits a positive and smooth density, as we assume in this section, then we can write $\pi \propto\exp(-V)$ without loss of generality (with $V = \log(1/\pi) + \text{const.}$).
In this case, a canonical MCMC algorithm is obtained by discretizing the \emph{Langevin diffusion}\index{Langevin diffusion}, which is the solution to the stochastic differential equation (SDE)
\begin{align*}
    \ud X_t
    = -\nabla V(X_t) \, \ud t + \sqrt 2\, \ud B_t\,,
\end{align*}
where ${(B_t)}_{t\ge 0}$ is a standard Brownian motion.
As soon as $\nabla V$ is, e.g., Lipschitz continuous, there is a unique strong solution to this SDE for any prescribed initial condition, and its stationary distribution is $\pi$.

In the next section, we show that when we track the evolution of the marginal law $\mu_t \deq \law(X_t)$ of the Langevin diffusion, then ${(\mu_t)}_{t\ge 0}$ follows the Wasserstein gradient flow of the KL divergence $\KL(\cdot \mmid \pi)$.
More broadly, this story is the starting point of a fruitful literature which has blossomed in recent years on an optimization perspective (i.e., the application of optimization algorithms such as gradient flows) on the problem of sampling.

\subsection{The Langevin diffusion as a Wasserstein gradient flow}\label{subsec:langevin}

To spoil the surprise, the fundamental reason why~\eqref{eq:wgf_kl} admits a stochastic implementation is because we can rewrite
\begin{align*}
    \divergence(\mu_t\, \nabla \log \mu_t)
    = \divergence\bigl(\mu_t\,\frac{\nabla \mu_t}{\mu_t}\bigr)
    = \divergence(\nabla \mu_t)
    = \Delta \mu_t
\end{align*}
where $\Delta f = \sum_{i=1}^d \partial_i^2 f$ is the \emph{Laplacian} of $f$.
On the other hand, second-order parabolic PDEs---the heat equation $\partial_t \mu_t = \Delta \mu_t$ being the most fundamental example---classically describe the evolution in law of stochastic differential equations driven by Brownian motion.
The rest of this subsection aims to make this connection precise.

Using the computation above, we rewrite the Wasserstein gradient flow of the KL divergence, given in~\eqref{eq:wgf_kl}, as
\begin{align}\label{eq:wgf_kl_form2}
    \partial_t \mu_t
    &= \Delta\mu_t + \divergence(\mu_t\,\nabla V)\,.
\end{align}

We next compute the marginal evolution of the Langevin diffusion in order to compare with~\eqref{eq:wgf_kl_form2}.
The usual method for doing so is to use It\^o's formula from stochastic calculus, but we instead proceed more informally.
First, let us condition on $X_0 = x_0$, and write the Langevin diffusion in integral form.
\begin{align*}
    X_t = x_0 - \int_0^t \nabla V(X_s) \, \ud s + \sqrt 2 \, B_t\,.
\end{align*}
Recall also that $B_t \sim \cN(0, tI)$.
In particular, $\|\int_0^t \nabla V(X_s)\,\ud s\| = O_\p(t)$ and $\|B_t\| = O_\p(t^{1/2})$ for small $t$, so the Brownian motion term dominates and $\|X_t - x_0\| = O_\p(t^{1/2})$.
By duality, to calculate the evolution of $\mu_t$, it suffices to compute the evolution of the expectation of any test function $\varphi : \R^d\to\R$.
A Taylor expansion yields
\begin{align*}
    \varphi(X_t)
    &= \varphi\Bigl(x_0 - \int_0^t \nabla V(X_s) \, \ud s + \sqrt 2 \, B_t\Bigr) \\
    &= \varphi(x_0) + \Bigl\langle \nabla \varphi(x_0), -\int_0^t\nabla V(X_s) \, \ud s + \sqrt 2 \, B_t\Bigr\rangle \\ 
    &\qquad{} + \frac{1}{2} \,\Bigl\langle \nabla^2 \varphi(x_0), \Bigl( -\int_0^t \nabla V(X_s) \, \ud s + \sqrt 2 \,B_t\Bigr)^{\otimes 2}\Bigr\rangle + O_\p(t^{3/2})\,.
\end{align*}
The first-order term equals
\begin{align*}
    -t\,\langle \nabla \varphi(x_0), \nabla V(x_0)\rangle + \sqrt 2 \, \langle \nabla \varphi(x_0), B_t\rangle + O_\p(t^{3/2})\,.
\end{align*}
The second-order term equals
\begin{align*}
    \langle \nabla^2 \varphi(x_0) \, B_t, B_t\rangle + O_\p(t^{3/2})\,.
\end{align*}
Therefore,
\begin{align*}
    \varphi(X_t)
    &= \varphi(x_0) - t\,\langle \nabla \varphi(x_0), \nabla V(x_0)\rangle + + \sqrt 2 \, \langle \nabla \varphi(x_0), B_t\rangle  \\
    & \qquad  +  \langle \nabla^2 \varphi(x_0)\,B_t, B_t\rangle+ O_\p(t^{3/2})\,.
\end{align*}
Taking expectations and using $\E[B_t] = 0$, $\E[B_t B_t^\T] = tI$,
\begin{align*}
    \E \varphi(X_t)
    &= \varphi(x_0) + t\,\bigl(\tr \nabla^2 \varphi(x_0) - \langle \nabla \varphi(x_0),\nabla V(x_0)\rangle\bigr) + O_\p(t^{3/2}) \\
    &= \varphi(x_0) + t\,\bigl(\Delta \varphi(x_0) - \langle \nabla \varphi(x_0),\nabla V(x_0)\rangle\bigr) + O_\p(t^{3/2})\,.
\end{align*}
Subtracting $\varphi(x_0)$, dividing by $t$, and letting $t\searrow 0$,
\begin{align}\label{eq:langevin_gen_1}
    \partial_t \E\varphi(X_t) \big|_{t=0}
     = \Delta \varphi(x_0) - \langle \nabla \varphi(x_0), \nabla V(x_0)\rangle\,.
\end{align}
In the language of Markov semigroup theory, we have computed the \emph{generator} of the Langevin diffusion to be the second-order differential operator $\cL$, defined by $\cL \varphi \deq \Delta \varphi - \langle \nabla \varphi, \nabla V\rangle$.

More generally, by first conditioning on the value of $X_t$ and using the Markov property and~\eqref{eq:langevin_gen_1}, it holds that
\begin{align*}
    \partial_t \E \varphi(X_t) = \E \cL\varphi(X_t)\,.
\end{align*}
Expressed in terms of the marginal law $\mu_t$, it reads
\begin{align*}
    \int \varphi\,\partial_t \mu_t
    &= \int (\Delta \varphi - \langle \nabla \varphi, \nabla V\rangle) \, \ud\mu_t\,.
\end{align*}
In order to identify an equation for $\partial_t \mu_t$, we must compute the adjoint (w.r.t.\ Lebesgue measure) of $\cL$.
This is accomplished through integration by parts, which shows that the right-hand side equals
\begin{align*}
    \int \varphi\,\bigl(\Delta \mu_t + \divergence(\mu_t\, \nabla V)\bigr)\,.
\end{align*}
We have established the following theorem.

\begin{theorem}[Fokker--Planck equation]\label{thm:fokker_planck}\index{Fokker--Planck equation}
    The marginal law $\mu_t \deq \law(X_t)$ of the Langevin diffusion with potential $V$ is given by the solution to the \emph{Fokker--Planck equation}
    \begin{align*}
        \partial_t \mu_t = \Delta \mu_t + \divergence(\mu_t\,\nabla V)\,.
    \end{align*}
\end{theorem}

Comparing with~\eqref{eq:wgf_kl_form2}, it yields:

\begin{corollary}\label{cor:langevin_is_wgf}
    The marginal law of the Langevin diffusion with potential $V$ is the Wasserstein gradient flow of $\KL(\cdot\mmid \pi)$, where $\pi$ has density proportional to $\exp(-V)$.
\end{corollary}

As a special case when $V = 0$, we also obtain the following corollary.

\begin{corollary}\label{cor:heat_flow_wgf}
    If ${(\mu_t)}_{t\ge 0}$ is the marginal law of a (rescaled) Brownian motion ${(\sqrt 2\,B_t)}_{t\ge 0}$, then ${(\mu_t)}_{t\ge 0}$ solves the heat equation $\partial_t \mu_t = \Delta \mu_t$, and it is the Wasserstein gradient flow of the entropy functional $\cH$.
\end{corollary}

We showed in Corollary~\ref{cor:vi_convergence} that the strong log-concavity of $\pi$ implies rapid convergence of the Wasserstein gradient flow of the KL divergence.
Therefore, we immediately obtain the following elegant convergence result for the Langevin diffusion.

\begin{corollary}\label{cor:langevin_conv}
    Let $\pi$ be an $\alpha$-strongly log-concave measure, and let ${(\mu_t)}_{t\ge 0}$ denote the marginal law of the Langevin diffusion with stationary distribution $\pi$. Then,
    \begin{align*}
        \KL(\mu_t\mmid \pi)
        &\le e^{-2\alpha t}\KL(\mu_0\mmid \pi)\,.
    \end{align*}
\end{corollary}

To conclude this section, we take stock of the situation at hand.
Recall that the Wasserstein gradient flow of the KL divergence can be implemented via the deterministic evolution~\eqref{eq:wgf_kl_particle}.
On the other hand, in this subsection, we started with the Langevin diffusion, which is a stochastic evolution:
\begin{align}\tag{$\msf{LD}$}\label{eq:langevin}
    \ud X_t = -\nabla V(X_t)\,\ud t + \sqrt 2 \, \ud B_t\,.
\end{align}
We showed in Theorem~\ref{thm:fokker_planck} that the marginal law $\mu_t \deq \law(X_t)$ evolves according to the Fokker{--}Planck equation
\begin{align}\tag{$\msf{FP}$}\label{eq:fokker_planck}
    \partial_t \mu_t
    &= \Delta \mu_t + \divergence(\mu_t\,\nabla V)\,,
\end{align}
and moreover that this evolution coincides with the Wasserstein gradient flow of $\KL(\cdot \mmid \pi)$.
Note that~\eqref{eq:fokker_planck} is strictly coarser than~\eqref{eq:langevin} because~\eqref{eq:langevin} also includes information about correlations between different time points of the stochastic process.

Ultimately, we have the equivalence $\eqref{eq:fokker_planck} \Leftrightarrow \eqref{eq:langevin} \Leftrightarrow \eqref{eq:wgf_kl_particle}$ in the sense that they correspond to the same curve on the space of probability measures---clearly at the particle level, they are different---but these three differing perspectives provide new avenues for algorithm design and theoretical study.

\subsection{Sampling as optimization}\label{sec:sampling_as_opt}

The gradient flow perspective on the Langevin diffusion is the starting point of a flourishing literature on an optimization perspective on sampling.
In this subsection, we study the basic properties of KL divergence minimization as an optimization problem, inspired by the treatment in~\cite{Wib18}.
Then, in the next subsection, we provide a glimpse of the recent impact of this perspective on the theory of log-concave sampling.
We refer to the monograph~\cite{Chewi24Book} for a detailed exposition.

Let $V : \R^d\to\R$ be a potential which is $\alpha$-convex, $\alpha > 0$, and recall that our goal is to output a sample from the target $\pi\propto \exp(-V)$.
Corollary~\ref{cor:langevin_conv} then ensures that the continuous-time Langevin diffusion converges rapidly to $\pi$, but in order to obtain algorithmic guarantees we must \emph{discretize} the process, and for this we also impose the dual assumption of smoothness, $\nabla^2 V \preceq \beta I$, to ensure stability.

Error estimates for discretizations of the Langevin diffusion are by now well-established.
Under our assumptions of strong convexity and smoothness of $V$, non-asymptotic convergence guarantees can be established for the Euler--Maruyama scheme
\begin{align}\label{eq:lmc}
    X_{k+1}
    &= X_k - h\,\nabla V(X_k) + \cN(0,2hI)\,,\qquad k=0,1,2,\ldots\,,
\end{align}
which parallels the complexity theory for optimization~\cite{Nes18}.
Unlike the situation in optimization, however, the discretization~\eqref{eq:lmc} is asymptotically biased: the stationary distribution of the Markov chain~\eqref{eq:lmc} does not equal the target $\pi$. This leads to slower rates of convergence, as the step size $h$ must be chosen small to mitigate the bias.
We now explain how an optimization perspective sheds light on the source of asymptotic bias and suggests a proximal scheme for removing it.

As in~\eqref{eq:kl_decomp}, we write the KL divergence as the sum of the potential energy and the entropy,
\begin{align*}
    \KL(\mu \mmid \pi)
    &= \int V \, \ud \mu + \int \mu \log \mu + \text{const.}
    = \cV(\mu) + \cH(\mu) + \text{const.}
\end{align*}
The problem of sampling from $\pi$ is cast as the minimization of this objective functional over $\cP_{2,\rm ac}(\R^d)$, and we have already begun studying its properties.
Namely:
\begin{itemize}
    \item The potential energy is strongly convex and smooth, $\alpha \le \gradW^2 \cV \le \beta$.
    We proved the lower bound in Theorem~\ref{thm:potential_cvx}, and the upper bound follows by a similar computation.
    \item The entropy is convex, $0 \le \gradW^2 \cH$. We proved this as Theorem~\ref{thm:entropy_cvx}. However, the entropy is \emph{non-smooth}.\footnote{One can in fact show that $\gradW^2 \cH(\mu)[v, v] = \int \|\nabla v - I\|_{\rm HS}^2\,\ud \mu$ and there is no constant $C > 0$ such that $\gradW^2 \cH(\mu)[v,v] \le C\,\|v\|_\mu^2$ for all $v \in T_\mu \cP_{2, \rm ac}(\R^d)$.}
\end{itemize}
The situation at hand is one that is commonly encountered in optimization, known as \emph{composite optimization}: minimize the sum $f+g$, where $f$ is (strongly) convex and smooth, and $g$ is convex but non-smooth.
The prototypical example is the $\ell_1$-penalized least squares objective (or LASSO), $\theta \mapsto \|y-X\theta\|^2 + \lambda \,\|\theta\|_1$.
In optimization theory, the canonical algorithm designed for such problems is the \emph{proximal gradient} method
\begin{align}\label{eq:prox_grad}
    x_{k+1}
    &= \prox_{hg}(x_k - h\,\nabla f(x_k))\,, 
\end{align}
where
\begin{align}\label{eq:prox}
    \prox_{hg}(y) \deq \argmin_{x\in\R^d}{\Bigl\{hg(x) + \frac{1}{2}\,\|y-x\|^2\Bigr\}}\,.
\end{align}
When $g$ has a simple structure, such as the $\ell_1$ norm, then the proximal mapping sometimes admits a closed-form solution.

The proximal gradient algorithm is unbiased, meaning that its only fixed points are minimizers of $f+g$.
Moreover, one can show that the iteration~\eqref{eq:prox_grad} converges at the same rate that gradient descent would for a convex and smooth objective, despite the non-smoothness of $g$.

More broadly, we have introduced two discretization schemes: gradient descent, and the proximal step~\eqref{eq:prox}. Each has its relative merits. Whereas gradient descent is cheaper to implement (especially for functions which do not have a simple structure like $\|\cdot\|_1$), the proximal scheme converges even without smoothness.
Therefore, we are motivated to apply one discretization method to $f$, and the other to $g$; this is known as a \emph{splitting scheme}.
Not all splitting schemes are unbiased, however, and the combination of gradient descent and the proximal map is an especially auspicious match.\footnote{In numerical analysis, these two discretizations are so-called ``adjoints'' to each other, see~\cite[Appendix B]{Wib18}.}

With these principles from optimization in mind, let us now consider the situation for sampling.
The discretization~\eqref{eq:lmc} can be viewed as the splitting scheme
\begin{align*}
    X_{k+1/2}
    &= X_k - h\,\nabla V(X_k)\,, \\
    X_{k+1}
    &= X_{k+1/2} + \cN(0, 2hI)\,.
\end{align*}
The two steps correspond, respectively, to \emph{gradient descent}\footnote{Check that if $\mu_k = \law(X_k)$, $\mu_{k+1/2} = \law(X_{k+1/2})$, and $h\le 1/\beta$, then $\mu_{k+1/2} = \exp_{\mu_k}(-h\,\gradW \cV(\mu_k))$, which is the Riemannian analogue of gradient descent.} for $\cV$ and the \emph{gradient flow} of $\cH$; the latter statement is Corollary~\ref{cor:heat_flow_wgf}. The combination of gradient descent and gradient flow does not produce an unbiased splitting scheme.

The intuition from optimization suggests to replace the gradient flow for $\cH$ with the proximal map for $\cH$.
Generalizing the definition~\eqref{eq:prox} to the Wasserstein space, we arrive at
\begin{align}\label{eq:prox_ent}
    \prox_{h\cH}(\mu)
    &\deq \argmin_{\nu \in \cP_{2,\rm ac}(\R^d)}{\Bigl\{h\cH(\nu) + \frac{1}{2}\,W_2^2(\mu,\nu)\Bigr\}}\,.
\end{align}
(In fact, the proximal operator on the Wasserstein space was the device through which~\cite{JorKinOtt98} first made precise the Wasserstein gradient flow interpretation of the Langevin diffusion in Subsection~\ref{subsec:langevin}.)
The resulting proximal gradient algorithm on the Wasserstein space can indeed be shown to converge rapidly to $\pi$~\cite{SalKorLui20WPG}.
However, the proximal map~\eqref{eq:prox_ent} is in general intractable, so the proximal gradient algorithm is merely wishful thinking.

Actually, it is not just $\prox_{h\cH}$ that is intractable; gradient \emph{descent} on $\cH$ is also impractical because it requires knowing the entire probability density (this is the motivation for our discussion of SVGD in Section~\ref{sec:svgd}).
It seems that the only operation we can reasonably implement for minimizing $\cH$ is the gradient \emph{flow}, due to the fortuitous link with Brownian motion.
This can be considered both as a blessing and curse.
The blessing is that the routine we can implement---gradient flow---succeeds despite the non-smoothness of $\cH$, which explains why sampling is possible at all.
The curse, however, is that the gradient flow is ``mismatched'' with our discretization for $\cV$, and hence the discretization~\eqref{eq:lmc} incurs asymptotic bias.

This is perhaps representative of the subject as a whole: although optimization theory suggests a huge number of algorithmic paradigms which we hope to port over to the world of sampling, the execution of these ideas requires care.
In the next subsection, we survey a number of examples in which this philosophy has been successfully carried out, including a surprising ``proximal'' algorithm for sampling.

\subsection{Some recent developments}

\paragraph{Algorithms}
Open any modern book on convex optimization to find a formidable arsenal of methods: coordinate descent, gradient descent, interior point, mirror descent, Newton's method, Nesterov's fast gradient method, proximal gradient, stochastic gradient descent, etc.
In recent years, a substantial research effort has been devoted to developing sampling analogues of all of these methods and more, of which we describe only a select few.

Our first example is the aforementioned proximal algorithm for sampling.
Since it is easier to motivate \emph{a posteriori}, we begin by defining the algorithm. Augment the target distribution $\pi$ to form a distribution $\bpi$ over $\R^d\times \R^d$ with density given by
\begin{align*}
    \bpi(x,y) \propto \exp\Bigl(-V(x) - \frac{1}{2h}\,\|y-x\|^2\Bigr)\,.
\end{align*}
The \emph{proximal sampler}\index{proximal sampler}, introduced in~\cite{LeeSheTia21RGO}, applies Gibbs sampling to the new target $\bpi$. Explicitly, repeat the following steps:
\begin{enumerate}
    \item Given $X = x$, resample $Y\sim \bpi^{Y \mid X=x} = \cN(x, hI)$.
    \item Given $Y = y$, resample $X \sim \bpi^{X\mid Y=y}$, where the conditional distribution $\bpi^{X\mid Y=y}$, called the \emph{restricted Gaussian oracle} (RGO), has density $\bpi^{X\mid Y=y}(x) \propto \exp(-V(x) - \frac{1}{2h}\,\|y-x\|^2)$.
\end{enumerate}

A few simple properties can be verified immediately.
First, the $X$-marginal of $\bpi$ is the original target $\pi\propto\exp(-V)$, so it suffices to sample from the augmented target.
Second, the proximal sampler is unbiased---its stationary distribution is $\bpi$---because Gibbs sampling is so.
As stated, however, it is still an idealized algorithm, since it is unclear how to implement step two.
But notice the following analogy: if minimizing $V$ (optimization) corresponds to sampling from $\pi\propto\exp(-V)$ (sampling), then computing the proximal map~\eqref{eq:prox} for $V$ corresponds to sampling from $\bpi^{X\mid Y=y}$; it is in this sense that the proximal sampler resembles a ``proximal'' algorithm for sampling.

But perhaps the most convincing justification is based on the adage ``if it looks like a duck\ldots'': the convergence analyses in~\cite{LeeSheTia21RGO, Chenetal22ProxSampler} show that the convergence rates for the proximal sampler \emph{exactly} replicate the rates for the proximal point method from convex optimization.
Through careful implementations of the RGO\@, the proximal sampler has played an essential role in extending sampling guarantees to both non-log-concave and non-log-smooth settings~\cite{FanYuaChe23ImprovedProx, AltChe24HighAcc}.

Our next example is the adaptation of mirror descent. Recall that the \emph{mirror descent} algorithm for minimizing $V$, originally introduced in~\cite{NemYud83} for optimization w.r.t.\ non-Euclidean norms, starts by choosing a strictly convex function (the ``mirror map'') $\phi : \R^d\to\R\cup\{\infty\}$ and iterating
\begin{align*}
    \nabla \phi(x_{k+1}) = \nabla \phi(x_k) - h\,\nabla V(x_k)\,, \qquad k=0,1,2,\dotsc\,.
\end{align*}
It turns out that the ``mirror'' analogue of~\eqref{eq:langevin} is the so-called \emph{mirror Langevin diffusion}\index{Langevin diffusion!mirror Langevin}~\cite{Zhang+20MLMC, CheLe-Lu20}:
\begin{align*}
    Y_t \deq \nabla \phi(X_t)\,, \qquad \ud Y_t = -\nabla V(X_t)\,\ud t + \sqrt{2\,\nabla^2 \phi(X_t)}\,\ud B_t\,.
\end{align*}
The mirror Langevin diffusion can also be suitably interpreted as a Wasserstein ``mirror'' flow of the KL divergence.\footnote{More specifically, it is the gradient flow of $\KL(\cdot\mmid \pi)$ with respect to the Wasserstein geometry induced by the Hessian metric induced by $\phi$.}
An important special case is obtained when $V$ is strictly convex and we take the mirror map $\phi = V$, leading to the \emph{Newton--Langevin diffusion}\index{Langevin diffusion!Newton--Langevin}---the sampling analogue of Newton's method:
\begin{align*}
    Y_t \deq \nabla V(X_t)\,, \qquad \ud Y_t = -Y_t\,\ud t + \sqrt{2\,\nabla^2 V(X_t)}\,\ud B_t\,.
\end{align*}
The Newton--Langevin diffusion inherits some appealing properties of Newton's method, such as its affine invariance. However, discretization of the Newton--Langevin diffusion (or of the more general mirror Langevin diffusion) is currently less well-understood than for~\eqref{eq:langevin}.

Finally, our last example is the adaptation of Nesterov's \emph{accelerated gradient descent}~\cite{Nes1983Accel}, which is an optimal first-order method for convex smooth minimization.
The corresponding SDE system, called the \emph{underdamped} (or \emph{kinetic}) \emph{Langevin diffusion}\index{Langevin diffusion!underdamped Langevin}, dates back at least to~\cite{Kol1934Underdamped} and is given by
\begin{align*}
    \ud X_t = P_t \,\ud t\,, \qquad \ud P_t = -\nabla V(X_t)\,\ud t - \gamma P_t\,\ud t + \sqrt{2\gamma}\,\ud B_t\,.
\end{align*}
Here, $P$ represents a momentum variable, and $\gamma > 0$ is the ``friction'' parameter.
This diffusion has already formed the basis for numerous state-of-the-art guarantees for log-concave sampling. But in the world of optimization, Nesterov's method is best known for improving the complexity of strongly convex and smooth minimization to $\widetilde O(\sqrt\kappa)$, where $\kappa$ is the ``condition number''. This remarkable result, which saves a factor of $\sqrt\kappa$ over the basic rate for gradient descent, has been dubbed the \emph{acceleration} phenomenon, and it remains an intriguing open question to establish such a phenomenon for sampling.

\paragraph{Complexity}

Since the work of~\cite{NemYud83}, a major goal of optimization research has been to precisely characterize the minimax complexity of optimization over various function classes and oracle models. With the advent of this mindset to MCMC\@, it became natural to do the same for sampling, starting with the non-asymptotic upper bounds in works such as~\cite{Dal17a, DurMou17NonAsymptotic, CheBar18Langevin, DurMajMia19LMCCvxOpt}.
The similarities are striking: both fields consider similar choices for the function classes (e.g., strongly convex and smooth functions) and for the oracle models.
This connection also motivated the search for oracle \emph{lower bounds} which could certify the optimality of our existing algorithms.
This has proven to be a challenging problem, with some modest progress made recently.

Another example of the transfer of ideas is the development of a sampling analogue of ``approximate first-order stationarity'' which provides an alternative approach to the quantitative study of sampling in general non-log-concave settings~\cite{Bal+22NonLogConcave}.

Despite rapid progress in this direction, there are still many fundamental unresolved questions regarding the complexity of log-concave sampling.
We refer to the monograph~\cite{Chewi24Book} for an introduction to this active field and for further references.

\section{Interacting particle systems}\label{sec:interacting}

The SDE systems we encountered in Section~\ref{sec:sampling} all involve evolving a single particle (possibly over an expanded state space) at a time.
More generally, we can consider an \emph{interacting system} of particles, either deterministic or stochastic.  This is highly relevant because as we discussed in Subsection~\ref{sec:mf_WGF}, Wasserstein gradient flows initialized at discrete measures can be implemented using mean-field interacting particle systems. Indeed, recall that 
if we initialize the Wasserstein gradient flow~\eqref{eq:wgf_general_func} at
$$
\mu_0^N = \frac1N\sum_{j=1}^N \delta_{X_0^j}\,,
$$
then $\mu_t$ is given by the empirical measure 
$$
\mu_t^N = \frac1N\sum_{j=1}^N \delta_{X_t^j}\,,
$$
where
\begin{align}\label{eq:interacting_wass_gf}
    \dot X_t^i
    &= -\gradW \cF\bigl(\frac{1}{N}\sum_{j=1}^N \delta_{X_t^j}\bigr)\,, \qquad i \in [N]\,.
\end{align}
This is true whenever the Wasserstein gradient is defined at the discrete measure $\mu_t^N$, and in this case, it is generally expected (and can be rigorously established under assumptions on $\cF$) that as $N\to\infty$ and $\mu_0^N \to \mu_0 \in  \cP_{2,\rm ac}(\R^d)$, the dynamics~\eqref{eq:interacting_wass_gf} converges to~\eqref{eq:wgf_general_func} initialized at $\mu_0$. Using additional tools, one may also show that if $\mu_0 \in  \cP_{2,\rm ac}(\R^d)$ then $\mu_t \in  \cP_{2,\rm ac}(\R^d)$ for all $t \ge 0$.

\subsection{McKean--Vlasov equations}\label{subsec:mkv}

In this subsection, we consider other examples of interacting systems arising from Wasserstein gradient flows.
Recall that in Section~\ref{sec:otto_calc}, we gave three fundamental examples of functionals over the Wasserstein space: potential energy, internal energy, and interaction energy.
What if we consider the sum of the three?
\begin{align}\label{eq:sum_of_three}
    \cF(\mu)
    &\deq \int V\,\ud \mu + \iint W(x-y)\,\mu(\ud x)\,\mu(\ud y) + \frac{\sigma^2}{2} \int \mu \log \mu\,,
\end{align}
where $W$ is even.
By using the trick at the beginning of Subsection~\ref{subsec:langevin} and by computing Wasserstein gradients, convince yourself of the following theorem.

\begin{theorem}
    The Wasserstein gradient flow ${(\mu_t)}_{t\ge 0}$ of~\eqref{eq:sum_of_three} can be described as follows: $\mu_t = \law(X_t)$, where ${(X_t)}_{t\ge 0}$ solves the SDE
    \begin{align}\label{eq:mkv}
        \D X_t
        &= -\nabla V(X_t)\,\ud t - \int \nabla W(X_t - y)\,\mu_t(\ud y)\,\ud t + \sigma\,\ud B_t\,.
    \end{align}
\end{theorem}

Note that the coefficients of the SDE system~\eqref{eq:mkv} depend on the law of the process.
Such systems are called \emph{McKean--Vlasov processes}~\cite{McK1966}\index{McKean--Vlasov}.
To approximate~\eqref{eq:mkv} by an interacting particle system, we replace the integral over $\mu_t$ with an average over particles:
\begin{align}\label{eq:mkv_fp}
    \D X_t^i
    &= -\nabla V(X_t^i)\,\ud t - \frac{1}{N-1} \sum_{j\in [N] \setminus i} \nabla W(X_t^i - X_t^j) \, \ud t + \sigma\,\ud B_t^i\,,
\end{align}
where ${\{B^i\}}_{i\in [N]}$ is a collection of independent Brownian motions.
A natural question that arises is to quantify how close the finite-particle 
system~\eqref{eq:mkv_fp} is to its mean-field limit~\eqref{eq:mkv}.
This problem is addressed by the mathematical theory of \emph{propagation of chaos}~\cite{Szn1991PoC}.

More generally, suppose that we have the general entropically-regularized functional
\begin{align}\label{eq:ent_reg_functional}
    \cF(\mu)
    \deq \cF_0(\mu) + \frac{\sigma^2}{2}\int \mu \log \mu\,.
\end{align}
One can show that the Wasserstein gradient flow ${(\mu_t)}_{t\ge 0}$ of~\eqref{eq:ent_reg_functional} can be described as the marginal law $\mu_t = \law(X_t)$ of the SDE system
\begin{align*}
    \ud X_t
    &= -\gradW \cF_0(\mu_t)(X_t)\,\ud t + \sigma\,\ud B_t\,.
\end{align*}
This system is known as the \emph{mean-field} (or \emph{interacting}) \emph{Langevin dynamics}\index{Langevin diffusion!mean-field}~\cite{CheRenWan23PoC, SuzNitWu23MFLangevin}.
The corresponding finite-particle system,
\begin{align*}
    \ud X_t^i
    &= -\gradW\cF_0\bigl(\frac{1}{N} \sum_{j=1}^N \delta_{X_t^j}\bigr)(X_t^i) \,\ud t + \sigma \,\ud B_t^i\,, \qquad i\in [N]\,,
\end{align*}
is actually the Langevin diffusion corresponding to the target
\begin{align*}
    \hat\pi_N(x^1,\dotsc,x^N)
    &\propto \exp\Bigl(-\frac{2N}{\sigma^2}\,\cF_0\bigl(\frac{1}{N}\sum_{j=1}^N \delta_{x_t^j}\bigr)\Bigr)\,.
\end{align*}
The complexity of sampling from the minimizers of the functionals~\eqref{eq:sum_of_three} and~\eqref{eq:ent_reg_functional} was studied in~\cite{Kook+24SamplingMeanField}.

We conclude this section with an application to the mean-field VI problem introduced in Subsection~\ref{ssec:mfvi}.
By writing down the stochastic implementation of the Wasserstein gradient flow therein,~\cite{Lac23IndepProj} arrived at the McKean--Vlasov SDE
\begin{align*}
    \D X_t
    &= -\int \nabla_1 W(X_t, y)\,\mu_t(\D y)\,\ud t + \sqrt 2\,\D B_t\,,
\end{align*}
where $\mu_t = \law(X_t)$, $\nabla_1$ denotes the gradient taken with respect to the first argument, and
\begin{align*}
    W(x,y)
    &\deq \sum_{i=1}^d V(y_1,\dotsc,y_{i-1}, x_i, y_{i+1},\dotsc,y_d)\,.
\end{align*}
When this SDE is initialized at $X_0 \sim \mu_0$ which is a product measure, $\mu_t$ remains a product measure for all $t\ge 0$ so that ${(\mu_t)}_{t\ge 0}$ is indeed the constrained Wasserstein gradient flow.

The corresponding interacting particle system is given by
\begin{align*}
    \D X_t^j
    &= -\frac{1}{N-1} \sum_{j'\in [N]\setminus i} \nabla_1 W(X_t^j, X_t^{j'}) \, \D t + \sqrt 2 \,\D B_t^j\,, \qquad j\in [N]\,.
\end{align*}
One then expects that if we take $N$ sufficiently large and run the particle system, then the law of (say) the first particle $X_t^1$ will be close to the mean-field minimizer $q_\star$.

\subsection{Birth-death sampling}\index{birth-death sampling}\index{Wasserstein--Fisher--Rao (WFR)}

We return to the sampling problem from Section~\ref{sec:sampling}.
Instead of following the Wasserstein gradient flow of the KL divergence, what if we follow the WFR gradient flow introduced in Section~\ref{sec:wfr}?
The fundamental difference between these approaches is that the Langevin diffusion is a \emph{local} algorithm and hence struggles to jump between well-separated modes. This manifests itself in the convergence rate in Corollary~\ref{cor:langevin_conv}, which depends on the log-Sobolev constant of the target $\pi$. On the other hand, the ``teleportation'' effect of the Fisher--Rao component gives rise to a universal exponential convergence rate~\cite{DomPoo23FR}.

The main challenge is to implement the flow.
Recall that a particle implementation for the WFR gradient flow was presented in Subsection~\ref{subsec:particle_wfr}, but it does not apply to the choice of functional $\cF = \KL(\cdot \mmid \pi)$ since it assumed there that the first variation $\delta \cF$ can be evaluated at an empirical measure.
An implementation was provided in~\cite{LuLuNol19} under the name of ``birth-death'' sampling, which we now describe.

The implementation is based on an interacting system of $N$ particles, which at time $t$ are denoted $X_t^1,\dotsc,X_t^N$.
The approach of Subsection~\ref{subsec:particle_wfr} would associate with each particle $X_t^i$ a weight $w_t^i$ which evolves via $\dot w_t^i = -\alpha_t(X_t^i) \,w_t^i$, where $\alpha_t(x) \deq \delta \cF(\mu_t)(x) - \int \delta \cF(\mu_t) \, \ud \mu_t$.
Here, $\alpha_t(x)$ represents an exponential rate of decay/growth (according to $\alpha_t(x) > 0$ or $\alpha_t(x) < 0$) of the density at $x$.
We instead replace the use of weights with a procedure that ``kills'' or ``duplicates'' the particle after a random wait time.
More precisely, associate with each particle $X_t^i$ an independent clock which rings in the next instantaneous time interval $[t, t+\ud t]$ with probability $\alpha_t(X_t^i)\,\ud t$.
Whenever one of these clock rings---corresponding to, say, $X_t^i$---we either remove $X_t^i$ from the system (if $\alpha_t(X_t^i) > 0$) or we duplicate $X_t^i$ (if $\alpha_t(X_t^i) < 0$).
To keep the total number of particles constant, in the former (resp.\ latter) case we randomly duplicate (resp.\ kill) one of the other particles.

The birth-death process implements the Fisher--Rao component of the WFR gradient flow.
To implement the Wasserstein component, we stipulate that each particle evolves independently according to a Langevin diffusion between birth-death events.

The preceding discussion is ambiguous: what is the measure $\mu_t$? Ideally it should be the marginal law of $X_t^i$ (which is independent of $i$ due to exchangeability), but we do not have access to this marginal law for implementation purposes.
The methodology from the previous subsection suggests to replace $\mu_t$ by the empirical measure $\mu_t^N \deq \frac{1}{N}\sum_{i=1}^N \delta_{X_t^i}$ over the particles, but for the KL divergence the first variation is not well-defined at such a measure.
Therefore, as suggested in~\cite{LuLuNol19}, we resort to a kernel density estimator, replacing $\mu_t$ with $k \star \mu_t^N$ for an appropriate kernel function $k : \R^d\to\R$.
This leads to the following algorithm.
\begin{enumerate}
    \item Associate with each particle $X_t^i$ an independent clock that rings with instantaneous rate $\widehat\alpha_t(X_t^i)$, where
    \begin{align*}
        \widehat\alpha_t(x) \deq \log \frac{k\star \mu_t^N}{\pi}(x) - \frac{1}{N} \sum_{j=1}^N \log \frac{k\star \mu_t^N}{\pi}(X_t^j)\,.
    \end{align*}
    Note that when $\pi$ is given as an unnormalized density $\pi\propto\exp(-V)$, the computation of $\widehat\alpha_t$ does not require knowledge of the normalization constant for $\pi$.
    \item When one of the clock rings, kill or duplicate the corresponding particle $X_t^i$, and randomly duplicate or kill another particle, according to $\widehat\alpha_t(X_t^i) > 0$ or $\widehat\alpha_t(X_t^i) < 0$.
    \item Between the rings of the clocks, each particle evolves according to a Langevin diffusion:
    \begin{align*}
        \ud X_t^i = -\nabla V(X_t^i)\,\ud t + \sqrt 2 \,\ud B_t^i\,, \qquad i \in [N]\,,
    \end{align*}
    where ${\{B^i\}}_{i\in [N]}$ are i.i.d.\ Brownian motions.
\end{enumerate}
As shown in~\cite{LuLuNol19}, the marginal laws of this process converge to the WFR gradient flow as $N\to\infty$ and the bandwidth of the kernel tends to zero appropriately.

\section{Non-parametric maximum likelihood}\index{non-parametric maximum likelihood (NPMLE)}\label{sec:NPMLE}

We have just seen that sampling can be viewed as optimization over the space of probability measures using the perspective originally put forward in~\cite{JorKinOtt98} and \cite{Wib18}. In this section we study a classical statistical problem that is readily of this nature.

Consider a Gaussian mixture on $\R^d$ with density given by:
$$
G_\rho = \int_{\R^d} \phi(\cdot-y)\, \rho(\ud y)=\phi \star \rho \,,
$$
where $\phi(z)=(2\pi)^{-d/2} \exp(-\|z\|^2/2)$ denotes the density of the standard isotropic Gaussian distribution $\cN(0, I)$ and $\rho$ is the mixing distribution of interest. Note that in comparison with the general Gaussian mixtures introduced in Section~\ref{sec:gaussian_mixtures}, we constrain all the Gaussian components to have identity covariance matrix for simplicity.

Let $\rho^\star$ be an unknown mixing distribution on $\R^d$. Given $n$ independent observations $X_1, \ldots, X_n$ drawn from $G_{\rho^\star}$, our goal is to estimate $\rho^\star$. This is a Gaussian deconvolution problem, for which the rates of convergence are known to be very slow~\cite{RigWee19}. When $\rho^\star$ is assumed to be smooth, a classical approach that leads to optimal rates of convergence uses kernel smoothing~\cite{Fan91a}. In this section, we explore a different approach called non-parametric maximum likelihood estimation following the paper~\cite{YanWanRig23}.

The negative log-likelihood for this problem is defined as:
$$
\ell_n(\rho) \deq  -\frac{1}{n} \sum_{i=1}^n \log G_\rho(X_i) =  -\frac{1}{n} \sum_{i=1}^n \log \phi \star \rho (X_i)\,.
$$
The non-parametric maximum likelihood estimator, or NPMLE, is defined as any minimizer of $\ell_n$:
\begin{equation}\label{eq:NPMLE}
    \hat \rho = \argmin_{\rho \in \cP(\R^d)} \ell_n (\rho)\,.
\end{equation}

Before turning to the computational aspects of this problem, we first note a surprising connection with entropic optimal transport, highlighted in~\cite{RigWee18}.
Writing $\mu_n$ for the empirical measure $\frac 1n \sum_{i=1}^n \delta_{X_i}$, it turns out that the NPMLE satisfies
\begin{equation}\label{eq:eot_NPMLE}
	\hat \rho = \argmin_{\rho \in \cP(\R^d)} S_{2 \sigma^2}(\mu_n, \rho)\,,
\end{equation}
where $S_{2 \sigma^2}(\cdot, \cdot)$ is the entropic optimal transport cost with regularization parameter $\eps = 2 \sigma^2$.
In other words, the NPMLE precisely minimizes the entropic OT cost to the data $\mu_n$.

This connection arises from duality.
A version of the Gibbs variational principle (see Proposition~\ref{prop:gibbs}) tailored to probability measures rather than general positive measures implies that for suitable $h: \R^d \to \R$, it holds
\begin{equation*}
	\log \int \exp(h) \, \ud Q = \sup_{P\in \cP(\R^d)}{\Bigl\{\int h \, \ud P - \KL(P \mmid Q)\Bigr\}}\,.
\end{equation*}
This result can be found for example as Proposition~1.4.2 in~\cite{DupEll97}.
We can use this expression to obtain a variational formulation of $\log G_\rho(X_i)$:
\begin{align*}
    &- \log G_\rho(X_i)
    = (2 \pi)^{d/2} - \log \int \exp\Bigl(-\frac{\|X_i - y\|^2}{2\sigma^2}\Bigr)\, \rho(\ud y) \\
    &\qquad = (2 \pi)^{d/2} + \inf_{P_i \in \cP(\R^d)} \int \frac{1}{2 \sigma^2}\,\|X_i - y\|^2\, P_i(\ud y) + \KL(P_i \mmid \rho)\,.
\end{align*}
This shows that the optimization problem in~\eqref{eq:NPMLE} is equivalent to
\begin{align*}
	\hat \rho = \argmin_{\rho \in \cP(\R^d)} \inf_{P_1, \dots, P_n \in \cP(\R^d)} \frac 1n \sum_{i=1}^n \Bigl[\int \|X_i - y\|^2\, P_i(\ud y) + 2 \sigma^2 \KL(P_i \mmid \rho)\Bigr]\,.
\end{align*}
The minimization problem over $P_1, \dots, P_n$ can be equivalently rewritten as a minimization over measures of the form $\gamma = \frac 1n \sum_{i=1}^n \delta_{X_i} \otimes P_i$, which are precisely joint measures whose first marginal is $\mu_n$.
It can be shown that the second marginal can be taken to be $\rho$, which leads to the representation in~\eqref{eq:eot_NPMLE}.

The convex infinite-dimensional optimization problem~\eqref{eq:NPMLE} has primarily been studied in the case where $d=1$ where it can be shown that the solution is unique and supported on a small number of atoms~\cite{Lin83, PolWu24}. Using these properties, various computational schemes have been proposed by restricting the set of measures to ones with a small support. In higher dimensions, much less is known.

The definition~\eqref{eq:NPMLE} of the NPMLE is precisely an optimization over probability measures and in the rest of this section we describe algorithms based on Wasserstein gradient flows to solve it.

We first compute the Wasserstein gradient. To that end, we need the first variation of $\ell_n$. Fix $\eps>0$ and $\xi$ be a perturbation such that $\rho + \eps \xi$ is a probability measure and observe that
\begin{align*}
   \ell_n (\rho+ \eps \xi) &=  -\frac{1}{n} \sum_{i=1}^n \log (\phi \star \rho +  \eps \phi \star \xi) (X_i) \\
   &= -\frac{1}{n} \sum_{i=1}^n \log (\phi \star \rho) (X_i)   -\frac{\eps}{n} \sum_{i=1}^n \frac{\phi \star \xi(X_i) }{\phi \star \rho(X_i)} + O(\eps^2)\,.  
\end{align*}
Hence
$$
\lim_{\eps \searrow 0} \frac{ \ell_n (\rho+ \eps \xi) -  \ell_n (\rho)}{\eps} =  -\frac{1}{n} \sum_{i=1}^n \frac{\phi \star \xi(X_i) }{\phi \star \rho(X_i)} = -\frac{1}{n} \sum_{i=1}^n \frac{\int \phi(\cdot - X_i)\, \ud \xi }{\phi \star \rho(X_i)} \,,
$$
and we readily identify that the first variation is given by
$$
\delta \ell_n(\rho) =-\frac1n \sum_{i=1}^n\frac{ \phi(\cdot - X_i)  }{\phi \star \rho(X_i)}\,.
$$
The Wasserstein gradient flow of $\ell_n$ correspond to the following ODE:
\begin{align}
    \dot \theta_t &= - \gradW \ell_n(\rho_t) (\theta_t) \nonumber \\
    &=  \frac1n \sum_{i=1}^n\frac{ \nabla \phi(\theta_t  - X_i)  }{\phi \star \rho_t(X_i)} \nonumber\\
    &= -\frac1n \sum_{i=1}^n\frac{ (\theta_t -X_i)\,\phi(\theta_t  - X_i)  }{\phi \star \rho_t(X_i)} \label{eq:WGF_NPMLE} \,,
\end{align}
where $\rho_t = \law(\theta_t)$. In light of the continuity equation~\eqref{eq:continuity}, we see that the Wasserstein gradient flow of $\ell_n$ is the curve described by the PDE:
\begin{align*}
\partial_t \rho_t = \frac1n \sum_{i=1}^n \divergence\Bigl( \rho_t\,\frac{ (\cdot -X_i)\,\phi(\cdot  - X_i)  }{\phi \star \rho_t(X_i)} \Bigr)\,.
\end{align*}

Since the velocity field in~\eqref{eq:WGF_NPMLE} depends on $\rho_t$, we use a particle implementation of this gradient flow: given $N$ particles $\theta_t^1, \ldots, \theta_t^N$ with marginal distribution $\rho_t$, we replace $\rho_t$ with the empirical distribution $N^{-1} \sum_{j=1}^N \delta_{\theta_t^j}$. It results in the system of coupled ODEs: for $j\in [N]$,
$$
\boxed{\dot \theta_t^j = -\frac1n \sum_{i=1}^n\frac{ (\theta_t^j -X_i)\,\phi(\theta_t^j  - X_i)  }{\frac1N  \sum_{k=1}^N \phi(\theta_t^k- X_i)}\,.}
$$

Unfortunately, this Wasserstein gradient flow is difficult to analyze. Moreover, time-discretizations of this Wasserstein gradient flow do not perform well in practice. Instead,~\cite{YanWanRig23} propose to study the Wasserstein--Fisher--Rao gradient flow of $\ell_n$. In this context, the measure $\rho_t$ is approximated by
$$
\sum_{j=1}^N w_t^j \delta_{\theta_t^j}
$$
where for any $j \in [N]$, we use the following dynamics:
\begin{align*}
    \boxed{\begin{aligned}
        \dot \theta_t^j &= - \frac1n \sum_{i=1}^n\frac{ (\theta_t^j -X_i)\,\phi(\theta_t^j  - X_i)  }{\sum_{k=1}^N w_t^k \phi(\theta_t^k - X_i)}\,,\\
        \dot w_t^j&= \Bigl[  \frac1n \sum_{i=1}^n\frac{\phi(\theta_t^j  - X_i)  }{\sum_{k=1}^N w_t^k \phi(\theta_t^k - X_i)} - 1 \Bigr]\, w_t^j\,.
    \end{aligned}}
\end{align*}
Under some conditions, the convergence guarantees of this system can be established but they are entirely driven by the Fisher--Rao part and the proof largely consists in finding conditions under which the Wasserstein part does not get in the way of convergence.

\section{Mean-field neural networks}\index{mean-field neural networks}\label{sec:mean_field_NN}

A \emph{two-layer\footnote{In other words, the network has one hidden layer.} neural network}\index{neural network} is a parameterized function
\begin{align}\label{eq:mean_field_NN}
    f(x; \theta)
    &= \frac{1}{m} \sum_{j=1}^m a_j \sigma(\langle w_j,x\rangle + b_j)\,,
\end{align}
where $\theta \deq \{(a_j,w_j,b_j),\,j\in [m]\}$ represents the \emph{parameters} (or \emph{weights}) of the network, and $\sigma(\cdot)$ is a non-linearity, e.g., the common ReLU activation $\sigma(\cdot) = {(\cdot)}_+ = \max(0, \cdot)$. 

The first layer is the map $\ell_1: \R^d \to \R^m$ defined by
\begin{equation}\label{eq:first_layer}
    \ell_1(x) = \bigl(\sigma(\langle w_1,x\rangle + b_1), \ldots, \sigma(\langle w_m,x\rangle + b_m)\bigr)^\T \eqqcolon \sigma( Wx + b)\,.
\end{equation}
It produces an internal \emph{representation} of the vector $x$ that is more suitable for the subsequent task; e.g. classification or regression. This representation is then passed on to the second and terminal layer $\ell_2: \R^m \to \R$ which collapses the representation $z \deq \ell_1(x)$ into a scalar prediction using a linear projection:
$$
\ell_2(z) = \frac{1}{m} \sum_{j=1}^m a_j z_j = \frac1m \langle a, z\rangle\,.
$$
This terminal layer is tailored to a regression task but it is also common to  employ terminal layers that are tailored to classification. For binary classification for example, it is desirable to have the output of the neural network lie in the interval $[0,1]$ so further process the output $y = \ell_2\circ \ell_1(x)$ using 
$$
\msf{logistic}(y) =  \frac{e^{y}}{1+e^y}\,.
$$
The reader will recognize here the logistic function\index{logistic function} employed in generalized linear models. In the rest of this section we focus on two-layer neural networks of the form $\ell_2\circ \ell_1$ and leave the study of $\msf{logistic}\circ \ell_2 \circ \ell_1$ as an exercise for the reader.

The parametrization~\eqref{eq:mean_field_NN} is called the \emph{mean-field} parametrization, and it lends itself to passing to the limit $m\to\infty$.

Consider a simple regression task in which we have $n$ data points $\{(X_i, Y_i),\; i\in [n]\}$ with $X_i\in\R^d$ and $Y_i \in\R$.
To train the neural network, we can minimize the squared error
\begin{align}\label{eq:nn_loss}
    L(\theta)
    &\deq \sum_{i=1}^n \bigl(Y_i - f(X_i; \theta)\bigr)^2\,.
\end{align}

The training dynamics for minimizing the objective~\eqref{eq:nn_loss} are complex because the parametrization $\theta \mapsto f(\cdot;\theta)$ is non-linear and consequently the loss $L$ is non-convex.
One approach to study these dynamics is to lift the optimization problem to one set over the space of probability measures, with the hope that the lifted problem affords simplifications.
In doing so, we must preserve the connection with the original dynamics, and this leads naturally to Wasserstein gradient flows.

The lifting is carried out as follows.
Let $\Omega = \R\times \R^d\times \R$ denote the space of $(a,w,b)$ triples, and let $\mu$ be a measure over $\Omega$.
Set
\begin{align*}
    f(x; \mu)
    &= \int \rho(x; \omega)\,\mu(\ud \omega)\,, \qquad \rho(x;\omega) \deq a\sigma(\langle w, x\rangle + b)\,,
\end{align*}
where $\omega = (a,w,b)$.
One can check that if we encode parameters $\theta = \{(a_i,w_i,b_i)\}_{i=1}^m$ via the empirical measure $\mu_\theta \deq \frac{1}{m} \sum_{i=1}^m \delta_{(a_i,w_i,b_i)}$, then $f(\cdot; \mu_\theta) = f(\cdot;\theta)$, so this definition indeed generalizes~\eqref{eq:mean_field_NN}.
However, we can now formulate the problem of optimizing, over the space $\cP_2(\Omega)$ of probability measures over $\Omega$, the objective
\begin{align}\label{eq:nn_loss_wass}
    L(\mu) \deq \sum_{i=1}^n \bigl(Y_i - f(X_i; \mu)\bigr)^2\,.
\end{align}

As discussed above, we require that the dynamics of minimizing $\mu\mapsto L(\mu)$ over $\cP_2(\Omega)$ be compatible with the original dynamics of minimizing $\theta\mapsto L(\theta)$ over $\Omega^m$.
Herein lies the utility of the Wasserstein geometry, as it was set up precisely to ensure that dynamics over the base space $\Omega$ lift gracefully to dynamics over $\cP_2(\Omega)$.
More precisely:

\begin{proposition}\label{prop:mfnn}
    The Wasserstein gradient flow ${(\mu_t)}_{t\ge 0}$ of~\eqref{eq:nn_loss_wass}, when initialized at a measure of the form $\mu_0 = \mu_\theta$, is such that $\mu_t = \mu_{\theta_t}$ for all $t\ge 0$, where ${(\theta_t)}_{t\ge 0}$ is the (time-rescaled) Euclidean gradient flow of~\eqref{eq:nn_loss} initialized at $\theta$.
\end{proposition}

We can also rewrite the objective~\eqref{eq:nn_loss_wass} as follows:
\begin{align*}
    L(\mu)
    &= \sum_{i=1}^n \Bigl(Y_i^2 - 2Y_i \int \rho(X_i; \omega)\, \mu(\ud \omega) \\
    &\qquad\qquad{} + \iint \rho(X_i; \omega) \,\rho(X_i;\omega')\,\mu(\ud \omega)\,\mu(\ud \omega')\Bigr) \\
    &= \text{const.} -2\int \sum_{i=1}^n Y_i \,\rho(X_i;\omega)\,\mu(\ud \omega) \\
    &\qquad\qquad{} + \iint \sum_{i=1}^n \rho(X_i; \omega) \,\rho(X_i;\omega')\,\mu(\ud \omega)\,\mu(\ud\omega')\,.
\end{align*}
We can recognize the second term as a potential energy (in the sense of Example~\ref{ex:potential}) and the third term as an interaction energy (in the sense of Example~\ref{ex:interaction}), albeit a generalized version in which the interaction is not of the form $(x,y)\mapsto K(x-y)$.
As expected, the loss $L$ is not in general geodesically convex.

Since the Wasserstein perspective is essentially a reformulation of the original neural network problem, a skeptic may ask what advantages it brings.
The answer is that we can now consider more general initializations than empirical measures (measures of the form $\mu_\theta$), and in particular, a well-known result of Chizat and Bach~\cite{ChiBac18GlobalConv} uses this approach to establish global convergence in the mean-field regime, under certain assumptions.
Their result requires the initialization to be absolutely continuous, and since such a measure can only be approximated by empirical measures in the limit $m\to\infty$, this corresponds in some sense to ``infinitely wide'' neural networks.
The error incurred for finite $m$ can be controlled and leads to insights for finite-width networks in various settings~\cite{MeiMonNgu18MeanField, AbbBoiMis22Staircase, AbbBoiMis23Leap}.

It has also been proposed to add an entropic regularization term to the loss~\eqref{eq:nn_loss_wass} and to train the network via the mean-field Langevin dynamics from Subsection~\ref{subsec:mkv}~\cite{CheRenWan23PoC, SuzNitWu23MFLangevin, TzeRag24MKV}.

\section{Transformers}\index{transformers}\label{sec:transformers}

Since their introduction in 2017 in the paper ``Attention is all you need"~\cite{VasShaPar17}, transformers have profoundly transformed practical deep neural networks, most notably in natural language processing (NLP), but also in computer vision and robotics. Central to this new architecture is the so-called \emph{attention}\index{attention} mechanism, a layer that is markedly different from a perceptron\index{perceptron} (a.k.a.\ feed-forward) layer such as the one in~\eqref{eq:mean_field_NN}. 

Unlike the neural networks that we have seen in the previous sections that are functions $f: \R^d \to \R$, an attention layer is a sequence-to-sequence map 
\begin{align*}
g: (\R^d)^N &\to  (\R^d)^N\,,\\
(x^1, \ldots, x^N) & \mapsto \big(g^1(x^1, \ldots, x^N), \ldots, g^N(x^1, \ldots, x^N)\big)\,.
\end{align*}
More specifically, a input (a sentence in NLP or an image in computer vision) is broken into \emph{tokens} $x^1, \ldots, x^N \in \R^d$ and processed through an attention layer $g=(g^1, \dotsc, g^N)$ where for each $i \in [N]$, 
$$
g^i(x^1, \ldots, x^N) = x^i + V\, \frac{\sum_{j=1}^N x^j e^{\langle Qx^i, K x^j\rangle}}{\sum_{j=1}^N e^{\langle Qx^i, K x^j\rangle}}\,,
$$
where $K$, $Q$, and $V$ are three $d \times d$ matrices called \emph{key}, \emph{query}, and \emph{value} respectively. 

While practical transformers combine perceptron layers with attention layers---and also normalization layers that are briefly discussed below---we focus here on composing multiple attention layers with the same matrices $(K,Q,V)$. This composition results in a iterative scheme where tokens are updated as:
$$
x^i_{t+1} =  x^i_t + V\, \frac{\sum_{j=1}^N\,x^j_t e^{\langle Qx^i_t, K x^j_t\rangle}}{\sum_{j=1}^N\,e^{\langle Qx^i_t, K x^j_t\rangle}}\,, \qquad i \in [N]\,.
$$
In turn, taking the same perspective as in neural ODEs~\cite{CheRubBet18}, we can view the above iterations as a time discretization of the following dynamical system of interacting particles:
\begin{equation}\label{eq:ode_trans}
    \dot{x}^i_t = V\, \frac{\sum_{j=1}^N\,x^j_t e^{\langle Qx^i_t, K x^j_t\rangle}}{\sum_{j=1}^N\,e^{\langle Qx^i_t, K x^j_t\rangle}}\,, \qquad i\in [N]\,.
\end{equation}
The above equation describes a system of $N$ ordinary differential equations (ODEs), one for each token/particle, that are called \emph{self-attention dynamics} by~\cite{GesLetPol23,GesLetPol24}. The way these tokens interact is not completely wild: each token  evolves according to its own position and the \emph{empirical distribution} of all the tokens. Indeed, let $\mu_t$ denote this empirical distribution at time $t$:
$$
\mu_t = \frac1N \sum_{i=1}^N \delta_{x_t^i}\,.
$$
We can rewrite~\eqref{eq:ode_trans} as
$$
\dot{x}^i_t = V\, \frac{\int y e^{\langle Qx^i_t, K y\rangle} \, \mu_t(\ud y)}{\int e^{\langle Qx^i_t, K y\rangle} \, \mu_t(\ud y)}\,, \qquad i\in [N]\,.
$$
It becomes now clear that the tokens all have the same mean-field dynamics so we can drop the index $i$:
\begin{equation}\label{eq:mean_field_trans}
    \dot{x}_t = V\, \frac{\int y e^{\langle Qx_t, K y\rangle} \, \mu_t(\ud y)}{\int e^{\langle Qx_t, K y\rangle} \, \mu_t(\ud y)}\,,
\end{equation}
where $\mu_t = \law(x_t)$. Using the continuity equation, we get that $\mu_t$ evolves according to the following PDE:
$$
\partial_t \mu_t + \divergence\Bigl( \mu_t V\, \frac{\int y e^{\langle Q\cdot, K y\rangle} \, \mu_t(\ud y)}{\int e^{\langle Q\cdot, K y\rangle} \, \mu_t(\ud y)}\Bigr)=0\,.
$$
This perspective on transformers was first put forward in~\cite{SanAblBlo22} which raised the question of whether this curve could be viewed as a Wasserstein gradient flow. To investigate this question, assume that $K=Q=V=I$ so that~\eqref{eq:mean_field_trans} becomes:
$$
\dot{x}_t = \frac{\int y e^{\langle x_t,  y\rangle} \, \mu_t(\ud y)}{\int e^{\langle x_t,  y\rangle} \, \mu_t(\ud y)} = \nabla\Bigl[\log \int  e^{\langle \cdot,  y\rangle} \, \mu_t(\ud y)\Bigr](x_t)\,.
$$
The form of the velocity field is suggestive and readily begs the question of whether there exists a functional $\cF$ such that its first variation is given by
$$
\delta \cF(\mu) = \log \int  e^{\langle \cdot ,  y\rangle} \, \mu_t(\ud y)\,.
$$
Unfortunately,~\cite{SanAblBlo22} also show that this is not the case due to a lack of symmetry. To overcome this limitation, one may consider instead the 
\emph{unnormalized self-attention} dynamics introduced in~\cite{GesLetPol24}. These dynamics are of the form
$$
\dot{x}_t = \int y e^{\langle x_t,  y\rangle} \, \mu_t(\ud y) = \nabla\Bigl[ \int  e^{\langle \cdot,  y\rangle} \, \mu_t(\ud y)\Bigr](x_t)\,.
$$
We readily get that these dynamics describe the Wasserstein gradient flow of the interaction energy
$$
\cF(\mu) \deq  - \iint e^{\langle x, y \rangle} \, \mu(\ud x) \, \mu (\ud y)\,.
$$
Unfortunately, it is easy to see that this functional does not admit a global minimum over the space of probability measure, indeed, for any Dirac delta $\mu=\delta_x$, we have $\cF(\delta_x) = -e^{\|x\|^2} \to -\infty$ as $x \to \infty$. In particular this suggests that tokens undergoing these dynamics will simply diverge to infinity.

In practice however, tokens are restricted to live on the unit sphere $\cS^{d-1}$ of $\R^d$ using a procedure know as \emph{layer normalization} (or simply ``layernorm"). With layernorm, the unnormalized self-attention dynamics then become
$$
\dot{x}_t = {\sf P}_{x_t} \int y e^{\langle x_t,  y\rangle} \, \mu_t(\ud y) = \nabla_{x_t} \int  e^{\langle x_t, y\rangle}\, \mu_t(\ud y)\,,
$$
where for any $x \in \cS^{d-1}$, $y \in  \R^d$, we write  ${\sf P}_{x} y \deq y - \langle x,y\rangle\, x$ for the projection of $y$ onto the tangent space of the sphere $\cS^{d-1}$ at $x$ and $\nabla_{x} \deq  {\sf P}_{x} \nabla$ denotes the spherical (Riemannian) gradient at $x \in \cS^{d-1}$. Using the version Otto calculus on Riemannian manifolds alluded to in Section~\ref{sec:gaussian_mixtures}, we get that these dynamics correspond to a Wasserstein gradient flow of the interaction energy $\cF$ now defined on the sphere. In fact, since for $x,y \in \cS^{d-1}$, it holds that $\|x-y\|^2 = 2\,(1-\langle x, y \rangle)$, we can write $\cF$ as
\begin{equation}\label{eq:functional_sphere}
    \cF(\mu) \deq  - e \iint_{\cS^{d-1} \times \cS^{d-1}}  e^{-\frac{\|x-y\|^2}{2}} \, \mu(\ud x) \, \mu (\ud y)\,.
\end{equation}

It can be readily seen that the maximizers of this functional are precisely Dirac deltas $\delta_x$ for any $x \in \cS^{d-1}$. Hence unnormalized self-attention is a Wasserstein gradient flow of a functional minimized at Dirac deltas, which correspond to states where all the tokens are clustered. Unfortunately, this functional is not geodesically convex and admits many stationary points where the Wasserstein gradient vanishes. Using this framework~\cite{CriRebMcR24, GesLetPol24} show that these points are in fact saddle points, guaranteeing asymptotic convergence to a single cluster when dynamics are initialized in a generic position.

\section{Discussion}

\noindent\textbf{\S\ref{sec:vi}.}
Another natural geometric approach to VI is \emph{natural gradient descent}, which is motivated by the parametrization invariance of the Fisher--Rao geometry~\cite{Ama98NatGrad, AmaNag00}. However, it has been challenging to analyze this approach since the VI problem is often non-convex.
See~\cite{AwaRis15VI} for an early work on algorithmic guarantees for VI\@.

Our discussion of Gaussian VI follows~\cite{Lametal22GVI}.
Algorithms for Gaussian VI which are closely related to~\eqref{eq:bw_gvi} have been proposed and studied in works such as~\cite{AlqRid20VI, Dom20VI, GalPerOpp21VarGauss}; here, our emphasis is on the derivation via Otto calculus.
There have been many subsequent works on Gaussian VI, both on the computational (e.g.,~\cite{Diaoetal23FBGVI, DomGowGar23BBVI, Kim+23BBVI, BonLamBac24LowRank}) and statistical (\cite{KatRig23GVI}) aspects, as well as applications to bandits~\cite{ClaHuiDur24VITS} and control~\cite{LamBonBac23Kushner, LamBacBon24VarControl}.

The potential application of Wasserstein geometry to mean-field VI was noticed by several authors~\cite{Gho+22MFVI, Lac23IndepProj, YaoYan23MFVI}.
The works~\cite{Gho+22MFVI, Lac23IndepProj} also wrote down interacting SDE implementations of the gradient flow described in Subsection~\ref{subsec:mkv}.
The wider literature on mean-field VI, which usually focuses on coordinate ascent variational inference (CAVI), is vast and we do not survey it here, but see~\cite{ArnLac24CAVI, LavZan24CAVI} for analyses leveraging Otto calculus.

SVGD was introduced in~\cite{LiuWan16SVGD}, and geometric interpretations of SVGD are given in~\cite{Liu17SVGD, Cheetal20SVGD, DunNueSzp23SVGD}.
Convergence theory remains underdeveloped, see, e.g.,~\cite{LuLuNol19SVGD, Kor+20SVGD, SalSunRic22SVGD, DasNag23SVGD, ShiMac23SVGD, PriBiaSal24SVGD}.

Corollary~\ref{cor:kl_cvx} can be generalized to measures over Riemannian manifolds, in which case the strong convexity parameter of the KL divergence captures information about the Ricci curvature.
The seminal work of Lott and Villani~\cite{LotVil09} and Sturm~\cite{Stu06,Stu06a} leverages this to define a synthetic notion of Ricci curvature lower bounds for measured geodesic spaces (see Chapter~\ref{chap:geometry}) that recover classical Ricci curvature lower bounds when specialized to the Riemannian setting.
These ideas were later extended to discrete settings, e.g.,~\cite{Oll10Survey, OllVil12RicciHypercube, Oll13Visual}.

Finally note that while the KL divergence plays a preponderant role in variation inference, other distances between measures can be considered for this task; see, e.g.,~\cite{ArbKorSal19} who use Maximum Mean Discrepancy.

\noindent\textbf{\S\ref{sec:sampling}.} As mentioned in the main text, the interpretation of the Langevin diffusion as a Wasserstein gradient flow goes back to the seminal work of~\cite{JorKinOtt98}.
Otto calculus was first applied to obtain quantitative results for the Langevin diffusion, as in Exercise~\ref{ex:apply_otto_to_langevin} below, in~\cite{OttVil00LSI}; in this context, Exercise~\ref{ex:weakly_cvx_rate} from Chapter~\ref{chap:WGF} can be sharpened by a factor of $2$~\cite{BobGenLed01Hyper, OttVil01Comment}.
See also~\cite{Cor02MassTransport} for proofs in this spirit which do not require as much differential structure.
For textbook treatments on stochastic calculus, see~\cite{Ste01StocCalc} or~\cite{LeG16StocCalc}. The Wasserstein PL inequality for the KL divergence functional is known as the \emph{log-Sobolev inequality} and it plays a key role in the study of high-dimensional probability and Markov processes.
Crucially, although we have presented strong log-concavity as a sufficient condition for the validity of the log-Sobolev inequality, it is not necessary.
See~\cite{BakGenLed14} for further detail and~\cite{OhtTak11DisplI, OhtTak13DisplII, BlaBol18Loj} for generalizations.

The optimization perspective on sampling dates back to early works such as~\cite{DalTsy12LMC}; our discussion largely follows~\cite{Wib18}.
For an exposition to the modern complexity theory of log-concave sampling and further references, see~\cite{ChewiThesis, Chewi24Book}.
Recent works also apply this perspective for parameter estimation~\cite{Aky+23MMLE, Cap+24ParticleGD}.

The proximal sampler has been applied to structured log-concave sampling~\cite{LeeSheTia21RGO}, to non-Euclidean~\cite{Gopi+23LogLaplace} and heavy-tailed~\cite{He+24StableOracle} sampling, and to sampling from convex bodies~\cite{KooVemZha24InOut}.

As noted in~\cite{CheLe-Lu20}, the Newton--Langevin diffusion converges to any strictly log-concave target with a universal exponential rate as a consequence of the Brascamp--Lieb inequality~\cite{BraLie76}.
There is a sense in which it is an optimal preconditioning of Langevin~\cite{CuiTonZah24OptPoincare}.

Convergence of the underdamped Langevin diffusion requires heavier machinery than Wasserstein gradient flows and is based on the theory of hypocoercivity, for which the standard reference is~\cite{Vil09Hypo}.

\noindent\textbf{\S\ref{sec:interacting}.}
Analysis of birth-death sampling was first carried out in~\cite{LuLuNol19} and improved in~\cite{LuSleWan23BirthDeath}.

\noindent\textbf{\S\ref{sec:NPMLE}.} The WFR gradient flow for NPMLE can be adapted to more general mixtures. Usually, asymptotic convergence of the gradient flow to the NPMLE is only established conditionally on convergence to a limit point. In fact, it is shown in~\cite{YanWanRig23} that the NPMLE is the only stationary point of the gradient flow initialized at a measure that is absolutely continuous. 

\noindent\textbf{\S\ref{sec:mean_field_NN}.} The study of training dynamics of two-layer neural networks from the mean-field perspective was proposed in four independent papers that were released within about a month period in 2018:~\cite{MeiMonNgu18MeanField} on April 18,~\cite{SirSpi20} and~\cite{RotVan22} both on May 2, and~\cite{ChiBac18GlobalConv} on May 24. It is worth noting that the normalization $1/m$ in~\eqref{eq:mean_field_NN} is critical for the mean-field interpretation of the problem. Other works have proposed to use the normalization $1/\sqrt{m}$ which results in the so called \emph{neural tangent kernel (NTK)}\index{neural tangent kernel (NTK)} (a.k.a.\  \emph{lazy training}) regime. In this regime, which will remind the reader of the normalization employed in the central limit theorem, it can be shown that the parameters do not move far away from a random initialization and the neural network can be studied using linear approximation around initialization~\cite{NTK}. For more details on NTK and its relationship with the mean-field regime see~\cite{MisMon23}.

\noindent\textbf{\S\ref{sec:transformers}.} The functional $\cF$ defined in~\eqref{eq:functional_sphere} has appeared in the literature on optimal configuration. In this line of work, the \emph{maximizers} of this functional are of interest. It is know that $\cF$ is maximized by the uniform distribution on the sphere~\cite{shuotan2017}. Finding maximizers subject to a cardinality constraint on the support of $\mu$ is directly connected to questions arising in sphere packing; see~\cite{cohn2007universally}.

\section{Exercises}

\begin{enumerate}
\item Let $K: \R^d \to \R$ be a symmetric function on $\R^d$, and consider the corresponding interaction energy as defined in Example~\ref{ex:interaction}: $\cF(\mu)  \deq  \frac 12 \iint K(x - y)\, \mu(\ud x) \, \mu(\ud y)$.
\begin{itemize}
    \item Show that if $K$ is convex, then $\cF$ is geodesically convex on $\cP_{2, \mathrm{ac}}(\R^d)$.
\emph{Hint}: let $X_t = (1-t) X_0 + t T(X_0)$, so that $(\mu_t)_{t \in [0, 1]}  \deq  (\mathrm{law}(X_t))_{t \in [0, 1]} $ is a Wasserstein geodesic, then apply~\eqref{eq:geod_cvx_def}.
\item Show that $\cF$ is never $\alpha$-geodesically convex for any $\alpha > 0$. \emph{Hint}: Consider the geodesic $(\cN(t v, I))_{t \in [0, 1]} $ for a nonzero vector $v \in \R^d.$
\end{itemize}

    \item\label{ex:mf_gradient} Show that the tangent space to the space of product measures at $\mu$ is given by the space of \emph{separable} vector fields:
    \begin{align*}
        &T_\mu {\cP_{2,\rm ac}(\R)}^{\otimes d} \\
        &\qquad = \overline{\{x \mapsto (\psi_1'(x_1),\dotsc,\psi_d'(x_d)) \mid \psi_1,\dotsc,\psi_d : \R\to\R\}}^{L^2(\mu)}\,,
    \end{align*}
    where $\psi_1,\dotsc,\psi_d$ are smooth and compactly supported.
    Then, show that the Wasserstein gradient projected to this subspace takes the form~\eqref{eq:mf_gradient}.

    \item Compute the gradient of the KL divergence restricted to the space of product measures.
    From the first-order optimality condition, write down a fixed-point equation for the density of the solution $q_\star$ to~\eqref{eq:mfvi}.

    \item\label{ex:mf_isometry} Prove Lemma~\ref{lem:mf_isometry}. \emph{Hint}: Use the separability of the transport maps to reduce to one-dimensional optimal transport, for which we can apply the results of Section~\ref{sec:ot_1d} and Proposition~\ref{prop:w21d}.

    \item Compute the derivative of $t\mapsto \KL(\mu_t \mmid \pi)$ when ${(\mu_t)}_{t\ge 0}$ evolves according to~\eqref{eq:svgd}.

    \item\label{ex:apply_otto_to_langevin} Interpret Exercises~\ref{ex:weakly_cvx_rate} and~\ref{ex:otto_villani} from Chapter~\ref{chap:WGF} for the Langevin diffusion.

    \item Consider the Euler{--}Maruyama scheme~\eqref{eq:lmc} where the initial distribution is $\cN(m, \Sigma)$ and the target distribution is $\pi = \cN(0, I)$.
    Compute the law of $X_k$ for each $k\ge 0$.
    Use this to compute the stationary distribution $\hat\pi$ of~\eqref{eq:lmc}, and compute the KL divergence $\KL(\hat\pi \mmid \pi)$.
    How small should we choose the step size $h$ if we want to ensure $\KL(\hat\pi \mmid \pi) \le \varepsilon^2$?
    
    \item Let $\mu = \cN(m, \Sigma)$ be a Gaussian measure with $\Sigma\succ 0$. Evaluate $\prox_{h\cH}(\mu)$, where $\prox_{h\cH}$ is the proximal map for the entropy defined in~\eqref{eq:prox_ent}.
    This computation is used as the basis for the Gaussian VI algorithm of~\cite{Diaoetal23FBGVI}.

    \item For $\cF \deq \frac{1}{2}\,W_2^2(\cdot,\nu)$, compute the iterations $\mu_{n+1} = \prox_{h\cF}(\mu_n)$ of the JKO scheme, where
    \begin{align*}
        \prox_{h\cF}(\mu) \deq \argmin_{\mu' \in \cP_2(\R^d)}\Bigl\{h\cF(\mu') + \frac{1}{2}\,W_2^2(\mu,\mu')\Bigr\}\,.
    \end{align*}
    Letting $h\searrow 0$ while $nh\to t$, show that one recovers the gradient flow from Exercise~\ref{ex:wgf_of_w2} from Chapter~\ref{chap:WGF}.

    \item Consider the proximal sampler with initial distribution $\cN(m,\Sigma)$ and target distribution $\pi = \cN(0, I)$. Compute the law of the $k$-th iterate $X_k$ for each $k\ge 0$ and estimate the rate of convergence to $\pi$.
    
    \item Let $V(x) = \frac{1}{2}\,\langle x,A\,x\rangle$ and $W(x) = \frac{\lambda}{2}\,\|x\|^2$, where $A \succ 0$ and $\lambda \ge 0$. Compute the stationary distributions $\pi$ and $\hat\pi_N$ of the McKean{--}Vlasov SDE~\eqref{eq:mkv} and the finite-particle system~\eqref{eq:mkv_fp} respectively.

    \item Prove Proposition~\ref{prop:mfnn}.
    Also, generalize to the case of a two-layer neural network composed with a logistic function when the training data satisfies $Y_i \in \{0,1\}$ for each $i\in [n]$ and we use the cross-entropy loss: $L(\mu) = -\sum_{i=1}^n \{(1-Y_i)\log(1-f(X_i;\mu)) + Y_i\log f(X_i;\mu)\}$.
\end{enumerate}

\chapter{Metric geometry of the Wasserstein space}
\label{chap:geometry}

In the previous two chapters, we studied the space $\cW_2 = (\cP_2(\R^\dd), W_2)$ through the lens of Riemannian geometry.
Although such an approach yields considerable geometric insight and can even be treated rigorously (see~\cite{AmbGigSav08}), it is important to keep in mind that $\cW_2$ is not a \emph{bona fide} Riemannian manifold, and consequently technical issues abound. 

Despite what its name suggests, \emph{metric geometry} requires a bit more structure than simply a metric space. Indeed, in the rest of this chapter, we will talk about length/geodesic spaces which have a continuous flavor. In particular it is possible to take derivatives of functions along smooth curves. This primitive differential structure is often sufficient to understand questions about curvature which we will employ to establish rates of convergence for Wasserstein barycenters in the next chapter.

The goal of this chapter is to gather basic material from metric geometry for a general metric space $(S,\md)$ following the classical book \cite{BurBurBur01}; see also \cite{AleKapPet22} for a more advanced coverage. Main concepts (curvature, tangent cone, logarithmic map, etc.)  are instantiated to the (2-)Wasserstein space.

\section{Geodesics}\label{subsec:geodesic}

We already appealed to an intuitive notion of geodesics in Chapter~\ref{chap:WGF}. In this chapter we properly define these objects as length minimizing.

\subsection{Length and geodesic spaces}\index{geodesic space}

Let $(S, \md)$ be a metric space.
A \emph{path}\index{path} in $S$ is a continuous map $\omega:I\to S$ where $I\subset \R$ is an interval.
The \emph{length}\index{length} $L(\omega)\in \R \cup\{\infty\}$ of a path $\omega:I\to S$ is defined by 
\begin{equation}
\label{EQ:deflength}
L(\omega) \deq \sup\sum_{i=1}^{n-1} \md(\omega(t_i),\omega(t_{i+1}))\,,
\end{equation}
where the supremum is taken over all $n\ge 1$ and all $n$-tuples $t_1<\dots< t_{n}$ in $I$.

A path is called \emph{rectifiable}\index{rectifiable path} if it has finite length. 

For any path $\omega$ and any interval $J \subset \R$, we write $\omega_{J}$ to denote the restriction of $\omega$ to $I \cap J$. The following lemma holds.

\begin{lemma}\label{lem:contpath}
For any rectifiable path $\omega:I \to S$, the function $t \mapsto \ell(t)=L(\omega_{(-\infty, t]})$ is continuous on $I$.
\end{lemma}
\begin{proof}
We prove left continuity, i.e., that for any $\eps>0$, there exists $\delta>0$ such if $t-\delta<t'\le t$, we have
$$
\ell(t)-\eps\le \ell(t')\le \ell(t)\,.
$$ 

By continuity of $\omega$, there exists $\delta_1>0$ such that  $t' \in (t-\delta_1, t]$ implies
\begin{equation}\label{EQ:contgamma}
\md(\omega(t'),\omega(t)) \le \frac{\eps}{2}\,.
\end{equation}
Next, let $n$ and $t_1 < \dots < t_{n}=t$ be such that 
\begin{equation}\label{EQ:deflengthpr}
\ell(t)-\frac{\eps}{2}\le \sum_{i=1}^{n-1} \md(\omega(t_i),\omega(t_{i+1}))\le \ell(t)\,,
\end{equation}
define $\delta_2=\min_{i=1,\dotsc,n-1}|t_{i+1}-t_i|>0$, and let $\delta=\min(\delta_1, \delta_2)>0$. Observe that for any $t'$ such that $t\ge t'>t-\delta$ it holds $t_{n-1} < t' \le t$ so that
\begin{align*}
\ell(t')&\ge  \sum_{i=1}^{n-2} \md(\omega(t_i),\omega(t_{i+1}))+ \md(\omega(t_{n-1}), \omega(t'))\\
&\ge \sum_{i=1}^{n-1} \md(\omega(t_i),\omega(t_{i+1})) - \md(\omega(t'),\omega(t))\\
&\ge \ell(t) -\eps
\end{align*}
where we used the triangle inequality in the second line and~\eqref{EQ:contgamma}--\eqref{EQ:deflengthpr} in the third. This completes the proof of left continuity. Right continuity follows using the same argument.
\end{proof}

Two paths $\omega_1: I_1 \to S$ and $\omega_2:I_2 \to S$ are  \emph{equivalent} if there exists a continuous, non-decreasing, and surjective function $\varphi:I_1 \to I_2$ such that $\omega_1=\omega_2\circ\varphi$. In this case, $\omega_2$ is a \emph{reparametrization} of $\omega_1$ (and vice-versa) and it is easy to check that $L(\omega_1)=L(\omega_2)$. 

Finally a path $\omega:[a,b]\to S$ is said to have \emph{constant speed} if for all $a\le s\le t\le b$,
\begin{equation}\label{cspeed}
    L(\omega_{[s,t]})=\frac{t-s}{b-a}\, L(\omega)\,.
\end{equation}

\begin{proposition}
Any rectifiable path $\omega:[a,b] \to S$  has a constant-speed reparametrization $\bar \omega:[0,1]\to S$.
\end{proposition}
\begin{proof}
Let us first reparametrize $\omega$ so that it is \emph{never locally constant} meaning that there exists no interval $[c,d] \subset [a,b]$ such that $\omega_{[c,d]}$ is constant. If such an interval exists, define 
$\pi:\R\to \R$ to be such that
$$
\pi(t)=\begin{cases}
t & \text{if } t\le c\,,\\
c& \text{if }  c<t \le d\,,\\
t-(d-c)&\text{if }  t> d\,.
\end{cases}
$$
Observe that $\pi([a,b])=[a,b-(d-c)]$ is an interval and that $\pi$ is continuous and non-decreasing on this interval. Then reparametrize $\omega$ into $\omega':[a,b-(d-c)]\to S$ such that $\omega=\omega'\circ \pi$ holds, which is possible since $\omega$ is constant on $[c,d]$.

By repeating this operation, we may assume that $\omega$ is never locally constant and, in particular, that the map $t \mapsto \varphi(t) \deq L(\omega_{[a,t]})/L(\omega)$ is strictly increasing on $[a,b]$ and continuous by Lemma~\ref{lem:contpath} and therefore invertible.  In particular, $\varphi^{-1}$ is also continuous, strictly increasing, and defined over $[0,1]$. We define $\bar \omega:[0,1]\to S$ by $\bar \omega=\omega\circ \varphi^{-1}$ which is a constant-speed reparametrization of $\omega$.
\end{proof}

Given $x,y \in S$, a path $\omega:[a,b]\to S$ is said to \emph{connect} (or \emph{join}) $x$ to $y$ if $\omega(a)=x$ and $\omega(b)=y$.
By construction of the length function $L$, $\md(x,y)\le L(\omega)$ for any path $\omega$ connecting $x$ to $y$.
The space $S$ is called a \emph{length space} if for all $x,y\in S$,
\begin{equation}
\label{dintrinsic}
\md(x,y)=\inf_{\omega} L(\omega),
\end{equation}
where the infimum is taken over all paths $\omega$ connecting $x$ to $y$. A length space is said to be a \emph{geodesic space} if for all $x,y\in S$, the infimum on the right hand side of \eqref{dintrinsic} is attained.

\begin{definition}
Let $(S,d)$ be a length space. A \emph{geodesic}\index{geodesic} between $x$ and $y$ is any path $\omega:[0,1]\to S$  attaining the infimum in \eqref{dintrinsic}.
\end{definition}
In other words, a geodesic is a shortest path between two points.
It follows from the minimizing property of a geodesic $\omega$ that 
\[ \md(\omega(s),\omega(t))=L(\omega_{[s,t]})\,,\]
for all $0\le s\le t\le 1$. Together with~\eqref{cspeed} it yields the following useful characterization of \emph{constant-speed geodesics}.

\begin{proposition}\label{prop:cspeedg}
A path $\omega:[0,1]\to S$ is a \emph{constant-speed geodesic} if and only if 
$$
\md(\omega(s),\omega(t))=(t-s)\,\md(\omega(0),\omega(1))\,,
$$
for all $0\le s\le t\le 1$.
\end{proposition}

\subsection{Midpoints}\label{subsec:midpoints}

We now obtain a characterization of geodesic spaces in terms of midpoints. For any two points $x, y$ in a metric space, a \emph{midpoint}\index{midpoint} of $(x,y)$ is any $z \in S$ such that
\[ \md(x,z)= \md(y,z)=\frac{1}{2}\,\md(x,y)\,.\]

\begin{proposition}\label{def:midpoint}
Let $(S,\md)$ be a complete metric space. Then the following are equivalent:
\begin{enumerate}[label=(\roman*)]
\item $(S, \md)$ is a geodesic space.
\item Any two points $x,y\in M$ admit a midpoint.
\end{enumerate}
\end{proposition}
\begin{proof}
We begin with the easy direction: $(i) \Rightarrow (ii)$. Let $\omega$ be a geodesic that connects $x$ to $y$, then clearly $\omega(1/2)$ is a midpoint.

To prove $(ii) \Rightarrow (i)$, we construct a path $\omega:[0,1]\to S$ such that $\omega(0)=x$, $\omega(1)=y$ and $L(\omega)=\md(x,y)$. To that end, we first define $\omega$ on the set $\cD$ of dyadic rationals of $[0,1]$ defined by
$$
\cD=\{k/2^m:m\ge 1,\, k\ge0\}\cap [0,1]\,.
$$
We proceed in a recursive fashion. Let $z$ be a midpoint of $(x,y)=(\omega(0),\omega(1))$ and define $\omega(1/2)=z$. Given $H_m \deq \{\omega(\frac{k}{2^m}),\; k \in [2^m]\}$, define $H_{m+1}=\{\omega(\frac{k}{2^{m+1}}),\; k \in [2^{m+1}]\}$ by setting $\omega(\frac{k}{2^{m+1}})=\omega(\frac{k/2}{2^{m}}) \in H_m$ if $k$ is even and letting $\omega(\frac{k}{2^{m+1}})$ be the midpoint of $(\omega_{\frac{(k-1)/2}{2^m}},\omega_{\frac{(k+1)/2}{2^m}})$ when $k$ is odd.
The union of $H_m$, $m \ge 0$ defines $\omega$ on $\cD$.

From our construction, for   $t,t' \in \cD$, it holds
\begin{equation}\label{EQ:lipext}
\md(\omega(t), \omega(t'))=|t-t'|\,\md(x,y)
\end{equation}
so that $\omega$ is $\md(x,y)$-Lipschitz on $\cD$. We now show that $\omega$ can be extended to a continuous function on $[0,1]$ that connects $x$ to $y$. To that end, fix $t \in [0,1]$ and let $(t_n)_{n\ge 0} \subseteq \cD$ be a sequence of dyadic integers that converges to $t$. Observe that $(\omega(t_n))_{n\ge 0}$ forms a Cauchy sequence in $(S,\md)$ since by~\eqref{EQ:lipext} it holds
$$
\md(\omega(t_n), \omega(t_m))\le |t_n-t_m|\,\md(x,y) \to 0\,, \qquad n,m \to \infty\,.
$$
Therefore since $S$ is complete, $(\omega(t_n))_{n\ge 0}$ converges and we set $\omega(t)$ to be its limit. To see that such an $\omega$ is continuous, note that for any $t,u \in [0,1]$, there exists sequences $(t_n)_{n\ge 0}, (u_n)_{n\ge 0} \subseteq \cD$ such that $t_n \to t$, $u_n \to u$ and
\begin{align*}
    \md(\omega(t), \omega(u))
    =\lim_{n \to \infty} \md(\omega(t_n), \omega(u_n))
    &\le \lim_{n \to \infty}|t_n-u_n|\,\md(x,y) \\
    &=|t-u|\,\md(x,y)
\end{align*}
where we used~\eqref{EQ:lipext} in the equality. Therefore, we have constructed a path that connects $x$ to $y$.

To conclude the proof, it suffices to observe that~\eqref{EQ:deflength} and~\eqref{EQ:lipext} imply that $L(\omega)=\md(x,y)$ as desired.
\end{proof}

\subsection{Geodesics in Wasserstein space}\label{subsec:geodesic-wass}

We are now in a position to place the Wasserstein space $\cW_2$ within the framework of metric geometry.
Compare the following theorem with Theorem~\ref{thm:benamou_brenier}.

\begin{theorem}\label{thm:W2geo}
The Wasserstein space $\cW_2$ is a geodesic space. Moreover, let $\pi_t(x,y) \deq (1-t)\,x+t\,y$, $t\in [0,1]$, and for any $\mu,\nu \in \cW_2$ let $\gamma \in \Gamma_{\mu,\nu}$ be an optimal transport plan in the sense that
$$
\int \|x-y\|^2\, \gamma(\ud x,\ud y)=W_2^2(\mu, \nu)\,.
$$
Then the path $\omega$ given by $\omega(t)=(\pi_t)_{\#}\gamma$ is a constant-speed geodesic in $\cW_2$ connecting $\omega(0)=\mu$ to $\omega(1)=\nu$.
\end{theorem} 
\begin{proof}
For any $0 \le s \le t \le 1$, define the coupling $\gamma_{s,t} \deq (\pi_s, \pi_t)_{\#}\gamma \in \Gamma_{\omega(s),\omega(t)}$. Then
\begin{align*}
W_2^2(\omega(s), \omega(t))&\le \int \|x-y\|^2\, \gamma_{s,t}(\ud x,\ud y)\\
&=\int \|\pi_s(x,y)-\pi_t(x,y)\|^2 \,\gamma(\ud x,\ud y)\\
&=\int \|(1-s)\,x+s\,y-((1-t)\,x+t\,y)\|^2 \,\gamma(\ud x,\ud y)\\
&=(t-s)^2\int \|x-y\|^2 \,\gamma(\ud x,\ud y)\\
&=(t-s)^2\,W_2^2(\omega(0), \omega(1))\,.
\end{align*}
We have proved that
$$
W_2(\omega(s), \omega(t))\le |t-s|\,W_2(\omega(0), \omega(1))\,.
$$
To show that this inequality is in fact an equality, note that together with the triangle inequality, it yields
\begin{align*}
W_2(\omega(0), \omega(1))&\le W_2(\omega(0), \omega(s))+W_2(\omega(s), \omega(t))+W_2(\omega(t), \omega(1))\\
&\le (s+ |t-s|+|1-t|)\,W_2(\omega(0), \omega(1))\\
&=W_2(\omega(0), \omega(1))\,.
\end{align*}
Therefore, the above inequalities are equalities and in particular,
\begin{align*}
    &W_2(\omega(0), \omega(s))+W_2(\omega(s), \omega(t))+W_2(\omega(t), \omega(1))\\
    &\quad = s\,W_2(\omega(0), \omega(1))+ |t-s|\,W_2(\omega(0), \omega(1))+|1-t|\,W_2(\omega(0), \omega(1))\,.
\end{align*}
Since each term on the left-hand side is smaller than its corresponding part in the right-hand side, we have that
$$
W_2(\omega(s), \omega(t))=|t-s|\,W_2(\omega(0), \omega(1))\,,
$$
and the conclusion follows from Proposition~\ref{prop:cspeedg}. This explicit construction of geodesics joining any pair $\mu,\nu \in \cW_2$ readily implies that $\cW_2$ is indeed a geodesic space.
\end{proof}

For any constant-speed geodesic $\omega$ connecting two measures $\mu, \nu \in \cW_2$ and any $t \in [0,1]$, the measure $\omega(t)$ is often called \emph{displacement interpolation} after~\cite{McC97}. Crucially, if $\mu$ and $\nu$ have densities $f_\mu$ and $f_\nu$, this interpolation differs from the usual interpolation given by the mixture with density $(1-t)\,f_\mu+t\, f_\nu$. This is a manifestation of the geometry of $\cW_2$.

Note that the proof of Theorem~\ref{thm:W2geo} above implies the following interesting corollary. 
\begin{corollary}\label{cor:geowass}
Let $\omega$ be any constant-speed geodesic in $\cW_2$ and let  $\gamma$ be an optimal coupling between $\omega(0)$ and $\omega(1)$. Then for any $0\le s\le t \le 1$, the coupling $\gamma_{s,t} \deq (\pi_s, \pi_t)_{\#}\gamma \in \Gamma_{\omega(s),\omega(t)}$, where $\pi_t(x,y) \deq (1-t)\,x+t\,y$, $t\in [0,1]$, is optimal in the sense that
$$
\int \|x-y\|^2\, \gamma_{s,t}(\ud x,\ud y)=W_2^2(\omega(s), \omega(t))\,.
$$
\end{corollary}

Finally, in the case where the geodesic emanates from a distribution that admits a density, we get from Brenier's Theorem~\ref{thm:improvedBrenier} the following useful corollary, which justifies Definition~\ref{def:w2_geod}.
\begin{corollary}\label{cor:geobrenier}
Let $\mu, \nu \in \cW_2$ be two probability measures such that $\mu$ has a density and let $T: \R^\dd \to \R^\dd$ be the (unique) Brenier map such that $T_{\#}\mu=\nu$. Then, the constant-speed geodesic $\omega:[0,1]\to \cW_2$ such that $\omega(0)=\mu$ and $\omega(1)=\nu$ is unique and given by
$$
\omega(t)=\bigl((1-t)\,{\id} + t\,T\bigr)_{\#}\mu\,, \qquad \forall\, t \in [0,1]\,.
$$
where ${\id} :\R^\dd\to \R^\dd$ denotes the identity map.

In other words, if $X \sim \mu$, then $(1-t)\,X +t\,T(X) \sim \omega(t)$.
\end{corollary}

\section{Curvature}\label{sec:curvature}

\subsection{Alexandrov curvature}\label{subsec:curvature}\index{Alexandrov curvature}

Given a real number $\kappa\in\R$, a geodesic space of special interest is the (complete and simply connected) $2$-dimensional Riemannian manifold with constant sectional curvature $\kappa$.
For given $\kappa\in\R$, this metric space $(M_{\kappa}, \md_{\kappa})$ is unique up to an isometry, and called a 
\emph{model space}. For each $\kappa \in \R$, we use the following representative of the equivalence class generated by the group of isometries.
\begin{itemize}
\item If $\kappa<0$, $(M_{\kappa}, \md_{\kappa})$ is the  hyperbolic plane of constant curvature $\kappa<0$.
\item If $\kappa=0$, $(M_0,\md_{0})$ is the Euclidean plane $\R^2$ equipped with its Euclidean metric.
\item If $\kappa>0$, $(M_{\kappa},\md_{\kappa})$ is the $2$-dimensional Euclidean sphere of radius $1/\sqrt{\kappa}$ equipped with the angular metric.
\end{itemize}
These model spaces play a central role in metric geometry. As described below, curvature bounds in general metric spaces  are formulated by comparison arguments involving these model spaces as benchmarks. 

The fundamental device allowing for this comparison is that of comparison triangles. Given a metric space $(S,\md)$, we define a triangle as any set of three distinct points $\{p,x,y\}\subset S$. For $\kappa\in\R$, a \emph{comparison triangle} for $\{p,x,y\}$ in $M_{\kappa}$ is an isometric embedding of $\{p,x,y\}$ in $M_{\kappa}$, i.e., a set $\{\bar p, \bar x, \bar y\}\subset M_{\kappa}$ such that  
\[ \md_{\kappa}(\bar p, \bar x)=\md(p,x)\,,\quad \md_{\kappa}(\bar p, \bar y)=\md(p,y)\,,\quad\text{and}\quad \md_{\kappa}(\bar x, \bar y)=\md(x,y)\,.\] 
When $\kappa\le 0$, such a comparison triangle always exists (and is unique up to an isometry). When $\kappa>0$, such a triangle exists (and is unique up to an isometry) provided it fits on the sphere of radius $\kappa^{-1/2}$. This condition may be specified in terms of its perimeter:
\begin{equation}\label{EQ:admissible_triangle}
\peri\{p,x,y\} \deq \md(p,x)+\md(p,y)+\md(x,y)<\frac{ 2\pi}{\sqrt{\kappa}}\,.
\end{equation}
For $\kappa>0$ say that a triangle $\{p,x,y\}$ that satisfies~\eqref{EQ:admissible_triangle} is \emph{admissible}. When $\kappa \le 0$, all triangles are admissible.

We are now in a position to define curvature bounds for general geodesic spaces.

\begin{definition}\label{def:curvk}
Let $\kappa\in \R$ and $(S,\md)$ be a geodesic space.
\begin{itemize}
\item We say that $\curv(S)\ge\kappa$ if for any admissible triangle $\{p,x,y\}\subset S$ and any comparison triangle $\{\bar p,\bar x,\bar y\}\subset M_{\kappa}$, the following holds. For any constant-speed geodesics  $\omega: [0,1]\to S$ and $\bar \omega: [0,1]\to M_\kappa$ joining   $x$ to $y$ and $\bar x$ to $\bar y$ respectively, it holds
\begin{equation}\label{boundkappabelow}
\md\big(p, \omega(t)\big) \ge \md_{\kappa}\big(\bar p, \bar \omega(t)\big)\,, \qquad \forall \, t \in [0,1]\,.
\end{equation}
\item We say that $\curv(S)\le\kappa$ if for any admissible triangle $\{p,x,y\}\subset S$ and any comparison triangle $\{\bar p,\bar x,\bar y\}\subset M_{\kappa}$, the following holds. For any constant-speed geodesics $\omega: [0,1]\to S$ and $\bar \omega: [0,1]\to M_\kappa$ joining   $x$ to $y$ and $\bar x$ to $\bar y$ respectively, it holds
\begin{equation}\label{boundkappaabove}
\md\big(p, \omega(t)\big) \le \md_{\kappa}\big(\bar p, \bar \omega(t)\big)\,, \qquad \forall \, t \in [0,1]\,.
\end{equation}

\end{itemize}
\end{definition} 

\noindent The previous definition admits a natural geometric interpretation: if $\curv(S)\ge\kappa$ (resp.\ $\curv(S)\le\kappa$), a triangle $\{p,x,y\}$ looks thicker (resp.\ thinner) than a corresponding comparison triangle $\{\bar p, \bar x, \bar y\}$ in the model space $M_{\kappa}$.

The case $\kappa=0$ is of special interest since the model space of reference is flat. In that case, one compares our geometry to a familiar Euclidean one. We say that $S$ is a space of non-positive curvature (NPC) when $\curv(S)\le0$, and a space of non-negative curvature (NNC) when $\curv(S)\ge  0$. 

In the flat case the following lemma holds.

\begin{lemma}\label{lem:hilbertflat}
Let $\bH$ be a Hilbert space equipped with inner product $\langle\cdot, \cdot\rangle$ and norm $\|\cdot\|$. Then, for any $p, x, y \in \bH$, the constant-speed geodesic joining $x$ to $y$ is unique and given by $\omega(t)=(1-t)\,x+ t\,y$ and for any $p \in \bH$,
$$
\|p-\omega(t)\|^2= (1-t)\,\| p-x\|^2+t\,\| p-y\|^2-t\,(1-t)\,\| x- y\|^2\,, \quad \forall \, t \in [0,1]\,.
$$
In particular, this holds for the model space $M_0=\R^2$.
\end{lemma}
\begin{proof}
It can be easily checked that $\omega$ is indeed a constant-speed geodesic joining $x$ to $y$. Fix $t \in [0,1]$. To check the equality, observe that on the one hand
\begin{align*}
\|p-\omega(t)\|^2&=\|p-(1-t)\,x- t\,y\|^2\\
&=\|(1-t)\,(p-x)+ t\,(p-y)\|^2\\
&=(1-t)^2\,\|p-x\|^2 +t^2\,\|p-y\|^2+2t\,(1-t)\,\langle p-x, p-y \rangle\,.
\end{align*}
On the other hand,
\begin{align*}
\| x- y\|^2=\| x-p+p- y\|^2=\|p-x\|^2+\|p-y\|^2-2\,\langle  p-x, p-y \rangle\,.
\end{align*}
Putting the above two displays together yields
\begin{align*}
\|p-\omega(t)\|^2
&=(1-t)^2\,\|p-x\|^2 +t^2\,\|p-y\|^2\\
&\qquad{} +t\,(1-t)\,\big[\| p-x\|^2+\|p-y\|^2-\| x- y\|^2\big]\\
&=(1-t)\,\| p-x\|^2+t\,\| p-y\|^2-t\,(1-t)\,\| x- y\|^2\,.
\end{align*}
It remains to show that $\omega$ is unique. To that end, let $\omega'$ by any constant-speed geodesic joining $x$ to $y$ and fix $t \in [0,1]$. Apply the above identity to $p=\omega'(t)$ to get
$$
\|\omega'(t)-\omega(t)\|^2= (1-t)\,\|\omega'(t)-x\|^2+t\,\| \omega'(t)-y\|^2-t\,(1-t)\,\| x- y\|^2\,.
$$
Since $\omega'$ is a constant-speed geodesic joining $x$ to $y$, we have by Proposition~\ref{prop:cspeedg} that $\|\omega'(t)-x\|=t\,\| x- y\|$ and $\|\omega'(t)-y\|=(1-t)\,\| x- y\|$. Therefore
$$
\|\omega'(t)-\omega(t)\|^2= \big((1-t)\,t^2+t\,(1-t)^2-t\,(1-t)\big)\,\| x- y\|^2=0\,,
$$
so that $\omega'=\omega$.
\end{proof}

Lemma~\ref{lem:hilbertflat} involves only squared distances and can be directly stated in geodesic spaces. It turns out that this generalization gives a useful characterization of NNC or NPC spaces. Note that this characterization does not extend to the cases where the reference space is not flat (i.e., curvature bounded by a non-zero quantity).

\begin{proposition}\label{pro:NNC}
Let $(S,\md)$ be a geodesic space. Then $\curv(S)\ge0$ if and only if for triangle $\{p,x,y\}\in S$ and any constant-speed geodesic $\omega$ joining $x$ to $y$, we have
\begin{equation}\label{EQ:curvge0}
\md^2(p,\omega(t))\ge (1-t)\,\md^2(p,x)+t\,\md^2(p,y)-t\,(1-t)\,\md^2(x,y)\quad \forall\,  t\in[0,1]\,.
\end{equation}
We have $\curv(S)\le0$ if and only if the same statement holds with the opposite inequality.
\end{proposition}
\begin{proof}
Consider a triangle $\{p, x, y\}$ together with a comparison triangle $\{\bar p, \bar x, \bar y\} \in \R^2$. 

Assume that $\curv(S)\ge 0$ in the sense of Definition~\ref{def:curvk}. Then for any constant-speed geodesic $\omega$ that connects $x$ to $y$ and $\bar \omega$ the unique constant-speed geodesic that connects $\bar x$ to $\bar y$, we have by Definition~\ref{def:curvk} and Lemma~\ref{lem:hilbertflat} respectively that
\begin{align*}
    \md^2(p,\omega(t))
    &\ge \|\bar p -\bar \omega(t)\|^2\\
    &=(1-t)\,\| \bar p-\bar x\|^2+t\,\| \bar p-\bar y\|^2-t\,(1-t)\,\| \bar x- \bar y\|^2\\
    &=(1-t)\,\md^2(p,x)+t\,\md^2(p,y)-t\,(1-t)\,\md^2(x,y)\,.
\end{align*}

To prove the converse, note that~\eqref{EQ:curvge0} yields
\begin{align*}
    \md^2(p,\omega(t))&\ge (1-t)\,\md^2(p,x)+t\,\md^2(p,y)-t\,(1-t)\,\md^2(x,y)\\
    &=(1-t)\,\| \bar p-\bar x\|^2+t\,\| \bar p-\bar y\|^2-t\,(1-t)\,\| \bar x- \bar y\|^2\\
    &= \|\bar p -\bar \omega(t)\|^2\,,
\end{align*}
by Lemma~\ref{lem:hilbertflat} so that Definition~\ref{def:curvk} holds.
\end{proof}

\begin{remark}
    Comparing with the definition of $\alpha$-convexity in Appendix~\ref{app:convex}, we see that $\frac{1}{2}\,\|p-\cdot\|^2$ is $1$-strongly convex in any Hilbert space.
    Similarly, $(S, \md)$ is an NPC space if and only if $\frac{1}{2}\,\md^2(p,\cdot)$ is $1$-strongly convex \emph{along the geodesics} of $(S,\md)$. The notion of geodesic convexity was also used in Section~\ref{sec:riem}, but here we work in the more general setting of geodesic spaces.
\end{remark}

A geodesic space $(S,\md)$ with any curvature bound is called an \emph{Alexandrov space}\index{Alexandrov space}. If $\curv(S)\le\kappa$ for some $\kappa\in \R$, then $(S,\md)$ is sometimes called a CAT($\kappa$) space in reference to E.\ Cartan, A.\ D.\ Alexandrov, and V.\ A.\ Toponogov. As noted before, a CAT($0$)\index{CAT($0$)|see {NPC}} space is also referred to as an NPC\index{NPC} (non-positively curved) or sometimes \emph{Hadamard} space\index{Hadamard|see {NPC}}. If $\curv(S) \ge 0$ we call the space \emph{non-negatively curved} or NNC\index{NNC}. It is worth noting that the previous definitions are of global nature as they require comparison inequalities to be valid for all triangles (that admit a comparison triangle in the relevant model space). Some definitions of curvature require the previous comparison inequalities to hold only locally. The local validity of these comparison inequalities is known, under suitable conditions depending on the value of $\kappa$, to imply their global validity (globalization theorems).

We conclude this subsection by giving a third equivalent definition of non-negative curvature.

\begin{proposition}\label{pro:NNC2}
Let $(S,\md)$ be a geodesic space. Then $\curv(S)\ge0$ if and only if for any triangle $\{p,x,y\}\subset S$, comparison triangle $\{\bar p,\bar x,\bar y\}\subset M_{0}$, and any constant-speed geodesics $\omega, \omega', \bar \omega, \bar \omega'$ joining $p$ to $x$, $p$ to $y$, $\bar p$ to $\bar x$, and $\bar p$ to $\bar y$ respectively, we have
\begin{equation}\label{EQ:curvge0-2}
\md^2(\omega(s), \omega'(t))\ge \|\bar \omega(s)-\bar \omega'(t)\|^2\,,\qquad \forall\,  s,t\in[0,1]\,.
\end{equation}
We have $\curv(S)\le0$ if and only if the same statement holds with the opposite inequality.
\end{proposition}
\begin{proof}
Assume first that $\curv(S)\ge 0$ and observe that by Proposition~\ref{pro:NNC} and Definition~\ref{def:curvk} respectively, it holds
\begin{align*}
    &\md^2(\omega(s), \omega'(t))
    \ge (1-s)\,\md^2(p, \omega'(t))+s\,\md^2(x, \omega'(t))-s\,(1-s)\,\md^2(p,x)\\
    &\qquad \ge (1-s)\,\|\bar p - \bar \omega'(t)\|^2+s\,\|\bar x- \bar \omega'(t)\|^2-s\,(1-s)\,\|\bar p- \bar x\|^2\,.
\end{align*}
The right-hand side of the above inequality is precisely $\|\bar \omega(s)-\bar \omega'(t)\|^2$ by Lemma~\ref{lem:hilbertflat}. We have proved~\eqref{EQ:curvge0-2}.

Conversely,  let $\omega_x, \omega_x'$ be constant-speed geodesics joining $x$ to $y$ and $x$ to $p$, respectively, and let $\bar \omega_{\bar x}, \bar \omega_{\bar x}'$ be constant-speed geodesics joining $\bar x$ to $\bar y$ and $\bar x$ to $\bar p$ respectively. Then, taking $s=1$ in~\eqref{EQ:curvge0-2}, we get for any $t \in [0,1]$,
\begin{align*}
    \md^2(p, \omega_x(t))
    &= \md^2(\omega_x'(1), \omega_x(t)) \\
    &\ge \|\bar p -\bar \omega_{\bar x}(t)\|^2\\
    &=(1-t)\,\|\bar p -\bar x\|^2+ t\,\|\bar p -\bar y\|^2-t\,(1-t)\,\|\bar x -\bar y\|^2\\
    &=(1-t)\,\md^2(p,x)+t\,\md^2(p, y) -t\,(1-t)\,\md^2(x,y)\,,
\end{align*}
which is the characterization of $\curv(S)\ge 0$ from Proposition~\ref{pro:NNC}.

The proof for $\curv(S) \le 0$ follows using the same argument.
\end{proof}

\subsection{Curvature of the Wasserstein space}\label{subsec:curv-wass}

Note that if $d=1$, the space $\cP_2(\R)$ equipped with the Wasserstein distance is actually \emph{flat}.

\begin{proposition}
The space $\cW_{2}(\R)$ is flat in the sense that
\begin{align*}
    \curv(\cW_{2}(\R)) \le 0 \qquad\text{and}\qquad \curv(\cW_2(\R)) \ge 0
\end{align*}
and it can be isometrically embedded into a Hilbert space.
\end{proposition}
\begin{proof}
Recall from Proposition~\ref{prop:w21d} that for any $\mu,\nu \in \cW_{2,\text{ac}}$, it holds
$$
W_2^2(\mu, \nu)=\int_0^1|F^{\dagger}_\mu(u)-F^{\dagger}_\nu(u)|^2\,\ud u=\|F^{\dagger}_\mu-F^{\dagger}_\nu\|^2\,,
$$
where $\|\cdot\| \deq \|\cdot\|_{L^2([0,1])}$. In particular, the map $\mu \mapsto F^{\dagger}_\mu$ is an isometry from $\cW_{2,\text{ac}}$ to $L^2([0,1])$.

Let now $\omega$ be a constant-speed geodesic that connects $\mu$ to $\nu$ and recall from  Proposition~\ref{prop:w21d} and Theorem~\ref{thm:W2geo} that $\omega$ is uniquely characterized by the fact that if $V=(1-t)\,F_\mu^{\dagger}(U)+ t\,F_\nu^{\dagger}(U)$, where $U \sim \unif([0,1])$, then $W \sim \omega(t)$. It yields that for any $v \in \R$, 
$$
\p(V \le v)=\p\big((1-t)\,F_\mu^{\dagger}(U)+ t\,F_\nu^{\dagger}(U)\le v\big)=\big((1-t)\,F_\mu^{\dagger}+t\,F_\nu^{\dagger}\big)^\dagger(v)\,.
$$
Hence,
$$
F_{\omega(t)}^{\dagger}=(1-t)\,F_\mu^{\dagger} + t\,F_{\nu}^{\dagger}\,.
$$

Next, let $\rho \in \cW_{2}$ and $t \in [0,1]$. Since $L^2(\R)$ is a Hilbert space, we get from Lemma~\ref{lem:hilbertflat} that
\begin{align*}
    &W_2^2(\rho, \omega(t))
    =\|F^{\dagger}_\rho-F^{\dagger}_{\omega(t)}\|^2
    =\|F^{\dagger}_\rho-(1-t)\,F_\mu^{\dagger} - t\,F_{\nu}^{\dagger}\|^2\\
    &\qquad =(1-t)\,\|F^{\dagger}_\rho-F_\mu^{\dagger}\|^2+t\,\|F^{\dagger}_\rho-F_\nu^{\dagger}\|^2-t\,(1-t)\,\|F_\mu^\dagger-F_\nu^{\dagger}\|^2\\
    &\qquad =(1-t)\,W_2^2(\rho,\mu)+t\,W_2^2(\rho,\nu)-t\,(1-t)\,W_2^2(\mu,\nu)\,.
\end{align*}
This completes the proof that $\curv(\cW_{2,\text{ac}}(\R))=0$. 
In turn, one can show that this implies $\curv(\cW_2(\R)) = 0$ as well.
\end{proof}

More generally, for any $d\ge 1$, $\cW_2(\R^d)$ is positively curved as indicated by the theorem below.

\begin{theorem}\label{thm:curvw2ge0}
The 2-Wasserstein space $\cW_2$ is non-negatively curved,
\begin{align*}
    \curv(\cW_2) \ge 0\,,
\end{align*}
i.e., for any $\mu, \nu, \rho \in \cW_2$ and any constant-speed geodesic $\omega$ that connects $\mu$ to $\nu$,  it holds
$$
W_2^2(\rho, \omega(t))\ge (1-t)\,W_2^2(\rho, \mu)+t\,W_2^2(\rho, \nu)-t\,(1-t)\,W_2^2(\mu, \nu)\,.
$$
\end{theorem}
\begin{proof}
Let $\gamma \in \Gamma_{\mu, \nu}$ be an optimal coupling and recall from Theorem~\ref{thm:W2geo} that for any $t\in [0,1]$, $\omega(t)=(\pi_t)_{\#}\gamma$ where $\pi_t(x,y)=(1-t)\,x+t\,y$. In particular, if $(X,Y)\sim \gamma$, then $V_t \deq (1-t)\,X+t\,Y\sim \omega(t)$. Next, let $\gamma_t \in \Gamma_{\omega(t), \rho}$ be an optimal coupling between $\omega(t)$ and $\rho$. In particular, it induces a conditional distribution on $Z\sim\rho$ given $V_t$. We have described a joint distribution $\Upsilon$ for $(X,Y,Z)$ that has marginals $\mu$, $\nu$, and $\rho$ respectively (the reader will have recognized a variant of the gluing lemma, Lemma~\ref{lem:gluing}).

With this notation, we have
\begin{align*}
    W_2^2(\rho, \omega(t))
    &= \int\|z-v\|^2\, \gamma_t(\ud x,\ud v) \\
    &=\int\|z-(1-t)\,x-t\,y\|^2\, \Upsilon(\ud x,\ud y,\ud z)\,.
\end{align*}
Next, observe that by Lemma~\ref{lem:hilbertflat} applied to $\bH=\R^\dd$, we have 
$$
\|z-(1-t)\,x-t\,y\|^2=(1-t)\,\|z-x\|^2+t\,\|z-y\|^2-t\,(1-t)\,\|x-y\|^2\,,
$$
so that
\begin{align*}
    W_2^2(\rho, \omega(t))&=\int \bigl[(1-t)\,\|z-x\|^2 + t\,\|z-y\|^2 \\
    &\qquad\qquad\qquad\qquad\qquad{} - t\,(1-t)\,\|x-y\|^2\bigr]\, \Upsilon(\ud x,\ud y,\ud z)\\
    &\ge (1-t)\,W_2^2(\rho, \mu)+t\,W_2^2(\rho, \nu)\\
    &\qquad{} -t\,(1-t)\,\int\|x-y\|^2\, \Upsilon(\ud x,\ud y,\ud z)\\
    &=(1-t)\,W_2^2(\rho, \mu)+t\,W_2^2(\rho, \nu)-t\,(1-t)\,W_2^2(\mu, \nu)\,,
\end{align*}
where in the inequality, we used the suboptimality of the first two couplings and in the last equality, we used the optimality of the coupling between $\mu$ and $\nu$ induced by $\Upsilon$.
\end{proof}

\section{Tangent cones}\label{subsec:angles}\index{tangent cone}

A geodesic space has a priori no differentiable structure but a surrogate for it may be built. It starts from the notion of \emph{angle} which can be defined on any metric space by analogy to the Hilbert case, akin to our definition of curvature bounds. When applied to a geodesic space, angles allow us to define the notion of \emph{direction}, which can be thought of as the initial velocity of a constant-speed geodesic. The collection of such directions forms the tangent cone.

\subsection{Angles}

We first define angles on a metric space and show that they provide alternative characterizations of 
curvature bounds for geodesic spaces.

Recall that for any three points $p,x,y \in \R^2$, the cosine of the angle $\measuredangle_p(x,y)$, formed by vectors $\overrightarrow{p\,x}$ and $\overrightarrow {p\,y}$ is given by
$$
\cos\measuredangle_p(x,y)=\frac{\langle x-p, y-p\rangle}{\|x-p\|\,\|y-p\|}\,.
$$
Note that
\begin{align*}
-2\,\langle x-p, y-p\rangle&=\|(x-p)-(y-p)\|^2-\|x-p\|^2-\|y-p\|^2\\
&=\|x-y\|^2-\|x-p\|^2-\|y-p\|^2\,.
\end{align*}
Therefore, we can rewrite this definition only in terms of squared distances to obtain
$$
\cos\measuredangle_p(x,y)=\frac{\|x-p\|^2+\|y-p\|^2-\|x-y\|^2}{2\,\|x-p\|\,\|y-p\|}\,.
$$
This definition generalizes to any metric space.

\begin{definition}\label{def:angle}
Let $(S,\md)$ be a metric space and  for any triangle $\{p,x,y\}$ in $S$, define the \emph{angle} $\measuredangle_p(x,y)\in[0,\pi]$ at $p$ by 
 \[\cos\measuredangle_p(x,y) \deq \dfrac{\md^2(p,x)+\md^2(p,y)-\md^2(x,y)}{2\,\md(p,x)\,\md(p,y)}\,.\]
\end{definition}

Similar comparisons may be made with model spaces $M_{\kappa}$ for $\kappa\neq 0$ but are beyond the scope of these lectures.

The next result presents a characterization of positively curved spaces in terms of the \emph{angle monotonicity.}

\begin{proposition}[Angle monotonicity]\label{pro:monotone}
 Let $(S,\md)$ be a geodesic space. Then, $\curv(S)\ge 0$ in the sense of Definition \ref{def:curvk}, if and only if for any triangle $\{p,x,y\}$ in $S$ and any geodesics $\omega$ and $\omega'$ joining $p$ to $x$ and  $p$ to $y$ respectively, the function
 \[(s,t)\in[0,1]^2\mapsto\measuredangle_p(\omega(s),\omega'(t))\]
 is non-increasing in each variable when the other is fixed.
\end{proposition}
\begin{proof}
Assume first that  $\curv(S)\ge 0$ and consider a triangle $\{p,x,y\}$ in $S$ and constant-speed geodesics $\omega$ and $\omega'$ joining $p$ to $x$ and $p$ to $y$ respectively. It is enough to prove that, for all $(s,t)\in[0,1]^2$, 
 \begin{equation}
 \nonumber  
\cos\measuredangle_p(\omega(s),\omega'(t))\le \cos\measuredangle_p(x,y)\,.
 \end{equation}
Let $\{\bar p,\bar x,\bar y\}$ be a comparison triangle for $\{p,x,y\}$ in $M_0=\R^2$ and let $\bar \omega$ and $\bar \omega'$ be constant-speed geodesics in $M_0$ connecting $\bar p$ to $\bar x$ and $\bar y$ respectively. 
It holds
 \begin{align*}
\cos\measuredangle_p(\omega(s),\omega'(t))&=\frac{\md^2(p,\omega(s))+\md^2(p,\omega'(t))-\md^2(\omega(s),\omega'(t))}{2\,\md(p,\omega(s))\,\md(p,\omega'(t))}\\
&=\frac{s^2\,\md^2(p,x)+t^2\,\md^2(p,y)-\md^2(\omega(s),\omega'(t))}{2st\,\md(p,x)\,\md(p,y)}\\
&=\frac{s^2\,\|\bar p- \bar x\|^2+t^2\,\|\bar p-\bar y\|^2-\md^2(\omega(s),\omega'(t))}{2st\,\|\bar p-\bar x\|\, \|\bar p-\bar y\|}\\
&\le \frac{s^2\,\|\bar p- \bar x\|^2+t^2\,\|\bar p-\bar y\|^2-\|\bar \omega(s)-\bar \omega'(t)\|^2}{2st\,\|\bar p-\bar x\|\, \|\bar p-\bar y\|}\\
&=\frac{\|\bar p- \bar \omega(s)\|^2+\|\bar p-\bar \omega'(t)\|^2-\|\bar \omega(s)-\bar \omega'(t)\|
^2}{2\,\|\bar p-\bar \omega(s)\|\, \|\bar p-\bar \omega'(t)\|}\\
&=\cos\measuredangle_{\bar p}(\bar \omega(s),\bar \omega'(t))\\
&=\cos\measuredangle_{\bar p}(\bar x,\bar y)\\
&=\cos\measuredangle_{p}(x,y)\,,
\end{align*}
where in the inequality we used the fact that $\curv(S)\ge 0$ and Proposition~\ref{pro:NNC2}. This completes the proof that the curvature lower bound implies angle monotonicity.

Conversely, assume that for any triangle $\{p,x,y\}$ in $S$, any constant-speed geodesics $\omega$ and $\omega'$ connecting $p$ to $x$ and $p$ to $y$ respectively, and all $(s,t)\in[0,1]^2$, we have 
$$
\cos\measuredangle_p(\omega(s),\omega'(t))\le \cos\measuredangle_p(x,y)\,.
$$
Then, the first part of the proof implies that
\[ \md^2(\omega(s),\omega'(t))\ge \|\bar \omega(s)-\bar \omega'(t)\|^2\,, \qquad \forall\, (s,t)\in[0,1]^2\]
with the same notation as above, which is the characterization of $\curv(S)\ge 0$ from Proposition~\ref{pro:NNC2}.
 \end{proof}

\subsection{Directions}

From the notion of angles between points, we can readily define an angle between constant-speed geodesics. 

Let $(S,\md)$ be a geodesic space such that $\curv(S)\ge 0$, $p\in S$, and $\omega$, $\omega'$ two constant-speed geodesics connecting $p$ to $x$ and $y$ respectively. We define the angle between $\omega$ and $\omega'$ as 
\[
 \measuredangle(\omega,\omega') \deq \lim_{s,t\searrow 0}\measuredangle_p(\omega(s),\omega'(t))\,.
 \]
It follows from Proposition \ref{pro:monotone} that this limit exists under the assumption $\curv(S)\ge 0$. In fact, under the same assumption,
\[
 \measuredangle(\omega,\omega') = \lim_{t\searrow 0}\measuredangle_p(\omega(t),\omega'(t))\,.
 \]

Given a third constant-speed geodesic $\omega'':[0,1]\to S$ such that $\omega''(0)=p$ and $\omega''(1)=z$, it can be shown (see Exercise~\ref{ex:triangle_cone}) that we have the triangular inequality \begin{equation}
 \label{angleti}
\measuredangle(\omega,\omega')\le \measuredangle(\omega,\omega'')+\measuredangle(\omega'',\omega')\,,
 \end{equation}
so that $\measuredangle$ is a pseudo-metric on the set $\cG(p)$ of all constant-speed  geodesics emanating from~$p$. Next, we define the equivalence relation $\sim$ on $\cG(p)$ by 
$$
\omega\sim\omega'\Leftrightarrow \measuredangle(\omega,\omega')=0\,.
$$
We can turn $\measuredangle$ into a proper metric (still denoted $\measuredangle$) on the quotient $\cG(p)/{\sim}$.

\begin{definition}\label{def:tangentcone}
The \emph{space of directions} emanating from $p$ is the completion $(\Sigma_p,\measuredangle)$ of $(\cG(p)/{\sim},\measuredangle)$. An element of $\Sigma_p$ is called a \emph{direction}.
\end{definition}

\subsection{Tangent cone}\label{subsec:cone}

An analog of a tangent space for geodesic spaces is provided by the notion of a tangent cone. 
  
\begin{definition}[Tangent cone]\index{tangent cone}\label{def:tan_cone}
Let $(S,\md)$ be a geodesic space with positive curvature and fix $p\in S$. The tangent cone $T_pS$ at $p$ is the Euclidean cone over the space of directions $(\Sigma_p,\measuredangle)$. In other words, $T_pS$ is the metric space:
\begin{itemize} 
\item whose underlying set consists in equivalence classes in $\Sigma_p\times[0,+\infty)$ for the equivalence relation $\sim$ defined by 
\[
(\omega,s)\sim (\omega',t) \Leftrightarrow \left\{\begin{array}{ll}
\phantom{\text{or }}s=t=0\\
\text{or } \omega=\omega' \ \text{and} \ s=t
\end{array}\right.
\]
\item and whose metric $\md_p$ is defined 
\[
\md_p((\omega,s),(\omega',t)) \deq \sqrt{s^2+t^2-2st\cos\measuredangle(\omega,\omega')}\,.
\]
\end{itemize}

For $u=(\omega,s)$ and $v=(\omega',t)\in T_pS$, we write $\|u-v\|_p \deq \md_p(u,v)$, $\|u\|_p \deq \md_p(o_p,u)$, where $o_p=(\omega,0) \in T_p S$ is the tip of the cone and 
$$ 
\langle u,v\rangle_p \deq \|u\|_p\,\|v\|_p\cos\measuredangle(\omega,\omega')=\frac12\,\bigl(\|u\|^2_p+\|v\|^2_p-\|u-v\|^2_p\bigr)\,.
$$
\end{definition}

The terminology \emph{cone} and the  notation $\|\cdot\|_p$ and $\langle\cdot,\cdot\rangle_p$ introduced above is justified by the fact that the cone $T_pS$ possesses a Hilbert-like structure described below. As often the case in metric geometry, the definition of $\md_p$ comes from rewriting a Euclidean notion using only notions that exist on a geodesic space, namely distances and angles in this case. Indeed, if $\omega, \omega'$ are points on the sphere and $(\omega,s) \deq s\cdot \omega$, $(\omega', t) \deq  t\cdot \omega'$ then it follows from the law of cosines that their squared distance is given by
\begin{equation}\label{eq:law_cosines}
    \|s\cdot \omega- t\cdot \omega'\|^2= s^2 + t^2 -2st \cos\measuredangle(\omega, \omega')\,.
\end{equation}

For a point $u=(\omega,t)$ and $\lambda\ge 0$, we define $\lambda\cdot u \deq (\omega, \lambda t)$. 
Moreover, it may be checked using the previous definitions that, for any $u,v\in T_pS$ and any $\lambda\ge 0$, we get 
\[\|\lambda\cdot u\|_p=\lambda\,\|u\|_p\quad\mbox{and}\quad \langle \lambda\cdot u,v\rangle_p=\langle u,\lambda\cdot v\rangle_p=\lambda\,\langle u,v\rangle_p\,.\]

Note that the tangent cone may not be geodesic but in cases when it is, the sum of points $u,v\in T_pS$ is defined as the midpoint of $2\cdot u$ and $2\cdot v$ as defined in Definition~\ref{def:midpoint}. In this case, $T_pS$ is indeed a cone.

An example to keep in mind is when $S$ is a filled-in square in the plane $\R^2$. Then, the tangent cone at one of the corners of $S$ is ``missing'' some directions (the ones that would lead out of $S$) and is therefore not a vector space, hence why we do not call it the tangent ``space''.

The logarithmic map plays an important role in the sequel. For all $x\in S$, we denote ${\Uparrow_p^{x}}\subset \Sigma_p$ the set of all equivalence classes of constant-speed geodesics connecting $p$ to $x$ in $S$. Then for every $x\in S$, we arbitrarily choose one direction ${\uparrow_p^x}\in{\Uparrow_p^{x}}$.  

\begin{definition}[Logarithmic map]\index{logarithmic map}
Let $(S,\md)$ be a positively curved geodesic space. Then, having chosen ${\uparrow_p^x}\in{\Uparrow_p^{x}}$ for every $x\in S$, the associated logarithmic map is defined by
 \[
 \log_p: x\in S\mapsto (\uparrow_{p}^{x},\md(p,x))\in T_pS\,.
 \]
  \end{definition}
  
At this level of generality, the definition of $\log_p(x)$ depends on the choice of directions $\{{\uparrow_p^x},\, x\in S\}$. This ambiguity may be removed by restricting the $\log_p$ map to an appropriate subset, namely the set of points $x\in S$ for which there is only one equivalence class of directions of constant-speed geodesics connecting $p$ to $x$. This set might be specified even more accurately by observing that, in an NNC space $S$,  if constant-speed geodesics $\omega$ and $\omega'$ from $p$ to $x$ satisfy $\measuredangle(\omega,\omega')=0$, then $\omega=\omega'$. In other words, the set of points $x\in S$ for which there is more than one equivalence class of constant-speed geodesics connecting $p$ to $x$ is exactly the set of points $x$ connected to $p$ by at least two distinct constant-speed geodesics. This set of points is denoted $C(p)$ and called the \emph{cut-locus} of $p$. Then, for any $x\in S\setminus C(p)$, $\log_p(x)$ is defined without ambiguity as 
 \[
 \log_p(x)=(\omega_{x},\md(p,x))\,,
 \]
 where $\omega_{x}$ denotes the unique constant-speed geodesic from $p$ to $x$ therefore identified to its direction. With this notation, we get in particular, for all $t\in[0,1]$ and all $x\in S\setminus C(p)$, that
 \[\log_p \omega_{x}(t) =t \log_p x \,.\]
More generally, if $\omega:[0,1] \to S$ is a constant-speed geodesic in $S$ and $p$ is such that $p=\omega(t)$ for some $t \in [0,1]$, i.e., $p$ is on the geodesic, then
\begin{equation}
\label{EQ:logstraight}
\log_p \omega(s)=(1-s)\log_p \omega(0) + s\log_p \omega(1)\,.
\end{equation}
In other words, the logarithmic maps turns geodesics into straight lines.

The following result shows that the $\log_p$ map is expanding in a space of positive curvature.

\begin{proposition}\label{lem:tangentcone2}
    Let $(S,\md)$ be a geodesic space and $p\in S$ be fixed. If $\curv(S)\ge0,$ then the logarithmic map is expansive in the sense that
    then for all $x,y\in S$,
     \[ \md(x,y)\le\|{\log_p(x)}-\log_p(y)\|_p\,,\]
     with equality if $x=p$ or $y=p$.
\end{proposition}
\begin{proof}
 Let $\omega, \omega'$ be two constant-speed geodesics connecting $p$ to $x$ and $y$ respectively. Then if $p\neq x$ and $p \neq y$, we have by definition of $\|\cdot\|_p$ and angle monotonicity that for all $s,t \in [0, 1]$, it holds
\begin{align*}
     &\|{\log_p(x)}-\log_p(y)\|_p^2
     =\md_p^2(\log_p(x), \log_p(y))\\
     &\qquad = \md^2(p, x) + \md^2(p,y)-2\,\md(p,x)\,\md(p,y)\cos \measuredangle(\omega, \omega') \\
     &\qquad \ge \md^2(p, x) + \md^2(p,y)-2\,\md(p,x)\,\md(p,y)\cos \measuredangle_p(\omega(s), \omega'(t))\,.
\end{align*}
Applying Definition~\ref{def:angle}, we get
\begin{align*}
     &\|{\log_p(x)}-\log_p(y)\|_p^2\\
     &\qquad \ge \md^2(p, x) + \md^2(p,y)-\frac{s^2 \md^2(p,x)+t^2 \md^2(p,y)-\md^2(\omega(s),\omega'(t))}{st}\,.
 \end{align*}
 Letting $s=t=1$ yields
\begin{align*}
 \|{\log_p(x)}-\log_p(y)\|_p^2\ge \md^2(x,y)\,.
 \end{align*}
It is easy to check the equality cases from the definition of $\md_p$.
 \end{proof}

\subsection{Tangent cone of the Wasserstein space}

Going to the very definition of the tangent cone, we can show that it takes a very simple form in the Wasserstein case. To that end, let $\mu,\nu,\rho \in \cW_2$ be three probability distributions. Moreover, let $\omega_\nu$ and $\omega_\rho$ be two geodesics in the 2-Wasserstein space $\cW_2$ joining $\mu$ to $\nu$ and $\mu$ to $\rho$ respectively and recall that the tangent cone at $\mu$ is the metric space of directions at $\mu$ equipped with distance $\md_\mu$ such that
\begin{align*}
    &\md^2_\mu\bigl((\omega_\nu,\md(\mu,\nu)),(\omega_\rho,\md(\mu,\rho))\bigr) \\
    &\qquad = W_2^2(\mu,\nu)+W_2^2(\mu,\rho)-2\,W_2(\mu,\nu) \,W_2(\mu,\rho)\cos\measuredangle(\omega_\nu,\omega_\rho)\,.
\end{align*}
What is the angle $\measuredangle(\omega_\nu,\omega_\rho)$ between these two Wasserstein geodesics? We can carry out a calculation assuming that $\mu$ has a density so that Brenier's theorem ensures the existence of two optimal transport maps $T_{\mu \to \nu}$ and $T_{\mu \to \rho}$ so that
\begin{align*}
    \cos\measuredangle(\omega_\nu,\omega_\rho) &= \lim_{t\searrow 0} \frac{W_2^2(\mu,\omega_\nu(t))+W_2^2(\mu,\omega_\rho(t))-W_2^2(\omega_\nu(t),\omega_\rho(t))}{2\,W_2(\mu,\omega_\nu(t))\,W_2(\mu,\omega_\rho(t))}\\
    &= \lim_{t\searrow 0} \frac{t^2\,W_2^2(\mu,\nu)+t^2\,W_2^2(\mu,\rho)-W_2^2(\omega_\nu(t),\omega_\rho(t))}{2t^2\,W_2(\mu,\nu)\,W_2(\mu,\rho)}\\
    &=\frac{W_2^2(\mu,\nu)+W_2^2(\mu,\rho)-\lim_{t\searrow 0} \frac{W_2^2(\omega_\nu(t),\omega_\rho(t))}{t^2}}{2\,W_2(\mu,\rho)\,W_2(\mu,\nu)}\,.
\end{align*}

\begin{lemma}\label{lem:Wasserstein_derivative}
Let $\mu \in \cP_{2,\rm ac}(\R^d)$ and denote by $T_{\mu \to \nu}$ and $T_{\mu \to \rho}$ the Brenier maps from $\mu$ to $\nu$ and $\mu$ to $\rho$ respectively. Then
$$
\lim_{t\searrow 0} \frac{W_2^2(\omega_\nu(t),\omega_\rho(t))}{t^2}=\|T_{\mu \to \nu} - T_{\mu \to \rho}\|^2_{L^2(\mu)}\,.
$$
\end{lemma}
\begin{proof}
It is easy to show one of the required inequalities. Indeed, let $X \sim \mu$ and observe that
\begin{align*}
    X^\nu_t
    &\deq (1-t)\,X + t\,T_{\mu \to \nu}(X) \sim \omega_{\nu}(t)\,, \\
    X^\rho_t
    &\deq (1-t)\,X + t\, T_{\mu \to \rho}(X) \sim \omega_{\rho}(t)\,.
\end{align*}
Therefore,
\begin{align}
    W_2^2(\omega_\nu(t),\omega_\rho(t))& \le \E \|X^\nu_t-X^\rho_t\|^2 \nonumber\\
    &=t^2\, \E \|T_{\mu \to \nu}(X)  - T_{\mu \to \rho}(X) \|^2 \nonumber\\
    &=t^2\, \|T_{\mu \to \nu} - T_{\mu \to \rho}\|^2_{L^2(\mu)}\,.\label{eq:ub_tan_cone_dist}
\end{align}

To prove the converse, for any $t \in [0,1]$, let $\Upsilon_t$ be the following coupling between five random variables: $(X,Y,Z,Y_t, Z_t) \sim \Upsilon_t$ if
\begin{enumerate}
    \item $X \sim \mu$ and $Y=T_{\mu \to \nu}(X) \sim \nu$ are optimally coupled, 
    \item $Y_t = (1-t)\,X + t\,Y \sim \omega_{\nu} (t)$,
    \item $Z_t\sim \omega_\rho(t)$ and $Y_t \sim \omega_{\nu} (t)$ are optimally coupled,
    \item $Z_t \sim \omega_\rho(t)$ and $Z \sim \rho$ are  are optimally coupled, 
    \item $Z_t \sim \omega_\rho(t)$ and $X' \sim \mu$ are  are optimally coupled. 
\end{enumerate}
Figure~\ref{fig:upsilon} indicates that this joint coupling can be realized using the gluing lemma since the resulting graph is acyclic. Note in particular that $Z_t= (1-t)\,X' + t\,Z$.
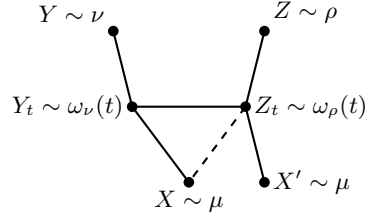
\begin{figure}[h]
\centering
\begin{tikzpicture}
    \coordinate (x) at (0, 0);
    \coordinate (y) at (-1, 2);
    \coordinate (z) at (1, 2);
    \coordinate (yt) at (-.75, 1);
    \coordinate (zt) at (.75, 1);
    \coordinate (z0) at (1, 0);

    \draw[thick]  (x)--(yt);
    \draw[thick]  (y)--(yt);
    \draw[thick]  (z0)--(zt);
    \draw[thick]  (zt)--(z);
    \draw[thick, dashed] (zt)--(x);
    \draw[thick]  (yt)--(zt);

    \fill (x) circle (2pt) node[below] {$X \sim \mu$};
    \fill (y) circle (2pt) node[above left] {$Y \sim \nu$};
    \fill (z) circle (2pt) node[above right] {$Z \sim \rho$};
    \fill (yt) circle (2pt) node[left] {$Y_t \sim \omega_{\nu}(t)$};
    \fill (zt) circle (2pt) node[right] {$Z_t \sim \omega_{\rho}(t)$};
    \fill (z0) circle (2pt) node[right] {$X' \sim \mu$};
\end{tikzpicture}
    \caption{The coupling $\Upsilon_t$ between $(\mu, \nu, \rho, \omega_{\nu}(t),  \omega_{\rho}(t))$. Solid lines indicate optimal couplings. Dashed lines and missing lines indicate potentially suboptimal ones.}\label{fig:upsilon}
\end{figure}

We have
\begin{align}
     W_2^2&(\omega_\nu(t),\omega_\rho(t))= \E\|Y_t -Z_t\|^2=\E\|(1-t)\,X + t\,Y -Z_t\|^2\nonumber\\
     &=(1-t)\,\E\|X-Z_t\|^2 + t\, \E\|Y - Z_t\|^2 -t\,(1-t)\, \E\|X-Y\|^2 \label{eq:pr:upsilon:1}\,,
\end{align}
where we used Lemma~\ref{lem:hilbertflat} for $\bH=\R^d$. Now note that $X$ and $Z_t$ are potentially coupled in a suboptimal way so that 
$$
(1-t)\,\E\|X-Z_t\|^2 \ge (1-t)\,W_2^2(\mu, \omega_\rho(t)) = t^2\,(1-t)\,  W_2^2(\mu, \rho)\,.
$$
Moreover, using Lemma~\ref{lem:hilbertflat} again, we get
\begin{align*}
    t\, \E\|&Y - Z_t\|^2 = t\, \E\|Y - (1-t)\,X' + t\,Z\|^2 \\
    &=t\, (1-t)\, \E\|Y-X'\|^2 + t^2\,\E\|Y - Z\|^2 - t^2\,(1-t)\, \E\|X'-Z\|^2\\
    &\ge t\,(1-t)\, \E\|Y-X\|^2 + t^2\, \E\|Y - Z\|^2 - t^2\,(1-t)\, \E\|X'-Z\|^2\,,
\end{align*}
where in the above inequality, we used the fact that $X$ and $Y$ are optimally coupled. 

Plugging the above two displays in~\eqref{eq:pr:upsilon:1}, we see that the terms in $\E\|Y-X\|^2 $ can cancel out. We get
\begin{align*}
       W_2^2&(\omega_\nu(t),\omega_\rho(t)) \\
       &\ge t^2\,(1-t)\, W_2^2(\mu, \rho) +  t^2\, \E\|Y - Z\|^2 -  t^2\,(1-t)\, \E\|X'-Z\|^2\\
       &= t^2 \,\E\|Y-Z\|^2\,.
\end{align*}
Recall from~\eqref{eq:ub_tan_cone_dist} that $W_2(\omega_\nu(t), \omega_\rho(t)) = O(t)$, so that $\E\|Y_t - Z_t\|^2 = O(t^2)$.
Since
\begin{align*}
    X
    &= (Y_t - tY)/(1-t) \qquad\text{and}\qquad X' = (Z_t - tZ)/(1-t)\,,
\end{align*}
it implies $\E\|X-X'\|^2 = O(t^2)$ as well.
Assuming (without justification) that $T_{\mu\to\rho}$ is Lipschitz\footnote{This assumption be lifted via approximation arguments, at the cost of additional technicalities.}
\begin{align*}
    \E\|Y-Z\|^2
    &= \E\|T_{\mu\to\nu}(X) - T_{\mu\to\rho}(X')\|^2 \\
    &= \E\|T_{\mu\to\nu}(X) - T_{\mu\to\rho}(X)\|^2 - O(t)\,.
\end{align*}
This readily yields
$$
\lim_{t\searrow 0} \frac{W_2^2(\omega_\nu(t),\omega_\rho(t))}{t^2}\ge \|T_{\mu \to \nu} - T_{\mu \to \rho}\|^2_{L^2(\mu)}\,,
$$
which concludes the proof of our Lemma.
\end{proof}

It follows from Lemma~\ref{lem:Wasserstein_derivative} that
\begin{align}
     \cos\measuredangle(&\omega_\nu,\omega_\rho) = \frac{W_2^2(\mu,\nu)+W_2^2(\mu,\rho)-\|T_{\mu \to \nu} - T_{\mu \to \rho}\|^2_{L^2(\mu)}}{2\,W_2(\mu,\nu)\,W_2(\mu,\rho)} \nonumber\\
     &=  \frac{\|T_{\mu \to \nu}-{\id}\|^2_{L^2(\mu)}+\|T_{\mu \to \rho}-{\id}\|^2_{L^2(\mu)}-\|T_{\mu \to \nu} - T_{\mu \to \rho}\|^2_{L^2(\mu)}}{2\,\|T_{\mu \to \nu}-{\id}\|_{L^2(\mu)}\,\|T_{\mu \to \rho}-{\id}\|_{L^2(\mu)}}\nonumber \\
     &=\cos \measuredangle(T_{\mu \to \nu}-{\id},T_{\mu \to \rho}-{\id}) \label{eq:cos_wasserstein}\,,
\end{align}
where the last $\cos$ is understood in the Hilbert space $L^2(\mu)$.

In turn, the law of cosines~\eqref{eq:law_cosines} implies that the metric on the tangent cone at $\mu$ is given by
\begin{align*}
    \md_\mu^2\bigl((\omega_\nu, s), (\omega_\rho,t)\bigr)
    &= s^2 +t^2 -2st\cos \measuredangle(\omega_\nu,\omega_\rho)\\
    &= s^2 +t^2 -2st\cos \measuredangle(T_{\mu \to \nu}-{\id},T_{\mu \to \rho}-{\id})\\
    &=\|s\,(T_{\mu \to \nu}-{\id})-t\,(T_{\mu \to \rho}-{\id})\|^2_{L^2(\mu)}\,.
\end{align*}
We have shown that the tangent cone equipped with the metric $\md_\mu$ is isometric to a Hilbert space. We have proved the following theorem which we had identified using the formalism of Otto calculus in Section~\ref{sec:otto_calc}.

\begin{theorem}\label{thm:tangentWass}
Let $\mu \in \cP_{2,\rm ac}(\R^d)$. Then the tangent cone $T_\mu\cW_2(\R^d)$ at $\mu$ is a convex subset of $L^2(\mu)$. Moreover, for any $\nu \in \cP_2(\R^\dd)$, we have
$$
\log_\mu(\nu)=T_{\mu \to \nu}-{\id} \in L^2(\mu)\,,
$$
where $T_{\mu \to \nu}$ is the Brenier map from $\mu$ to $\nu$.
\end{theorem}

In cases such as the one above, where the tangent cone $T_\mu\cW_2$ equipped with $\langle \cdot, \cdot\rangle$ from Definition~\ref{def:tangentcone} is, in fact, a convex subset of a Hilbert space, we call it, by abuse of notation, the \emph{tangent space} at $\mu$. It follows readily from the definition of the logarithmic map that the inner product $\langle\cdot, \cdot\rangle_\mu$  is given for any $\nu, \rho \in \cP_2
(\R^\dd)$ by
\begin{align*}
\langle \log_\mu(\nu), \log_\mu(\rho) \rangle_\mu
&=\langle T_{\mu \to \nu}-{\id}, T_{\mu \to \rho}-{\id}\rangle_{L^2(\mu)}\\
&=\int \bigl\langle T_{\mu \to \nu}(x)-x, T_{\mu \to \rho}(x)-x\bigr\rangle\,\mu(\ud x)\,.
\end{align*}

\section{Discussion}

\noindent\textbf{\S\ref{subsec:geodesic}.} Wasserstein geodesics are discussed in detail in~\cite[Chapter 5]{Vil03} and~\cite[Chapter 7]{AmbGigSav08}.
More generally, the Wasserstein space over any length space is also a length space.

\noindent\textbf{\S\ref{sec:curvature}.}
The non-negative curvature of the Wasserstein space is proven in~\cite[Section 7.3]{AmbGigSav08}.
More generally, the Wasserstein space over a non-negatively curved Alexandrov space is also non-negatively curved.

The curve in Exercise~\ref{ex:gen_geod} is called a \emph{generalized geodesic}\index{geodesic!generalized} and it plays an important role in the theory of Wasserstein gradient flows, as well as occasionally in other applications of optimal transport.

As mentioned in the discussion notes for Section~\ref{sec:vi}, there is a notion of synthetic Ricci curvature lower bounds which makes sense on geodesic spaces.
It is a natural to ask whether this notion recovers the non-negative Alexandrov curvature of the Wasserstein space when equipped with an appropriate measure. Unfortunately, \cite{Cho12} shows that does not yield any finite lower bound even for even for the (flat) Wasserstein space on the real line.

\noindent\textbf{\S\ref{subsec:angles}.}
The tangent cone of the Wasserstein space and its relationship to the tangent space is discussed in~\cite[Section 12.4]{AmbGigSav08}.

\section{Exercises}

\begin{enumerate}
    \item Show that the space of probability measures endowed with MMD (Definition~\ref{def:MMD}) defines a flat geometry.
    
    \item\label{ex:gen_geod} Let $\mu,\nu,\rho \in \cP_2(\R^d)$. Prove that there \emph{exists} a curve $\omega : [0,1]\to \cP_2(\R^d)$ with $\omega(0) = \mu$, $\omega(1) = \nu$ such that the opposite inequality to Theorem~\ref{thm:curvw2ge0} holds, i.e.,
    \begin{align*}
        W_2^2(\rho,\omega(t)) \le (1-t)\,W_2^2(\rho,\mu) + t\,W_2^2(\rho,\nu) - t\,(1-t)\,W_2^2(\mu,\nu)\,.
    \end{align*}
    \emph{Hint}: for $X_\mu \sim \mu$, $X_\nu \sim \nu$, $X_\rho\sim\rho$, optimally couple $(X_\mu, X_\rho)$ and $(X_\nu,X_\rho)$.
    Define $\omega(t)$ to be the law of a suitable interpolation of $X_\mu$ and $X_\nu$.
    Compare with Exercise~\ref{ex:gen_geod_cvx} from Chapter~\ref{chap:WGF}.

    \item Generalize the proof of Theorem~\ref{thm:W2geo} to show that for any $p\ge 1$, the space $\cP_p(\R^d)$ of probability measures with finite $p$-th moment, equipped with the $p$-Wasserstein metric $W_p$, is a geodesic space.
    What are the geodesics?

    \item\label{ex:triangle_cone} Generalizing Definition~\ref{def:tan_cone}, use the following outline to show that if $(S, \md)$ is a metric space with diameter at most $\pi$, and $\cone(S)$ is the set $X \times [0,\infty)$ with all points of the form $(x, 0)$ identified, then
    \begin{align*}
        \md((x,s), (y,t))
        &\deq \sqrt{s^2 + t^2 - 2st \cos \md(x,y)}
    \end{align*}
    defines a metric on $\cone(S)$.
    To do so, let $(x_1,r_1)$, $(x_2, r_2)$, $(x_3, r_3)$ be three points in the cone and construct three points $y_1,y_2,y_3 \in \R^2$ so that their distances from the origin equal $r_1$, $r_2$, $r_3$ respectively, and so that the angles between $y_1$ and $y_2$, and between $y_2$ and $y_3$, equal $\md(x_1, x_2)$ and $\md(x_2, x_3)$ respectively. Show that $\|y_1 - y_2\| = \md((x_1,r_1), (x_2,r_2))$ and $\|y_2-y_3\| = \md((x_2,r_2), (x_3, r_3))$. (Caution: $\|y_1-y_3\|$ does not necessarily equal $\md((x_1,r_1), (x_3,r_3))$.)
    Now establish the triangle inequality for the cone metric, splitting into two cases according to whether or not $\md(x_1,x_2) + \md(x_2,x_3) \le \pi$.
\end{enumerate}

\chapter{Wasserstein barycenters}
\label{chap:barycenters}

\emph{Averaging} data is among the most fundamental of the statistician's tools, but its implementation on non-Euclidean spaces, capturing data modalities that differ from the typical vector-valued covariates common in traditional statistical literature, often requires care.
Suppose, for instance, that our dataset consists of \emph{images} and we wish to define a suitable notion of an average image which captures representative aspects of the whole.
This model problem arises in situations such as the aggregation of information from repeated MRI scans.

We can represent a $p$-pixel image via its values in the R, G, B channels for each pixel, thereby considering it as a vector taking values in $\{0,1,\dotsc,255\}^{3p}$.
A na\"{\i}ve approach to the averaging problem would be to simply compute the usual average of these vector representations of the image.
Attempting this method on a few images, however, should readily convince the reader that this notion of average is unsatisfactory.

A closer inspection of the na\"{\i}ve approach reveals the nature of the problem: when we average the vector representations of the image, we tacitly endow the space $\R^{3p}$ with the Euclidean geometry, and there is no reason to expect that this geometry should be compatible with our embedding of images into $\R^{3p}$. Indeed, the representation of an image as a vector in $\R^{3p}$ is an engineering choice, not an intrinsic quality of the image. A perhaps more principled approach would be to regard the images as living in an abstract space $S$ endowed with a metric $\md$ which captures closeness with respect to the attributes we regard as important for the data under consideration. Our task can then be formulated as follows: given points $x_1,\dotsc,x_n$ inside a metric space $(S,\md)$, what is a suitable notion for the \emph{average} of $x_1,\dotsc,x_n$?

Fortunately there is a general and useful answer to this question, which is motivated as follows. It is not hard to see that for $x_1,\dotsc,x_n$ belonging to a Hilbert space $\bH$, the average $\frac{1}{n}\sum_{i=1}^n x_i$ is characterized as the unique minimizer of the functional
\begin{align*}
    x\mapsto \frac{1}{n}\sum_{i=1}^n \|x_i-x\|^2\,.
\end{align*}
This formulation only involves squared distances and is amenable to generalization to metric spaces.

\begin{definition}[Barycenter]\index{barycenter}
    Given any probability measure $P$ over a metric space $(S, \md)$, we say that $b$ is a \emph{barycenter} of $P$ if it is a minimizer of the functional
    \begin{align*}
        b\mapsto \int \md^2(b, x) \, P(\ud x)\,.
    \end{align*}
\end{definition}

In particular, if we take $P$ to be an empirical measure $\frac{1}{n}\sum_{i=1}^n \delta_{x_i}$, then a barycenter of $P$ is an average of the points $x_1,\dotsc,x_n$.
Note that at this level of generality, a barycenter may not exist, and even if one exists, it may not be unique.

The case when $(S,\md)$ is the Wasserstein space is already of interest and provides motivation for the theory we develop in this chapter.
For example, Wasserstein barycenters provide a geometrically meaningful solution to the image averaging problem with which we opened, as well as to many other problems such as curve registration; see~\cite{Rabin+12Barycenter, CutDou14, GraPeyCut15Neuro, SolGoePey15, BonPeyCut16BaryCoord, PanZem16Point, SriLiDun18Bary, PeyCut19, LeGLouRig20Fairness} and the references therein.
However, since the framework we develop fits naturally within metric geometry, as developed in Chapter~\ref{chap:geometry}, we work in this setting and specialize later.

Here, we develop statistical theory to justify the use of geometric averaging methods in practice.
Namely, assume that we have i.i.d.\ data $X_1,\dotsc,X_n$ drawn from a distribution $P$ over $(S,\md)$, and let $b^\star$ denote the barycenter of $P$. This \emph{population barycenter} is our quantity of interest and is unknown. In order for the statistical problem to be well-posed, we always work under assumptions which guarantee that $b^\star$ exists and is unique.

There is a natural plug-in estimator for this problem: the barycenter $b_n$ of the empirical measure or the \emph{empirical barycenter}, defined as the minimizer of the functional
\begin{align*}
    b\mapsto \frac{1}{n} \sum_{i=1}^n \md^2(b,X_i)\,.
\end{align*}
Our goal is to quantify the error $\md(b_n, b^\star)$ based on natural geometric features of the space $(S,\md)$.

\section{The Hilbert case}\label{sec:hilbertbarycenter}

In the case where $(S,\md)$ is a Hilbert space, the empirical barycenter converges at the so-called parametric rate. This is easy to see. To that end, let $\bH$ be a Hilbert space and recall that in this case 
$$
 b_n=\frac{1}{n}\sum_{i=1}^n X_i\,, \qquad b^\star=\E X=\int x\, P(\ud x)\,,
$$
where we used a Pettis integral to define $\bs$. We have
\begin{align*}
\E\| b_n -b^\star\|^2&=\frac{1}{n^2}\sum_{i,j=1}^n \E\langle X_i-\E X, X_j-\E X\rangle\\
&=\frac{1}{n}\var(X)\,, \qquad \text{where} \qquad \var(X)=\E\| X-\E X\|^2\,,
\end{align*}
where we used the independence of the $X_i$'s and bilinearity of the inner product. Unfortunately, this proof, while concise and leading to an equality (!) is not very instructive since it makes crucial use of the closed form for the barycenter, as well as the inner product structure, which do not extend beyond Hilbert spaces.

Remarkably, the same parametric rate of estimation for the barycenter continue to hold in NPC spaces, see Exercise~\ref{ex:npc_bary}.
Unfortunately, as we saw in Theorem~\ref{thm:curvw2ge0}, our main space of interest, namely the Wasserstein space, is an \emph{NNC space}.
Therefore, our goal is to develop statistical theory for the more difficult setting of $\curv \ge 0$.

Returning to the Hilbert case for inspiration, we propose a second proof which still leads to qualitatively the same result but is off by a factor 4. By definition of $ b_n$, we have
$$
P_n\| b_n -\bcdot\|^2 \le P_n\|b^\star -\bcdot\|^2\,.
$$
Here and in the sequel, we use the shorthand operator notation: for any integrable function $f$, $Pf(\bcdot)= \int f(x)\, P(\ud x)$ and in particular
$$
P_n f(\bcdot)=\frac1n\sum_{i=1}^n f(X_i)\,,
\quad \text{where}  \quad
P_n=\frac1n\sum_{i=1}^n \delta_{X_i}
$$
denotes the empirical distribution of the $X_i$'s.

Next, note that
$$
\|b_n -\bcdot\|^2=\|b_n - b^\star\|^2+ \| b^\star - \bcdot \|^2+ 2\,\langle  b_n-b^\star, b^\star - \bcdot\rangle\,.
$$
Therefore, applying operator $P_n$, we get
$$
\|b_n - b^\star\|^2+P_n \| b^\star - \bcdot \|^2+ 2 P_n \langle  b_n-b^\star, b^\star - \bcdot\rangle \le P_n\|b^\star -\bcdot\|^2\,,
$$
so that
$$
\|b_n - b^\star\|^2 \le 2  P_n \langle  b_n-b^\star, \bcdot - b^\star\rangle\,.
$$
Now the above inequality simply says that $\|b_n - b^\star\|^2 \le 2\,\|b_n - b^\star\|^2 $, which is not very useful but we are going to keep going with it for the sake of argument. 

Note first that by linearity of the inner product
$$
 P \langle  b_n-b^\star, \bcdot - b^\star\rangle=0\,.
$$
Therefore, we have
$$
\|b_n - b^\star\|^2 \le 2 \, (P_n-P) \langle  b_n-b^\star, \bcdot - b^\star\rangle= 2 \, \langle  b_n-b^\star,  (P_n-P)(\bcdot - b^\star)\rangle\,.
$$
Dividing on both sides by $\|b_n - b^\star\|$ and applying the Cauchy--Schwarz inequality, we get
\begin{equation}\label{EQ:boundhilbert4}
\E\|b_n - b^\star\|^2 \le 4\,\E\|(P_n-P)(\bcdot - b^\star)\|^2=\frac4n\var(X) \,.
\end{equation}

What did we learn in this proof? First we have only an inequality and lost a factor 4. Our major gain was that we never used the closed form for $b_n$. Instead, we only used the fact that
$$
\E\|(P_n-P)(\bcdot - b^\star)\|^2 \le \var(X) /n\,,
$$
which applies more broadly. On the downside, we used the linearity of the inner product and more generally the Hilbert structure quite extensively. It turns out that this is quite necessary to obtain our results. Therefore, we force the Hilbert structure in through the tangent space of $\cW_2$ and keep track of how much we lose.

\section{Barycenters on positively curved spaces}\label{sec:bary_nnc}

Let $P$ be a probability measure on a positively curved geodesic space $(S,\md)$. Let $\bs$ be any barycenter of $P$ and let $b_n$ be an empirical barycenter built from $n$ independent copies $X_1, \ldots, X_n$ of $X \sim P$:
$$
\bs \in \argmin_{b \in S} \int \md^2(b, x)\,P(\ud x)\,, \qquad b_n \in \argmin_{b \in S} \sum_{i=1}^n \md^2(b,X_i)\,.
$$

Before we proceed, we discuss our overall approach.
The argument in the Hilbert case rests on the inequality $P_n \md^2(b_n, \bcdot) \le P_n \md^2(\bs, \bcdot)$, which holds true by definition of $b_n$ and highlights that the empirical barycenter is an instance of the empirical risk minimization (ERM)\index{empirical risk minimization (ERM)} framework within the statistical estimation literature.
Using ERM techniques, we could hope to control the estimation error via measures of the ``complexity'' of the space $(S,\md)$, and this approach has been pursued in the literature (see~\cite{AhiLeGPar20Bary}).
However, for our application of interest in which $(S,\md)$ is taken to be the Wasserstein space, the complexity is prohibitively large and it leads to non-parametric rates of estimation; in particular, they suffer from the curse of dimensionality, similarly to what we saw in Chapter~\ref{chap:primal-dual}.

However, we have just seen that the rates of estimation in a Hilbert space escape the curse, despite the fact that Hilbert spaces can even be \emph{infinite}-dimensional.
Our intuition therefore leads us to believe that, even if we are working over a curved space $(S,\md)$, as long we can restrict ourselves to sufficiently ``flat'' parts of $S$, then perhaps we could recover the Hilbertian rates.
In the sequel, we seek natural geometric conditions---morally, they encode curvature bounds---which enable fast, \emph{parametric} rates of estimation.

\subsection{Master theorem}

We begin with a general result that mimics the proof of Section~\ref{sec:hilbertbarycenter} in the Hilbert case.

Before stating it, we introduce a quantity that measures how much the tangent space at the barycenter ``hugs" the original space at the barycenter $b^\star$. 

\begin{definition}[Hugging]\index{hugging}
Let $(S,\md)$ be a geodesic space such that $\curv(S)\ge 0$. For any $b^\star, b \in S$, let $h_{b^\star}^b$  be the \emph{hugging function of $S$ at $b^\star$ in direction $b$} defined by
\begin{equation}
\label{eq:kb}
 h^b_{\bs}(x)=1-\frac{\|{\log_{\bs}(x)}-\log_{\bs}(b)\|^2_{\bs}-\md^2(x,b)}{\md^2(b,\bs)}\,, \qquad x \in S\,.
\end{equation}
\end{definition}

Note that it follows from Proposition~\ref{lem:tangentcone2} that $h^b_{\bs}(x)\le 1$ for all $x \in S$. Moreover,
if $S$ is a Hilbert space, then $\|{\log_{\bs}(x)}-\log_{\bs}(b)\|^2_{\bs}=\md^2(x,b)$ and $ h^b_{\bs}\equiv 1$. In general, $h^b_{\bs}(x)$ may be negative when there is a lot of curvature around $\bs$ but it remains non-negative in average when computed at barycenter $\bs$. This result follows from the following simple but important observation.

\begin{theorem}[Variance equality]\label{thm:vi}\index{variance equality}
Let $(S,\md)$ be a geodesic space with $\curv(S)\ge 0$. Let $Q\in\cP_2(S)$ be a probability distribution on $S$ with barycenter $\bs$. Assume further that the tangent cone of $S$ at $\bs$ equipped with $\langle\cdot, \cdot \rangle_{\bs}$ is a convex subset of a Hilbert space. Then, for all $b\in S$,
\begin{equation}\label{eq:vi}
\md^2(b,\bs) \int h^{b}_{\bs}(x)\,Q(\ud x)= \int(\md^2(x,b)-\md^2(x,\bs))\, Q(\ud x)\,,
\end{equation}
where $h^b_{\bs}$ is the hugging function defined in \eqref{eq:kb}. 
\end{theorem}
\begin{proof}
By definition of $h_{\bs}^b$, we have 
\begin{align*}
   \md^2(b,\bs) \, h_{\bs}^b(\bcdot)
   &= \md^2(b, \bs) + \md^2(\bcdot, b) - \|{\log_{\bs} b} -\log_{\bs} \bcdot\|^2_{\bs} \\
   &= \md^2(b, \bs) + \md^2(\bcdot, b) \\
   &\qquad{} - \|{\log_{\bs} b}\|_{\bs}^2 -\|{\log_{\bs} \bcdot}\|^2_{\bs} +2\,\langle \log_{\bs} \bcdot, \log_{\bs} b \rangle_{\bs}\\
      &= \md^2(b, \bs) + \md^2(\bcdot, b) \\
   &\qquad{} - \md^2(b, \bs) - \md^2(\bs, \bcdot)+2\,\langle \log_{\bs} \bcdot, \log_{\bs} b \rangle_{\bs}\\
   &= \md^2(\bcdot, b) - \md^2(\bcdot, \bs)+2\,\langle \log_{\bs} \bcdot, \log_{\bs} b \rangle_{\bs}\,.
\end{align*}
Therefore applying the linear operator $Q$, we get
$$
    \md^2(b,\bs) \, Q h_{\bs}^b(\bcdot)=Q\md^2(\bcdot, b) - Q\md^2(\bcdot, \bs) + 2\, \langle \log_{\bs}\bs , \log_{\bs} b \rangle_{\bs}\,,
$$ 
where we use the fact that $Q\log_{\bs}\bcdot=\log_{\bs} \bs$ or, in other words, that $\log_{\bs}\bs$ is the barycenter of $(\log_{\bs})_{\#}Q$. Indeed, we have by Proposition~\ref{lem:tangentcone2} that for all $b \in S$,
$$
Q\|{\log_{\bs}\bcdot}-\log_{\bs}\bs \|_{\bs}^2=Q\md^2(\bcdot, \bs)\le Q\md^2(\bcdot, b)\le Q\|{\log_{\bs}\bcdot}-\log_{\bs}b \|_{\bs}^2
$$
with equality if $b=\bs$ so that $\log_{\bs}\bs$ is a barycenter for $(\log_{\bs})_{\#}Q$ and therefore $Q\log_{\bs}\bcdot=\log_{\bs}\bs$.

Finally since $\log_{\bs}\bs=o_{\bs}$ is the tip of the tangent cone, we have $\|{\log_{\bs}\bs}\|_{\bs}=0$, which, in turn, yields $\langle \log_{\bs}\bs , \log_{\bs} b \rangle_{\bs}=0$. This completes the proof.
\end{proof}

A direct consequence of the variance equality is that if $ \int h^b_{\bs}\,\ud Q>0$, then $\bs$ is the unique barycenter of $Q$. Moreover, since the right-hand side of the variance equality is non-negative by definition of a barycenter $\bs$, we readily get the following corollary.

\begin{corollary}\label{cor:kbounds}
Under the same assumptions as Theorem~\ref{thm:vi}, we have for any $b \in S$,
$$
0 \le \int h^b_{\bs}(x)\,Q(\ud x) \le 1\,.
$$
Moreover, for any $b, x \in S$, we have $h^b_{\bs}(x) \le 1$.
\end{corollary}

It turns out the hugging function at $\bs$ plays a key role in obtaining parametric rates of convergence for empirical barycenters.

\begin{theorem}[Master theorem]\label{thm:master}
Let $P$ be a probability measure on a NNC geodesic space $(S,\md)$ and denote by $\bs$ and $b_n$ a barycenter of $P$ and an empirical barycenter respectively. Assume further that the tangent cone of $S$ at $\bs$ equipped with $\langle\cdot, \cdot \rangle_{\bs}$ is a convex subset of a Hilbert space. Then, the following holds: if for any $b \in S$, 
\begin{equation}\label{EQ:kpositif}
h_{\bs}^{b}(\bcdot) \ge \mathsf{h}_{\min} >0\,,
\end{equation}
then $b_n$ and $\bs$ are both unique and
$$
\E\md^2(b_n, \bs)\le \frac{4\sigma^2}{ \mathsf{h}_{\min}^2n}\,,
$$
where $\sigma^2$ denotes the variance of $P$ defined by
\begin{equation}\label{EQ:defvarP}
\sigma^2=\int \md^2(\bs, x)\, P(\ud x)\,.
\end{equation}
\end{theorem}
\begin{proof}
Note first that uniqueness follows directly from the variance equality and~\eqref{EQ:kpositif}.

Next, we start as in the Hilbert case by observing that
$$
P_n \md^2(b_n, \bcdot) \le P_n \md^2(\bs, \bcdot)\,.
$$
It yields 
\begin{equation}\label{EQ:prmaster1}
    \begin{aligned}
        &P_n \md^2(b_n, \bcdot) -P_n\|{\log_{\bs}b_n} - \log_{\bs}\bcdot \|^2_{\bs} \\
        &\qquad{} +P_n\|{\log_{\bs}b_n} - \log_{\bs}\bcdot \|^2_{\bs}-P_n \md^2(\bs, \bcdot)\le 0\,.
    \end{aligned}
\end{equation}
We now make use of the fact that the tangent cone has a Hilbert structure so that
\begin{align*}
&\|{\log_{\bs}b_n} - \log_{\bs}\bcdot \|^2_{\bs} \\
&\qquad =\|{\log_{\bs}b_n}\|^2_{\bs}+\|{\log_{\bs}\bcdot}\|^2_{\bs}-2\,\langle \log_{\bs}b_n, \log_{\bs}\bcdot\rangle_{\bs}\\
&\qquad = \md^2(\bs, b_n)+ \md^2(\bs, \bcdot)-2\,\langle \log_{\bs}b_n, \log_{\bs}\bcdot\rangle_{\bs}
\end{align*}
where in the second identity, we used twice the equality case in Proposition~\ref{lem:tangentcone2}. Plugging this into~\eqref{EQ:prmaster1} yields
$$
 \md^2(\bs, b_n)\le  P_n \big[\|{\log_{\bs}b_n} - \log_{\bs}\bcdot \|^2_{\bs}-\md^2(b_n, \bcdot)\big]+2P_n\langle \log_{\bs}b_n, \log_{\bs}\bcdot\rangle_{\bs}\,.
 $$
Next, by definition of the hugging function, we get
$$
\|{\log_{\bs}b_n} - \log_{\bs}\bcdot \|^2_{\bs}-\md^2(b_n, \bcdot)=(1-h_{\bs}^{b_n}(\bcdot))\, \md^2(b_n, \bs)\,.
$$
It yields
$$
 \mathsf{h}_{\min}\, \md^2(b_n, \bs) \le 2P_n\langle \log_{\bs}b_n, \log_{\bs}\bcdot\rangle_{\bs}\,.
$$
The right-hand side is simply an average in a Hilbert space so, dividing by $\|{\log_{\bs} b_n}\|_{\bs}=\md(b_n, \bs)$ on both sides and applying Cauchy--Schwarz, we get
$$
\mathsf{h}_{\min}^2\,\E \md^2(b_n,\bs) \le \frac{4\sigma^2}{n}\,,
$$
where
$$
\sigma^2=\int \|{\log_{\bs} x}\|_{\bs}^2\, P(\ud x)=\int \md^2(\bs,x) \,P(\ud x)
$$
as desired.
\end{proof}

It follows from inspecting the proof of the master theorem that in order to obtain parametric rates of estimation for $\bs$, it suffices to have the weaker condition $P_nh_{\bs}^{b_n}(\bcdot)\ge \mathsf{h}_{\min}>0$. Since $P_n$ is a random measure, we prefer not to impose conditions on it and focus instead on the stronger condition~\eqref{EQ:kpositif}. We are going to obtain such results using the notion of \emph{extendable geodesics}.

\subsection{Extendable geodesics}

We now present a compelling synthetic geometric condition that implies this lower bound in the context of NNC spaces: the extendability, by a given factor, of all geodesics emanating from and arriving at the barycenter $\bs$.

\begin{definition}[Extendable geodesic]\index{geodesic!extendable}
Consider a constant-speed geodesic $\omega:[0,1]\to S$.
For $(\lambda_\mathrm{in},\lambda_\mathrm{out})\in [0, \infty]^2$, we say that $\omega$ is $(\lambda_\mathrm{in},\lambda_\mathrm{out})$-\emph{extendable} if there exists a path $\omega^{+}:[-\lambda_\mathrm{in},1+\lambda_\mathrm{out}]\to S$ which agrees with $\omega$ on $[0,1]$, called an \emph{extension} of $\omega$, which remains a geodesic between its endpoints $\omega^{+}(-\lambda_\mathrm{in})$ and $\omega^{+}(1+\lambda_\mathrm{out})$.
\end{definition}

Before we state the main result of this subsection, we need the following fact. 

\begin{theorem}\label{thm:extendvariance}
Suppose that $\curv(S) \ge 0$. Let $Q\in \cP_2(S)$ be a probability measure on $S$ with a barycenter~$\bs$. 
Suppose that, for each $x \in \supp(Q)$, there exists a constant-speed geodesic $\omega_{x}:[0,1]\to S$ connecting $\bs$ to $x$ which is $(0,\lambda)$-extendable for $\lambda>0$. Suppose in addition that $\bs$ remains a barycenter of the distribution $Q_{\lambda}=(e_{\lambda})_{\#} Q$ where $e_{\lambda}(x)=\omega^{+}_{x}(1+\lambda)$. Then for all $b \in S$,
\begin{equation}
\label{eq:extend}    
Qh_{\bs}^b(\bcdot) \ge \frac{\lambda}{1 + \lambda} \,.
\end{equation}
In particular, it implies that $\bs$ is the unique barycenter of $Q$.
\end{theorem}
\begin{proof}
Fix $y \in \supp(Q)$ and define $y_\lambda=e_{\lambda}(y)$. Let $\omega:[0,1]\to S$ be a constant-speed geodesic connecting $b^\star$ to $y_\lambda$. By definition, $\omega(\tau)=y$ for $\tau=1/(1+\lambda)$.
Since $\curv(S)\ge 0$, we have for any $b \in S$, 
\begin{align*}
\md^2(b, y)&\ge (1-\tau)\,\md^2(b, \bs)+\tau\, \md^2(b,y_\lambda)-\tau\,(1-\tau)\,\md^2(\bs, y_\lambda)\\
& = \frac{\lambda}{1+\lambda}\,\md^2(b, \bs)+\frac{1}{1+\lambda} \,\md^2(b,y_\lambda)-\frac{\lambda}{(1+\lambda)^2}\,\md^2(\bs, y_\lambda)\,.
\end{align*}
Next, observe that
$$
\md^2(\bs, y_\lambda)=(1+\lambda)^2\,\md^2(\bs, y)
$$
so that
\begin{align}
    &\frac{\lambda}{1+\lambda}\,\md^2(b, \bs)
    \le \md^2(b, y)+\lambda \,\md^2(\bs, y)- \frac{1}{1+\lambda}\, \md^2(b,y_\lambda) \nonumber\\
    &\qquad =\big(\md^2(b, y)-\md^2(\bs, y)\big)+ (1+\lambda)\, \md^2(\bs, y)- \frac{1}{1+\lambda}\, \md^2(b,y_\lambda)\,.\label{pr:extgeo1}
\end{align}
Moreover,
\begin{align*}
    &(1+\lambda)\, \md^2(\bs, y)- \frac{1}{1+\lambda}\, \md^2(b,y_\lambda) \\
    &\qquad =\frac{1}{1+\lambda}\,\big( (1+\lambda)^2\, \md^2(\bs, y) -\md^2(b,y_\lambda)\big)\\
    &\qquad =\frac{1}{1+\lambda}\,\big( \md^2(\bs, y_\lambda) -\md^2(b,y_\lambda)\big)\,.
\end{align*}
Thus, writing $Q_\lambda \deq  (e_\lambda)_{\#}Q$, we get
\begin{align*}
    &\int \Big((1+\lambda)\, \md^2(\bs, y)- \frac{1}{1+\lambda}\, \md^2(b,y_\lambda)\Big)\, Q(\ud y)\\
    &\qquad =\frac{1}{1+\lambda}\int \big( \md^2(\bs, y) -\md^2(b,y)\big)\, Q_\lambda(\ud y) \le 0\,,
\end{align*}
where in the last inequality, we used the fact that $\bs$ remains a barycenter of $Q_\lambda$. Together with~\eqref{pr:extgeo1} integrated with respect to $Q$, we get
\begin{align}\label{eq:var_ineq}
    \frac{\lambda}{1+\lambda}\,\md^2(b, \bs) \le \int \big(\md^2(b, y)-\md^2(\bs, y)\big)\,Q(\ud y)\,.
\end{align}
Combined with the variance equality (Theorem~\ref{thm:vi}), this completes the proof.
\end{proof}

The above notion of extendable geodesics gives a lower bound on $Ph_{\bs}^b(\bcdot)$ uniformly in $b$. While this is already an attractive feature that implies uniqueness of the barycenter, it suffers from two deficiencies. First, we need to control $P_nh_{\bs}^{b_n}(\bcdot)$ and $b_n$ is data-dependent, and it is unclear how to control the deviation $|P_nh_{\bs}^{b_n}-Ph_{\bs}^{b_n}|$ in a suitable fashion. Second, the condition that $P_\lambda=(e_\lambda)_{\#}P$ keeps the same barycenter is difficult to check and appears to be restrictive. 

To overcome both limitations, we allow for geodesics emanating from $\bs$ to be extendable in both directions.

\begin{theorem}\label{thm:extendgeod} 
Suppose that $\curv(S)\ge 0$ and let $x, b, \bs \in S$. Suppose that there exist $\lambda_\mathrm{in},\lambda_\mathrm{out}>0$ and a geodesic connecting $\bs$ to $x$ which is $(\lambda_\mathrm{in},\lambda_\mathrm{out})$-extendable. Then 
\[
h_{\bs}^b(x)\ge \mathsf{h}_{\min}= \frac{\lambda_\mathrm{out}}{1+\lambda_\mathrm{out}}-\frac{1}{\lambda_\mathrm{in}}\,.
\]
\end{theorem}
\begin{proof} 
Let $\omega_x:[0,1]\to S$ be a $(\lambda_\mathrm{in},\lambda_\mathrm{out})$-extendable geodesic connecting $\bs$ to $x$ and denote by $\omega^+_x:[-\lambda_\mathrm{in},1+\lambda_\mathrm{out}]\to S$  its extension. Let $z=\omega_x^+(-\xi)$ where $\xi=\lambda_\mathrm{in}/(1+\lambda_\mathrm{out})$ and consider the measure $Q$ defined by
\[
Q \deq \frac{\xi}{1+\xi}\,\delta_x+\frac{1}{1+\xi}\,\delta_{z}\,.
\] 
Since $Q$ is supported on $\omega^+$ we can easily compute its barycenter. Indeed, note that $x=\omega_x^+(1)$ so the barycenter of $Q$ is given by
$$
\omega^{+}_x\big( 1\cdot \frac{\xi}{1+\xi}-\xi\cdot \frac{1}{1+\xi} \big)=\omega^+(0)=\bs\,.
$$

Now, we wish to apply Theorem~\ref{thm:extendvariance} to $Q$. To that end, note that the constant-speed geodesics $\omega_x$ connecting $\bs$ to $x$ and $\sigma$ connecting $\bs$ to $z$ and defined by $\sigma(t)=\omega^+_x(-t\xi)$ are both $(0,1+\lambda_\mathrm{out})$-extendable by assumption and by construction respectively. 

Finally, we check that $\bs$ remains a barycenter of the probability measure $Q_{\lambda_\mathrm{out}}=(e_{\lambda_\mathrm{out}})_{\#}Q$ where $e_{\lambda_\mathrm{out}}(x)=\omega^+_x(1+\lambda_\mathrm{out})$. Indeed, by construction, $Q_{\lambda_\mathrm{out}}$ is the two-point probability measure given by
$$
Q_{\lambda_\mathrm{out}}=\frac{\xi}{1+\xi}\,\delta_{\omega^+(1+\lambda_\mathrm{out})}+\frac{1}{1+\xi}\,\delta_{\omega^+(-\lambda_\mathrm{in})}\,.
$$
Therefore, the barycenter is given by
\begin{align*}
    &\omega^+\big((1+\lambda_\mathrm{out})\cdot \frac{\xi}{1+\xi} -\lambda_\mathrm{in}\cdot \frac{1}{1+\xi}\big) \\
    &\qquad = \omega^+\big(\frac{(1+\lambda_\mathrm{out})\,\lambda_\mathrm{in}}{1+\lambda_\mathrm{in}+\lambda_\mathrm{out}} -\frac{\lambda_{\rm in}\,(1+\lambda_\mathrm{out})}{1+\lambda_\mathrm{in}+\lambda_\mathrm{out}}\big)
    =\omega^+(0)=\bs\,.
\end{align*}
 As a result, Theorem \ref{thm:extendvariance} implies that
\begin{align*}
    \frac{\lambda_\mathrm{out}}{1 + \lambda_\mathrm{out}}
    &\le Qh_{\bs}^b(\bcdot)= \frac{\xi}{1 + \xi}\, h_{\bs}^b(x) + \frac{1}{1 + \xi}\,h_{\bs}^b(z) \\
    &\le \frac{\xi}{1 + \xi} \,h_{\bs}^b(x) + \frac{1}{1 + \xi}\,,
\end{align*}
where we used Corollary~\ref{cor:kbounds} to bound $h^b_{\bs}(z)\le 1$ for all $b,z\in S$. 

Hence, we obtain
\begin{align*}
    h_{\bs}^b(x) &\ge \frac{1 + \xi}{\xi} \,\Bigl( \frac{\lambda_\mathrm{out}}{1 + \lambda_\mathrm{out}} - \frac{1}{1 + \xi}\Bigr) \\
    &=\frac{1 + \xi}{\xi}\, \Bigl( \frac{\lambda_\mathrm{out}}{ \lambda_\mathrm{in}}\,\xi - \frac{1}{1 + \xi}\Bigr) \\
    &=\frac{\lambda_\mathrm{out}}{ \lambda_\mathrm{in}}\,(1+\xi)-\frac1\xi\\
    &=\frac{\lambda_\mathrm{out}}{ \lambda_\mathrm{in}}+\frac{\lambda_\mathrm{out}}{ \lambda_\mathrm{in}}\cdot \frac{\lambda_\mathrm{in}}{1+\lambda_\mathrm{out}}-\frac{1+\lambda_\mathrm{out}}{\lambda_\mathrm{in}}\\
    &= \frac{\lambda_\mathrm{out}}{1 + \lambda_\mathrm{out}} -
    \frac{1}{\lambda_\mathrm{in}}\,,
\end{align*}  
which completes the proof.
\end{proof}

Note that Theorem~\ref{thm:extendgeod} gives a lower bound on $h_{\bs}^b(x)$ that is uniform in both $b$ and $x$. It is of course possible to make this result depend on $x$ only and get a result of the form
\[
h_{\bs}^b(x)\ge \frac{\lambda_\mathrm{out}(x)}{1+\lambda_\mathrm{out}(x)}-\frac{1}{\lambda_\mathrm{in}(x)}\,.
\]
If we assume that 
$$
P\Big(\frac{\lambda_\mathrm{out}(\bcdot)}{1+\lambda_\mathrm{out}(\bcdot)}-\frac{1}{\lambda_\mathrm{in}(\bcdot)}\Big)>0\,,
$$
then standard concentration tools ensure that $P_n h_{\bs}^{b_n}(\bcdot)>0$ for $n$ large enough as desired. 

Instead of going into these details, let us inspect the uniform bound more closely. From the master theorem, and Theorem~\ref{thm:extendgeod}, we get the following corollary.

\begin{corollary}
Let $P$ be a probability measure on an NNC geodesic space $(S,\md)$ and denote by $\bs$ and $b_n$ a barycenter of $P$ and an empirical barycenter respectively. Assume that the tangent cone of at $\bs$ equipped with $\langle\cdot, \cdot \rangle_{\bs}$ is a convex subset of a Hilbert space. Moreover,  let $\lambda_\mathrm{in},\lambda_\mathrm{out}\in [0,\infty]$ be such that
$$
 \mathsf{h} \deq \frac{\lambda_\mathrm{out}}{1 + \lambda_\mathrm{out}} -
    \frac{1}{\lambda_\mathrm{in}}>0
$$
and assume further that for any $x \in \supp(P)$, there exists a geodesic connecting $\bs$ to $x$ that is $(\lambda_\mathrm{in},\lambda_\mathrm{out})$-extendable. Then $\bs$ is unique and the empirical barycenter satisfies
$$
\E\big[\md^2(b_n, \bs)\big]\le \frac{4\sigma^2}{ \mathsf{h}n}
$$
where $\sigma^2$ denotes the variance of $P$ that is defined in~\eqref{EQ:defvarP}.
\end{corollary}

As a result, we get parametric rates when geodesics may be sufficiently extended. In particular, if $S$ is a Hilbert space, then all geodesics are infinitely extendable. Therefore $ \mathsf{h}=1$ and we recover~\eqref{EQ:boundhilbert4}.

\section{Parametric rates for Wasserstein barycenters}\label{sec:wass_bary}

To conclude these notes, we study Wasserstein barycenters as an example. Note that our result readily applies to this case. One may ask the question: how does the condition of extendable geodesics translate in terms of optimal transport? It turns out that it can be characterized in terms of regularity conditions on the Brenier maps.

\begin{theorem}\label{thm:w2_extendable}
Let $\mu, \nu \in \cW_2$ be two probability measures such that $\mu$ has a density and let $\varphi: \R^d \to \R$ be the convex function defined by $\varphi(x)=(\|x\|^2-f(x))/2$, where $f$ is the Kantorovich potential given in Definition~\ref{def:potentials}. In particular, $\nabla \varphi$ is defined $\mu$-almost surely and is the Brenier map. Recall that the unique constant-speed geodesic $\omega$ connecting $\mu$ to $\nu$ is given by $\omega(t)=((1-t)\,{\id} + t\,\nabla \varphi)_{\#}\mu$. Then, for any $\lambda>0$,  $\omega$ is $(0, \lambda)$-extendable if and only if $\varphi$ is $\frac{\lambda}{1+\lambda}$-strongly convex.
\end{theorem}
\begin{proof}
Assume first that $\omega$ is $(0, 1+\lambda)$-extendable and let
\begin{align*}
    \omega^+:[0,1+\lambda]\to \cW_2
\end{align*}
denote its extension. Let $Y_\lambda\sim \omega^{+}(1+\lambda)$ and observe that there exists a convex function $\varphi_\lambda$ defined $\mu$-almost everywhere such that $Y_\lambda = \nabla \varphi_\lambda(X)$, where $X\sim \mu$. Moreover, since $\omega^+$ is a geodesic and $\nabla \varphi(X)\sim \omega^+(1)$, we also have
$$
\nabla\varphi(X)=\frac{\lambda}{1+\lambda}\,X+\frac{1}{1+\lambda}\,Y_\lambda\,,
$$
so that $Y_\lambda = \nabla \varphi_\lambda(X)=(1+\lambda)\,\nabla\varphi(X)-\lambda X$. In particular, it means that we can choose
$$
\varphi(x)=\frac1{1+\lambda}\,\varphi_\lambda(x)+\frac{\lambda}{2\,(1+\lambda)}\,\|x\|^2\,.
$$
Since $\varphi_\lambda$ is convex, so is $\varphi_\lambda/(1+\lambda)$ and the above display implies that $\varphi$ is $\frac{\lambda}{1+\lambda}$-strongly convex.

\smallskip

Conversely, assume that $\varphi$ is $\frac{\lambda}{1+\lambda}$-strongly convex and define $Y_\lambda =(1+\lambda)\,\nabla\varphi(X)-\lambda X$ where $X \sim \mu$. We are going to show that $Y_\lambda$ and $X$ are optimally coupled. To that end, note that $Y_\lambda=\nabla\varphi_\lambda(X)$ where
$$
\varphi_\lambda(x)=(1+\lambda)\,\varphi(X)-\frac{\lambda}{2}\, \|x\|^2\,.
$$
Since $\varphi$ is $\frac{\lambda}{1+\lambda}$-strongly convex, $\varphi_\lambda$ is convex and thus $\nabla \varphi_\lambda$ is the Brenier map. It follows that $Y_\lambda$ and $X$ are optimally coupled so that $\omega^+:[0, 1+\lambda]\to \cW_2$ defined by 
$$
\omega^+(t)= \Big({\id} +\frac{t}{1+\lambda}\,(\nabla\varphi_\lambda-{\id})\Big)_{\#}\mu
$$
is a geodesic connecting $\mu$ and the distribution of $Y_\lambda$ such that $\omega^+(t)=\omega(t)$ for $t \in [0,1]$. Therefore $\omega$ is $(0, 1+\lambda)$-extendable.
\end{proof}

Recall that if $\mu$ and $\nu$ both have a density such that the Brenier map from $\mu$ to $\nu$ is given by $\nabla \varphi$, then the Brenier map from $\nu$ to $\mu$ is given by $\nabla \varphi^*$. Therefore, if $\varphi$ is $\beta$-smooth in the sense that for any $x,y \in \R^d$, 
$$
\varphi(x) -\varphi(y) \le \langle \nabla \varphi(y), x-y\rangle +\frac{\beta}{2}\,\|x-y\|^2\,,
$$
then $\varphi^*$ is $1/\beta$-strongly convex (see Lemma~\ref{thm:strcvx_smooth_dual}), which, in turn implies that the geodesic connecting $\nu$ to $\mu$ is $(0,\frac{1}{\beta+1})$-extendable.

These facts yield the following theorem but we provide an alternate, more direct, proof.

\begin{theorem}
    Let $P$ be a probability measure on $\cW_2$ with a barycenter $\bs$ that admits a density. Assume further that for any $\mu \in \supp(P)$ the Brenier map from $\bs$ to $\mu$ is $\alpha$-strongly convex and $\beta$-smooth with $\beta-\alpha \in [0,1)$. Then $\bs$ is unique and the empirical Wasserstein barycenter $b_n$ satisfies
$$
\E\bigl[W_2^2(b_n, \bs)\bigr]\le \frac{4\sigma^2}{(1-(\beta-\alpha))^2\, n}\,.
$$
\end{theorem}
\begin{proof}
For any $\mu \in \supp(P)$, let  $\varphi$ be such that $\nabla \varphi$ is the Brenier map from $\bs$ to $\mu$. For any $b,\mu \in \cW_2$, let $X,X' \sim \mu$, $Y,Y' \sim b$ and $Z,Z' \sim \bs$. In view of the gluing lemma, we may assume that $(X,Z)$ and $(Y,Z)$ are optimally coupled whereas we assume that $(X',Y')$ and $(X',Z')$ are optimally coupled.

Note that rearranging terms in the definition of the hugging function, our goal is to prove that
\begin{equation}\label{EQ:prkwass1}
\E\|X - Y\|^2\le \E\|X'-Y'\|^2+(\beta-\alpha)\,\E\|Y-Z\|^2\,.
\end{equation}
By assumption, for $\bs$-almost all  $z\in \R^\dd$ and any $y \in \R^\dd$, we have
\begin{equation}\label{EQ:strcvxsmth}
\frac{\alpha}{2}\, \|y-z\|^2\le \varphi(y)-\varphi(z) - \langle \nabla\varphi(z),y-z\rangle \le\frac{\beta}{2}\, \|y-z\|^2\,.
\end{equation}
It holds
\begin{align}
    &\E\|X - Y\|^2
    =\E\|X-Z\|^2+\E\|Y-Z\|^2+2\,\E\langle X-Z, Z-Y\rangle\nonumber\\
    &\quad =\E\|X-Z\|^2+\E\|Y-Z\|^2-2\,\E\langle Z, Z-Y\rangle+2\,\E\langle \nabla \varphi(Z), Z-Y\rangle \,.\label{EQ:prwassext1}
\end{align}
Next, note that on the one hand, it follows from \eqref{EQ:strcvxsmth} that
\begin{align*}
&2\,\E\langle \nabla \varphi(Z), Z-Y\rangle
\le 2\,\E\varphi(Z)-2\,\E\varphi(Y)+\beta\, \E\|Y-Z\|^2\\
&\qquad =2\,\E\varphi(Z')-2\,\E\varphi(Y')+\beta \,\E\|Y-Z\|^2\\
&\qquad \le 2\,\E\langle \nabla \varphi(Z'), Z'-Y'\rangle -\alpha\,\E\|Y'-Z'\|^2+ \beta\,\E\|Y-Z\|^2\\
&\qquad =2\,\E\langle X', Z'-Y'\rangle -\alpha\,\E\|Y'-Z'\|^2+ \beta\,\E\|Y-Z\|^2\,.
\end{align*}
Since $\E\|Y'-Z'\|^2 \ge \E\|Y-Z\|^2$, we get,
$$
2\,\E\langle \nabla \varphi(Z), Z-Y\rangle \le 2\,\E\langle X', Z'-Y'\rangle +(\beta-\alpha)\,\E\|Y-Z\|^2\,.
$$
\noindent On the other hand,
$$
\E\|Y-Z\|^2-2\,\E\langle Z, Z-Y\rangle=\E\|Y\|^2-\E\|Z\|^2=\E\|Y'\|^2-\E\|Z'\|^2\,.
$$

\noindent Together, with~\eqref{EQ:prwassext1}, the above two displays yield
\begin{align*}
\E\|X - Y\|^2
&\le \E\|X'-Z'\|^2+\E\|Y'\|^2-\E\|Z'\|^2 \\
&\qquad{}+ 2\,\E\langle  X', Z'-Y'\rangle +(\beta-\alpha)\,\E\|Y-Z\|^2\\
&=\E\|X'-Y'\|^2+(\beta-\alpha)\,\E\|Y-Z\|^2\,,
\end{align*}
which completes the proof of~\eqref{EQ:prkwass1}.

We have proved that $h_{\bs}^b(\mu)\ge 1 -(\beta -\alpha)$ which, together with the master theorem, completes the proof.
\end{proof}

Barycenters are the equivalent of averages on curved spaces. As such they are the building block of many statistical tools including regression~\cite{CheLinMul23}, analysis of variance~\cite{DubMul19}, change-point detection~\cite{DubMul20}, discriminant analysis~\cite{FlaCutCou18}, and principal component analysis~\cite{Big+17GeoPCA,CazSegBig18}. Despite initial work, many questions about the statistical properties of these statistical objects remain to be understood.

\section{Discussion}\label{sec:bary_discussion}

\noindent\textbf{\S\ref{sec:hilbertbarycenter}.} Beyond the setting of Hilbert spaces, quantitative laws of large numbers are obtained over Banach spaces in relation to the theory of type and cotype, see~\cite{LedTal91}.
Also, see the excellent exposition~\cite{Sturm03NPC} for barycenters over NPC spaces, from which Exercise~\ref{ex:npc_bary} is taken.

\noindent\textbf{\S\ref{sec:bary_nnc}.} 
The material in this section is taken from~\cite{AhiLeGPar20Bary, LeGetal22Bary}.

\noindent\textbf{\S\ref{sec:wass_bary}.}
The basic theory of Wasserstein barycenters (existence, duality, etc.) was developed in~\cite{AguCar11}; see Exercise~\ref{ex:bary_multimarginal}. Statistical consistency for Wasserstein barycenters was established in~\cite{LeGLou17Bary}.
The variance inequality in Exercise~\ref{ex:w2_var_ineq} is from~\cite{CheMauRig20a}.

Substantial attention has also been devoted to the computation of barycenters.
For discrete distributions, the work of~\cite{AltBoi21BaryPoly, AltBoi22BaryNPHard} established polynomial-time tractability of Wasserstein barycenters in fixed dimension, and \textbf{NP}-hardness in general dimension; see the references therein for other approaches, such as parametrization via neural networks and application of continuous optimization methods.

Another line of work, more closely related to Chapter~\ref{chap:WGF}, develops algorithms for computing the barycenter via gradient methods in the Wasserstein space~\cite{Alvetal16FixedPt, ZemPan19, CheMauRig20a, AltCheGer21, Bacetal22Bary, Kum+22Bary, BraRubTom24FixedPt}.
The descent lemma in Exercise~\ref{ex:bary_descent} is from~\cite{ZemPan19}, which interpreted the fixed-point approach of~\cite{Alvetal16FixedPt} as Wasserstein gradient descent.

Barycenters for Gaussians were studied earlier than the general case, dating back to~\cite{KnoSmi1994Cyclic, RusUck02nCoupling}.
Statistical estimation was studied in~\cite{KroSpoSuv21BWBary}, and non-asymptotic computational guarantees for Wasserstein gradient descent were given in~\cite{CheMauRig20a, AltCheGer21}.

Similarly to Chapter~\ref{chap:entropic}, one can add entropic regularization to the Wasserstein barycenter, at the level of the Wasserstein distance or the barycenter objective or both; see~\cite{Kro18BaryMK, BigCazPap19PenBary, Li+20RegBary, CarEicKro21EntBary, Chi23Doubly, VasChi23DoublyEntropic}.

\section{Exercises}

\begin{enumerate}
    \item Let $P$ be a distribution over $\cP_2(\R)$.
    Give a closed-form expression for the $W_2$ barycenter of $P$ in terms of the CDFs of the measures in $\supp P$.
    
    \item\label{ex:npc_bary} Suppose that $P$ is a probability measure over an \emph{NPC space} $(S,\md)$ with barycenter $\bs$. It turns out that statistical estimation of barycenters over NPC spaces is far easier, as we demonstrate in this exercise.
    \begin{enumerate}
        \item Show that for any $b\in S$,
        \begin{align}\label{eq:npc_var_ineq}
            P[\md^2(b,\bcdot) - \md^2(\bs, \bcdot)] \ge \md^2(b,\bs)\,.
        \end{align}
        \item Suppose that ${(X_i)}_{i=1}^n$ is an i.i.d.\ sequence drawn from $P$ and form the following estimator $b_n$ inductively: set $b_1 = X_1$, and for $n\ge 2$ let $b_n \deq \omega_{b_{n-1},X_n}(1/n)$ where $\omega_{x,y} : [0,1]\to S$ is the constant-speed geodesic joining $x$ to $y$.
        Prove by induction that for all $n\ge 1$,
        \begin{align*}
            \E \md^2(b_n, \bs)
            &\le \frac{\sigma^2}{n}\,, \qquad\text{where}~\sigma^2 = P \md^2(\bs, \bcdot)\,.
        \end{align*}
        \emph{Hint:} Apply the NPC inequality from Proposition~\ref{pro:NNC} together with the inequality~\eqref{eq:npc_var_ineq}.
    \end{enumerate}
    
    \item\label{ex:bary_multimarginal} Let $\mu_1,\dotsc,\mu_n \in \cP_2(\R^d)$ and let $\Gamma(\mu_1,\dotsc,\mu_n)$ denote the set of couplings of $\mu_1,\dotsc,\mu_n$.
    Consider the \emph{multi-marginal} optimal transport problem
    \begin{align*}
        \min_{\gamma \in \Gamma(\mu_1,\dotsc,\mu_n)} \int \sum_{i=1}^n {\bigl\lVert x_i - \frac{1}{n}\sum_{j=1}^n x_j\bigr\rVert^2} \,\gamma(\ud x_1,\dotsc, \ud x_n)\,.
    \end{align*}
    Let $\gamma^\star$ denote an optimal solution.
    Prove that if $(X_1,\dotsc,X_n) \sim \gamma^\star$, then the law of $\frac{1}{n}\sum_{i=1}^n X_i$ is the Wasserstein barycenter of $\mu_1,\dotsc,\mu_n$.
    
    \item\label{ex:w2_var_ineq} Due to Theorems~\ref{thm:extendvariance} and~\ref{thm:w2_extendable}, in the case of the Wasserstein space we know that as long as the transport maps from the barycenter $\bs$ to elements in the support of $P$ are obtained from $\alpha$-strongly convex potentials, \emph{and the barycenter of the extended distribution is still $\bs$}, then $Ph_{\bs}^b(\bcdot) \ge \alpha$.
    It turns out that due to the structure of the Wasserstein space, the second condition is unnecessary.

    To prove this, use the following dual characterization of the Wasserstein barycenter: for each $\mu\in\supp(P)$, $\varphi_\mu$ is such that $(\nabla \varphi_\mu)_\# \bs = \mu$, and $\int (\frac{\|\cdot\|^2}{2}-\varphi_\mu)\,P(\ud \mu) = 0$.
    Assume that each $\varphi_\mu$ is $\alpha(\mu)$-strongly convex.
    Use this to show that
    \begin{align*}
        \varphi_\mu^*(x) + \varphi_\mu(y) \ge \langle x,y\rangle + \frac{\alpha(\mu)}{2}\,\|y-\nabla\varphi_\mu^*(x)\|^2\,.
    \end{align*}
    By integrating this inequality, prove that~\eqref{eq:var_ineq} holds with $\frac{\lambda}{1+\lambda}$ replaced by $\int\alpha(\mu) \, P(\ud \mu)$.

    \item\label{ex:bary_descent} Let $P$ be a probability measure over $\cW_2$ and let $\cF : \cP_2(\R^d)\to\R$ denote the barycenter functional $\cF(b) \deq \frac{1}{2}\,PW_2^2(b,\bcdot)$.
    Using~\eqref{eq:grad_of_w2}, the Wasserstein gradient of $\cF$ is given by $\gradW \cF(b) = {\id} - PT_{b\to\bcdot}$, and a Wasserstein \emph{gradient descent} step with step size $h > 0$ is given by the iteration $b^+ \deq ({\id}-h\,\gradW\cF(b))_\# b$.
    Prove the \emph{descent lemma}
    \begin{align*}
        \cF(b^+) - \cF(b) \le -h\,\bigl(1-\frac{h}{2}\bigr)\,\|\gradW \cF(b)\|_b^2\,,
    \end{align*}
    which quantifies the progress made in one step of GD on $\cF$.
    Deduce that $h=1$ is a reasonable choice of step size and write out the form of the GD updates in this case.
    
    \item Specialize the GD updates (with step size $h=1$) from the previous exercise to the case when $P$ is supported on centered, non-degenerate Gaussians.
    In particular, when initialized at a centered Gaussian, show that all of the iterates are centered Gaussians, and write down the update equations for the covariance matrix.
\end{enumerate}

\appendix
\chapter{Convex analysis}
\label{app:convex}

In this appendix, we provide a quick review of convex analysis. We refer to the book~\cite{Roc1997CvxAnalysis} for a comprehensive treatment.

\section{Convex functions, subdifferentials, and duality}\label{sec:cvx_sub_dual}

\begin{definition}\label{def:cvx}
    A function $f : \R^d\to \R\cup\{\infty\}$ is \emph{convex} if for all $x,y\in\R^d$ and all $t\in [0,1]$,
    \begin{align*}
        f((1-t)\,x+t\,y) \le (1-t)\,f(x) +t\,f(y)\,.
    \end{align*}
    Also, a set $C \subseteq \R^d$ is \emph{convex} if for all $x,y\in\R^d$ and all $t\in [0,1]$,
    \begin{align*}
        (1-t)\,x + t\,y \in C\,.
    \end{align*}
\end{definition}

The \emph{domain} of $f$, $\dom(f)$, is the set $\{f < \infty\}$ of points where $f$ takes finite values.
If $f$ is convex, then $\dom(f)$ is a convex set.

We say that $f$ is \emph{proper} if it does not take the value $-\infty$ (note that this is already assumed in the definition of convexity given above) and it is not identically $+\infty$.
We assume without further mention that the convex functions we work with are proper.
We say that $f$ is \emph{closed} or \emph{lower semicontinuous} if for any sequence $x_k \to x$ in $\R^d$, it holds that $\liminf_{k\to\infty} f(x_k) \ge f(x)$; equivalently, all of the sublevel sets $\{f \le t\}$ for $t\in\R$ are closed.

Suppose that $f$ takes values in $\R$. Then, convexity of $f$ implies that $f$ is automatically continuous, and in fact locally Lipschitz, hence differentiable almost everywhere by Rademacher's theorem.
If $f$ is continuously differentiable, then convexity of $f$ is equivalent to $f$ always lying above its tangent line:
\begin{align}\label{eq:first_order_cvx}
    f(y) \ge f(x) + \langle \nabla f(x), y-x\rangle\,, \qquad \forall \,x,y\in\R^d\,.
\end{align}
If $f$ is twice continuously differentiable, then convexity of $f$ is equivalent to a condition on its Hessian:
\begin{align*}
    \nabla^2 f(x) \succeq 0\,, \qquad \forall\,x\in\R^d\,.
\end{align*}

In general, a convex function may not be differentiable. One reason why differentiability can fail is simply because $f$ takes on infinite values ($\dom(f) \ne \R^d$). However, as discussed above, $f$ is always differentiable \emph{almost} everywhere on the interior of its domain.
Moreover, we can find a useful substitute for differentiability through the notion of a subgradient, which is based on the ``above tangent line'' property encapsulated in~\eqref{eq:first_order_cvx}.

\begin{definition}[Subdifferential]\index{subdifferential}
    Let $f : \R^d\to\R\cup\{\infty\}$ be convex and let $x\in\R^d$.
    We say that $g$ is a \emph{subgradient} of $f$ at $x$ if
    \begin{align*}
        f(y) \ge f(x) + \langle g, y-x\rangle\,, \qquad\forall \,y\in\R^d\,.
    \end{align*}
    The set of all subgradients of $f$ at $x$ is called the \emph{subdifferential} of $f$ at $x$, denoted $\partial f(x)$.
    Also, $\partial f \deq \{(x, g) : x\in\R^d,\;g\in \partial f(x)\}$ is called the \emph{subdifferential} of $f$.
\end{definition}

Importantly, the following lemma holds:

\begin{lemma}
    If $f : \R^d\to\R\cup\{\infty\}$ is convex and $x$ lies in the interior of $\dom(f)$, then $\partial f(x)$ is non-empty.
    Also, if $f$ is differentiable at $x$, then the subdifferential at $x$ is single-valued and satisfies $\partial f(x) = \{\nabla f(x)\}$.
\end{lemma}

We next turn towards the crucial concept of duality.

\begin{definition}[Convex conjugate]\index{convex duality}
For any function $f:\R^d \to \R\cup\{\infty\}$, we define its convex conjugate\footnote{The convex conjugate is also known as the \emph{Fenchel--Legendre transform}, the \emph{Fenchel dual} or variations of these terms.} $f^*$ via
$$
f^*(y) \deq \sup_{x\in \R^d}\left\{\langle x, y\rangle-f(x)\right\}\,,\qquad y \in \R^d\,.
$$
\end{definition}

\begin{example}
    Let $A \succ 0$ be a positive definite matrix.
    Then, the convex conjugate of $x \mapsto \frac{1}{2}\,\langle x, A\,x\rangle$ is the function $y\mapsto \frac{1}{2}\,\langle y, A^{-1}\,y\rangle$.
    See Lemma~\ref{lem:quadratic_conjugate} for the proof.
    The reader is invited to write down other examples of convex functions and to compute their conjugates.
\end{example}

As a supremum of affine functions, the convex conjugate of $f$ is always a closed convex function, even if $f$ is not.
Conversely, if $f$ is closed and convex, then $f = f^{**}$.

The inequality $f(x) + f^*(y) \ge \langle x,y\rangle$ is trivial from the definition of the convex conjugate.
However, it is important enough to deserve a name, and we need the equality case for later use.

\begin{theorem}[Fenchel--Young inequality]\label{thm:fenchel_young}\index{Fenchel--Young inequality}
    For a convex function $f : \R^d\to\R\cup\{\infty\}$ and any $x,y\in\R^d$,
    \begin{align*}
        f(x) + f^*(y) \ge \langle x,y\rangle\,.
    \end{align*}
    Equality holds if and only if $y \in \partial f(x)$.
\end{theorem}

Note that by symmetry, equality holds if and only if $x \in \partial f^*(y)$.
In particular, when $f$ and $f^*$ are differentiable, then the subdifferentials are single-valued, so that the equality condition reads $y = \nabla f(x)$ and $x = \nabla f^*(y)$.
This says that the gradient mappings are inverse to each other: $\nabla f^* = {(\nabla f)}^{-1}$.

We conclude this section by proving Rockafellar's theorem (Theorem~\ref{thm:rock}), which characterizes subdifferentials of closed convex functions as maximally monotone subsets of $\R^d \times \R^d$. In Section~\ref{sec:grad_cvx_fn}, we show that if $\varphi : \R^d\to\R$ is convex, then its subdifferential $\partial\varphi$ is cyclically monotone.
Here, we prove the converse.

\begin{proof}[Proof of Theorem~\ref{thm:rock}]
Let $A$ be cyclically monotone and fix $(x_0, y_0) \in A$. Define for any $x\in \R^d$ the function
$$
\varphi(x)=\sup_{k \ge 0}\sup_{\substack{(x_i,y_i) \in A \\ i=1, \ldots, k}}\big\{ \langle x_1-x_0, y_0\rangle+\langle x_2-x_1, y_1\rangle   + \cdots + \langle x-x_k, y_k\rangle \big\}\,.
$$
Clearly $\varphi$ is closed and convex as a supremum of affine functions. Moreover, $\varphi(x_0)\le 0$ by cyclical monotonicity\footnote{This is in fact the only place we use cyclical monotonicity!} and $\varphi(x_0)\ge 0$ (take $k=1$ and $(x_1, y_1)=(x_0, y_0)$) so that $\varphi(x_0)=0$ and $\varphi$ is a \emph{proper} convex function. Finally note that for any $(x,y)=(x_{k+1}, y_{k+1}) \in A$ and any $z \in \R^\dd$, it holds
\begin{align*}
\varphi(z)&\ge\sup_{k \ge 0}\sup_{(x_i,y_i) \in A,\, i=1, \ldots, k}\big\{ \langle x_1-x_0, y_0\rangle+\langle x_2-x_1, y_1\rangle   + \cdots \\
&\qquad\qquad\qquad\qquad\qquad\qquad\qquad{} + \langle x-x_k, y_k\rangle + \langle z-x, y\rangle \big\}\\
&= \varphi(x) + \langle z-x, y\rangle\,.
\end{align*}
Therefore, $y \in \partial \varphi(x)$. 
\end{proof}

\section{Strong convexity and smoothness}\label{sec:strcvx_smooth}

\begin{definition}
    A function $f : \R^d\to\R\cup\{\infty\}$ is called \emph{$\alpha$-convex} (for $\alpha\in\R$) if for all $x,y\in\R^d$ and all $t\in [0,1]$,
    \begin{align*}
        f((1-t)\,x+t\,y) \le (1-t)\,f(x) + t\,f(y) - \frac{\alpha\,t\,(1-t)}{2}\,\|y-x\|^2\,.
    \end{align*}
\end{definition}

The case when $\alpha = 0$ corresponds to convexity, as in Definition~\ref{def:cvx}.
When $\alpha > 0$, then $f$ is called \emph{strongly convex}, and the above inequality strengthens the usual convexity inequality.
When $\alpha < 0$, then $f$ is called \emph{semi-convex}.

When $f$ is continuously differentiable, $\alpha$-convexity is equivalent to the first statement below.
When $f$ is twice continuously differentiable, $\alpha$-convexity is equivalent to both statements below.
\begin{enumerate}
    \item $f(y) \ge f(x) + \langle \nabla f(x), y-x\rangle + \frac{\alpha}{2}\,\|x-y\|^2$, for all $x,y\in\R^d$.
    \item $\nabla^2 f(x) \succeq \alpha I$ for all $x\in\R^d$.
\end{enumerate}

We also formulate the dual property of an upper bound on the second derivative.

\begin{definition}
    A continuously differentiable function $f : \R^d\to\R$ is called \emph{$\beta$-smooth} ($\beta \ge 0$) if for all $x,y\in\R^d$,
    \begin{align*}
        f(y) \le f(x) + \langle\nabla f(x), y-x\rangle + \frac{\beta}{2}\,\|y-x\|^2\,.
    \end{align*}
\end{definition}

When $f$ is twice continuously differentiable, then this is equivalent to $\nabla^2 f(x) \preceq \beta I$ for all $x\in\R^d$.
If, in addition, $f$ is convex, then $\nabla^2 f(x) \succeq 0$, so in particular the operator norm of $\nabla^2 f(x)$ is at most $\beta$. In turn, this is equivalent to the $\beta$-Lipschitzness of the mapping $\nabla f : \R^d\to\R^d$ (thus, convex and smooth functions are often referred to as \emph{gradient Lipschitz}).

The properties of strong convexity and smoothness are dual.
For simplicity, we state the following lemma assuming that $f$ is continuously differentiable, but the assumptions can be somewhat relaxed.

\begin{lemma}\label{thm:strcvx_smooth_dual}
    Let $f : \R^d\to\R$ be continuously differentiable, convex, and $\|\nabla f(x)\| \to\infty$ as $\|x\|\to\infty$.
    Let $\alpha > 0$.
    Then, $f$ is $\alpha$-strongly convex if and only if its convex conjugate $f^*$ is $\frac{1}{\alpha}$-smooth.
\end{lemma}
\begin{proof}
    From classical results in convex analysis (see~\cite[Theorem 25.5, Theorem 26.6, and Lemma 26.7]{Roc1997CvxAnalysis}), under our assumptions, $\nabla f : \R^d\to\R^d$ is a diffeomorphism with inverse $\nabla f^*$.
    
    $(\Rightarrow)$ By taking the first-order condition for strong convexity and adding it to the inequality with $x$ and $y$ interchanged, we obtain
    \begin{align*}
        \langle \nabla f(x) - \nabla f(y), x-y \rangle \ge \alpha\,\|x-y\|^2\,.
    \end{align*}
    Let $x = \nabla f^*(x')$ and $y = \nabla f^*(y')$ and recall also that $\nabla f \circ \nabla f^* = \id$.
    The above inequality yields, for all $x',y'\in\R^\dd$,
    \begin{align*}
        \langle x' - y', \nabla f^*(x') - \nabla f^*(y') \rangle
        &\ge \alpha \, \|\nabla f^*(x') - \nabla f^*(y')\|^2\,.
    \end{align*}
    Applying Cauchy{--}Schwarz to the left-hand side and rearranging, it follows that $\nabla f^*$ is $\frac{1}{\alpha}$-Lipschitz, which is equivalent to $f^*$ being $\frac{1}{\alpha}$-smooth, as discussed above.

    $(\Leftarrow)$ By smoothness of $f^*$,
    \begin{align*}
        f(y)
        &= \sup_{y'\in\R^\dd}\{\langle y,y'\rangle - f^*(y')\} \\
        &\ge \sup_{y'\in\R^\dd}\Bigl\{\langle y,y'\rangle - f^*(x') - \langle \nabla f^*(x'), y'-x'\rangle - \frac{1}{2\alpha}\,\|y'-x'\|^2\Bigr\} \\
        &= -f^*(x') + \langle y, x'\rangle + \frac{\alpha}{2}\,\|y-\nabla f^*(x')\|^2\,.
    \end{align*}
    Choose $x' = \nabla f(x')$ so that $\nabla f^*(x') = x$ and recall that $f(x) + f^*(x') = \langle x,x'\rangle$.
    It yields
    \begin{align*}
        f(y) - f(x) - \langle \nabla f(x), y-x\rangle
        &\ge \frac{\alpha}{2}\,\|y - x\|^2\,,
    \end{align*}
    completing the proof.
\end{proof}

Note that if $f$, $f^*$ are twice continuously differentiable, then $f$ is $\alpha$-strongly convex iff $\nabla^2 f \succeq \alpha I$, and $f^*$ is $\frac{1}{\alpha}$-smooth iff $\nabla^2 f^* = (\nabla^2 f){}^{-1}\circ \nabla f^* \preceq \alpha^{-1} I$, which provides a more transparent proof.

Strong convexity also implies the following property, which can be viewed as a strong quantitative form of the principle that locally optimal points are globally optimal under convexity.

\begin{definition}[Polyak--\L{}ojasiewicz inequality]\index{Polyak--\L{}ojasiewicz (P\L{}) inequality}
    We say that a continuously differentiable function $f : \R^d\to\R$ satisfies a \emph{Polyak--\L{}ojasiewicz (P\L{}) inequality} with constant $\alpha > 0$ if for all $x\in\R^d$,
    \begin{align*}
        \|\nabla f(x)\|^2 \ge 2\alpha\,(f(x) - \inf f)\,.
    \end{align*}
\end{definition}

\begin{lemma}\label{lem:strcvx_implies_pl}
    If $f : \R^d\to\R$ is continuously differentiable and $\alpha$-convex for $\alpha > 0$, then it satisfies a P\L{} inequality with constant $\alpha$.
\end{lemma}
\begin{proof}
    Let $x_\star$ denote the minimizer of $f$. Then,
    \begin{align*}
        \inf f
        = f(x_\star)
        \ge f(x) + \langle \nabla f(x), x_\star - x\rangle + \frac{\alpha}{2}\,\|x_\star - x\|^2\,.
    \end{align*}
    By Cauchy{--}Schwarz and Young's inequality,
    \begin{align*}
        \langle \nabla f(x), x_\star - x\rangle
        &\ge -\|\nabla f(x)\|\,\|x_\star - x\| \\
        &\ge -\frac{1}{2\alpha}\,\|\nabla f(x)\|^2 - \frac{\alpha}{2}\,\|x_\star - x\|^2\,.
    \end{align*}
    Substituting and rearranging finishes the proof.
\end{proof}

We also note that strong convexity implies quadratic growth around the minimizer.

\begin{lemma}\label{lem:strcvx_implies_growth}
    If $f : \R^d\to\R\cup\{\infty\}$ is $\alpha$-convex, and if $x_\star$ denotes the minimizer of $f$, then for all $x\in\R^d$,
    \begin{align}\label{eq:euclidean_quad_growth}
        f(x) - f(x_\star) \ge \frac{\alpha}{2}\,\|x-x_\star\|^2\,.
    \end{align}
\end{lemma}
\begin{proof}
    The strong convexity inequality gives
    \begin{align*}
        f((1-t)\,x_\star + t\,x) \le (1-t)\,f(x_\star) + t\,f(x) - \frac{\alpha\,t\,(1-t)}{2}\,\|x-x_\star\|^2\,,
    \end{align*}
    or
    \begin{align*}
        0
        &\le f((1-t)\,x_\star + t\,x) - f(x_\star) \\
        &\le t\,\Bigl[f(x) - f(x_\star) - \frac{\alpha\,(1-t)}{2}\,\|x-x_\star\|^2\Bigr]\,.
    \end{align*}
    Rearranging, dividing by $t$, and letting $t\searrow 0$ proves the result.
\end{proof}

Actually, in Exercise~\ref{ex:otto_villani} in Chapter~\ref{chap:WGF}, we refine this statement to show that the P\L{} inequality itself implies the growth inequality~\eqref{eq:euclidean_quad_growth}. We refer the reader to~\cite[Appendix~A]{KarNutSch16} for a concise exposition of the interplay between these inequalities.

\section{Convex conjugate of a quadratic function}

In this section we establish the following useful lemma which states that the convex conjugate of a quadratic function is an explicit quadratic function.

\begin{lemma}\label{lem:quadratic_conjugate}
	If $f(x) = \frac 12\, x^\top A x + b^\top x$ for $A \succ 0$, then
	\begin{equation}
		f^*(y) = \frac 12\, (y-b)^\top A^{-1} (y-b)\,.
	\end{equation}
	If $A$ is not invertible, the same expression holds if we interpret $A^{-1}(y-b)$ as the solution to $y = A x + b$ if it exists, and $f^*(y) = +\infty$ otherwise.
\end{lemma}
\begin{proof}
	The definition of $f^*(y)$ implies
	\begin{equation*}
		f^*(y) = \sup_{x\in\R^d}{\Bigl\{y^\top x -  \frac 12\, x^\top A x - b^\top x\Bigr\}}\,.
	\end{equation*}
	Differentiating the objective yields that if a maximizer $x^\star$ exists, then it satisfies
	\begin{equation*}
		y = A x^\star + b\,,
	\end{equation*}
	which yields $x^\star = A^{-1}(y- b)$ if $y - b \in \operatorname{span}(A)$.
	In this case, we obtain
	\begin{equation*}
		f^*(y) =  y^\top x^\star -  \frac 12\, (x^\star)^\top A x^\star - b^\top x^\star = \frac 12\, (y- b)^\top A^{-1} (y- b)\,.
	\end{equation*}
	On the other hand, if $y - b \not \in \operatorname{span}(A)$, then we can find a vector $z \in \operatorname{ker}(A)$ such that $z^\top (y - b) \neq 0$.
	Considering $x = \lambda z$ for $\lambda \in \R$, we obtain
	\begin{align*}
		f^*(y)
            &\geq \sup_{\lambda\in\R}{\Bigl\{y^\top (\lambda z) -  \frac 12 \,(\lambda z)^\top A (\lambda z) - b^\top (\lambda z)\Bigr\}} \\
            &= \sup_{\lambda\in\R} \lambda z^\top (y - b) = + \infty\,,
	\end{align*}
	as claimed.
\end{proof}

\chapter{Probability}
\label{app:proba}

In this appendix, we gather together some background material on probability theory.
See, e.g., the textbook~\cite{Bil99} for further discussion.

We begin with the notion of convergence of probability measures.

\begin{definition}\index{weak convergence}
A sequence of probability measures $(\mu_n)_n$ on $\R^d$ is said to \emph{converge (weakly)} to a probability measure $\mu$ if for all bounded continuous functions $f : \R^d\to\R$, it holds that
\begin{align*}
    \int f \, \ud \mu_n \to \int f \, \ud \mu\,.
\end{align*}
\end{definition}

For the topology of $\R^d$, we know exactly which subsets are compact: namely, a set $A$ is compact if and only if it is closed and bounded (Heine{--}Borel theorem). This provides a useful criterion for when a sequence $(x_n)_n$ in $A$ converges, upon passing to a subsequence, to a point in $A$.
The following definition and theorem characterize compact sets of probability measures in the topology of weak convergence.

\begin{definition}
    A set $\cA$ of probability measures on $\R^d$ is \emph{tight} if for all $\varepsilon > 0$, there is a compact set $K$ such that $\mu(K^\comp) \le \varepsilon$ for all $\mu \in \cA$.
\end{definition}

\begin{theorem}[Prokhorov's theorem]\index{Prokhorov's theorem}\label{thm:prokhorov}
Any weakly convergent sequence of probability measures is tight.
Conversely, any tight sequence of probability measures has a subsequential weak limit.

Equivalently, a set $\cA$ of probability measures on $\R^d$ is compact if and only if it is closed and tight.
\end{theorem}

We omit the proof, but the intuition can be gleaned via a simple example: on $\R$, let $\mu_n = \delta_{x_n}$ for all $n$, where $x_n\to\infty$; then, $(\mu_n)_n$ clearly has no weakly convergent subsequence.
The condition of tightness ensures that the mass does not run off to infinity, and once this is ensured then a weakly convergent subsequence is guaranteed.

The following theorem provides useful reformulations of weak convergence of measures.

\begin{theorem}[Portmanteau theorem]\label{thm:portmanteau}\index{portmanteau theorem}
    Let ${(\mu_n)}_{n}$ be a sequence in $\cP(\R^d)$ and let $\mu \in \cP(\R^d)$.
    The following are equivalent.
    \begin{enumerate}
        \item $\mu_n \to \mu$ weakly.
        \item $\int f \, \D \mu_n \to \int f \, \D \mu$ for all bounded Lipschitz continuous $f : \R^d \to \R$.
        \item $\int f \, \D \mu \le \liminf_{n\to\infty} \int f \, \D \mu_n$ for all lower semicontinuous functions $f : \R^d \to [0,\infty]$.
        \item $\int f \, \D \mu \ge \limsup_{n\to\infty} \int f \, \D \mu_n$ for all upper semicontinuous functions $f : \R^d \to [-\infty, 0]$.
        \item $\mu(G) \le \liminf_{n\to\infty} \mu_n(G)$ for all open $G \subseteq\R^d$.
        \item $\mu(F) \ge \limsup_{n\to\infty} \mu_n(F)$ for all closed $F \subseteq \R^d$.
        \item $\lim_{n\to\infty} \mu_n(A) = \mu(A)$ for all Borel $A \subseteq \R^d$ such that $\mu(\partial A) = 0$.
    \end{enumerate}
\end{theorem}
\begin{proof}
    (1) $\Rightarrow$ (2) is trivial.
    Also, it is easy to see that (3) is equivalent to (4) by replacing $f$ by $-f$, and that (5) is equivalent to (6) by taking complements.

    (2) $\Rightarrow$ (3): It is known that one can approximate $f$ from below by a sequence ${(f_k)}_{k}$ of Lipschitz continuous functions $\R^d\to [0,\infty)$.
    For any $k,n\in\N$, it holds that $\int f_k \, \D \mu_n \le \int f \, \D \mu_n$.
    Taking the limit $n\to\infty$, we get $\int f_k \, \D \mu \le \liminf_{n\to\infty} \int f \, \D \mu_n$.
    Then, take $k\to\infty$ using the monotone convergence theorem.

    (3) $\Rightarrow$ (5): The indicator function $\one_G$ is lower semicontinuous.
    Similarly, we obtain (4) $\Rightarrow$ (6) since the indicator function $\one_F$ is upper semicontinuous.

    (5) and (6) $\Rightarrow$ (7): The condition $\mu(\partial A) = 0$ means that $\mu(\interior A) = \mu(A) = \mu(\overline A)$.
    Applying (5) to the open set $\interior A$ and (6) to the closed set $\overline A$ proves (7).

    (7) $\Rightarrow$ (1): Let $f$ be bounded and continuous; we may as well assume $f$ is non-negative.
    Observe that for $t\in\R$, it holds that $\partial f^{-1}([t,\infty)) \subseteq f^{-1}(\{t\})$, and this set can have positive $\mu$-measure for at most countably many values of $t$.
    Applying (7), we obtain
    \begin{align*}
        \int f \, \D \mu_n
        &= \int_0^\infty \mu_n\{f \ge t\}\,\ud t
        \to \int_0^\infty \mu\{f \ge t\}\,\ud t
        = \int f \, \D \mu\,,
    \end{align*}
    where to justify the convergence we can use, e.g., bounded convergence (since the integral can actually be taken over a finite interval).
\end{proof}

The following lemma, as the name suggests, allows us to ``glue'' together couplings which share a marginal and is useful for some constructions in optimal transport (e.g., the proof of the triangle inequality for Wasserstein distances in Proposition~\ref{prop:wp_is_metric}).

\begin{lemma}[Gluing lemma]\label{lem:gluing}\index{gluing lemma}
Let $\gamma$ and $\gamma'$ be two measures on $\R^d\times \R^d$ such that for any Borel set $A\subset \R^d$, it holds $\gamma(\R^d\times A)=\gamma'(A \times\R^d$), i.e., the second marginal of $\gamma$ coincides with the first marginal of $\gamma'$. Then there exists three random variables $X,Y,Z \in \R^d$ such that
$(X,Z)\sim \gamma$ and $(Z,Y) \sim \gamma'$.
\end{lemma}
\begin{proof}
We are going to explicitly construct such a triplet $(X,Y,Z)$. To that end, let $Z$ be distributed according the  second marginal of $\gamma$ (which corresponds to the first marginal of $\gamma'$ by assumption).  Then $\gamma$ (resp.\ $\gamma'$) determines the conditional distribution of $X$ (resp.\ $Y$) given $Z$. For example, we may draw $X$ and $Y$ to be conditionally independent given $Z$. This gives a valid triplet $(X,Y,Z)$. \end{proof}

\begin{figure}
	\centering
	\begin{tikzcd}
		X \arrow[dash, "\text{opt}"]{r} & Z \arrow[dash, "\text{opt}"]{r} & Y
	\end{tikzcd}
	\caption{The diagram above is a convenient way to represent couplings between multiple random variables. An edge represent the constraint that the coupling needs to be optimal. In general, any coupling described as a graph with no cycle can be realized using the gluing lemma.}\label{fig:diagcouplings}
\end{figure}

\backmatter

\bibliographystyle{aomalpha}
\bibliography{rigolletmain,extra}
\printindex

\end{document}